\documentclass[10pt,reqno,fleqn]{amsart}

\usepackage[utf8]{inputenc}		
\usepackage[english]{babel}		

\usepackage{amsmath}				
\usepackage{mathtools}			
\usepackage{amsthm}				
\usepackage{amssymb}				
\usepackage[nointlimits]{esint}	
\usepackage{mathrsfs} 			

\usepackage{graphicx}			
\usepackage{epsfig}				
\usepackage{tikz}				
\usepackage{xxcolor}				
\usepackage{wrapfig}
\usepackage{caption}

\usepackage{enumitem}			
\usepackage{ifdraft} 			

\usepackage[]{todonotes}			

\usepackage[linktocpage=true,colorlinks=true,linkcolor=blue,citecolor=green]{hyperref}
\usepackage[
	backend=biber,style=alphabetic, natbib=false, mcite=false, casechanger=auto,
	sorting=nyt, sortcites=false, pluralothers=true, maxnames=4, minnames=1,
	backref=false, arxiv=abs, 
	date=year,
	isbn=false, url=false, doi=true, eprint=true, related=false,
	giveninits=true
]{biblatex}
\renewbibmacro{in:}{}	

\usepackage[hmarginratio=1:1,top=32mm,columnsep=20pt]{geometry} 

\numberwithin{equation}{section}				

\theoremstyle{plain}							
\newtheorem{theorem}{Theorem}[section]		
\newtheorem*{theorem*}{Theorem}				
\newtheorem*{maintheorem*}{Main Result}
\newtheorem{lemma}[theorem]{Lemma}			
\newtheorem{proposition}[theorem]{Proposition}
\newtheorem{corollary}[theorem]{Corollary}

\theoremstyle{definition}					
\newtheorem{definition}[theorem]{Definition}	
\newtheorem{example}[theorem]{Example}

\theoremstyle{remark}						
\newtheorem{remark}[theorem]{Remark}

\theoremstyle{definition}					
\providecommand{\customtheoremname}{}			
\newtheorem{placeholderthhm}{\customtheoremname}
\newcommand{\newcustomtheorem}[2]{%
  \newenvironment{#1}[1]						
  {											
   \renewcommand\customtheroremname{#2}
   \renewcommand\theplaceholderthhm{##1}	
   \placeholderthhm							
  }											
  {\endplaceholderthhm}						
}
\newcustomtheorem{customass}{Assumption}		
\newcustomtheorem{customdef}{Definition}

\DeclarePairedDelimiter{\Vpair}{\|}{\|} 			
\DeclarePairedDelimiter{\vpair}{|}{|} 			
\DeclarePairedDelimiter{\bpair}{[}{]} 			
\DeclarePairedDelimiter{\ppair}{(}{)} 			
\DeclarePairedDelimiter{\Bpair}{\{}{\}} 			
\DeclarePairedDelimiterX{\cond}[2]{[}{]}{#1  \,\delimsize\vert\,  \mathopen{}#2}

\newcommand\mathvshift[2]{\raisebox{#1}{$\displaystyle{#2}$}}
\newcommand\strutb{\vphantom{\bigm|}}

\newcommand\norminner[2][0pt]{\mathchoice{\mathvshift{#1}{\strutb\mathopen{}#2}}{\smash{#2}}{#2}{#2}}
\newcommand\bracketinner[2][0pt]{\mathchoice{\mathvshift{#1}{\strutb\mathopen{}#2}}{#2}{#2}{#2}}
\ifoptiondraft{
	\newcommand\norm[2][0pt]{\Vpair*{#2}} 	
	\newcommand\abs[2][0pt]{\vpair*{#2}} 		
	\newcommand\pp[2][0pt]{\ppair*{#2}} 		
	\newcommand\bp[2][0pt]{\bpair*{#2}} 		
	\newcommand\Bp[2][0pt]{\Bpair*{#2}} 		
}{
	\newcommand\norm[2][0pt]{\Vpair*{\norminner[#1]{#2}}}	 	
	\newcommand\abs[2][0pt]{\vpair*{\norminner[#1]{#2}}} 		
	\newcommand\pp[2][0pt]{\ppair*{\bracketinner[#1]{#2}}} 	
	\newcommand\bp[2][0pt]{\bpair*{\bracketinner[#1]{#2}}} 	
	\newcommand\Bp[2][0pt]{\Bpair*{\bracketinner[#1]{#2}}}		
}
\newcommand\delimstyle[2]{\setlength{\delimitershortfall}{#1}\delimiterfactor=#2}
\delimstyle{5pt}{850}

\newcommand{\ov}{\overline}
\newcommand{\wt}{\widetilde}
\newcommand{\eps}{\varepsilon}

\newcommand{\vphi}{\varphi}
\newcommand{\vtheta}{\vartheta}

\newcommand{\data}{\mathrm{data}}

\newcommand{\N}{\mathbb{N}}
\newcommand{\Nz}{\mathbb{N}_0}
\newcommand{\Z}{\mathbb{Z}}

\newcommand{\R}{\mathbb{R}}

\newcommand{\Zd}{{\Z^d}}
\newcommand{\Rd}{{\R^{d}}}
\newcommand{\Rdd}{{\R^{d\times d}}}
\newcommand{\Rddsym}{{\R^{d\times d}_{\mathrm{sym}}}}
\newcommand{\Rt}{{\R^{2}}}

\newcommand{\Sone}{{\mathbb{S}^1}}
\newcommand{\Sd}{{\mathbb{S}^{d-1}}}
\newcommand{\Ut}{\mathcal{O}}				
\newcommand{\hsp}[1]{{\mathcal{H}_{#1}^+}}	
\newcommand{\hsm}[1]{{\mathcal{H}_{#1}^-}}	
\newcommand{\hsma}[4]{{S_{#1,#2,#3}^-\pp{#4}}}			

\newcommand\sorder[1]{o\pp{#1}}	

\newcommand{\pat}{\partial_t}

\newcommand{\m}{\,\mathrm{d}}				

\newcommand{\diam}{\operatorname{diam}}
\newcommand{\dist}{\operatorname{dist}}

\newcommand{\charfun}[1][_]{{\if_#1 \chi \else \chi_{#1} \fi}}			
\newcommand{\comp}[1]{{{#1}^c}}										
\newcommand{\Id}{\operatorname{Id}}
\newcommand{\Idd}{\operatorname{Id}_{d\times d}}

\newcommand{\tr}{\operatorname{tr}}		
\newcommand{\supp}{\operatorname{supp}}

\newcommand{\Ev}{\mathbb{E}}												
\newcommand{\EV}[2][_]{ {\if_#1 \Ev\bp{#2} \else \Ev\cond*{#2}{#1} \fi} }	
\newcommand{\Pm}{{\mathbb{P}}}											
\newcommand{\PM}[2][_]{ {\if_#1 \Pm\bp{#2} \else \Pm\cond*{#2}{#1} \fi} }	
\newcommand{\F}{\mathcal{F}}												
\newcommand{\G}{\mathcal{G}}												

\newcommand{\vmin}{v_{\mathrm{min}}}
\newcommand{\vmineff}{v_{\mathrm{min}}}
\newcommand{\vmax}[1][_]{{\if_#1 v_{\mathrm{max}} \else v_{\mathrm{max},#1} \fi}}

\newcommand{\restcomp}{Q^{\oslash}}		

\newcommand{\Eup}[2][_]{{E^{\if_#1 \scalebox{0.6}{$\uparrow$} \else \scalebox{0.6}{$\uparrow$},#1 \fi}_{#2}}}

\newcommand{\Qup}[1]{{Q^{\scalebox{0.5}{$\uparrow$}}_{#1}}}

\newcommand{\Espeed}[2][_]{{E^{\if_#1 \scalebox{0.65}{$\circledcirc$} \else \scalebox{0.65}{$\circledcirc$},#1 \fi}_{#2}}}
\newcommand{\ratespeed}{{\alpha_{\scalebox{0.5}{$\circledcirc$}}}}
\newcommand{\Cspeed}{{C_{\scalebox{0.5}{$\circledcirc$}}}}
\newcommand{\cspeed}{{c_{\scalebox{0.5}{$\circledcirc$}}}}

\newcommand{\supJ}[1][]{\mathcal{J}^{2,+}_{#1}}			
\newcommand{\subJ}[1][]{\mathcal{J}^{2,-}_{#1}}			
\newcommand{\supP}[1][]{\mathcal{P}^{2,+}_{#1}}			
\newcommand{\supPc}[1][]{\ov{\mathcal{P}}{}^{2,+}_{#1}}	
\newcommand{\subP}[1][]{\mathcal{P}^{2,-}_{#1}}			
\newcommand{\subPc}[1][]{\ov{\mathcal{P}}{}^{2,-}_{#1}}	
\newcommand{\supPt}[1][]{\wt{\mathcal{P}}{}^{2,+}_{#1}}	
\newcommand{\subPt}[1][]{\wt{\mathcal{P}}{}^{2,-}_{#1}}	

\newcommand{\reach}[3][]{{{\mathscr{R}}_{#2}^{#1}\pp{#3}}}
\newcommand{\met}[4][]{{m^{#1}_{#2}\pp{#3,#4}}}			
\newcommand{\metb}[4][]{{m^{#1}_{#2}\pp{#3,#4}}}	
\newcommand{\uls}[2][_]{\if_#1 u_{#2} \else u_{#2, #1} \fi}	
\newcommand{\ur}[2][_]{\if_#1 u_{#2}^{\mathscr{R}} \else u_{#2, #1}^{\mathscr{R}} \fi}		

\newcommand{\hsubs}[1][_]{\if_#1 \subset_h \else \subset_{#1} \fi}
\newcommand{\hsups}[1][_]{\if_#1 \supset_h \else \supset_{#1} \fi}

\newcommand{\Cmeas}{c_{\F}}
\newcommand{\cstable}{c_s}

\newcommand{\M}{\mathcal{M}}						
\newcommand{\MA}{X}								
\newcommand{\delM}{Z}					
\newcommand{\timebound}[3]{T_{#1}\pp{#2,#3}}	
\newcommand{\maxM}[4][]{\if_#1 M \else M_{#2}\pp{#3,#4}\fi}		
\newcommand{\Cflu}{C_{\mathrm{fl}}}	
\newcommand{\cflu}{c_{\mathrm{fl}}}
\newcommand{\Cfluco}{C_{\mathrm{co}}}	
\newcommand{\CG}[1][]{C_{\G}}	
\newcommand{\Ccond}{C_K}

\newcommand{\Tfull}[1][]{{T^{#1}_{\mathrm{full}}}}		
\newcommand{\Tfirst}[1][]{{T^{#1}_{\mathrm{first}}}}		
\newcommand{\delF}{{\Delta F}}							
\newcommand{\delFdel}{{\delta_F}}							
\newcommand{\vhom}[1][]{\ov{v}_{#1}}
\newcommand{\limm}[4][]{\ov{m}^{#1}_{#2,#3}\pp{#4}}	
\newcommand{\genhoel}{{\beta^\ast}}
\newcommand{\hoel}{\eta}
\newcommand{\Choel}{{C^\ast}}

\graphicspath{ {./MCF_fluct-hom_images/} }
\counterwithin{figure}{section}

\begin{document}

\title[Homogenization of geometric motions through obstacles]{Quantitative homogenization of forced geometric motions through random fields of obstacles}

\author{Julian Fischer}
\address{Institute of Science and Technology Austria (ISTA), Am~Campus~1, 
3400 Klosterneuburg, Austria}
\email{julian.fischer@ista.ac.at}
\author{Jonas Ingmanns}
\address{Institute of Science and Technology Austria (ISTA), Am~Campus~1, 
3400 Klosterneuburg, Austria}
\email{jonas.ingmanns@ista.ac.at}

\begin{abstract}
We establish a quantitative homogenization result for an interface moving through a field of sufficiently sparse but possibly impenetrable random obstacles. From a physical viewpoint, such problems arise e.\,g.\ in the context of the motion of dislocations or magnetic domain walls in a material with impurities.

More precisely, given an interface moving by forced mean curvature flow -- with a positive global driving force plus a spatially fluctuating (negative) driving force modeling the obstacles -- , we prove that the effective large-scale behavior of the forward front is governed by a constant-speed effective motion. For typical values of the global forcing, on large scales of the order $\varepsilon^{-1}$ we obtain a (relative) error estimate for arrival times of the front of the order $\varepsilon^{1/9-}$.

Previous stochastic homogenization results for forced mean curvature motion in the literature have required a positive pointwise lower bound on the combined forcing, which implies the absence of any actual obstacles capable of locally blocking the interface motion. In contrast, our results are valid even in the presence of islands with locally negative forcing, potentially allowing for locally pinned interfaces and eventually enclosures left behind the main front. Thus, our homogenization result applies to settings closer to (but still strictly away from) the pinning-depinning transition.
\end{abstract}

\thanks{
This project has received funding from the European Research Council (ERC) under the European Union's Horizon 2020 research and innovation programme (grant agreement No 948819)
\smash{
\begin{tabular}{@{}c@{}}\includegraphics[width=6ex]{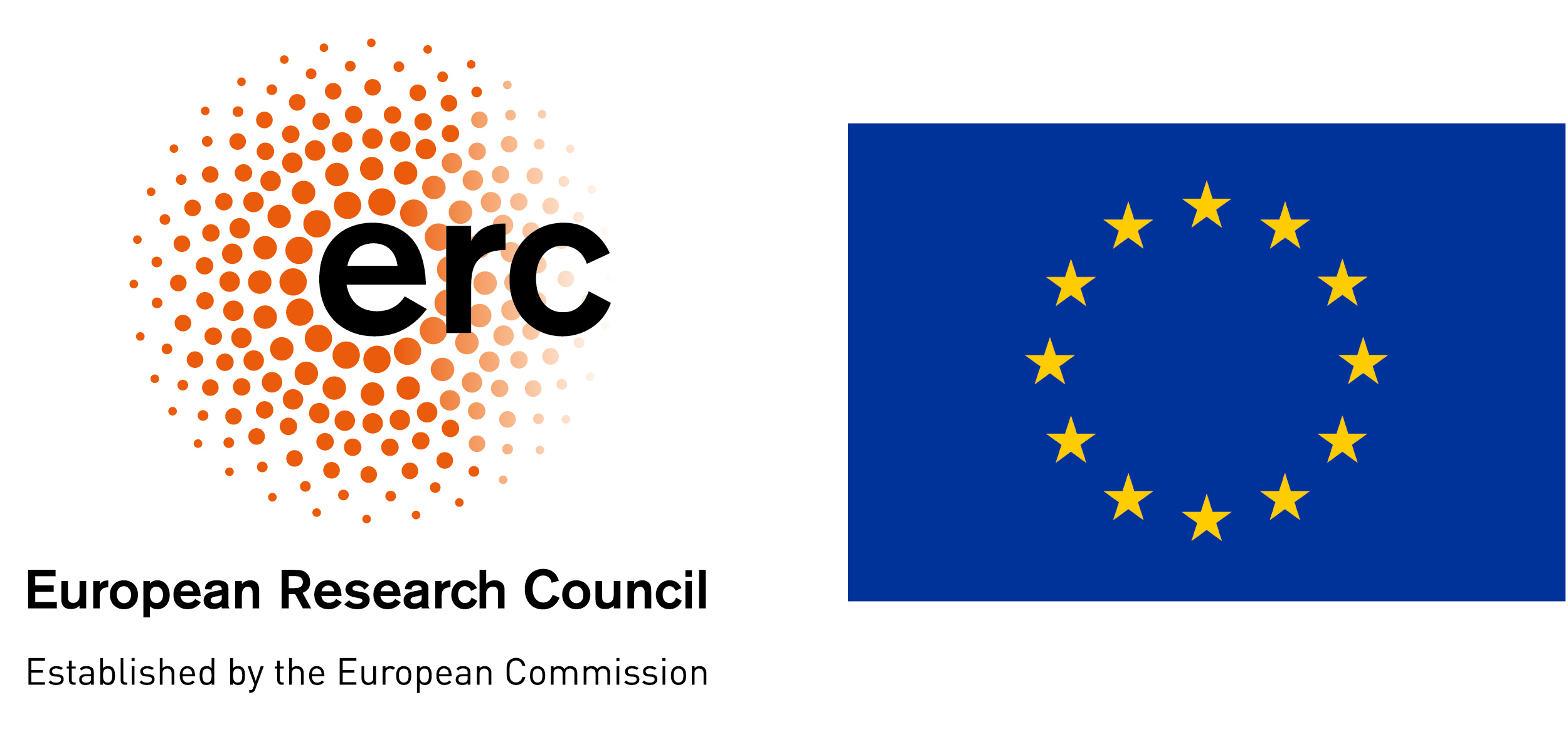}\end{tabular}
}
}

\maketitle
{\small{\bfseries \href{https://mathscinet.ams.org/mathscinet/msc/msc2020.html}{MSC2020 subject classification}:} 35B27, 60K35, 53E10, 35R60, 35F21}


\setcounter{tocdepth}{1}		
\tableofcontents
\setcounter{tocdepth}{3}		


\section{Introduction}
In the physics of materials, lower dimensional features in the material -- like dislocation lines, magnetic domain walls, or grain boundaries -- are central to the quantitative understanding of many phenomena, like e.\,g.\ plasticity, ferromagnetism, or fracture.
The motion of lines and interfaces in materials is often governed by the combination of a global force, driving the line or interface in a certain direction, with the general preference of minimization of the interface energy. In the simplest setting, the corresponding evolution is thus given by a forced mean curvature motion \cite{IoVi87, Kar98, BraNa04}. In a material with impurities, these impurities may act as obstacles to the evolution of the line or interface, giving rise to a complex range of expected phenomena, such as pinning \cite{BoTe14,CDO22,DiYi06,DDS11}, de-pinning \cite{BoTe14,CDL10,CDO22,DiYi06,DoSch12}, and possibly homogenization \cite{ArCa18,CaMo14,DKY08,DiYi06,Fel19,LiSo05a}.

In the present work, we consider the motion of an interface driven by curvature and a uniform forcing through a field of sufficiently sparse but possibly impenetrable random obstacles. We prove that this motion can be approximated on large scales by a homogeneous first-order motion, i.\,e., in the case of isotropic media, by a constant speed motion; see Figure \ref{intro_fig_hom} for an illustration.

More precisely, we treat the case of an interface motion with the normal velocity $V$ at $x\in\Rd$ given in terms of the normal $\nu\in\Sd$ and the mean curvature $H$ by
\begin{equation}\label{intro_eq_normal}
	V= H + F\pp{x, \nu},
\end{equation}
with the random forcing field $F(x,\nu)=F_{uni}(\nu)-F_{obst}(x,\nu)$ combining the uniform positive forcing and the obstacles.
On large scales, that is zooming out by $\eps^{-1}$ and accelerating time by $\eps^{-1}$ for some small $\eps>0$, the interface motion is given by
\begin{equation}\label{intro_eq_SurfM_normal}
	V= \eps H + F\pp{\frac{x}{\eps}, \nu}.
\end{equation}
For small $\eps>0$ and under suitable conditions on the obstacle distribution, one expects that this motion can be approximated by a first-order evolution with a homogeneous speed, that is with
\begin{equation}\label{intro_eq_HomM_normal}
	V= \vhom(\nu)
\end{equation}
for a suitable function $\vhom\colon\Sone\rightarrow\R_{>0}$ capturing the effective speed in each direction.
For our result, the effective speed $\vhom$ will be independent of the direction since we assume isotropy of the random obstacle distribution, that is that the law of $F(\cdot)$ is invariant with respect to rotations.
Under suitable further assumptions on the probability distribution of $F$ -- assumptions that in particular cover the case of sufficiently sparse but possibly impenetrable random obstacles $F_{obst}$ and generic positive forcings $F_{uni}$ -- , we provide a quantitative error estimate for this approximation of the order of $\varepsilon^{1/9-}$.

Previous stochastic homogenization results in the literature for forced mean curvature flow \eqref{intro_eq_SurfM_normal} and similar PDEs had been limited to strictly positive forcings $F>0$:
The first homogenization result for \eqref{intro_eq_SurfM_normal} in the random setting was obtained by Armstrong and Cardaliaguet \cite{ArCa18}, who required that $\inf_x F^2-(d-1)|\nabla F|>0$. In $d=2$ dimensions, Feldman \cite{Fel19} weakened this positivity assumption to $\inf_x F>0$.
Analogous results in the periodic setting had previously been derived by Lions and Souganidis \cite{LiSo05a} and by Caffarelli and Monneau \cite{CaMo14} respectively.
In all of these cases, a consequence of the uniform lower bound on the forcing is the existence of a minimum speed for the expansion of sufficiently regular sets, and thus the absence of any impenetrable obstacles.

\begin{figure}[h]
\centering
\includegraphics[width=0.4\textwidth]{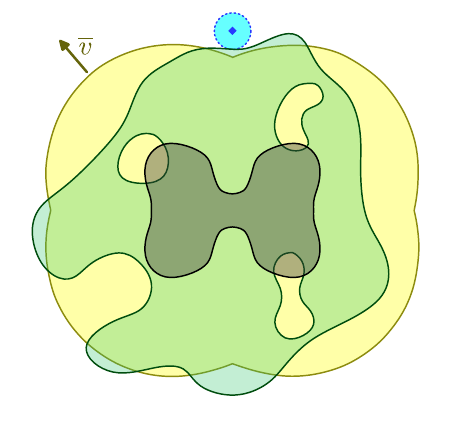}
$\quad$
\includegraphics[width=0.4\textwidth]{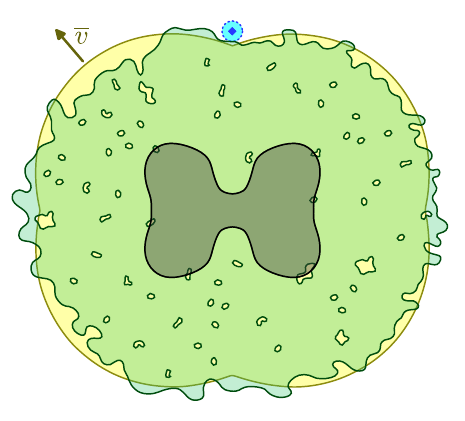}
\caption{	A comparison of the evolution by forced mean curvature flow in a field of random obstacles with the evolution according to a homogeneous first-order motion. 
			The initial set is depicted in grey, the evolution by forced mean curvature flow in green, and the evolution by the first-order motion in yellow.
			The target area for which the arrival times are compared is depicted in blue.
			Left: The length scale of the initial set is of the same order as the range of dependence.
			Right: The length scale of the initial set is much larger than the range of dependence, and homogenization is observed.}
\label{intro_fig_hom}
\end{figure}

In the present paper, we allow for potentially negative forcing and hence actual impenetrable obstacles. 
This means that locally there is a non-zero probability that the interface gets stuck.
In particular, some locations might never be reached.
Hence, a pointwise comparison between the evolutions \eqref{intro_eq_SurfM_normal} and \eqref{intro_eq_HomM_normal} is not possible. 
Instead, we compare the interface motions \eqref{intro_eq_SurfM_normal} and \eqref{intro_eq_HomM_normal} with respect to arrival times of interfaces at some $x_0\in\Rd$.
In fact, for \eqref{intro_eq_SurfM_normal} we consider the arrival time in a neighborhood $\ov{B}_{h_\eps}(x_0)$ to ensure that the probability of the arrival time being infinity converges to $0$ as $\eps\rightarrow 0$.
In addition, we truncate the arrival time to have a well defined expected value.
The quantitative error estimate then consists of two parts:
\begin{itemize}
\item	For the fluctuations of the truncated arrival time for \eqref{intro_eq_SurfM_normal} in $\ov{B}_{h_\eps}(x_0)$ we obtain a bound of order $\eps^{\frac{1}{2}-}$ with essentially exponential tail bounds.
\item 	For the bias, the difference between the expected value and the arrival time for the homogenized evolution \eqref{intro_eq_HomM_normal}, we obtain a bound of order $\eps^{\frac{1}{9}-}$.
\end{itemize}

Briefly summarizing our setting and assumptions, the coefficient field $F$ is sampled from a stationary, isotropic probability measure $\Pm$ satisfying a finite range of dependence assumption. 
Regarding the regularity of the coefficient field, $F\in C^{0,1}(\Rd\times\Sd)$ is assumed to be uniformly bounded and uniformly Lipschitz.

Moreover, there are two less straightforward assumptions:
First, we assume that actual obstacles are sufficiently sparse.
That is, we require that for any sufficiently large and regular set with overwhelming probability there exists an interior approximation, which does not shrink under \eqref{intro_eq_normal} and expands at an `effective minimum speed' $\vmineff>0$ on large scales.
Here, expanding at an `effective minimum speed' $\vmineff>0$ on a fixed scale means that at any time $t\geq 0$ the evolved set grows at speed $\vmineff$ up to left-behind holes of size below this scale, see Figure \ref{sett_fig_vmin}.
In fact, this is a generalization of the assumptions on the forcing from \cite{Fel19} and \cite{ArCa18}, since their assumptions guarantee that the original sets themselves are non-shrinking and expand with an actual minimum speed.

Second, we require that our candidate for the homogenized speed $\vhom$ is H\"older continuous at $\delF=0$ with respect to changing the forcing by a constant from $F$ to $F+\delF$.
We will show that this assumption holds generically in the sense that $\delF\mapsto\vhom[\delF]$ is increasing and hence Lipschitz continuous at almost every $\delF$.

In order to illustrate our result, we first state it in the special case of sufficiently sparse Poisson-distributed obstacles of bounded diameter in dimension $d=2$ as in \cite{DDS11, CDO22}. 
In the companion paper \cite{FI26} we will rigorously prove that this setting satisfies the assumptions for the present paper's main result.
We further expect that the same holds in arbitrary dimension $d\geq 3$ and that these arguments should also be applicable to essentially Gaussian obstacle fields,  where $F_{obst}$ arises from suitable local modifications of a Gaussian random field with finite range of dependence and where $F_{uni}>0$ is sufficiently large.

\subsection{An application of the main result: Poisson-distributed obstacles}
Before we discuss the informal version of our main result and the technical notions necessary to state it, we present this application for the special case of potentially impenetrable obstacles distributed according to a Poisson point process  in dimension $d=2$ and a uniform forcing which is sufficiently large compared to the intensity of the point process.
In the companion paper \cite{FI26} we will show that the requirements for the application of the results in the present paper are satisfied.
We expect that an analogous result holds in dimension $d\geq 3$ up to a different scaling of the lower bound for the uniform forcing.

\begin{theorem}[Homogenization for Poisson-distributed obstacles]\label{intro_thm_Poisson}
Consider Poisson-distributed obstacles in dimension $d=2$ represented by 
\begin{equation*}
	F_{obst}\colon\R^2\rightarrow \R,\quad F_{obst}(x) \coloneqq \max_{i\in\N}\vphi(\abs{x-X_i})
\end{equation*}
for
\begin{itemize}
\item	a Poisson point process $\Bp{X_i}_{i\in\N}\subset\Rt$ with intensity $\rho>0$, and
\item	a Lipschitz continuous obstacle shape $\vphi\in C^{0,1}(\R)$ with support $\operatorname{\supp}\vphi\subset\ov{B}_1$.
\end{itemize}
There exist $\rho_0>0$ and $C_0>0$ independent of the obstacle shape $\vphi$ such that -- given any intensity $\rho\in(0, \rho_0]$ -- for almost every $F_{uni}\geq C_0\sqrt{\rho}$ the interface motion given by
\begin{equation}\label{intro_eq_Poisson-motion}
	\mathrm{Normal~velocity} = \mathrm{mean~curvature} + F_{uni}-F_{obst}(x) 
	\qquad\text{at $x\in \R^2$}
\end{equation}
is approximated by a first-order motion with constant positive speed $\vhom=\vhom(\rho,\vphi, F_{uni}) >0$ in the following sense.

Let $x_0\in\R^2$ and let $S\subset\R^2$ be a union of balls with radius $r>0$.
For $\eps>0$, let $(\reach{t}{\eps^{-1}S})_{t\geq 0}$ be the set evolution of $\eps^{-1}S$ corresponding to the interface motion \eqref{intro_eq_Poisson-motion}, see Definition \ref{not_def_reach} below.
Let $\vtheta\in(0,1)$.
Then the arrival time of the set evolution in the target area $\ov{B}_{\eps^{-\vtheta}}(\eps^{-1}x_0)$, given by
\begin{equation*}
	\met{\eps^{-\vtheta}}{\eps^{-1}x_0}{\eps^{-1}S} \coloneqq \min\Bp{t\geq 0\,:\,\ov{B}_{\eps^{-\vtheta}}(\eps^{-1}x_0) \cap\reach {t}{\eps^{-1}S}\neq \emptyset},
\end{equation*}
is described well by a constant-speed motion in the sense that
\begin{equation*}
	\PM{\abs{ \eps \met{\eps^{-\vtheta}}{\eps^{-1}x_0}{\eps^{-1}S} - \vhom^{-1}\dist(x_0,S)} \geq C\eps^{\frac{1-\vtheta}{9}}}
	\leq C \exp\pp{-C^{-1}\eps^{-\alpha\vtheta}}
\end{equation*}
for the probability measure $\Pm$ associated to the Poisson point process and for some $\alpha>0$ and some $C=C(\rho,\norm{\vphi}_{C^{0,1}}, F_{uni}, r, \dist(x_0,S), \vtheta)>0$.
\end{theorem}

Since $\rho_0$ and $C_0$ are independent of the obstacle shape, the shape can in particular be chosen such that the obstacles are impenetrable for a large range of uniform forcings $F_{uni}\geq C_0\sqrt{\rho}$.
The scaling $\sqrt{\rho}$ of the lower bound for the uniform forcing corresponds to the Taylor scaling of the pinning to depinning transition (see Subsection \ref{subs_lit_pinn}) established by Courte, Dondl and Ortiz in \cite{CDO22}.
Note however that the assumption of an `effective minimum speed' in our main result below is stronger than the existence of a ballistic evolution proven in \cite{CDO22}.
For a brief discussion of the companion paper \cite{FI26} and its application to Poisson distributed obstacles see Example \ref{ex_ex_box} and Example \ref{veff_ex_Poisson}.

\subsection{Set evolutions and arrival times}
Next, we want to state the informal version of our main result, comparing the arrival times for the heterogeneous and homogenized geometric motions.
In order to state this informal version, we first define these arrival times at points $x_0\in\Rd$ or in target areas $\ov{B}_h(x_0)\in\Rd$.

We view the geometric motions as set evolutions with \eqref{intro_eq_normal}, \eqref{intro_eq_SurfM_normal} and \eqref{intro_eq_HomM_normal} describing the speed at which the interface moves.
To obtain well-defined set evolutions, we work with level-set formulations in the framework of viscosity solutions, see \cite{Giga06} for a self-contained account of the theory and Appendix \ref{s_visc} for a summary of the definitions and results used by us.
In this framework, a set evolves according to for example \eqref{intro_eq_SurfM_normal} if it corresponds to the sublevel set of a function $u^\eps$ which in the case of \eqref{intro_eq_SurfM_normal} is a viscosity solution to
\begin{equation}\label{intro_eq_SurfM_ls}
		\pat u^\eps - \eps \tr\pp{A\pp{\frac{\nabla u^\eps}{\abs{\nabla u^\eps}}}D^2 u^\eps} 
		+ F\pp{\frac{x}{\eps}, \frac{\nabla u^\eps}{\abs{\nabla u^\eps}}}\abs{\nabla u^\eps}
		=0
\end{equation}
in $\Rd\times(0,\infty)$ with $A(e)=\Id-e\otimes e$ for $x\in\Rd$, $e\in\Sd$.
More precisely, we work with the following notion of set evolutions.

For some partial results, we allow $A(\cdot,\cdot)$ to take a more general form and in particular to be space dependent, see Remark \ref{intro_rem_aniso} below.
This is why we state the definition of set evolutions for a larger class of coefficient fields corresponding to geometric motions.

\begin{definition}[The set evolution in the heterogeneous environment]\label{not_def_reach}
Let $(A,F)\in\Omega$, with $\Omega$ denoting the admissible coefficient fields as specified in Section \ref{subs_sett_AF}.
Given a closed set $S\subset\Rd$, we define the evolution $\pp{\reach{t}{S}}_{t\geq 0}$ of $S$ with respect to $(A,F)$ as
\begin{equation*}
	\reach{t}{S}=\reach[(A,F)]{t}{S}\coloneqq \Bp{x\in\Rd \,:\, \uls{S}(x,t)\leq 0},
\end{equation*}
where $\uls{S}$ with continuous initial data $\uls{S}(\cdot,0)$ satisfying $S=\Bp{x\in\Rd \,:\, \uls{S}(x,0)\leq 0}$ is a (viscosity) solution of 
\begin{align}
	&\pat u + G_{(A,F)}(x,\nabla u, D^2 u)=0 
	\quad\,\text{in }\Rd\times(0,\infty)\label{intro_eq_uls}\\
	&	\qquad\qquad\quad \text{with }\quad
	G_{(A,F)}(x,\xi,X) \coloneqq - \tr\pp{A\pp{x, \frac{\xi}{\abs{\xi}}}X} + F\pp{x, \frac{\xi}{\abs{\xi}}}\abs{\xi}\label{intro_eq_GAF}
\end{align}
for $(x,\xi,X)\in \Rd\times\pp{\Rd\!\setminus\!\Bp{0}}\times \Rddsym$.
Such a solution always exists and $\reach{t}{S}$ is independent of the specific initial data $\uls{S}(\cdot,0)$.
Further, we can assume that $\uls{S}$ is uniformly continuous on $\Rd\times(0,T)$ for any $T>0$, see Theorem \ref{comp_thm_exis}, Remark \ref{comp_rem_iniData}, Theorem \ref{comp_thm_unique}.
\end{definition}

\begin{remark}[Hyperbolic scaling regime]
The scaling of $\eps\rightarrow 0$ in \eqref{intro_eq_SurfM_normal} and \eqref{intro_eq_SurfM_ls} corresponds to zooming out and accelerating time by a factor of $\eps^{-1}$.
The evolution of a set $S$ with respect to these equations is hence given by $t\mapsto \eps\reach[(A,F)]{\eps^{-1}t}{\eps^{-1}S}$.
\end{remark}

\begin{remark}[Selection criterion for fattening]
Instead of a set evolution, one might hope to directly obtain an interface evolution from \eqref{intro_eq_uls} with the initial hypersurface given by $\Bp{x\in\Rd \,:\, \uls{S}(x,0)= 0}$ with Hausdorff-dimension $d-1$.
However, the level set $\Bp{x\in\Rd \,:\, \uls{S}(x,t)= 0}$ might at some point develop a non-empty interior, a phenomenon which is called fattening \cite{BSS93}.
In this case, there can be multiple interfaces within the level set for which we could say that they evolve by \eqref{intro_eq_normal}.
Our choice for $\reach{t}{S}$ implies that the interface $\partial\reach{t}{S}$ is the fastest or outermost of the interfaces within the level set.
\end{remark}

\begin{definition}[The homogenized set evolution]\label{intro_def_homR}
Let $v\in C(\Sd,\R{>0})$.
Given a closed set $S\subset\Rd$, we define 
\begin{equation*}
	\reach[hom,v]{t}{S}\coloneqq	\Bp{x\in\Rd \,:\, u^{hom}_S(x,t)\leq 0}
\end{equation*}
where $u^{hom}_S$ with $S=\Bp{x\in\Rd \,:\, u^{hom}_S(x,0)\leq 0}$ is a (viscosity) solution of 
\begin{equation}\label{intro_eq_HomM_ls}
		\pat u^{hom} + v\pp{\frac{\nabla u^{hom}}{\abs{\nabla u^{hom}}}}\abs{\nabla u^{hom}}=0
		\quad\,\text{in }\Rd\times(0,\infty).
\end{equation}
\end{definition}

As stated above, we compare the set evolutions $(\reach[(A,F)]{t}{S})_{t\geq 0}$ and $(\reach[hom,\vhom]{t}{S})_{t\geq 0}$ in terms of their arrival times at appropriately scaled target areas.

\begin{definition}[Arrival times]\label{intro_def_met}
For admissible coefficient fields $(A,F)$, $x_0\in\Rd$ and a closed set $S\subset\Rd$ we define the (actual) \textit{arrival time}
\begin{equation*}
	\met{}{x_0}{S}	=	\met[(A,F)]{}{x_0}{S}	\coloneqq	\min\Bp{t\geq 0\,:\,x_0\in\reach[(A,F)]{t}{S}}.
\end{equation*} 	
We further denote the \textit{truncated arrival time in the target area $\ov{B}_h(x_0)$} for $h>0$ with
\begin{equation*}
	\metb{h}{x_0}{S}	=	\metb[(A,F)]{h}{x_0}{S}	\coloneqq	\min\Bp{\min_{y\in\ov{B}_h(x_0)}\met[(A,F)]{}{y}{S},\,\timebound{h}{x_0}{S}}.
\end{equation*}
Here, $\timebound{h}{x_0}{S}$ is the time for which with high probability we will have a guarantee that the set evolution has reached the target area $\ov{B}_h(x_0)$, given by
\begin{equation*}
	\timebound{h}{x_0}{S}\coloneqq \vmineff^{-1}\dist(x_0,S)+C(\data)h.
\end{equation*}
The guarantee is due to the assumption on the existence of the effective minimum speed $\vmineff$ from Assumption \ref{sett_aP_pseudoStar} below, see Lemma \ref{stable_lem_TvsEspeed}.

Considering the homogenized set evolutions, for $v\in C(\Sd,\R_{>0})$ we set
\begin{equation*}
	\met[hom,v]{}{x_0}{S}	\coloneqq	\min\Bp{t\geq 0\,:\,x_0\in\reach[hom,v]{t}{S}}.
\end{equation*}
\end{definition}

\subsection{The informal main result}
Before we state the informal version of our main theorem, we provide an informal version of the key assumption, which effectively replaces and relaxes the deterministic, uniform bounds on the forcing required in \cite{ArCa18} or \cite{Fel19}. 
For the rigorous formulation see Assumption \ref{veff_aP_star} in Subsection \ref{subs_aP} and Subsection \ref{subs_sett_veff}, where the respective notions are made precise.
\begin{enumerate}[label=`(P$\circledcirc$)', left=0pt, resume]
	\item	\label{sett_aP_pseudoStar}\textbf{Almost guaranteed effective minimum speed}.
			We assume that there exist $\vmineff,C>0$, $h_0\geq 1$ and a rate $\ratespeed>0$ such that for all $h\geq h_0$ and $M\subset\Rd$
			\begin{align*}
				&\Pm\Big[\big\lbrace
					\text{In $M$, on the scale $h$, every `fat' set has a `stable approximation',} \\
				&\qquad\qquad\qquad\quad			
					\text{which propagates with an `effective minimum speed' of $\vmineff$}
				\big\rbrace\Big]\\
				&\qquad\quad	\geq		1-C\pp{1+\diam\pp{M}}^d\exp\pp{-C^{-1}h^\ratespeed},
			\end{align*}
			where `fat' sets $S\subset\Rd$ are unions of sufficiently large balls and `stable approximations' are non-shrinking interior approximations $S_h\subset S$ (see Figure \ref{sett_fig_stabapp}).
			An `effective minimum speed' on the scale $h$ means that for any $t\geq 0$ the evolved set $\reach{t}{S_h}$ after another time step of $\vmineff^{-1}h$ has advanced by at least $h$ in all directions while allowing for enclosures of width smaller than $2h$ to be left behind the main front (see Figure \ref{sett_fig_vmin}).
\end{enumerate}

Up to replacing this assumption with its rigorous formulation stated in Assumption \ref{veff_aP_star}, the following informal version of our main result is a precise statement.

\begin{theorem}[Informal main result]
Let $F_{obst}=F_{obst}(x,e)$ be a stationary random field with a finite range of dependence, which is uniformly bounded and Lipschitz.
Let $F_{uni}\in \R$ be a uniform forcing such that $F(x,e)=F_{uni}-F_{obst}(x,e)$ satisfies Assumption \ref{sett_aP_pseudoStar}.
Let $x_0\in\Rd$ and let $S\subset\Rd$ be a fat set with bounded boundary.

Then the truncated arrival time in the target area $\ov{B}_{h_\eps}(\eps^{-1}x_0)$ for the set evolution $(\reach{t}{\eps^{-1}S})_{t\geq 0}$ corresponding to \eqref{intro_eq_normal} with $h_\eps= \eps^{-\vtheta}$ for arbitrary $0<\vtheta<1$ has fluctuations of order at most $\eps^\frac{1-\vtheta}{2}$ with essentially exponential tail bounds, that is
\begin{multline}\label{intro_eq_infmain}
	\PM{\abs{\eps\metb{h_\eps}{\eps^{-1}x_0}{\eps^{-1}S}-\EV{\eps\metb{h_\eps}{\eps^{-1}x_0}{\eps^{-1}S}}}\geq C\eps^\frac{1-\vtheta}{2}\lambda}\\
	\leq \max\Bp{C\exp\pp{-\lambda}, C\exp\pp{-\frac{\eps^{-\vtheta\ratespeed}}{C\log(\eps^{-1})}}}
\end{multline}
for $\ratespeed$ from Assumption \ref{sett_aP_pseudoStar} and some constant $C>0$ independent of $\eps$.

Moreover, if $F_{obst}$ is sampled from an isotropic probability distribution, then for Lebesgue-almost every $F_{uni}\in\R$ with  $F(x,e)=F_{uni}-F_{obst}(x,e)$ satisfying Assumption \ref{sett_aP_pseudoStar} there is a  homogenized speed $\vhom\geq \vmin$ such that
\begin{equation*}
	\abs{\EV{\eps\metb{h_\eps}{\eps^{-1}x_0}{\eps^{-1}S}}-\met[hom,\vhom]{}{x_0}{S}}
	\leq C\eps^{\frac{1-\vtheta}{9}}.
\end{equation*}
\end{theorem}

\begin{remark}\label{intro_rem_aniso}
In fact, we obtain the fluctuation bounds in \eqref{intro_eq_infmain} also for space-dependent, anisotropic curvature terms, that is for interface motions with the normal velocity $V$ at $x\in\Rd$ given in terms of the normal $\nu\in\Sd$ by
\begin{equation*}
	V= \tr\pp{a(x,\nu)\nabla\nu} + F\pp{x, \nu}
\end{equation*}
for a positive semi-definite coefficient field $a(x,e)\in\Rddsym$.
In terms of the level-set formulation, this corresponds to \eqref{intro_eq_uls} with $A(x,e) = \pp{\Id-e\otimes e}a(x,e)\pp{\Id-e\otimes e}$.
\end{remark}

\subsection{Structure of the paper}
In Section \ref{s_lit}, we first collect related literature and then compare our result and the strategy of our proofs to \cite{ArCa18}.
In Section \ref{s_assump}, we lay out our assumptions and rigorously introduce the notions used in Assumption \ref{sett_aP_pseudoStar}.
In Section \ref{s_str}, we state our main results and then describe the strategy of the overall proof in detail. 
In Section \ref{s_pre}, we collect frequently used basic properties of the set evolutions and stable sets, as well as provide a measurability criterion, with their proofs given in Appendix \ref{s_pre-proofs}.
The major parts of the proof for the main results are presented in Section \ref{s_fluct}, Section \ref{s_inf}, Section \ref{s_lin}, Section \ref{s_hom}, and Appendix \ref{s_azuma}. 
A more detailed description of these Sections is given at the start of Subsection \ref{subs_str_main}.

\section{Related literature}\label{s_lit}
In this section we will discuss some of the literature which is related to our setting. 
In Subsection \ref{subs_lit_appl}, we provide motivation to study curvature-driven interface motions in random environments based on applications in physics.
In Subsection \ref{subs_lit_perio}, we review homogenization results for geometric motions in periodic environments.
In Subsection \ref{subs_lit_HJrand}, we review the known homogenization results for geometric motions in random environments, including some of the work on Hamilton--Jacobi equations which these results are based on.
In Subsection \ref{subs_lit_pinn}, we discuss results on pinning and de-pinning phenomena for geometric motions in random obstacle fields.
In Subsection \ref{subs_lit_perc}, we discuss some results from first passage percolation, which can be viewed as an edge case of the type of problem we study.
Finally, in Subsection \ref{subs_lit_ArCa} we compare the methods of our proof to those used in the existing literature on the stochastic homogenization of interface motions, in particular in \cite{ArCa18}.
 
\subsection{Physical motivation}\label{subs_lit_appl}
Equation \eqref{intro_eq_normal} describes the motion of an interface under the combined effects of surface tension -- represented by the mean curvature -- and a forcing field.
The forcing field can be thought of as the combination of a uniform external driving field and the effects of the random spatially heterogeneous environment including obstacles.  
Models of type  \eqref{intro_eq_normal} are used to describe the evolution of dislocation lines in a slip plane \cite{IoVi87, BraNa04} ($d=2$), and the evolution of domain boundaries in magnetic materials with impurities \cite{IoVi87, Kar98, BraNa04}, in the random-field Ising model \cite{KoLe85, Kar98}, or in a simplified version of immiscible-fluid displacement in porous media \cite{KoLe85, Kar98} ($d=2,3$).

In addition, there are phenomenologically related models for charge density waves \cite{FuLe78, IoVi87, Kar98, BraNa04}, contact lines \cite{Ge85, Kar98}, Wigner crystals \cite{NaSch00, BraNa04}, and vertex (or flux) lines in type-II superconductors \cite{Kar98, NaSch00, BraNa04}.

\subsection{Results for geometric motions in periodic environments}\label{subs_lit_perio}
A detailed survey on the homogenization of surfaces in periodic media up until 2013 can be found in \cite{Caf13}.
In the periodic (and almost periodic) setting, Lions and Souganidis \cite{LiSo05a} have provided a qualitative homogenization result for a class of second-order degenerate PDE in $\R^d$, which includes the level set formulation \eqref{intro_eq_SurfM_ls}. 
The coefficient fields are required to satisfy a technical assumption which in our setting corresponds to
\begin{equation}\label{lit_eq_mcf_forcingCond}
	\inf_{x\in\R^d}\left(F(x)^2-(d-1)|A||DF(x)|\right)>0
\end{equation}
and implies that for sufficiently large and regular sets the arrival times are uniformly Lipschitz continuous, which corresponds to a uniform minimum speed for the set evolutions.
Caffarelli and Monneau \cite{CaMo14} have shown that for $d=2$ this assumption can be relaxed to $\inf_x F(x)>0$, which yields a minimum speed on large scales.
In addition, a counterexample was provided showing that homogenization might fail in $d\geq 3$ for the same relaxation.
Gao and Kim \cite{GaKi19} have shown that for $d\geq 3$, under the assumption of a strictly positive forcing with $\inf_x F(x)>0$, for each direction there is a head and a tail speed for the expansion of sets.
For the counterexample from \cite{CaMo14}, head and tail speed differ, which is the reason why there can be no homogenization.

For signed, but very small forcing compared to the strength of the curvature, it was shown by Dirr, Karali, and Yip \cite{DKY08} that surfaces initialized as a plane stay within a fixed distance of a plane which moves with a fixed velocity.
For signed forcing, pinning and de-pinning phenomena occur when adding a constant $F_{uni}$ to the forcing, see Subsection \ref{subs_lit_pinn} for a brief introduction. 
In the work by Dirr and Yip \cite{DiYi06} it was shown for approximations of forced mean curvature flow in the graph setting that homogenization occurs and further that the homogenized speed scales like $(F_{uni}-F_*)^{\frac{1}{2}}$ close to the de-pinning threshold $F_*$.

\subsection{Results for Hamilton--Jacobi equations in random environments with geometric motions as a special case}\label{subs_lit_HJrand}
The level set formulation \eqref{intro_eq_uls} of the set evolution is a viscous Hamilton--Jacobi equation with a geometric and hence degenerate diffusion matrix and a $1$-homogeneous, non-coercive Hamiltonian.
Our setting corresponds to a hyperbolic scaling regime with the strength of the curvature proportional to the average distance between obstacles.
We limit our summary to results on stochastic homogenization  for Hamilton--Jacobi equations which can be interpreted as the level set formulation for geometric motions and an overview of the works which led to these results.

For first order Hamilton--Jacobi equations, the first qualitative homogenization results were obtained in \cite{ReTa00, Sou99} for convex Hamiltonians assuming just ergodicity.
Qualitative results for viscous, that is semilinear second-order  equations were first proven by Lions and Souganidis \cite{LiSo05b} and Kosygina, Rezakhanlou, and Varadhan \cite{KRV06}.
Extensions to larger classes of Hamilton--Jacobi equations and new proofs followed \cite{LiSo10, Sch09, ArSo12, ArSo13} with many more since.

However, there have been no results for Hamilton--Jacobi equations corresponding to curvature-driven geometric motions with general ergodicity assumptions for the random environment.
The first stochastic homogenization result for geometric motions was obtained for random environments with a finite range of dependence:
It is included as a special case in the last paper of a series of works by Armstrong, Cardaliaguet, and Souganidis \cite{ACS14} and by Armstrong and Cardaliaguet \cite{ArCa14, ArCa18} on the quantitative homogenization of increasingly larger classes of Hamilton--Jacobi equations.
As in the work \cite{LiSo05a} on periodic environments, the coefficient fields are assumed to satisfy a technical assumption, which in our setting corresponds to \eqref{lit_eq_mcf_forcingCond}, implying a guaranteed minimum speed of propagation.
Feldman \cite{Fel19} has shown for forced mean curvature flow in $d=2$ that this technical assumption can be relaxed to $\inf_x F(x)>0$, which again implies a minimum speed on large scales as in \cite{CaMo14} for the periodic case.
Our arguments also cover the setting from \cite{Fel19} and the case of geometric motions from \cite{ArCa18} with the additional assumption of statistical isotropy, substantially improving the convergence rates. 
However, the main novel feature of our work is that we allow for actual obstacles, which might locally pin the interface and thus cause enclosures left behind a main front.
Compared to the existing homogenization results in the literature, one can thus view our setting to be much closer to the de-pinning threshold, see Subsection \ref{subs_lit_pinn} below.

\subsection{Pinning and de-pinning phenomena in random environments}\label{subs_lit_pinn}
For curvature driven geometric motions through random obstacle fields, there has been a lot of interest in pinning and de-pinning phenomena.
The forcing field is often written as $F(\cdot)=F_{uni}-F_{obst}(\cdot)$ with a uniform forcing $F_{uni}\geq 0$ and a random field $F_{obst}(\cdot)\geq 0$ encoding the environment and obstacles.
For a fixed obstacle field $F_{obst}$, the choice of $F_{uni}$ can lead to qualitatively very different behavior.
First, there is the pinning regime $0\leq F_{uni}\leq F_*$, where $F_*$ is the maximal forcing for which there exists an interface above a given hyper-plane which corresponds to a stationary subsolution (in the sign-convention we use for our level-set formulation). 
By comparison principle, this interface is a barrier for the evolution of any interface initialized below.
Second, there is the ballistic regime $F_{uni}>F^*$, where $F^*$ is the infimum over all forcings such that the hyper-plane expands according to some minimum speed, that is that the expected height of the main front above each point of the hyperplane grows linear in time.
Due to ergodicity, these critical forcings are deterministic.
While $F_*\leq F^*$, in general it is not clear that $F_*=F^*$: There might be a so-called sub-ballistic regime inbetween, where the interfaces advance only at a sub-linear pace. 
Clearly our Assumption \ref{veff_aP_star} can hold only in the ballistic regime, but it is unclear whether it is satisfied for all $F_{uni}>F^*$.

In the prototypical case, $F_{obst}(\cdot)$ corresponds to obstacles of sufficient strength which are distributed according to a Poisson point process.
Alternatively, the obstacles can be placed on a grid and only the obstacle strength is random. 
Many results on pinning and de-pinning have been derived for the quenched Edwards--Wilkinson model, which is a linearized version of forced mean curvature flow in a graph setting and hence does not develop topological singularities.
For this model or (semi-)discrete versions, under varying assumptions on the obstacle distribution it has been shown that it is possible to have 
	$F_*>0$ \cite{DDS11, DJS21}, 
	$F_*<+\infty$ \cite{CDL10}, 
	$F^*<+\infty$ \cite{DoSch12, BoTe14, DoSch17}, 
	$F_*=F^*$ \cite{BoTe14},
	$F_*<F^*$ \cite{DoSch17},
	or $F_*=+\infty$	\cite{DJS22}.
	
For forced mean curvature flow as in \eqref{intro_eq_normal} with obstacles distributed according to a Poisson point process, it was shown by Dirr, Dondl, and Scheutzow \cite{DDS11} for $d\geq 2$ that $F_*>0$.
The proof relies on a percolation result proven in \cite{DDGHS10}.  
In the recent work by Courte, Dondl, and Ortiz \cite{CDO22}, it was shown in $d=2$ also for more general interface motions that the estimate $c\sqrt{\rho}\leq F_* \leq F^* \leq C\sqrt{\rho}$ holds, where $c,C>0$ are constants and where $\rho>0$ denotes the intensity of the Poisson point process. Note that $\sqrt{\rho}$ is proportional to the average obstacle distance.

\subsection{Related results from first passage percolation}\label{subs_lit_perc}
Our setting can be seen as a generalized version of first passage percolation.
The theory of first passage percolation studies the large-scale behavior of the passage time $T(x,y)$ between two points in a random environment, that is the optimal cost among all paths connecting the two points. 
For a detailed survey of the theory up until 2015 see \cite{ADH17}.
The most common choice for the underlying space is the lattice $\Zd$ with a random cost or time associated to the traversal of each edge.

In our setting, for $A=0$ and direction independent forcing $F(x,\xi)=F(x)$, the arrival times from Definition \ref{intro_def_met} correspond to the passage times in Riemannian first passage percolation on the continuum as introduced in \cite{LaGW10}.
For a description of this connection in the discrete setting see \cite[Section 2.2]{ADH17} .
In particular, the scaling of the fluctuation bounds from Theorem \ref{main_thm_fluct} and the convergence of the mean in Theorem \ref{main_thm} in this general setting can not be better than the best possible bounds for first passage percolation.

It has been predicted that the fluctuations of the passage times $T(0,\eps^{-1}e)$ scale like $\sim \eps^{-\frac{1}{3}}$ for $d=2$ and that they decrease further for higher dimensions.
Proving this without some stronger, unproven additional hypothesis is still open.
The best upper bound proven and published to date is $\lesssim \frac{\eps^{-\frac{1}{2}}}{\log\pp{\eps^{-1}}}$ for $d\geq 2$ which has first been achieved for Bernoulli distributed edge costs by Benjamini, Kalai, and Schramm \cite{BKS03} and was later extended to more general settings.
The best prior bound was $\lesssim \eps^{-\frac{1}{2}}$ by Kesten \cite{Kes93}, for the proof of which the martingale method of bounded increments was used.
This proof inspired the fluctuation estimates for the metric problem in the stochastic homogenization of Hamilton--Jacobi equations \cite{ACS14, ArCa14, ArCa18, Fel19} and thus in turn our proof of Theorem~\ref{main_thm_fluct}.

Regarding the rate of convergence to the mean, in \cite{Ale97} it was shown that 
\begin{equation*}
	\abs{\EV{T(0,\eps^{-1}e)}-\eps^{-1}\mu(e)} \lesssim \pp{\eps^{-1}\log\pp{\eps^{-1}}}^{\frac{1}{2}}
\end{equation*}
for all $e\in\Sd$.
The existence of the limit $\mu(e)$ had already been established beforehand in the form of shape theorems, see \cite{ADH17} for an overview.

\subsection{Comparison to \cite{ArCa18}}\label{subs_lit_ArCa}
Even though the work by Armstrong and Cardaliaguet \cite{ArCa18} also treats geometric motions, we use quite a different language.
This is due to both regularity issues in our setting and the fact that they are treating the level-set equation for forced mean curvature flow as a special case of more general Hamilton--Jacobi equations rather than just those corresponding to geometric motions.
The basic building block in \cite{ArCa18} is the metric problem, whose solutions for geometric motions would be exactly our arrival times.
To be more precise, for $\mu>0$ and a set $S\subset\Rd$ the metric problem is given by
\begin{equation}\label{intro_eq_metricPlanar}
	\left\lbrace\begin{aligned}
		&- \tr\pp{A\pp{x, \frac{\nabla \wt{m}_\mu}{\abs{\nabla\wt{m}_\mu}}}D^2 \wt{m}_\mu} + F\pp{x, \frac{\nabla \wt{m}_\mu}{\abs{\nabla \wt{m}_\mu}}}\abs{\nabla \wt{m}_\mu}
		=\mu
		&&\text{in }\Rd\setminus S,\\
		&\wt{m}_\mu=0
		&&\text{on }S.
	\end{aligned}\right.
\end{equation}
Here, since the underlying Hamilton--Jacobi equation corresponds to a geometric motion, we have $\wt{m}\coloneqq\wt{m}_1=\mu^{-1}\wt{m}_\mu$ for all $\mu>0$.
If a solution $\wt{m}$ to the metric problem existed, then $(x,t)\mapsto \max\Bp{0,\, -t+\wt{m}(x)}$ would be a solution to the level-set equation \eqref{intro_eq_uls} with the initial 0-sublevel set given by $S$.
In particular, our arrival time for $S$ at some $x_0\in\Rd$ would be given by $\met{}{x_0}{S}=\wt{m}(x_0)$.
However, in our setting the metric problem is ill-posed because our arrival times might attain infinity as a value since large areas might never be reached by the set evolution.

Translating the properties of the metric problem from \cite{ArCa18}, one of their central assumptions is that all realizations of $F$ satisfy an inequality which implies the Lipschitz continuity of the metric problem for sufficiently `fat' sets $S$.
In our language this corresponds to a guarantee that these sets expand at a general minimum speed  and hence the absence of any actual obstacles.
They heavily use this property both on large scales for the general concept of the proofs and on small scales for specific regularity issues. 
With our potentially negative forcing, there is no guaranteed minimum speed. 
In addition, it might happen that the front is not strictly expanding but might shrink locally.
Both of these possibilities potentially damage the effective independence between the `past' and `future' of the evolution, which is central in \cite{ArCa18}.
In order to salvage this property, we introduce a suitable notion of `stability', preventing shrinkage on large scales, as well as the notion of an `effective minimum speed of propagation' on large scales, ignoring enclosures behind the main front. 
While the `stability' property and the `effective minimum speed' are not guaranteed to be applicable, we can however reasonably assume that on large scales they are very likely to be applicable.

Applied on large enough scales, we use these two notions to obtain the fluctuation bounds for the truncated arrival times in a similar fashion to the fluctuation bounds for the metric problem in \cite{ArCa18}:
We also use the method of bounded martingale differences, writing the truncated arrival time as the endpoint of a suitably defined martingale.
However, in addition to significant technical adaptions due to the lack of an actual minimum speed, we obtain bounds for the martingale differences not almost surely, but just with high probability since our `effective minimum speed of propagation' is only applicable with high probability.
Despite this, we are still able to obtain good fluctuation bounds by introducing a suitable replacement for Azuma's inequality.

On small scales we have much worse regularity properties assuming an effective instead of an actual minimum speed.
This has two major consequences:
First, we only obtain an `approximate finite speed of propagation of perturbations' property corresponding to \cite[Proposition 4.2]{ArCa18} at the cost of increasing the forcing by a small constant, that is replacing $(A,F)$ with $(A,F+\delF)$ for some small $\delF>0$.
To obtain the overall quantitative homogenization result, we hence have to introduce an assumption on the H\"older continuity of the homogenized speed with respect to these forcing modifications.
Second, since the metric problem is ill-posed in our setting, we are not able to reproduce the classical homogenization argument for these kind of problems, which would be to use the perturbed test function method to lift the homogenization bounds from the metric problem first to an approximate corrector equation and then to the level-set equation. 
Instead, we use a more constructive approach specific to geometric motions, which is not applicable for the general viscous Hamilton--Jacobi equations treated in \cite{ArCa18}.

\section{Setting and Assumptions}\label{s_assump}
For convenience, we denote dependencies on the parameters introduced with the assumptions below via
\begin{equation*}
	\data\coloneqq	\pp{d, C_{1A}, C_{1F}, \vmineff, h_0, \ratespeed, \cspeed, \Cspeed, \Cmeas, \cstable, \hoel, \Choel }.
\end{equation*}

\subsection{Admissible coefficient fields}\label{subs_sett_AF}
We say that $(A,F)$ is an admissible coefficient field if the following conditions \eqref{sett_eq_assAbd}, \eqref{sett_eq_assAgeom} and \eqref{sett_eq_assF} are satisfied.
$A$ is assumed to be continuously differentiable, degenerate elliptic and uniformly bounded and Lipschitz, that is there exist $\sigma\in C^{0,1}(\Sd\times\Rd;\,\Rddsym)$ and $C_{1A}>0$ such that
\begin{equation}\label{sett_eq_assAbd}
	A=\sigma\sigma^\intercal
	\,\,\text{ with }
	\norm{\sigma}_{C^{0,1}(\Sd\times\Rd;\,\Rddsym)},\norm{A}_{C^{0,1}(\Sd\times\Rd;\,\Rddsym)} \leq C_{1A}.
\end{equation}
To guarantee that \eqref{intro_eq_SurfM_ls} is the level-set formulation of a surface motion, we require
\begin{equation}\label{sett_eq_assAgeom}
	A(x,e)\pp{\Idd-e\otimes e} = A(x,e)
	\text{ for all }e\in\Sd,\, x\in\Rd.
\end{equation}
$F$ is also assumed to be uniformly bounded and Lipschitz, that is there exists $C_{1F}>0$ such that
\begin{equation}\label{sett_eq_assF}
	F\in C^{0,1}(\Sd\times\Rd;\,\R) 
	\text{ with }
	\norm{F}_{C^{0,1}\pp{\Sone\times\Rd;\,\R}} \leq C_{1F}.
\end{equation}
These conditions on $(A,F)$ guarantee that the theory for surface evolution equations is applicable to \eqref{intro_eq_uls}, see Appendix \ref{s_visc} for a summary of the results relevant to us.
In addition, the uniform bounds will be used in some later sections.
To simplify the notation we use the $0$-homogeneous extension of $A$ and $F$ to $\pp{\Rd\setminus\Bp{0}}\times\Rd$, for $A$ given by $A(\xi,x)\coloneqq A\pp{\abs{\xi}^{-1}\xi,x}$.

\subsection{Assumptions on the probability distribution}\label{subs_aP}
In short, we assume that the coefficient field $(A,F)$ in \eqref{intro_eq_uls} is a random variable sampled from a stationary probability measure with finite range of dependence.
Some of the later results we are only able to prove under two additional assumptions: 
The assumption of $A$ being deterministic and the assumption of isotropy of the probability measure. 
Finally, we introduce two non-standard assumptions:
First, one of our key assumptions is that for `fat' sets with overwhelming probability there are `stable approximations', for which an `effective minimum speed' is applicable.
Second, we require that our candidate for the homogenized speed, which is well-defined due to the previous assumptions, is H\"older continuous with respect to changing the forcing $F$ by a constant.

More explicitly, the probability space is defined as the set of admissible coefficient fields
\begin{equation}\label{sett_eq_omega}
	\Omega\coloneqq \Bp{(A,F)\,:\,A\text{ satisfies }\eqref{sett_eq_assAbd}\text{ and } \eqref{sett_eq_assAgeom},\,
								F\text{ satisfies }\eqref{sett_eq_assF}}.
\end{equation}
On $\Omega$ we define the $\sigma$-algebras $\Bp{\F(U)}$ for all Borel sets $U\subset\Rd$,
\begin{multline}\label{sett_eq_sigmaF}
  	\F(U)\coloneqq 	\text{the $\sigma$-algebra generated by the family of maps}\\
  					\Omega\rightarrow \Rdd\times\R,\,
  					(A,F)\mapsto \pp{A(x,e), F(x,e)}\quad
  					\text{with }x\in U,\,e\in\Sd.
\end{multline}  
The $\sigma$-algebras $\F(U)$ can be seen as containing all of the knowledge which can be gained by observing the coefficient field on $U\subset\Rd$.

We assume that the randomness of the coefficient field is described by a probability measure $\Pm$ on $(\Omega,\F(\Rd))$, which satisfies the following conditions.
For some of our results, only a subset of these assumptions will be required.
First, we have two classical key assumptions from the theory of stochastic homogenization:
\begin{enumerate}[label=(P\arabic*), left=8pt]
	\item	\label{sett_aP_stat} \textbf{Stationarity.}
			 $\Pm$ is invariant under translations of the coefficient field: 
			\begin{equation*}
				\Pm=\Pm\circ \tau_x, \quad\text{for any }x\in\Rd,
			\end{equation*}
			 where 
				$\tau_x(\wt{\Omega})=\lbrace\pp{A(\cdot+x,\cdot),F(\cdot+x,\cdot)}\,:\,(A,F)\in\wt{\Omega}\rbrace$  
			for $\wt{\Omega}\in\F(\Rd)$.
	\item	\label{sett_aP_fin} \textbf{Finite range of dependence.}
			For all Borel sets $U,V\subset\Rd$ with $\dist(U,V)\geq 1$
			\begin{equation*}
				\F(U)\text{ and }\F(V)\text{ are $\Pm$-independent.}
			\end{equation*}
\end{enumerate}
Regarding the two additional assumptions, we require that the following holds:
\begin{enumerate}[label=(P\arabic*), left=8pt, resume]
	\item	\label{sett_aP_Aconst}	\textbf{Spatially constant and thus deterministic second order-term.}
			There holds
			\begin{equation*}
				\PM{(A,F)\in\Omega\,:\,A(x,e)=A(0,e)\text{ for all }x\in\Rd, e\in\Sd}=1.
			\end{equation*}
	\item	\label{sett_aP_ani} \textbf{Isotropy.}
			$\Pm$ is invariant under rotations, that is there exists a non-negative constant $a\geq 0$ such that 
			\begin{equation*}
				\PM{(A,F)\in\Omega\,:\,A(x,e)=a\pp{\Idd-e\otimes e} \text{ for all }x\in\Rd, e\in\Sd}=1,
			\end{equation*}	
			and $\operatorname{Law}(F(\cdot,\cdot)) = \operatorname{Law}{F(R\cdot,R\cdot)}$ for all $R\in\operatorname{SO}(d)$.
\end{enumerate}
Regarding our first non-standard assumption, we require that on large scales there is an almost guaranteed effective minimum speed for the evolution of sets.
\begin{enumerate}[label=(P$\circledcirc$), left=0pt]
\item	\label{veff_aP_star}	\textbf{Almost guaranteed effective minimum speed}.
		There is $\vmineff>0$, $h_0\geq 1$ and a rate $\ratespeed>0$ with constants $\cspeed, \Cspeed, \Cmeas, \cstable>0$, such that for each $h\geq h_0$ there is a family of events $\pp{\Espeed{h}(M)}_{M\subset\Rd}$ with 
		\begin{equation}\label{veff_eq_probBound}
			\quad	\Pm[\Espeed{h}(M)]\geq 1 - \Cspeed\pp{1+\frac{\diam(M)}{h}}^d\exp(-\cspeed h^{\ratespeed}),
		\end{equation}			
		which satisfies
		\begin{enumerate}[label=($*$\roman*), left=5pt]
			\item	\label{veff_Pmeas}
					$\Espeed{h}(M)$ is measurable with respect to $\F(M+\Cmeas \ov{B}_h)$,
			\item	\label{veff_Porder}
					for all $M_1\subset M_2$ there holds $\Espeed{h}(M_1)	\subset	\Espeed{h}(M_2)$,
			\item	\label{veff_Pspeed}
					for all $(A,F)\in \Espeed{h}(M)$, every $h$-fat set $S$ with $\partial S\subset M$ has a stable $(h,\cstable)$-approximation, 
					for which $(A,F)$ admits $\vmineff>0$ as an effective minimum speed of propagation on the scale $h$ in $M\subset\Rd$, see Definition \ref{veff_def_fatstab} below for the precise introduction of these notions.
		\end{enumerate}	
\end{enumerate} 
\begin{remark}
We do not explicitly define $\Espeed{h}(M)$ directly as the set such that \ref{veff_Pspeed} holds, because then in general we would not have measurability as in \ref{veff_Pmeas}.
\end{remark}

Regarding our second non-standard assumption, we require H\"older continuity of the homogenized speed with respect to changing the forcing by a constant $\delF\in\R$.
Note that generically this assumption is satisfied, see Remark \ref{sett_rem_hoel} below.
We want to choose this homogenized speed as the large scale limit of the average speed corresponding to truncated arrival times for half-spaces.
Since we are not able to guarantee the existence of this limit under just the previous assumptions, we approximate the half-spaces with bounded sets as in Figure \ref{lin_fig_hs-approx} and define the homogenized speed for the forcing $F+\delF$ as
\begin{equation*}
				\delF\mapsto \vhom[\delF,\beta](e)\coloneqq \pp{\lim_{r\rightarrow\infty}	\frac{1}{r}\EV{\metb[(A,F+\delF)]{h(r)}{re}{\hsma{e}{\beta}{h(r)}{r}}}}^{-1},
\end{equation*}
where $r\mapsto \hsma{e}{\beta}{h(r)}{r}$ are disks of radius $\sim r^{\beta}$ with $\beta\geq \frac{3}{2}$ and thickness $\sim h(r)<<r$ approximating the boundary of the half space $\hsm{e}=\Bp{x\in\Rd\,:\,x\cdot e\leq 0}$ for $e\in\Sd$, see Figure \ref{lin_fig_hs-approx} and Corollary \ref{lin_cor_bd} for the precise definition.
In Section \ref{s_lin} we will prove the existence of these limits for all $\beta$ and all $\delF\in\R$ if Assumptions \ref{sett_aP_stat}, \ref{sett_aP_fin}, \ref{sett_aP_Aconst} and \ref{veff_aP_star} hold for the random coefficient field $(A,F+\delF)$.
Assumption \ref{sett_aP_ani} further guarantees that these quantities are the same for all $e\in\Sd$.
We assume that the following holds.
\begin{enumerate}[label=(P$\hoel$), left=0pt, resume]
	\item	\label{sett_aP_gen}\textbf{H\"older-continuity of the homogenized speed with respect to changing the forcing by a constant.} 
			There is $0< \hoel\leq 1$ and a width scaling $\genhoel> \frac{3}{2}$ for the half-space approximations $\pp{\hsma{e}{\genhoel}{h(r)}{r}}_{e\in\Sd,r>0}$ defined in Corollary \ref{lin_cor_bd}, such that for all $e\in\Sd$ the function
			\begin{equation*}
				\delF\mapsto \vhom[\delF](e)
				\coloneqq 	\vhom[\delF,\genhoel](e)
				=			\pp{\lim_{r\rightarrow\infty}	\frac{1}{r}\EV{\metb[(A,F+\delF)]{h(r)}{re}{\hsma{e}{\genhoel}{h(r)}{r}}}}^{-1}
			\end{equation*}
			is H\"older-continuous at $\delF=0$ with H\"older coefficient $\hoel$.
			That is, there exists $\Choel>0$ such that
			\begin{equation*}
				\abs{\vhom[\delF](e)-\vhom[0](e)}\leq	\Choel\abs{\delF}^{\hoel}
				\,\, \text{ for all $\delF\in\bp{-\frac{1}{\Choel},\frac{1}{\Choel}}$ for which $\vhom[\delF](e)$ exists.}
			\end{equation*}
			For the homogenized speed given the unchanged forcing we write $\vhom=\vhom[0]$.
\end{enumerate}

\begin{remark}[On Assumption \ref{sett_aP_gen}]\label{sett_rem_hoel}
\begin{enumerate}[left=0pt]
\item	\label{sett_item_hoel_1}
		We will show in Lemma \ref{lin_lem_hoelInv} that the width scaling $\genhoel$ of the half-space approximations is arbitrary: 
		If all the previous assumptions hold, then Assumption \ref{sett_aP_gen} implies that the function $\delF\mapsto \vhom[\delF,\beta](e)$ is $\hoel$-H\"older continuous at $\delF=0$ with the same constant $\Choel$ and $\vhom[0,\beta](e)=\vhom[0,\genhoel](e)=\vhom$ for all $\beta>\frac{3}{2}$.
\item	Assumption \ref{sett_aP_gen} can be expected to hold generically in the following sense: 
		Since by the comparison principle $\delF\mapsto \vhom[\delF]$ is a nondecreasing function, it is Lipschitz-continuous at Lebesgue-almost every $\delF$.
\item	Assumption \ref{sett_aP_gen} is satisfied if there is an actual minimum speed as for example in the setting of \cite{ArCa18} for forced mean curvature flow.
		This can be seen by manipulating the metric problem, whose solutions in this case are well posed, Lipschitz continuous and correspond to our arrival times.
\item	Heuristic arguments suggest that $\delF\mapsto \vhom[\delF]$ is always at least $\frac{1}{2}$-H\"older continuous when all the previous assumptions hold.
		One reason for the failure of Lipschitz continuity can be the sufficiently frequent occurrence of local de-pinning when increasing $F$ by $\Delta F$, thus unlocking `shortcuts'.
		In this case, morally the effective speed can be expected to increase at most like $\sim\Delta F^{\frac{1}{2}}$, see the discussion on de-pinning transitions in the introduction of \cite{DiYi06}.
		However, further investigating this is out of scope for this paper.
\end{enumerate}
\end{remark}

\begin{remark}[Core assumptions versus additional requirements]
While \ref{sett_aP_fin} and \ref{veff_aP_star} are central assumptions for this paper, \ref{sett_aP_stat}, \ref{sett_aP_Aconst}, \ref{sett_aP_ani} and \ref{sett_aP_gen} are only required in the latter half of the paper for the deterministic error estimate, see the beginning of Subsection \ref{subs_str_main} for a description of where the assumptions enter the overall proof.
Note that in particular the isotropy assumption \ref{sett_aP_ani} is only required for the very last part of the proof and can be replaced by a less restrictive but more technical assumption on the homogenized speed $\vhom\colon\Sd\mapsto \R_{>0}$, see Remark \ref{hom_rem_aniso}.
\end{remark}

\begin{remark}[On Assumption \ref{sett_aP_ani}]\label{sett_rem_ani_actual}
Strictly speaking, Assumption \ref{sett_aP_ani} is stronger than just assuming that the probability distribution is isotropic in the context of the level-set formulation in \eqref{intro_eq_uls}, which would correspond to the following more general assumption.
\begin{enumerate}[label=\ref{sett_aP_ani}', left=0pt]
\item	\label{sett_aP_ani_actual} \textbf{Isotropy with respect to \eqref{intro_eq_uls}.}
		$\Pm$ is invariant under rotations:
		\begin{equation*}
			\Pm=\Pm\circ \operatorname{rot}_R,		\quad\text{for any }R\in\operatorname{SO}(d),
		\end{equation*}
		where $\operatorname{rot}_R(\wt{\Omega})$ for any set of coefficients $\wt{\Omega}\in\F(\Rd)$ is given by
		\begin{equation*}
			\operatorname{rot}_R(\wt{\Omega})=\lbrace\pp{(x,e)\mapsto R^\intercal A(Rx,Re)R,F(Rx,Re)}\,:\,(A,F)\in\wt{\Omega}\rbrace.
		\end{equation*}
\end{enumerate}
However, taken together with the other assumptions, this indeed requires that the second-order term corresponds to mean curvature up to a constant factor and hence \ref{sett_aP_ani} holds, as we state in the following lemma, which we prove in Appendix \ref{s_pre-proofs}.
\end{remark}

\begin{lemma}\label{sett_lem_isotropy}
	In view of \eqref{sett_eq_assAbd} and \eqref{sett_eq_assAgeom}, if Assumption \ref{sett_aP_stat}, Assumption \ref{sett_aP_fin}, Assumption \ref{sett_aP_Aconst}, and Assumption \ref{sett_aP_ani_actual} hold, then there exists a constant $a\geq 0$ such that 
	\begin{equation*}
		\PM{(A,F)\in\Omega\,:\,A(x,e)=a\pp{\Idd-e\otimes e} \text{ for all }x\in\Rd, e\in\Sd}=1.
	\end{equation*}	
\end{lemma}

For $d= 2$, Assumptions \ref{sett_aP_stat}, \ref{sett_aP_fin}, \ref{sett_aP_Aconst}, \ref{sett_aP_ani}, and \ref{veff_aP_star} are satisfied for mean curvature flow with a uniform, positive forcing except for Poisson-distributed symmetric obstacles of diameter smaller than one, as long as the intensity of the underlying Poisson process is low enough.
Assumption \ref{sett_aP_gen} holds for Lebesgue-almost every choice of the uniform forcing as stated in Remark \ref{sett_rem_hoel}.
For more details see Theorem \ref{intro_thm_Poisson} above and Example \ref{veff_ex_Poisson} below.

\subsection{Stability and the notion of effective minimum speeds}\label{subs_sett_veff}
In this subsection, we will properly introduce the notions used to state Assumption \ref{veff_aP_star} from above and its informal version \ref{sett_aP_pseudoStar} from the introduction.
That is, we will make precise what we mean with `fat' sets, `stable approximations' and an `effective minimum speed'.
In particular, we first need to introduce stable sets.

\begin{figure}[h]
\centering
\includegraphics[width=0.36\textwidth]{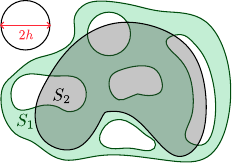}
\caption{	An illustration of our notion of h-envelopment $S_2 \hsubs S_1$.
			A set $S_1 \subset \Rd$ (depicted in green) which $h$-envelops a set $S_2\subset \Rd$ (depicted in grey).}
\label{sett_fig_envelop}
\end{figure}

\begin{figure}[h]
\centering
\includegraphics[width=0.35\textwidth]{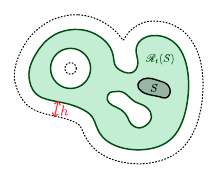}
$\quad$
\includegraphics[width=0.35\textwidth]{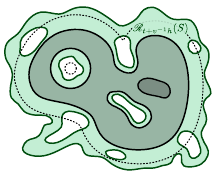}
\caption{An illustration of our notion of effective minimum speed.
	Left: An initial set S (depicted in grey) and its evolution $\reach{t}{S}$ after some time $t$ (depicted green). 
	The dotted line corresponds to a further expansion of $\reach{t}{S}$ by $\ov{B}_h$.
	Right: If an effective minimum speed $v$ holds for $S$, then after an additional time step $v^{-1} h$ the evolved set $\reach{t+v^{-1} h}{S}$ (depicted in green) must $h$-envelop the dotted area.}
\label{sett_fig_vmin}
\end{figure}

\begin{figure}[h]
\centering
\includegraphics[width=0.3\textwidth]{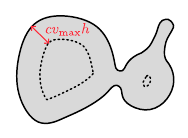}
$\quad$
\includegraphics[width=0.3\textwidth]{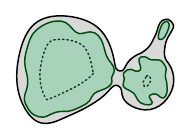}
$\quad$
\includegraphics[width=0.3\textwidth]{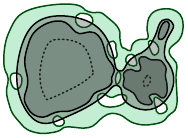}
\caption{	An illustration of our notion of a stable stable $(h,c)$-approximation $S_h$
of a set $S$.
			Left: The set $S\subset \Rd$ to be approximated is depicted in grey. 
			The subset of $S$ that must be included in $S_h$ is depicted with a dotted boundary.
			Center: In addition to the aforementioned sets, the approximation $S_h$ is depicted in green.
			Right: After a time step $ch$, the evolved set $\reach{ch}{S_h}$ (depicted in green) must $h$-envelop the set $S$.}
\label{sett_fig_stabapp}
\end{figure}

\subsubsection{Introduction of Stable Sets}
Throughout this paper, sets which do not shrink under the interface evolution will be central to many arguments.
We call these sets or the corresponding interfaces \textit{stable}.
We will discuss properties of these sets in Section \ref{s_pre}.

\begin{definition}[Stable Sets]\label{box_def_stable}
	A closed set $S\subset\Rd$ is called stable (with respect to $(A,F)\in\Omega$) if $u\coloneqq 1-\charfun[S]$ is a supersolution of 
	\begin{equation}\label{box_eq_stable}
		- \tr\pp{A\pp{x, \frac{\nabla u}{\abs{\nabla u}}}D^2 u} 
		+ F\pp{x, \frac{\nabla u}{\abs{\nabla u}}}\abs{\nabla u}
		= 0
		\quad\text{in }\Rd.
	\end{equation}
\end{definition}

\begin{remark}
Due to the forcing possibly being negative, there is no condition that would guarantee that any `flat enough' set is stable with respect to all coefficient fields $(A,F)\in\Omega$.
This is different from \cite{ArCa18}, where admissible coefficient fields have to satisfy stricter conditions. 
There, essentially the ``ball condition'' $S=\bigcup\Bp{\ov{B}_1(x)\,:\,\ov{B}_1(x)\subset S}$ guarantees that $S\subset\R^d$ is stable with respect to all admissible coefficient fields.
\end{remark}

\subsubsection{Approximate stability and effective minimum speeds}
We now make precise what we mean in Assumption \ref{veff_aP_star} with an `effective minimum speed' of propagation, `approximate stability' and `fat' sets.
We first give a brief intuitive description of these notions and then provide the formal definition.

Loosely speaking, a coefficient field $(A,F)$ admits $v>0$ as an \textit{effective minimum speed of propagation} on the scale $h$ in an area $M\subset\Rd$ for some set $S\subset\Rd$, if at any time $t\geq 0$ the evolved set $\reach{t}{S}$ within a time step $v^{-1}h$ further expands by $h$ in each direction, see Figure \ref{sett_fig_vmin}.
However, for this expansion we ignore enclosed holes of width smaller than $2h$. Morally we can view such thin enclosed areas as being part of the set since we look at arrival times in target areas of diameter $2h$, see Definition \ref{intro_def_met}.

A set $S$ is \textit{$(h,c)$-approximately stable} if it has a stable subset $S_h$, which approximates $S$ in two ways: First, it covers the part of the interior of $S$, which is further than $c\vmax h$ from the boundary. Second, after evolving for $ch$, $S_h$ envelops all of $S$ except for potential holes of width smaller than $2h$, see Figure \ref{sett_fig_stabapp}.

A set $S$ is \textit{$h$-fat}, if it is same as the union of all included balls with radius $h$ - that is, if there are no parts on a scale smaller than $h$ protruding from the set.

\begin{definition}\label{veff_def_fatstab}
Let $(A,F)\in\Omega$.
For a set $S\subset\Rd$, let $t\mapsto \reach[(A,F)]{t}{S}$ denote the set evolution from Definition \ref{not_def_reach}.
Let $h>0$, $c>0$.
\begin{enumerate}[label=(\roman*)]
\item	We say that a set $S_1\subset\Rd$ \textit{$h$-envelops} another set $S_2\subset\Rd$, denoted by $S_2\hsubs S_1$,
		if for each connected component $H$ of $\comp{(S_1)}$ either $H\cap S_2 = \emptyset$ or $H\subset S_1+\ov{B}_h$, see Figure \ref{sett_fig_envelop}.	
\item	We say that $(A,F)$ admits $v>0$ as an \textit{effective minimum speed of propagation} on the scale $h$ in $M\subset\Rd$ for a set $S\subset\Rd$ if for any $t\geq 0$ there holds
		\begin{equation}\label{veff_eq_effspeed}
			\quad	\pp{\reach[(A,F)]{t}{S}\cap M} + \ov{B}_h  \hsubs \reach[(A,F)]{t+v^{-1}h}{S}.
		\end{equation} 
\item	We say that a set $S\subset\Rd$ is \textit{$(h,c)$-approximately stable} with respect to $(A,F)$
		if there exists a set $S_h\subset S$ such that
		\begin{align*}
				\quad	S_h \text{ is stable},
			&&		S	\hsubs \reach[(A,F)]{c h}{S_h},
			&&		S_h	\supset \Rd\setminus\pp{\comp{S}+c\vmax B_h},
		\end{align*}		
		where $\vmax=\vmax(\data)$ corresponds to the maximum speed with which the interface propagates on larger scales, see Lemma \ref{not_lem_vmax}.
		In this case, $S_h$ is called a \textit{stable $(h,c)$-approximation} of $S$.
\item	We say that a set $S\subset \Rd$ is \textit{$h$-fat} if 
		\begin{equation*}
			\quad	S=\bigcup\Bp{\ov{B}_h(x)	\,:\,	x\in\Rd \text{ with }	\ov{B}_h(x)\subset S}.
		\end{equation*}
		We say that $S$ is \textit{fat}, if there exists some $h_0>0$ such that $S$ is $h_0$-fat.
\end{enumerate}
\end{definition}

Note that the notion of $h$-envelopes is chosen to be compatible with the truncated arrival times from Definition \ref{intro_def_met}, as we will show in the following lemma. For the proof see Appendix \ref{s_pre-proofs}.

\begin{lemma}[Compatibility of the arrival times {$\met[]{h}{\cdot}{\cdot}$} and ``{$\hsubs$}'']\label{veff_lem_henv}
Let $(A,F)\in\Omega$. 
Let $h>0$, $x_0\in\Rd$ and $S_1, S_2\subset\Rd$ be closed sets with
\begin{align}\label{veff_eq_reach-hsubs}
	\reach[(A,F)]{t}{S_1}\cap\ov{B}_{\max\Bp{h, \dist(x_0,S_2)}}(x_0) \hsubs\reach[(A,F)]{t}{S_2}  
	\qquad\text{for all }t\leq \timebound{h}{x_0}{S_2}
\end{align}
with $\timebound{h}{x_0}{S_2}=\vmineff^{-1}\dist(x_0,S_2)+Ch$ from Definition \ref{intro_def_met}.
Then there holds 
\begin{align*}
	\met[(A,F)]{h}{x_0}{S_1}	\geq \met[(A,F)]{h}{x_0}{S_2}.
\end{align*}
\end{lemma}

\subsection{Examples}\label{subs_sett_ex}
In this subsection, we list a range of settings for which the assumptions from Subsection \ref{subs_sett_AF} are applicable.

\begin{example}[In the absence of actual obstacles as in \cite{ArCa18}]
Let $\Pm$ satisfy Assumptions \ref{sett_aP_stat} and \ref{sett_aP_fin}.
In addition, for almost every $(A,F)\in\Omega$ let 
\begin{align*}
	A(x,\xi)	&=\Id-\frac{\xi}{\abs{\xi}}\otimes\frac{\xi}{\abs{\xi}},
	&
	F(x,\xi)	&=F(x),
\end{align*}
which corresponds to mean curvature flow with an additional forcing of $F(\cdot)$.
As shown in \cite{ArCa18}, if we have $\delta>0$ such that
\begin{equation*}
	\inf_{x\in\Rd}\pp{F(x)^2-(d-1)\abs{DF(x)}}\geq \delta>0
	\quad\text{$\Pm$-almost surely},
\end{equation*}
then there exists $R_0>0$ and a guaranteed minimum speed $\vmineff>0$, which applies to any $R_0$-fat set $S\subset\Rd$, that is 
\begin{equation*}
	\reach{t+s}{S}\supset \reach{t}{S}+\vmineff s\ov{B}_1
	\quad\text{$\Pm$-almost surely for all $t,s\geq 0$}.
\end{equation*}
Hence, Assumption \ref{veff_aP_star} is satisfied with the stable approximations given by the sets themselves. 
Moreover, Assumption \ref{sett_aP_ani} also is satisfied, see Remark \ref{sett_rem_hoel}.

In this special case and additionally requiring the isotropy from Assumption \ref{sett_aP_ani}, our result hence improves the error rate for the bias from \cite{ArCa18}.
Of course, \cite{ArCa18} covers a much more general setting and there had been no attempt to optimize this error rate. 
We do not believe that our error rate is optimal either.
\end{example}

\begin{example}[First passage percolation in a continuous environment]
For $A=0$ and direction independent forcing $F(x,\xi)=F(x)$, our setting corresponds to first passage percolation on continuous space, that is Riemannian first passage percolation as introduced in \cite{LaGW10}.
In this case, the arrival time $\met{}{x_0}{S}$ can be seen as the fastest time at which $x_0$ can be reached via an arbitrary path from a closed set $S\subset\Rd$, with the speed along the path at some point $x\in\Rd$ given by $F(x)$, see \cite{LiSo05b, ADH17}.
Areas with $F\leq 0$ would be impassable in this case.

Hence, the best possible bounds on fluctuations and bias in our general setting can not be better than the best possible bounds for first passage percolation.
\end{example}

\begin{example}[A box criterion for $d=2$]\label{ex_ex_box}
In the companion paper \cite{FI26}, we will provide a practical, constructive criterion which we prove to be sufficient for Assumption \ref{veff_aP_star} in dimension $d=2$ if Assumptions \ref{sett_aP_stat} and \ref{sett_aP_fin} hold.
In order to check the criterion, it suffices to sample the coefficient field in a thin, long box $\Qup{r,w}=[0,w]\times[-w,r]$ with width $w>0$ and height $r>1$ not necessarily much larger than the range of dependence.
If we did not have the isotropy from Assumption \ref{sett_aP_ani}, then we would also have to sample the coefficient field in a 90$^\circ$ rotation of the box.

The criterion is satisfied if it is likely enough that there is a stable set in the bottom of the box, which at some point reaches the top of the box when evolving with respect to $(A,F)$ restricted to the box, see Figure \ref{ex_fig_box}. 

To be more precise, we define the event
\begin{align*}
	\Eup{r,w}	\coloneqq
				\Big\lbrace	(A,F)\in\Omega	\,:\,	
						&\text{there is a w.r.t. $(A,F)$ stable set $S\subset[0,w]\times[-w,0]$}\\
				&\text{with }\reach[{(A,F)_{\Qup{r,w}}}]{t}{S}\cap [0,w]\times\Bp{r} \neq\emptyset \text{ for some }t>0		
				\Big\rbrace,
\end{align*}
where $(A,F)_{\Qup{r,w}}$ is the restriction of $(A,F)$ to the box $\Qup{r,w}$ as in Definition \ref{restr_def}.
The criterion is satisfied if for some $p_{r,w}\geq 0$ 
\begin{align}\label{ex_eq_box}
	\PM{\Eup{r,w}}\geq 1-p_{r,w},
\end{align}
where $p_{r,w}>0$ for sufficiently long and thin boxes.
\end{example}

\begin{figure}[h]
\centering
\includegraphics[width=0.25\textwidth]{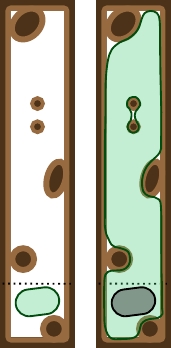}
\caption{	An illustration of our box criterion (Example \ref{ex_ex_box}) sufficient for Assumption \ref{veff_aP_star} in dimension $d=2$.
			Left: A rectangular box with a sample of the obstacles in the interior (depicted in brown) and the artificial obstacle from the restriction surrounding the boundary (also depicted in brown) with an initial stable set $S$ in the bottom of the box (depicted in green).
			Right: After some time $t$, we require the evolution $\reach{t}{S}$ (depicted in green) of the set $S$ to have reached the top of the box.
}
\label{ex_fig_box}
\end{figure}

\begin{example}[Poisson-distributed obstacles]\label{veff_ex_Poisson}
An explicit setting in which our results are applicable is given by forced mean curvature flow through a field of obstacles distributed according to a Poisson point process with small enough intensity. 

To be more precise, we assume that our coefficient fields are given by
\begin{align*}
	A(x,\xi)	&=\Id-\frac{\xi}{\abs{\xi}}\otimes\frac{\xi}{\abs{\xi}},
	&
	F(x,\xi)	&=F_{uni}-F_{obst}(x),
\end{align*}
where $F_{uni}>0$ represents a uniform forcing and where $F_{obst}$ representing the obstacles is the maximum of compactly supported, symmetric, smooth but steep bumps with their centers chosen according to the Poisson point process, see Theorem \ref{intro_thm_Poisson}.

Assumptions \ref{sett_aP_stat}, \ref{sett_aP_fin}, \ref{sett_aP_Aconst}, and \ref{sett_aP_ani} are satisfied for all $F_{uni}>0$, with the finite range of dependence given by the diameter of the obstacles.
If Assumption \ref{veff_aP_star} is satisfied for some $F_{uni}^*>0$, then with Remark \ref{sett_rem_hoel} Assumption \ref{sett_aP_gen} will be satisfied for Lebesgue-almost every $F_{uni}>F_{uni}^*$.

For $d=2$, we can check that these coefficient fields satisfy the box criterion from Example \ref{ex_ex_box} and hence Assumption \ref{veff_aP_star} if $F_{uni}$ is large enough compared to the intensity or average density $\rho$ of the Poisson point process.
That is, we can show that \eqref{ex_eq_box} holds for sufficiently large boxes if $F_{uni}\geq C\sqrt{\rho}$ for some $C>0$ with the support of the obstacles small enough, where $C$ is independent of the individual strength of the obstacles.
In particular, \eqref{ex_eq_box} holds even if the obstacles are chosen to individually be impenetrable for a given uniform forcing. 

For $d\geq 3$, we also expect that Assumption \ref{veff_aP_star} is satisfied if $F_{uni}$ is large enough compared to the intensity of the Poisson point process, based on the properties of sufficiently wide obstacle free areas.
Exploring this specific setting is left for future research.
\end{example}

\section{Main result and the strategy for the proof}\label{s_str}

\subsection{The main results}\label{subs_str_main}
Our main result can be split into two parts:
First the fluctuation bounds, for which we need much less assumptions and the arguments of which are applicable to a much broader range of set evolutions, as long as they satisfy a comparison principle.
Second the homogenization error, for which we need the remaining more specific requirements to make up for the lack of regularity from not having a guaranteed minimum speed of propagation.

\begin{theorem}[Main result for fluctuation bounds]\label{main_thm_fluct}
Let $(A,F)\in\Omega$ be a random coefficient field with a respective probability measure $\Pm$ on $(\Omega,\F(\Rd))$, which
\begin{itemize}
\item	has a finite range of dependence (Assumption \ref{sett_aP_fin}), and
\item	admits an effective minimum speed of propagation $\vmineff$ with overwhelming probability (Assumption \ref{veff_aP_star}).
\end{itemize}
Choose $0<\vtheta<1$ for the scaling of the target areas and set $h(r)\coloneqq r^{\vtheta}$.
Let $x_0\in\Rd$ and $S\subset\Rd$ have bounded boundary and be a fat set, as in Definition \ref{veff_def_fatstab}.
Set $R=\dist\pp{x_0,S}$.
Then there exists $\eps_0=\eps_0(\data,\vtheta,R,\diam(\partial S))>0$ such that for all $0<\eps\leq \eps_0$ the fluctuations of the truncated arrival times on the scale $\eps^{-1}$ can be bounded by
\begin{multline*}
	\PM{\abs{\eps\met[(A,F)]{h(\eps^{-1}R)}{\eps^{-1}x_0}{\eps^{-1}S}
		-	\EV{\eps\met[(A,F)]{h(\eps^{-1}R)}{\eps^{-1}x_0}{\eps^{-1}S}}}\geq C R \pp{\frac{\eps}{R}}^{\frac{1-\vtheta}{2}}\lambda}\\
	\leq		\max\Bp{	C	\exp\pp{-\lambda},	
					C	\exp\pp{-\frac{R^{\vtheta\ratespeed}\eps^{-\vtheta\ratespeed}}{C\log\pp{\pp{\frac{R}{\eps}}^{1-\vtheta}}}}}
\end{multline*}
for any $\lambda>0$ and for some $C=C(\data)>0$.
Here, $\ratespeed>0$ is the exponential rate for the probability bounds from Assumption \ref{veff_aP_star}.
\end{theorem}

This result follows from Proposition \ref{fluct_prop}, which treats general large scale fluctuations, by carefully plugging in the $\eps$-scaling.

\begin{theorem}[Main result for homogenization]\label{main_thm}
Let $(A,F)\in\Omega$ be a random coefficient field with a respective probability measure $\Pm$ on $(\Omega,\F(\Rd))$, which
\begin{itemize}
\item	is stationary, that is invariant with respect to translations (Assumption \ref{sett_aP_stat}),
\item	has a finite range of dependence (Assumption \ref{sett_aP_fin}), 
\item	has a constant, deterministic second-order coefficient $A$ (Assumption \ref{sett_aP_Aconst}),
\item	is isotropic, that is invariant with respect to rotations (Assumption \ref{sett_aP_ani}),
\item	admits an effective minimum speed of propagation $\vmineff$ with overwhelming probability for all forcings $F+\delF$ with $\delF$ in some interval around $0$ (Assumption \ref{veff_aP_star} for all $(A,F+\delF)$ with $\delF\in \bp{-\delFdel, \delFdel}$ for some $\delFdel>0$), and
\item	generates a candidate $\vhom$ for the homogenized speed that is $\hoel$-H\"older continuous with respect to changing the forcing by a constant (Assumption \ref{sett_aP_gen}).
\end{itemize}
Choose $0<\vtheta<1$ for the scaling of the target areas and set $h(r)\coloneqq r^{\vtheta}$.
Let $x_0\in\Rd$ and $S\subset\Rd$ be a fat set, as in Definition \ref{veff_def_fatstab}.
Set $R=\dist\pp{x_0,S}$.
Then for all $\eps>0$ the bias of the truncated arrival times on the scale $\eps^{-1}$ compared to the arrival times for the homogeneous evolution can be bounded by
\begin{multline*}
	\abs{	\EV{\eps\met[(A,F)]{h(\eps^{-1}R)}{\eps^{-1}x_0}{\eps^{-1}S}}
			-\met[hom,\vhom]{}{x_0}{S}}\\
	\leq		C R \pp{\frac{\eps}{R}}^\alpha \log(\max\Bp{R,2})^{\frac{3-\vtheta}{2}}\log(\eps^{-1})^{\frac{3-\vtheta}{2}}
\end{multline*}
for $C=C(\data,\delFdel, \vtheta)>0$ with $\alpha\coloneqq	\pp{1+4\pp{\frac{1}{1-\vtheta}+\frac{1}{\hoel}}}^{-1}$, where $\hoel$ is the H\"older coefficient from Assumption \ref{sett_aP_gen}.
For $\hoel=1$ we can hence obtain any rate $0<\alpha<\frac{1}{9}$ if we choose $\vtheta>0$ small enough.
\end{theorem} 

This result will be proven in Section \ref{s_hom}, using the fluctuation bounds from Section \ref{s_fluct}, the control on the propagation of influence via increasing the forcing from Section \ref{s_inf} and the quantitative convergence of the arrival times for the half-space approximations from Section \ref{s_lin}, which are used to define the homogenized speed $\vhom>0$.

\subsection{Overall strategy for the proof}
We give a rough overview of the overall proof and then go into more detail regarding the major sections in the following paragraphs.

In Section \ref{s_fluct}, we will  prove Theorem \ref{main_thm_fluct}, controlling the fluctuations of truncated arrival times for sets with bounded boundary.
This estimate will be heavily used for the overall homogenization proof.
For this section we only need the finite range of dependence from Assumption \ref{sett_aP_fin} and the very likely guaranteed minimum speed and stability from Assumption \ref{veff_aP_star}.
The proof is based on the martingale method of bounded differences, for which a quantity is written as the endpoint of a martingale with bounded increments and Azuma's inequality -- a martingale concentration inequality -- is applied.
We provide a detailed description of this method for set evolutions in Subsection \ref{subs_str_fluct}. 
However, in our setting we only obtain \textit{likely} bounded martingale differences.
We exploit these bounds with an alternative martingale concentration inequality, which we introduce and prove in Appendix \ref{s_azuma}.

In Section \ref{s_inf} we introduce the other major tool for the overall homogenization proof.
We provide a comparison principle that allows us to ignore the propagation of far-away influences at the cost of increasing the forcing by a small constant for one of the sets.
This is the only major part of the proof for which we have to work on the level of the specific level-set equation \eqref{intro_eq_uls}.
Since without an actual minimum speed of propagation we lack regularity on small scales, we are not able to deal with space-dependent second-order coefficients, such that from here on we have to use Assumption \ref{sett_aP_Aconst}, that is that $A(\cdot,\cdot)$ is constant in space and hence deterministic.

Next, we introduce the homogenized speed in Section \ref{s_lin}.
In the already established frameworks, this could be translated as the average speed corresponding to large-scale arrival times for half spaces.
In our case, we have to replace the half spaces with bounded half-space approximations in order to apply our fluctuation bounds.
Combining these fluctuation bounds from Section \ref{s_fluct} with the comparison principle from Section \ref{s_inf} and using the stationarity from Assumption \ref{sett_aP_stat}, we show that the large scale arrival times behave approximately linearly and hence obtain a quantitative convergence result for the averages.
However, as in Section \ref{s_inf} we are only able to do so at the cost of slightly changing the forcing. 
To properly apply this intermediate result, we will hence from here on use Assumption \ref{sett_aP_gen}, quantifying the continuity of the homogenized speed with respect to changing the forcing by a constant.
Hence, at the end of this section we prove Remark \ref{sett_rem_hoel} \eqref{sett_item_hoel_1}, that is that the choice of half-space approximations is arbitrary with respect to Assumption \ref{sett_aP_gen}.

Finally, we look at general fat sets in Section \ref{s_hom} to prove Theorem \ref{main_thm}.
On a mesoscopic level, we align the half-space approximations from Section \ref{s_lin} with the interface of the homogeneous first-order evolution in order to compare it to the heterogeneous second-order evolution. 
We use the isotropy from Assumption \ref{sett_aP_ani} to ensure the smoothness of the homogeneous evolution, allowing us to properly align the half-space approximations.
With this method we can lift the quantitative convergence result from Section \ref{s_lin} to general sets fat sets and obtain Theorem \ref{main_thm}.

\subsubsection{Fluctuation bounds for large-scale arrival times}\label{subsub_str_fluct}
In Section \ref{s_fluct} we will provide fluctuation bounds for large sets with bounded boundaries.
Let $h\geq h_0$ and $S\subset\Rd$ be an $h$-fat set with $x_0\in\Rd$ such that $\diam\pp{\partial S}$ and $\dist\pp{x_0,S}$ are not exponentially larger than $h$.
Then for the truncated arrival time for $S$ in the target area $\ov{B}_h(x_0)$ we obtain
\begin{equation*}
	\PM{\abs{\metb{h}{x_0}{S}-\EV{\metb{h}{x_0}{S}}}\geq \lambda}
	\leq		C\exp\pp{-\frac{\lambda}{C\sqrt{h\dist(x_0,S)}}}
	\quad\text{for }\lambda\leq \lambda_0
\end{equation*}
with $\lambda_0=h^\ratespeed\frac{\sqrt{h\dist(x_0,S)}}{\log\pp{\frac{\dist\pp{x_0,S}}{h}}}$, where $\ratespeed>0$ is the exponential rate from Assumption \ref{veff_aP_star}.
That is, we have exponential tail bounds which are capped essentially by the probability bounds from Assumption \ref{veff_aP_star}.
This cap is natural since we do not expect any control for coefficient fields in areas with no effective minimum speed of propagation on the scale $h$.

The proof for these bounds is based on the martingale method of bounded differences, using the framework developed in \cite{ACS14, ArCa14, ArCa18} as orientation.  
The truncated arrival time can be viewed as the endpoint of a martingale, whose increments are bounded to then apply a concentration result.
This martingale is the conditional expectation of the truncated arrival time with respect to the environment explored by the set evolution at different times.
We will describe this established framework in more detail in Subsection \ref{subs_str_fluct}.
The novelty of our result comes from applying this framework in a setting with no guaranteed minimum speed of propagation. 
This requires two types of innovation:
First, we use our concepts of an effective minimum speed and stable approximations to cover all the relevant use cases of an actual minimum speed on larger scales with some technical adjustments.
Second, we have no full guarantee that there actually are stable approximations and that the effective minimum speed is applicable, only a very high probability.
This yields bounds for the martingale increments which do not hold almost surely as usual, but only on a set of very high probability. 
We hence need a replacement for Azuma's inequality, the concentration result via which the fluctuation bounds are usually obtained.
We provide a suitable replacement in Appendix \ref{s_azuma}, via which we obtain the result.

Note that this method is applicable to much more general set evolutions, there are only four prerequisites:
\begin{itemize}
\item	The evolution admits a comparison principle.
\item	The range of dependence for the coefficient field is finite (Assumption \ref{sett_aP_fin}).
\item	With overwhelming probability, stable approximations exist and an effective minimum speed of propagation is applicable (Assumption \ref{veff_aP_star}).
\item	The set evolution only depends on the values of the coefficient field on the interface or in a small neighborhood of the overall interface.
\end{itemize}

\subsubsection{Bounding the propagation of influence at the cost of increasing the forcing}\label{subsub_str_inf}
Our building blocks for homogenizing the evolution of general large-scale sets will be a range of standardized disks.
Later in Section \ref{s_lin}, we will obtain good control on the evolution close to the center of these disks. 
For the final homogenization argument, we will align these disks with the boundary of much larger sets in order to extend the control on their evolution to arbitrary sets.
Therefore, we need a comparison principle for the evolution close to the center of the disk, which allows us to ignore the relation between the disk and the general set far away from this area.

In Section \ref{s_inf} we will show that for such a local comparison we can indeed ignore far away influences but at the cost of slightly increasing the forcing.
More precisely, we consider a stable set $S_1$, that is `ahead' of a second set $S_2$ within a large ball of radius $R>0$ -- outside of this ball, $S_2$ might be much larger than $S_1$, see Figure \ref{inf_fig_ball}. 
We then show that within in a much smaller ball of radius $r>0$ and for a limited time $T>0$, the evolution of $S_1$ stays ahead.
However, this only works at the cost of increasing the forcing in the evolution of $S_1$. 
More precisely, if for some $h\geq 1$
\begin{align*}
 	S_1							&\supset	\pp{S_2+\ov{B}_h	}\cap \ov{B}_R,\\
\intertext{then}
	\reach[(A,F+\delF)]{t}{S_1}	&\supset	\reach[(A,F)]{t}{S_2}\cap \ov{B}_r	
	&&\text{for all }0\leq t\leq T
\end{align*}
if the change to the forcing satisfies
\begin{align*}
	\delF\geq	C\frac{T^\frac{3}{4}}{(R-h-r)^{\frac{1}{2}}}.
\end{align*}
Since we are interested in truncated arrival times on the scale $r$, due to the truncation we only need to compare the sets up to time $T\approx r$. 
If we choose $R\approx r^\beta>>r,h$ for some $\beta\geq 1$, then we need $\delF\gtrsim r^{\frac{3}{4}-\frac{\beta}{2}}$.
For only a small change of the forcing to be sufficient, we hence require $\beta>\frac{3}{2}$ for the scaling of the larger radius.

The same result holds true if we replace the balls $\ov{B}_r$ and $\ov{B}_R$ with the half-balls $\ov{B}_r\cap\hsp{e}$ and $\ov{B}_R\cap\hsp{e}$ with $\hsp{e}\coloneqq \Bp{x\in\Rd\,:\,x\cdot e\geq 0}$ for $e\in\Sd$.
In this case we have to assume in addition that $S_1$ covers the boundary corresponding to the half-space, that is
\begin{equation*}
	S_1	\supset	\pp{\partial\hsp{e}+\ov{B}_h}\cap \ov{B}_R.
\end{equation*}
The proof for these results is an adaptation of the proof for the comparison principle.
For this, it is crucial that the locally large set is stable, which for the  solution of the level-set equation implies increasing sub-level sets. 
Thus, the solution is decreasing and we obtain a one-sided bound for the time derivative.
This is the only way for us to impose any control on the time derivatives of the level set evolution.
If we knew that there was a general minimum speed of propagation as in the setting from \cite{ArCa18}, then we would obtain a decreasing, Lipschitz-continuous solution.
Since we lack this regularity, we can not deal with space dependent second order terms and additionally need to increase the forcing for the evolution of the stable set to compensate.
Hence from here on we require Assumption \ref{sett_aP_Aconst}, that is that the second-order coefficient $A$ is constant in space and thus deterministic due to the stationarity.

\subsubsection{Introduction of the homogenized speed and quantitative convergence bounds.}\label{subsub_str_lin}
Next, in Section \ref{s_lin} we introduce our candidate for the homogenized speed $\vhom$, for which we need both the fluctuation bounds from Section \ref{s_fluct} and the bounds on the propagation of influence from Section \ref{s_inf} corresponding to the Subsections \ref{subsub_str_fluct} and \ref{subsub_str_inf}.
We would have liked to straight-up define $\vhom$ as the average large scale speed for the evolution of the half-spaces, that is
\begin{equation}\label{str_eq_vhom}
	\vhom=\vhom(e)	=	\pp{\lim_{r\rightarrow\infty}	\frac{1}{r}\EV{\metb{h(r)}{re}{\hsm{e}}}}^{-1}
\end{equation}
with the half-spaces $\hsm{e}=\Bp{x\in\Rd\,:\,x\cdot e \leq 0}$ for $e\in\Sd$.
The invariance with respect to $e\in\Sd$ is due to the isotropy from Assumption \ref{sett_aP_ani}, although technically we will only use this property later in Section \ref{s_hom}.
If we were able to control the fluctuations of arrival times for sets with unbounded boundary, then we would be able to show the approximate sub-linearity of the truncated arrival times $r\mapsto \met{h(r)}{re}{\hsm{e}}$ for half spaces $\hsm{e}=\Bp{x\in\Rd\,:\,x\cdot e \leq 0}$ with $h(r)<<r$, which would yield the existence of the limit in \eqref{str_eq_vhom}.
If in addition we had approximate super-linearity, we would obtain quantitative bounds for the convergence.

However, we only control the fluctuations for sets with bounded boundary. 
Thus, we approximate the boundary of the half space with disks $r\mapsto \hsma{e}{\beta}{h(r)}{r}$ of radius $\sim r^{\beta}$ for $\beta\geq \frac{3}{2}$ and thickness $\sim h(r)<<r$, see Figure \ref{lin_fig_hs-approx}.
In order to justify the existence of the limit and quantify the convergence in \eqref{str_eq_vhom}, we will replace $\hsm{e}$ with $\hsma{e}{\beta}{h(r)}{r}$.
We then show that the arrival times are approximately linear in $r$, which yields the convergence of the averages and quantitative error bounds.

First, just using the fluctuation bounds from Section \ref{s_fluct}, we obtain approximate superlinearity for the arrival times 
\begin{equation*}
	r\mapsto \met{h(r)}{re}{\hsma{e}{\beta}{h(r)}{r}}
\end{equation*}
in the sense that 
\begin{multline}\label{str_eq_sub}
	\EV{\metb{h(r_1)}{r_1 e}{\hsma{e}{\beta}{h_1(r_1)}{r_1}}}
				+ \EV{\metb{h(r_2)}{r_2 e}{\hsma{e}{\beta}{h_2(r_2)}{r_2}}}\\
	\geq			\EV{\metb{h(r_1+r_2)}{\pp{r_1+r_2}e}{\hsma{e}{\beta}{h(r_1+r_2)}{r_1+r_2}}}
				- C\sqrt{h(r_2)r_2}\log\pp{r_1}.
\end{multline}
The main idea for proving this is to place translations of the small disk $\hsma{e}{\beta}{h_2(r_2)}{r_2}$ in $\hsma{e}{\beta}{h(r_1+r_2)}{r_1+r_2}$ such that the translated target areas $\ov{B}_{h_2(r_2)}(r_2e)$ cover the smaller disk $r_2e+\hsma{e}{\beta}{h_1(r_1)}{r_1}$.
For the first step, we wait until all of these disks have reached the corresponding target area $r_2$ above them, see Figures \ref{lin_fig_Nplus} and \ref{lin_fig_Nplus_smallDisk}. 
For the second step, now that the upper, smaller disk is completely covered, we take the time until this disk reaches the target area $\ov{B}_{h(r_1)}{(r_1+r_2)}$.
By comparison principle, adding step 1 and step 2 takes longer than the time which the larger disk needs to reach the target area $\ov{B}_{h(r_1+r_2)}{(r_1+r_2)}$.
Taking the expected value, we can resolve the maximum for the arrival times over the distance $r_2$ from the first step at the cost of the typical-fluctuation-sized error $\sim C\sqrt{h_2r_2}\log\pp{r_1}$ because due to the exponential tail bounds from Section \ref{s_fluct} it is unlikely that any of the single arrival times is much larger.

We then show the approximate superlinearity.
The idea of this proof is similar to the one of \eqref{str_eq_sub}:
Instead of waiting until the intermediate disk is fully covered and then restarting with that disk, we now restart with that disk as soon as a larger intermediate disk is reached for the first time by one of the lower disks, see Figure \ref{lin_fig_Nminus}.
This yields a faster time than the original arrival time for the larger disk.
However, in this case, in addition to the fluctuation bounds from Section \ref{s_fluct} we need the comparison principle from Section \ref{s_inf}:
This comparison principle allows us to favorably compare the smaller disks $\hsma{e}{\beta}{h_1(r_1)}{r_1}$ and $\hsma{e}{\beta}{h_2(r_2)}{r_2}$ to the large disk at the cost of increasing the forcing for their evolution.
Thus, we obtain 
\begin{multline}\label{str_eq_sup}
	\EV{\metb[(A,F+f(r_1))]{h(r_1)}{r_1 e}{\hsma{e}{\beta}{h_1(r_1)}{r_1}}}
				+ \EV{\metb[(A,F+f(r_2))]{h(r_2)}{r_2 e}{\hsma{e}{\beta}{h_2(r_2)}{r_2}}}\\
	\leq		\EV{\metb[(A,F)]{h(r_1+r_2)}{\pp{r_1+r_2}e}{\hsma{e}{\beta}{h(r_1+r_2)}{r_1+r_2}}}
				+ C\sqrt{h(r_1\!+\!r_2)r_2}\log\pp{r_1},
\end{multline}
where the forcing is changed by the constant $f(r)=r^{-\frac{\beta}{2}+\frac{3}{4}}$, depending on the scaling of the width of the disks.

We now look at the limit of the averaged arrival times, having established the approximate linearity.
Already with just the sub-linearity, we obtain the existence of the limits
\begin{align*}
	\delF\mapsto		\vhom[\delF,\beta]
	\coloneqq	 \pp{\lim_{r\rightarrow\infty}\frac{1}{r}\EV{\metb[(A,F+\delF)]{h(r)}{re}{\hsma{e}{\beta}{h(r)}{r}}}}^{-1},
\end{align*}
where we use that if the underlying assumptions hold for the random coefficient field $(A,F-\delFdel)$, then they hold for all fields $(A,F+\delF)$ with $\delF\geq -\delFdel$.
The reason that we need to consider these limits after changing the forcing by different constants $\delF$ is that the quantitative convergence bounds, which we obtain from \eqref{str_eq_sub} and \eqref{str_eq_sup}, are given by 
\begin{multline}\label{str_eq_linerr}	
	\vhom[\delF,\beta]^{-1}r - C\sqrt{h(r)r}\log(r)	
	\leq 	\EV{\metb[(A,F+\delF)]{h(r)}{re}{\hsma{e}{\beta}{h(r)}{r}}}\\ 
	\leq		\vhom[\delF-Cf(r),\beta]^{-1}r + C\sqrt{h(r)r}\log(r).
\end{multline}
We would hence get proper quantitative bounds for the convergence if we knew that $\delF\mapsto \vhom[\delF,\beta]$ is H\"older continuous, which  Assumption \ref{sett_aP_gen} yields for some $\beta=\genhoel>\frac{3}{2}$. 
We justified this assumption in Remark \ref{sett_rem_hoel}.
At the end of Section \ref{s_lin}, we will show that if Assumption \ref{sett_aP_gen} holds for $\beta=\genhoel$, then it holds for all $\beta>\frac{3}{2}$ via the result from Section \ref{s_inf} bounding the propagation of far-away influence. 
In fact, if Assumption \ref{sett_aP_gen} holds, then the limit in \eqref{str_eq_vhom} for the full half spaces exists and is the same as $\vhom[0,\beta]$ for all $\beta>\frac{3}{2}$.

\subsubsection{Homogenization for general sets using half-space approximations as building blocks}
The final argument for the homogenization result from Theorem \ref{main_thm} is based on using the half-space approximations from Section \ref{s_lin} and Subsection \ref{subsub_str_lin} as building blocks.
The first step is to reduce the case of a general fat set $S\subset\Rd$ to two much simpler sets:
Given a target point $x_0\in\Rd$ with $R=\dist(x_0,S)$, we sandwich $\eps^{-1}S$ between
\begin{equation*}
	\ov{B}_{h(\eps^{-1}R)}(y_\eps)	\subset \eps^{-1}S	\subset	\Rd\setminus B_{\eps^{-1}R}(\eps^{-1}x_0)
\end{equation*}
with $y_\eps$ chosen such that $\dist(\eps^{-1}x_0,\ov{B}_{h(\eps^{-1}R)}(y_\eps))=\eps^{-1}\dist(x_0,S)$, see Figure \ref{hom_fig_sandwich}.
This is possible if we choose $\eps$ small enough such that $S$ is $h(\eps^{-1}R)$-fat, see Definition \ref{veff_def_fatstab}.  
The reason why the ball and the round hole are much simpler to treat is because they have self-similar first-order evolutions:
\begin{align}
	\reach[hom,\vhom]{t}{\ov{B}_{h(\eps^{-1}R)}}	&=	\pp{h(\eps^{-1}R)+\vhom t}\ov{B}_{1},\label{str_eq_ball}\\
	\reach[hom,\vhom]{t}{\Rd\setminus B_{\eps^{-1}R}}	&=	\Rd\setminus \pp{\eps^{-1}R-\vhom t}B_{1}\label{str_eq_hole}.
\end{align}
In particular, the arrival times at $\eps^{-1}x_0$ with respect to the homogenized first-order evolution are the same for the small ball, the original set $\eps^{-1}S$, and the large hole.
Thus, in order to obtain Theorem \ref{main_thm}, it is sufficient to show two things:
\begin{enumerate}
\item	On large scales the expansion of a ball with respect to the heterogeneous second-order evolution is not much slower than \eqref{str_eq_ball}.
\item	On large scales the shrinking of a hole with respect to the heterogeneous second-order evolution is not much faster than \eqref{str_eq_hole}.
\end{enumerate}
We obtain this by comparing the evolutions for discrete time-steps.
For the ball, starting with radius $R_0=h(\eps^{-1}R)$, we incrementally choose expansions $(r_n)_n$ until the radius $R_N=R_0+\sum_{n=1}^N r_n$ covers the distance between $\eps^{-1}S$ and $\eps^{-1}x_0$.
We show that the expected value for the time $t_{n+1}$ until the ball $\ov{B}_{R_n}$ fully envelops the next larger ball $\ov{B}_{R_{n+1}}$ is less than $\vhom^{-1}(R_{n+1}-R_n)+\tau_{n+1,ball}$ for a controlled error $\tau_{n+1,ball}$.
For the hole, starting with radius $R_0=\eps^{-1}R$, we incrementally choose expansions $(r_n)_n$ until $R_N=R_0-\sum_{n=1}^N r_n$ corresponds to the target area of radius $h(\eps^{-1}R)$.
We show that the expected value for the time $t_{n+1}$ until $\Rd\setminus B_{R_n}$ would first leave $\Rd\setminus B_{R_{n+1}}$ is more than $\vhom^{-1}\pp{R_{n}-R_{n+1}}-\tau_{n+1,hole}$ for an error term $\tau_{n+1,hole}$ similar to $\tau_{n+1,ball}$.
By comparison principle, the overall homogenization error can thus be bounded by adding up these error terms.
  
Since the shapes are self-similar, each step of this iteration is essentially the same.
For first-order motions with an arbitrary continuous speed $e\mapsto v(e)$, there does always exist such a self-similar, convex grower $K$ given by $\reach[hom,v]{t}{\Bp{0}}=tK$, which in our isotropic case is the ball. 
There are two potential issues in the anisotropic case though: 
First, if $K$ is convex but not smooth, then it is not possible to properly align the half-space approximations.
Second, in general it is not true that the mirror image $-K$ yields a self-similar shrinking hole as in the isotropic case. 

For the error terms $(\tau_{n,ball})_n$ and  $(\tau_{n,hole})_n$, there are three contributions, which we will exemplify for the shrinking hole; the proof for the growing ball is very similar.
We will choose the discrete step sizes $r_{n}=R_{n}-R_{n-1}$ to optimally balance these three error contributions.
In order to relate the time $t_{n}$ until $\Rd\setminus B_{R_{n-1}}$ first leaves $\Rd\setminus B_{R_{n}}$ to $\vhom^{-1}r_{n}$, 
we use half-space approximations $\hsma{e}{\beta}{h(r_{n})}{r_{n}}$ with target area $\ov{B}_{h(r_n)}(r_{n}e)$ from Subsection \ref{subsub_str_lin}, which were used to define the homogenized speed.
We align translations of these half-space approximations with the boundary of $\Rd\setminus B_{R_{n-1}}$, but place them slightly in front, see Figure \ref{hom_fig_hole-gap}.
The size $\rho_{n}$ of the resulting gap due to the curvature is given by 
\begin{equation*}
 	\pp{R_{n-1}-\rho_{n}}^2+\pp{C_\beta r_{n}^\beta}^2
 	=	R_{n-1}^2,
 	\quad\text{implying}\quad
 	\rho_{n}\leq C(\beta,\data)\frac{r_{n}^{2\beta}}{R_{n-1}},
 \end{equation*} 
The combined target areas for these half-space approximations approximate $\partial B_{R_{n}}$ up to an error corresponding to $\rho_{n}$.
Making up for the propagation of the far away influences with the half-space approximations locally ahead of the boundary of $\Rd\setminus B_{R_{n-1}}$ as described in Subsection \ref{subsub_str_inf}, it takes longer for the hole $\Rd\setminus B_{R_{n-1}}$ to reach these target areas than for the respective half-space approximation as long as for their evolution we slightly increase the forcing by $f(r_{n})=f_{\beta}(r_{n})$.
The time $t_{n}$ is hence larger than the minimum of these arrival times for half-space approximations up to an error of approximately $\vmineff^{-1}\rho_{n}$.
Taking the expected value, we can resolve this minimum at the cost of a typical-fluctuation-sized error since the essentially exponential tail bounds for the fluctuations make it unlikely that any of the single arrival times is much smaller.
Plugging in the convergence rate of the arrival times for half-space approximations from \eqref{str_eq_linerr} and $\rho_{n}\approx\frac{r_{n}^{2\beta}}{R_{n-1}}$, we obtain
\begin{align*}
	\EV{t_{n}}	&\geq		\vhom^{-1}r_{n}
						-\pp{\vhom^{-1}-\vhom[f_\beta(r_{n})]^{-1}}r_{n}
						-C \sqrt{h(r_{n})r_{n}}\log\pp{R_{n-1}}
						-C \frac{r_{n}^{2\beta}}{R_{n-1}}\\
				&\eqqcolon	\vhom^{-1}r_{n}	-	\tau_{n}.
\end{align*}
The first error term comes from comparing the homogenized speed for $(A,F+\delF)$ with $\delF=f_{\beta}(r_{n})$ and $\delF=0$.
The second error term comes from using the fluctuation bounds and the convergence rate to the homogenized speed.
The third error term comes from the gap between the half-space approximations and the boundary of $\Rd\setminus B_{R_{n-1}}$.

We choose $r_{n}=r_{n}(R_{n-1})$ and $\beta>\frac{3}{2}$ in order to optimally balance these three error contributions and obtain
\begin{align*}
	\tau_{n}\lesssim \sqrt{r_n^{1+\vtheta}}\log\pp{R_{n-1}}
	\quad\quad\text{for }\quad
	r_n\approx R_{n-1}^{\frac{2}{1-\vtheta+4\pp{\frac{1-\vtheta}{\hoel}+1}}},
\end{align*}
yielding a recursive scheme for the choice of the enlargements, where $\hoel$ is the H\"older coefficient from Assumption \ref{sett_aP_gen}.
Adding up these errors for each step until $R_N=h(\eps^{-1}R)$ and hence the original target area $\ov{B}_{h(\eps^{-1}R)}$ is reached yields the overall error.
Similar considerations yield an analogous result for the growing hole, such that in summary we obtain Theorem \ref{main_thm} after multiplying with $\eps$.

\subsection{The proof of the fluctuation bounds}\label{subs_str_fluct}
We will illustrate the main ideas behind the proof for the fluctuation bounds of the arrival times in Section \ref{s_fluct} in more detail.
For this, we prove an analogous result under the much more restrictive assumption that there exists a guaranteed minimum speed of propagation $\vmin>0$:
\begin{equation}\label{str_eq_vmin}
	\reach[(A,F)]{t}{S}\supset\reach[(A,F)]{s}{S}+\vmin(t-s)\ov{B}_1
	\,\,\,\text{for all }t\geq s,\,S=\bigcup\Bp{\ov{B}_1(x)\subset S}.
\end{equation}
We know from \cite{ArCa18} that this is satisfied for example if
\begin{align*}
	A(x,\xi)&=a\pp{\Id-\frac{\xi}{\abs{\xi}}\otimes\frac{\xi}{\abs{\xi}}},
	&	F(x,\xi)&=F(x)	\text{ with }F^2-(d-1)a|DF|\geq c,
\end{align*}
for some $c>0$ large enough since \eqref{str_eq_vmin} corresponds to Lipschitz continuity of the metric problem for $S$. 
This is because for geometric motions the metric problem corresponds to the arrival times, if it is well-posed -- which this assumption implies.
Hence, this case has been covered by the fluctuation bounds for the metric problem in \cite{ArCa18}.

The ideas presented here originate from \cite{Kes93} and have been extended and adapted to increasingly general Hamilton--Jacobi equations in the series \cite{ACS14, ArCa14, ArCa18}, finally including this case in \cite{ArCa18}. 
However, the proof of \cite[Proposition 3.1]{ArCa18} does not just cover interface motions but also general viscous Hamilton--Jacobi equations and therefore comes with significant additional technical baggage. 
Therefore, we think that it is worthwhile to present a concise proof in order to illustrate the already established main components and provide orientation for readers unfamiliar with these kind of arguments.
For this, we use our language of arrival times and set evolutions instead of working on the level of the equation describing the metric problem.

The technical innovations allowing us to adapt the ideas presented here to our setting later in Section \ref{s_fluct} are fourfold:
\begin{itemize}
\item	Usage of arrival times in a target area instead of solutions to the metric problem,
\item	Usage of stable approximations and an \textit{effective} minimum speed of propagation instead of a strict minimum speed of propagation,
\item	Applying the arguments on a set of high instead of full probability,
\item	An alternative to Azuma's inequality requiring good martingale bounds with just high probability (instead of full probability).
\end{itemize}
The simpler result which we will prove to demonstrate the main ideas is the following:
\begin{proposition}[Fluctuation bounds for guaranteed minimum speed of propagation]\label{str_prop_fluct}
Assume $\Pm$ is a probability measure on $(\wt{\Omega},\F(\Rd))$ with $\wt{\Omega}=\Bp{(A,F)\in\Omega\,:\,\eqref{str_eq_vmin}\text{ holds}}$ with a finite range of dependence as in Assumption \ref{sett_aP_fin}.
Then for the arrival times from Definition \ref{intro_def_met} there exists $C=C(\data,\vmin)>0$ such that for all $x_0\in\Rd$ and $S\subset\Rd$ with $S=\bigcup\Bp{\ov{B}_1(x)\subset S}$ as well as $\partial S$ bounded, for any $\lambda>0$ there holds
\begin{equation*}
	\PM{\abs{\met[]{}{x_0}{S}-\EV{\met[]{}{x_0}{S}}}\geq \lambda}
	\leq		C\exp\pp{-\frac{\lambda^2}{C(1+\dist(x_0,S))}}.
\end{equation*}
\end{proposition}

\subsubsection{Fluctuation bounds via Azuma's inequality}
We obtain this inequality by viewing $\met[]{}{x_0}{S}-\EV{\met[]{}{x_0}{S}}$ as the endpoint of a martingale, for which we apply a standard concentration result:
We set $(\wt{\M}_n)_{n\in\Nz}\coloneqq (\M_{n\Delta t})_n$ with $\Delta t>0$ to be chosen later and
\begin{equation*}
	\M_t\coloneqq \EV[\G_t]{\met[]{}{x_0}{S}} - \EV{\met[]{}{x_0}{S}},
\end{equation*}
with
\begin{align}
	\G_0		&\coloneqq	\Bp{\emptyset, \Omega},\notag\\
	\G_t		&\coloneqq 	\text{``$\sigma$-algebra corresponding to information on the state of $\reach{s}{S}$}\label{str_eq_defGt1}\\ 
			&			\qquad\qquad\qquad\qquad\qquad\quad		\text{and the coefficient field covered by it for $s\leq t$''},\notag
\end{align}
essentially uncovering the information on the coefficient field as the evolution of $S$ passes through. 
In particular, \eqref{str_eq_vmin} implies that $x_0\in\reach[(A,F)]{\timebound{}{x_0}{S}}{S}$ for all $(A,F)\in\wt{\Omega}$ with $\timebound{}{x_0}{S}\coloneqq \vmin^{-1}\dist(x_0,S)$, therefore $\met[]{}{x_0}{S}$ should be measurable with respect to $\G_{\timebound{}{x_0}{S}}$ and hence indeed the `endpoint' of the martingale would be
\begin{align*}
	\wt{M}_N	&=	\met[]{}{x_0}{S} - \EV{\metb[]{}{x_0}{S}}
	&	\text{for }N	\coloneqq \left\lceil\frac{\timebound{}{x_0}{S}}{\Delta t}\right\rceil
						=	\left\lceil\frac{\dist(x_0,S)}{\vmin\Delta t}\right\rceil.
\end{align*}
The concentration result which we want to apply to $(\wt{M}_n)_n$ is Azuma's inequality:

\begin{proposition}[Azuma's inequality \cite{Az67}]\label{str_prop_azuma}
Let $(\M_n)_n\in\Nz$ be a martingale which almost surely satisfies
\begin{equation*}
	\abs{\M_{n}-\M_{n-1}}\leq c_{n}	\quad\text{for all }n\in\N
\end{equation*}
for some sequence $(c_n)_n\subset\R_+$.
Then for all $N\in\N$ and $\lambda>0$ there holds
\begin{equation*}
	\PM{\abs{\M_N-\M_0}\geq \lambda}
	\leq		2\exp\pp{-\frac{\lambda^2}{2\sum_{n=1}^Nc_k^2}}.
\end{equation*}
\end{proposition}

If we could show that there was $C>0$ such that for all $0\leq s\leq t$ there holds 
\begin{equation}\label{str_eq_MtMs}
	\abs{\M_t-\M_s}\leq C(t-s)+C,
\end{equation}
then we could apply Azuma's inequality to obtain
\begin{equation*}
	\PM{\abs{\M_N-\M_0}\geq \lambda}
	\leq		2\exp\pp{-\frac{\lambda^2}{2C^2N(\Delta t + 1)^2}}.
\end{equation*}
Choosing for example $\Delta t=1$ then yields Proposition \ref{str_prop_fluct}.
Essentially, we need to show that `information' gained while evolving for $\Delta t$ impacts the current best guess for the arrival time at most linearly in $\Delta t$.
Due to the finite range of dependence, this will come down to comparing the slowest to the fastest possible evolution for this time step.

\subsubsection{Measurability and the application of the finite range of dependence assumption}
In order to point out some of the key considerations, we first need to make the definition of $\G_t$ from \eqref{str_eq_defGt1} more precise:
Due to measurability issues, we need to approximate the possible states of $\reach{t}{S}$ with countable many sets.
Fixing $h>0$, we define the discrete approximation of a set $M\subset\Rd$ as
\begin{equation*}
	K_h(M)\coloneqq \bigcup\Bp{x+[0,h]^d\,:\,x\in h\Zd,\,M\cap \pp{x+[0,h]^d}\neq\emptyset}.
\end{equation*}
Clearly, $\Bp{K_h(M)\,:\,M\subset\Rd}$ is not countable and hence too large for the discrete state space approximating the set evolutions.
However, due to the minimum speed from \eqref{str_eq_vmin} and the maximum speed of propagation (see Lemma \ref{not_lem_vmax}), for any $t\geq 0$ there holds
\begin{equation*} 
	\reach{t}{S}\in\Bp{M\subset\Rd\,:\,S\subset M\subset S+B_R\text{ for some }R>0},
\end{equation*}
and since $\partial S$ is bounded, there are countable many sets $(K_i)_{i\in\N}$ such that
\begin{equation*}
	\Bp{K_i\,:\,i\in\N}=\Bp{K_h(M)\,:\,M\subset\Rd \text{ with }S\subset M\subset S+B_R\text{ for some }R>0}.
\end{equation*}
Based on these sets we define the events $\Bp{E_i(t)}_{i\in\N}$ capturing the approximate state of the set evolution at time $t\geq 0$ as
\begin{equation*}
	E_i(t)	\coloneqq 
	\Bp{(A,F)\in\wt{\Omega}\,:\, K_h\pp{\reach[(A,F)]{s}{S}}=K_i}.
\end{equation*}
Note that $\bigcup_i E_i(t)=\wt{\Omega}$ for all $t\geq 0$.
Finally, this allow us to properly define $\G_t$ as
\begin{equation*}
	\G_t		\coloneqq 	\text{$\sigma$-algebra generated by events }E_i(s)\cap F, ~0\leq s\leq t,~F\in\F\pp{K_i},
\end{equation*}
where $\F\pp{K_i}$ was the $\sigma$-algebra containing the information of the coefficient field in $K_i$ as introduced in \eqref{sett_eq_sigmaF}.
Essentially, $\G_t$ contains the information of the approximate states of the set evolution up to time $t$ and on the coefficient field in the area covered by these approximate states.
In order to work with the conditional expectations of arrival times with respect to $\G_t$ from the definition of the martingale, we need the following measurability properties.

\begin{lemma}[Measurability of arrival times]\label{str_lem_meas}
In the setting of Proposition of \ref{str_prop_fluct} and with $\G_t$ and $\M_t$ defined as above, the following holds with the $\sigma$-algebras $\Bp{\F(U)})$ from \eqref{sett_eq_sigmaF} containing the information on the coefficient field in $U\subset\Rd$ respectively.
\begin{enumerate}[label=(\alph*), left=0pt]
\item	\label{str_meas_met}
		$\met{}{x_0}{S}$ is $\F(\ov{\comp{S}})$-measurable,
\item	\label{str_meas_Ei}
		$E_i(t)$ is $\F(K_i)$-measurable for all $i\in\N$ and $t\geq 0$,
\item	\label{str_meas_metchar}
		$\Bp{(A,F)\,:\,\met[(A,F)]{}{x_0}{S}\leq t}$ and $\met{}{x_0}{S}\charfun[\Bp{(A,F)\,:\,\met[(A,F)]{}{x_0}{S}\leq t}]$ are $\G_t$-measurable,
\item	\label{str_meas_condE}
		$\EV[\G_t]{\met{}{x_0}{K_h\pp{\reach{t}{S}}+\ov{B}_1}}
			=\sum_{i=1}^\infty\EV{\met{}{x_0}{K_i+\ov{B}_1}}\charfun[E_i(t)]$ for all $t\geq 0$.
\end{enumerate}
\end{lemma}

The proof of \ref{str_meas_condE} is the one and only point where we use the finite range of dependence from Assumption \ref{sett_aP_fin}.
Resolving the conditional expectation will be key when we bound the martingale increments below -- the combination of \ref{str_meas_metchar} and \ref{str_meas_condE} allows us to do so.
Further note that \ref{str_meas_metchar} implies that indeed 
	$\wt{M}_N = \met[]{}{x_0}{S} - \EV{\metb[]{}{x_0}{S}}$ 
since due to the minimum speed we have $\met{}{x_0}{S}\leq \timebound{}{x_0}{S}$ and thus the $\G_{N\Delta t}$-measurability of $\met[]{}{x_0}{S}$.

\begin{proof}[Sketch proof of Lemma \ref{str_lem_meas}]
In order to show measurability with respect to information about the coefficient field on only a certain region $U\subset\Rd$, we use some localization $(A,F)_{U}\in\wt{\Omega}$ essentially depending only on $A|_U$ and $F|_U$ with $(A,F)_U=(A,F)$ on $U$.
The precise definition of this localization is flexible, see \cite[Section 3.2]{ArCa18} or our restricted coefficient fields from Definition \ref{box_def_stable}.
For example, one could choose $(A,F)_U=0$ on $\Rd\setminus\pp{U+\ov{B}_1}$ with a smooth transition between $U$ and $\Rd\setminus\pp{U+\ov{B}_1}$.

Regarding \ref{str_meas_met}, due to the guaranteed growth from \eqref{str_eq_vmin} we have $\partial \reach[(A,F)]{t}{S}\subset \ov{\comp{S}}$ for all $t\geq 0$ and hence the evolution depends only on the coefficient field on $\ov{\comp{S}}$, yielding $\reach[(A,F)]{t}{S}=\reach[(A,F)_{\ov{\comp{S}}}]{t}{S}$ for all $t\geq 0$. 
From here it is straightforward to show \ref{str_meas_met}, see our main measurability criterion in Lemma \ref{meas_lem_closure}.

Analogously, regarding \ref{str_meas_Ei}, $E_i(t)$ is $\F(K_i)$-measurable because we have 
\begin{equation*}
	E_i(t)	=
	\Bp{(A,F)\in\wt{\Omega}\,:\, K_h\pp{\reach[(A,F)_{K_i}]{t}{S}}=K_i}.
\end{equation*}
The inclusion ``$\subset$'' holds since if $K_h\pp{\reach[(A,F)]{t}{S}}=K_i$, then $\reach[(A,F)]{s}{S}\subset K_i$ for all $0\leq s\leq t$ due to the minimum speed from \eqref{str_eq_vmin} and hence 
	$\reach[(A,F)]{s}{S}=\reach[(A,F)_{K_i}]{s}{S}$
because the evolution only depends on the coefficient field in $K_i$ for all $0\leq s \leq t$. 
The reverse inclusion ``$\supset$'' follows in the same way.

Regarding \ref{str_meas_metchar}, with $\reach[(A,F)]{s}{S}=\reach[(A,F)_{K_i}]{s}{S}$ for $s\leq t$ and $(A,F)\in E_i(t)$ we have
\begin{equation*}
	\Bp{(A,F)\,:\,\met[(A,F)]{}{x_0}{S}\leq t}
	=	\bigcup_{i\in\N}\pp{\Bp{(A,F)\,:\,x_0\in\reach[(A,F)_{K_i}]{t}{S}}\cap E_i(t)}
\end{equation*}
which is $\G_t$-measurable due to the definition of $\G_t$. 
For the same reason, we obtain
\begin{equation*}
	\met[(A,F)]{}{x_0}{S}
	=	\sum_{i\in\N}\met[(A,F)_{K_i}]{}{x_0}{S}\charfun[E_i(t)]
\end{equation*}
for $(A,F)\in\Bp{(A',F')\,:\,\met[(A',F')]{}{x_0}{S}\leq t}$ and thus $\G_t$-measurability.

Finally, \ref{str_meas_condE} follows immediately by summing over $i\in\N$ if we can show 
\begin{equation*}
	\EV[\G_t]{\met{}{x_0}{K_i+\ov{B}_1}\charfun[E_i(t)]}	=	\EV{\met{}{x_0}{K_i+\ov{B}_1}}\charfun[E_i(t)],
\end{equation*}
which follows if for any $E\in\G_t$ we can show that
\begin{equation*}
	\EV{\met{}{x_0}{K_i+\ov{B}_1}\charfun[E_i(t)]\charfun[E]}
	=	\EV{\met{}{x_0}{K_i+\ov{B}_1}}\PM{E_i(t)\cap E}.
\end{equation*}
This is the case due to the finite range of dependence from Assumption \ref{sett_aP_fin}:
From \ref{str_meas_met} we know that $\met{}{x_0}{K_i+\ov{B}_1}$ is $\F(\ov{\comp{\pp{K_i+\ov{B}_1}}})$-measurable, while $E\cap E_i(t)\in\F(K_i)$:
It is sufficient to prove this for $E=E_j(s)\cap F$ with $s\leq t$ and $F\in\F(K_j)$ due to the definition of $\G_t$.
We know from \ref{str_meas_Ei} that $E_j(s)\in\F(K_j)$ and hence $E\in\F(K_j)$.
If $E_i(t)\cap E_j(s)=\emptyset$, then we are done.
However, if there is $(A,F)\in E_i(t)\cap E_j(s)$, then due to the guaranteed minimum speed from \eqref{str_eq_vmin} there holds
\begin{equation*}
	K_j	=	K_h\pp{\reach[(A,F)]{s}{S}}	\subset	K_h\pp{\reach[(A,F)]{t}{S}}	=	K_i
\end{equation*}
and hence $E_j(s)\cap F\cap E_i(t)\in\F(K_i)$.
\end{proof}

\subsubsection{Bounds for the martingale increments}\label{subsub_str_MtMs}
With the measurability properties from Lemma \ref{str_lem_meas} established, we can now obtain bounds for the martingale increments as in \eqref{str_eq_MtMs} and thus prove Proposition \ref{str_prop_fluct}.

\begin{proof}[Sketch proof of Proposition \ref{str_prop_fluct}]
It remains to find $C>0$ such that \eqref{str_eq_MtMs} holds for all $0\leq s\leq t$.
For $\EV[\G_s]{\met[]{}{x_0}{S}}$, that is the `guess' for the arrival time given the information from $\G_s$, we consider two cases: 
The case where $x_0$ has already been reached at time $s$, and the case where $x_0$ is still outside of the already covered area. 
We have
\begin{align*}
	\M_s		&=	\EV[\G_s]{\met[]{}{x_0}{S}} - \EV{\metb[]{}{x_0}{S}}\\
			&=	\met{}{x_0}{S}\charfun[\Bp{\met{}{x_0}{S}\leq s}]
					+	\EV[\G_s]{\pp{s+\met[]{}{x_0}{\reach{s}{S}}}\charfun[\Bp{\met{}{x_0}{S}>s}]} \\
			&\quad	- \EV{\metb[]{}{x_0}{S}}\\
			&=	\met{}{x_0}{S}\charfun[\Bp{\met{}{x_0}{S}\leq s}]
					+	s\charfun[\Bp{\met{}{x_0}{S}>s}] \\
			&\quad	+	\EV[\G_s]{\met[]{}{x_0}{\reach{s}{S}}}\charfun[\Bp{\met{}{x_0}{S}>s}] 
					- \EV{\metb[]{}{x_0}{S}},
\end{align*}
where we used Lemma \ref{str_lem_meas}\ref{str_meas_metchar} and the definition of the arrival times.
We can match this for $t$ by writing
\begin{align*}
	\M_t		&=	\met{}{x_0}{S}\pp{\charfun[\Bp{\met{}{x_0}{S}\leq s}]+\charfun[\Bp{\met{}{x_0}{S}\in(s,t]}]}
					+	t\charfun[\Bp{\met{}{x_0}{S}>t}] \\
			&\quad	+	\EV[\G_t]{\met[]{}{x_0}{\reach{t}{S}}\charfun[\Bp{\met{}{x_0}{S}>t}]} 
					- \EV{\metb[]{}{x_0}{S}}\\
			&=	\met{}{x_0}{S}\charfun[\Bp{\met[(A,F)]{}{x_0}{S}\leq s}]
					+	t\charfun[\Bp{\met{}{x_0}{S}>s}] \\
			&\quad	+	\EV[\G_t]{\met[]{}{x_0}{\reach{t}{S}}}\charfun[\Bp{\met{}{x_0}{S}>s}] 
					- \EV{\metb[]{}{x_0}{S}}
					+R_t,
\end{align*}
introducing a small error term $R_t$ for changing arrival times in $(s,t]$ to $t$, given by
\begin{multline*}
	R_t	=	\pp{\met{}{x_0}{S}- t} \charfun[\Bp{\met{}{x_0}{S}\in(s,t]}]
			-\EV[\G_t]{\met[]{}{x_0}{\reach{t}{S}}\charfun[\Bp{\met{}{x_0}{S}\in(s,t]}]}\\
		=	\pp{\met{}{x_0}{S}- t} \charfun[\Bp{\met{}{x_0}{S}\in(s,t]}]\in [s-t,0].
\end{multline*}
Taking the difference, we obtain
\begin{align}
	\M_t-\M_s	&=	\pp{\EV[\G_t]{\met[]{}{x_0}{\reach{t}{S}}}	-	\EV[\G_s]{\met[]{}{x_0}{\reach{s}{S}}}}\charfun[\Bp{\met{}{x_0}{S}>s}]\notag\\
				&\quad	+	(t-s)\charfun[\Bp{\met{}{x_0}{S}>s}] + R_t.\label{str_eq_MtMspreK}
\end{align}
Due to the finite range of dependence, the future evolution should essentially not depend on the already explored environment, but just on the current state of the set. 
Accordingly, we want to resolve the conditional expectation and replace it with just an expected value, which corresponds to Lemma \ref{str_lem_meas}\ref{str_meas_condE}. 
In order to use the finite range of dependence in the form of Lemma \ref{str_lem_meas}\ref{str_meas_condE}, we need to slightly enlarge the set. 
Note that due to the minimum speed from \eqref{str_eq_vmin} we have $\reach{t}{S}\subset K_h\pp{\reach{t}{S}}+\ov{B}_1\subset	\reach{t+\vmin^{-1}\pp{\sqrt{d}h+1}}{S}$ and hence there holds
\begin{equation*}
	\met[]{}{x_0}{\reach{t}{S}}	\geq	 \met[]{}{x_0}{K_h\pp{\reach{t}{S}}+\ov{B}_1}	
								\geq	 \met[]{}{x_0}{\reach{t}{S}}	-	\vmin^{-1}\pp{\sqrt{d}h+1}.
\end{equation*}
Thus, enlarging $\reach{t}{S}$ and $\reach{s}{S}$ in \eqref{str_eq_MtMspreK} and then using the independence to resolve the conditional expectation via Lemma \ref{str_lem_meas}\ref{str_meas_condE} and that $\sum_i \charfun[E_i(t)]=\sum_j \charfun[E_j(s)]=1$ we obtain
\begin{align}\label{str_eq_MtMspreij}
	&\M_t-\M_s	=	(t-s)\charfun[\Bp{\met{}{x_0}{S}>s}] + R_t + R_K\\
	&\qquad				+	\sum_{i,j=1}^\infty	\pp{\EV{\met{}{x_0}{K_i+\ov{B}_1}}	-	\EV{\met{}{x_0}{K_j+\ov{B}_1}}}\charfun[E_i(t)\cap E_j(s)]\charfun[\Bp{\met{}{x_0}{S}>s}],\notag
\end{align}
with $\abs{R_K}\leq	\vmin^{-1}\pp{\sqrt{d}h+1}$.
In order to bound this term, we compare the expected arrival times from $K_i+\ov{B}_1$ and $K_j+\ov{B}_1$ if $E_i(t)\cap E_j(s)\neq \emptyset$, which can be done by comparing the fastest and slowest evolution during the time step $t-s$:
Due to the minimum speed from \eqref{str_eq_vmin} and the maximum speed from Lemma \ref{not_lem_vmax} there holds
\begin{equation*}
	\reach{s}{S}+\vmin(t-s)\ov{B}_1	\subset \reach{t}{S}	\subset	\reach{s}{S}+\pp{1+\vmax(t-s)}\ov{B}_1.
\end{equation*}
Since $K_i=K_h\pp{\reach[(A,F)]{t}{S}}$ and $K_j=K_h\pp{\reach[(A,F)]{s}{S}}$ if there is $(A,F)\in E_i(t)\cap E_j(s)$, this implies
\begin{equation*}
	K_j+\pp{\vmin(t-s)-\sqrt{d}h}\ov{B}_1	\subset K_i	\subset	K_j+\pp{1+\sqrt{d}h+\vmax(t-s)}\ov{B}_1.
\end{equation*}
On the one hand, the maximum speed implies that
\begin{equation*}
	\reach{\tau}{K_j+\ov{B}_1}	\subset K_i+\ov{B}_1	
	\quad\text{for all }\tau\leq \frac{\vmin}{\vmax}(t-s) - \frac{1+\sqrt{d}h}{\vmax},
\end{equation*}
while on the other hand, the minimum speed implies that
\begin{equation*}
	\reach{\tau}{K_j+\ov{B}_1}	\supset	K_i+\ov{B}_1
	\quad\text{for }\tau= \frac{\vmax}{\vmin}(t-s) + \frac{1+\sqrt{d}h}{\vmin}.
\end{equation*}
Hence, by the definition of the arrival times, for $E_i(t)\cap E_j(s)\neq \emptyset$ there holds
\begin{equation*}
	\met{}{x_0}{K_j+\ov{B}_1}-\met{}{x_0}{K_i+\ov{B}_1}
	\left\{\begin{aligned}
		&\geq	\frac{\vmin}{\vmax}(t-s) - \frac{1\!+\!\sqrt{d}h}{\vmax},	&&	\!\!\text{if }x_0\notin  K_i\!+\!\ov{B}_1\\
		&\geq	0,														&&	\!\!\text{if }x_0\in  K_i\!+\!\ov{B}_1\\
		&\leq	\frac{\vmax}{\vmin}(t-s) + \frac{1\!+\!\sqrt{d}h}{\vmin}.
	\end{aligned}\right.
\end{equation*}
Plugging these bounds into \eqref{str_eq_MtMspreij} and using that $1-\frac{\vmin}{\vmax}\leq \frac{\vmax}{\vmin}-1$, we obtain
\begin{equation*}
	\abs{\M_t-\M_s}	\leq		\pp{\frac{\vmax}{\vmin}-1}(t-s) + C(\data,h) + \abs{R_t} + \abs{R_{x}}
\end{equation*}
with 
\begin{align*}
	\abs{R_t}	&\leq (t-s)\charfun[\Bp{\met{}{x_0}{S}\in(s,t]}],\\
	\abs{R_{x}}	&\leq	\max\Bp{1-\pp{\frac{\vmax}{\vmin}-1},0}(t-s)\sum_{i\in\N}\charfun[E_i(t)\cap\Bp{x_0\in K_i\!+\!\ov{B}_1}\cap\Bp{\met{}{x_0}{S}>s}]
\end{align*}
which yields \eqref{str_eq_MtMs} once we fix an arbitrary small $h>0$.
In summary, during a time step $\Delta t$ the best guess for the arrival time changes at most linearly in $\Delta t$ with rate at most $\frac{\vmax}{\vmin}-1$ up to error terms.
With a more careful approach, the error terms $R_t$ and $R_{x}$ can be mitigated so that they do not influence this scaling.
\end{proof}

\section{Preliminaries}\label{s_pre}
Throughout this Section, if not specified further, we always assume $(A,F)\in\Omega$ or $(A,F)=(\wt{A},\wt{F})_Q$, that is that $(A,F)$ is the restriction of some admissible coefficient field $(\wt{A},\wt{F})\in\Omega$ to a closed set $Q\subset\Rd$, see Definition \ref{restr_def} below.
We first state some basic properties of the set evolutions defined in Definition \ref{not_def_reach}.
Then we will present some more specific properties related to stable sets and restricted coefficient fields.
Finally, we will provide a tool to check measurability with respect to $\F(U)$.
The proofs of these basic properties can be found in Appendix \ref{s_pre-proofs}.

\begin{remark}[Semigroup property of $\reach{t}{\cdot}$]\label{not_rem_SGprop}
Letting closed sets $S\subset\Rd$ evolve with respect to a fixed coefficient field $(A,F)$ for different times can be seen as a semigroup, that is
\begin{equation*}
	\reach[(A,F)]{t+s}{S}=\reach[(A,F)]{t}{\reach[(A,F)]{s}{S}}
	\qquad\text{for any $t,s\geq 0$,}
\end{equation*}
see e.g. \cite[Theorem 4.5.1]{Giga06}.
We will often implicitly use this property.
\end{remark}

Throughout this paper, we will often use that the (negative) characteristic function of the evolution of a set is a supersolution to the level-set equation \eqref{intro_eq_uls}.

\begin{lemma}\label{not_lem_ur}
Let $S\subset\Rd$ be a closed set.
The function $\ur{S}=\ur[(A,F)]{S}$ given by
\begin{equation*}
	\ur[(A,F)]{S}(x,t)
	\coloneqq 1-\charfun[{\reach[(A,F)]{t}{S}}](x)
\end{equation*}
is a supersolution of \eqref{intro_eq_uls} with the coefficient field $(A,F)$ in $\Rd\times(0,\infty)$.
\end{lemma}

A key property of the set evolutions is that -- while for example around closing holes or small gaps the interface might move at arbitrary high speeds as typical for the standard mean curvature flow -- on larger scales, there is a maximum speed which is essentially dominated by the first order term.

\begin{lemma}[Maximum speed of front propagation]\label{not_lem_vmax}
Let $S\subset\Rd$ be a closed set.
Then for any $\delta>0$ there holds 
	$\reach[(A,F)]{T}{S}\subset S+\pp{\delta+T\vmax[\delta]}\ov{B}_1$ 
for all $T>0$, where
\begin{equation*}
 	\vmax[\delta]\coloneqq \max\Bp{C_{1F},\,6C_{1A}}+ C_{1A}\delta^{-1}.
\end{equation*} 
We write $\vmax\coloneqq \vmax[1]$.
\end{lemma}

\subsection{Stable sets}
One of the key concepts in this work is the guarantee that a set will not shrink with respect to the evolution according to a given coefficient field $(A,F)$.
As stated in Definition \ref{box_def_stable}, we call a set \textit{stable} with respect to $(A,F)$ if its negative characteristic function is a stationary supersolution to the level-set equation \eqref{intro_eq_uls}.
The first lemma of this subsection states that these stable sets indeed do not shrink under the evolution with respect to $(A,F)$ and in fact the sets preserve the stability under the evolution.

\begin{lemma}[Preservation of Stability]\label{stable_lem_preserv}
Let $S\subset\Rd$ be a closed set.
If $S$ is stable with respect to $(A,F)$,
then also $\reach[(A,F)]{t}{S}$ is stable and $\reach[(A,F)]{s}{S}\subset \reach[(A,F)]{t}{S}$ for any $0\leq s \leq t$.
\end{lemma}

Next, we translate the comparison principle for viscosity solutions to our set evolutions:
When a set envelops another set, at any time the evolution of the first set will envelop that of the other.
As long as the larger set is stable, this stays true if we replace the full envelopment with $h$-envelopment from Definition \ref{veff_def_fatstab}, that is when ignoring holes of width smaller than $2h$.

\begin{lemma}[The comparison principle for set evolutions]$\quad$ \label{not_lem_comp}
\begin{enumerate}[label=(\alph*), left=0pt]
\item	\label{not_list_lemcomp1}
		Let $S_1\subset S_2$. 
		Then $\reach{t}{S_1}\subset \reach{t}{S_2}$ holds for all $t\geq 0$.
\item	\label{not_list_lemcomp2}
		Let $S_1\hsubs S_2$ for some $h>0$. 
		If in addition $S_2$ is stable, then $\reach{t}{S_1}\hsubs \reach{t}{S_2}$ holds for all $t\geq 0$.
\end{enumerate}
\end{lemma}

We use the above comparison principle to obtain the following lemma for the iterative application of the effective minimum speed from Definition \ref{veff_def_fatstab}.

\begin{lemma}[Iterative application of the effective minimum speed]\label{stable_lem_vmin-iterative}
Let $S\subset\Rd$ be a stable set with respect to some coefficient field $(A,F)$.
Assume that $(A,F)$ admits $v>0$ as an effective minimum speed of propagation on a scale $h>0$ in some set $U\subset\Rd$ as in Definition \ref{veff_def_fatstab}.
Then for any $n\in\N$ with $U\supset\pp{S+n\ov{B}_h}\cap\pp{\partial S\cup\comp{S}}$ there holds
\begin{equation*}
	S+n\ov{B}_h\hsubs \reach[(A,F)]{nv^{-1}}{S}.
\end{equation*}
\end{lemma}

Assumption \ref{veff_aP_star} provides an effective minimum speed $\vmineff$ on scales $h\geq h_0$, which applies for an area $U\subset\Rd$ if the event $\Espeed{h}(U)$ holds. 
For bounded $U$ and large enough $h\geq h_0$, this event has high probability.
We now use Lemma \ref{stable_lem_vmin-iterative} to show that this implies that the arrival time of an $h$-fat set in a target area $\ov{B}_h(x_0)$ is very likely smaller than the time $\timebound{h}{x_0}{S}$, which essentially corresponds to the arrival time of the same set expanding continuously at that minimum speed.

\begin{lemma}[Guaranteed upper bound for arrival times in target areas]\label{stable_lem_TvsEspeed}
Let $x_0\in\Rd$ and $S\subset\Rd$ be an $h$-fat set with bounded boundary.
Then with $\vmineff>0$, $\cstable\geq 0$ from Assumption \ref{veff_aP_star} there holds
\begin{equation*}
	\min\Bp{\met[]{}{y}{S}\,:\,y\in\ov{B}_h(x_0)}\charfun [\Espeed{h}(U)]
	\leq		\timebound{h}{x_0}{S}
	\coloneqq \vmineff^{-1}\dist(x_0,S)+ \cstable h
\end{equation*}
for any $U\subset\Rd$ with 
\begin{equation*}
	U\supset \pp{S+\dist(x_0,S)\ov{B}_1}\cap\pp{\partial S\cup\comp{S}}.
\end{equation*}
This means that the artificial upper bound in the definition of the truncated arrival times $\met[(A,F)]{h}{x_0}{S}$ in the target area $\ov{B}_h(x_0)$ only kicks in for $(A,F)\notin\Espeed{h}(U)$ from Assumption \ref{veff_aP_star}, which by this assumption is very unlikely, see \eqref{veff_eq_probBound}.
\end{lemma}

\subsection{Restricted solutions}
In order to apply the finite range of dependence condition \ref{sett_aP_fin}, we need to guarantee that some properties of solutions depend on the coefficient field only in a fixed area, as discussed for example in the proof of Lemma \ref{str_lem_meas}.
In the companion paper \cite{FI26}, we further want to ensure that a set evolves within or along a fixed area.
For this purpose, we introduce suitably restricted coefficient fields.

\subsubsection*{Introduction of restricted coefficient fields}
Given an admissible coefficient field $(A,F)$, we restrict it to some set $Q\subset\Rd$ by leaving it untouched on $Q$ and then continuously transition to constant coefficients within $Q+\ov{B}_1\setminus Q$. 
In fact, instead of $1$ we could choose an arbitrary small scale for the transition, but this would only impact the constants in our results.
While in principle we could transition to arbitrary coefficients, we choose a very negative forcing, such that sets starting within $Q$ are guaranteed to stay trapped within a small neighborhood of $Q$.

\begin{definition}[Restricted coefficient fields]\label{restr_def}
Let $Q\subset\Rd$ be a Borel set. 
Given $(A,F)\in\Omega$ with $A=\sigma\sigma^\intercal$ we define $(A,F)_Q\coloneqq \pp{A_Q,F_Q}$ with $A_Q=\sigma_Q\sigma_Q^{\intercal}$, the coefficient field restricted to $Q$, as 
\begin{equation}\label{box_eq_defAQ}
	\sigma_Q(x,e)\coloneqq
	\left\lbrace\begin{aligned}
		&\sigma(x,e)							&&\text{for }\dist(x,Q)\leq \frac{2}{3},\\
		&3\pp{1-\dist(x,Q)} \sigma(x,e)		&&\text{for }\dist(x,Q)\in\bp{\frac{2}{3},1},\\
		&0								&&\text{for }\dist(x,Q)\geq 1,\\
	\end{aligned}\right.
\end{equation}
\begin{equation}\label{box_eq_defFQ}
	F_Q(x,e)\coloneqq
	\left\lbrace\begin{aligned}	
		&F(x,e)			\qquad\qquad\qquad	\qquad\qquad\qquad\,		\text{for }\dist(x,Q)\leq 0,\\					
		&3\pp{\frac{1}{3}-\dist(x,Q)} F(x,e) - 3\dist(x,Q)\max\Bp{C_{1F},\,6C_{1A}}	\\			
		&\qquad\quad		\qquad\qquad\qquad	\qquad\qquad\qquad\,		\text{for }\dist(x,Q)\in\bp{0,\frac{1}{3}},\\				
		&-\max\Bp{C_{1F},\,6C_{1A}}	\,		\qquad\qquad\qquad\,		\text{for }\dist(x,Q)\geq \frac{1}{3}.
	\end{aligned}\right.
\end{equation}
\end{definition}

By definition, the evolution of a set with respect to $(A,F)_Q$ only depends on $(A,F)$ restricted to $Q+B_1$ and hence should be $\F(Q+B_1)$ measurable, see Lemma \ref{meas_lem_closure}.
As a consequence of this definition, any $S\subset Q$ stays stuck in $Q+B_{\frac{2}{3}}$ when evolving by \eqref{intro_eq_uls} with coefficients $(A,F)_Q$.
In addition, in this case the evolution of a set with respect to $(A,F)_Q$ provides a lower bound to the evolution of the same set with respect to the unmodified coefficient field $(A,F)$.
For the rigorous statement see Lemma \ref{restr_lem_PropRestr} below.

\begin{remark}[Well-posedness of evolution for restricted coefficient fields]
Existence, uniqueness and uniform continuity for solutions of the level-set equation as well as invariance for the evolution of the level sets as stated in Definition \ref{not_def_reach} for $(A,F)\in\Omega$ also hold if $(A,F)=(\wt{A},\wt{F})_Q$ for some $(\wt{A},\wt{F})\in\Omega$ and $Q\subset\Rd$.
\end{remark}

\subsubsection*{Properties of restricted coefficient fields}
As already stated above, the choice for the constant coefficients in the definition of restricted coefficient fields $(A,F)_Q$ means that sets stay trapped in a neighborhood of $Q$, which we make precise in the following lemma.

\begin{lemma}[Properties of restricted coefficient fields]\label{restr_lem_PropRestr}
Let $u$ be a supersolution to \eqref{intro_eq_uls} with coefficients $(A,F)_{Q}$ for $(A,F)\in\Omega$ and some closed set $Q\subset\Rd$, which satisfies 
\begin{align*}
	u(\cdot,0)&\leq C	\quad\text{in }\Rd,	&
	\Bp{x\,:\, u(\cdot,0)<C} \subset \comp{\pp{\restcomp}},
\end{align*}
where, as in Figure \ref{restr_fig_restcomp},
\begin{equation}\label{restr_eq_SQ}
	\restcomp\coloneqq \bigcup\Bp{B_{\frac{1}{3}}(x)\,:\,\dist(x,Q)\geq \frac{2}{3}},
	\quad\text{ see Figure \ref{restr_fig_restcomp}.}
\end{equation}
Then for any $t\geq 0$ there holds
\begin{align*}
	u(\cdot,t)&\leq C	\quad\text{in }\Rd,	&
	\Bp{x\,:\, u(\cdot,t)<C} \subset \comp{\pp{\restcomp}}.
\end{align*}
In addition, $u$ is a supersolution also to \eqref{intro_eq_uls} with coefficients $(A,F)_{\wt{Q}}$ for any $\wt{Q}\supset Q$. 

In particular, for any $S\subset \comp{\pp{\restcomp}}$ and $\wt{Q}\supset Q$ there holds
\begin{enumerate}[label=(\alph*), left=0pt]
\item	\label{restr_prop1_restr}
		$\reach[(A,F)_Q]{t}{S}\subset \comp{\pp{\restcomp}}\subset Q+ \ov{B}_{\frac{2}{3}}$ for any $t\geq 0$.
\item	\label{restr_prop2_stable}
		If $S$ is stable with respect to $(A,F)_Q$, then the set is stable with respect to $(A,F)_{\wt{Q}}$.
\item	\label{restr_prop3_super}
		$\ur[(A,F)_Q]{S}=1-\charfun[{\reach[(A,F)_Q]{\cdot}{S}}]$ as in Lemma \ref{not_lem_ur} is a supersolution to \eqref{intro_eq_uls} with coefficients $(A,F)_{\wt{Q}}$.
\end{enumerate}
\end{lemma}

\begin{figure}[h]
\centering
\includegraphics[width=0.6\textwidth]{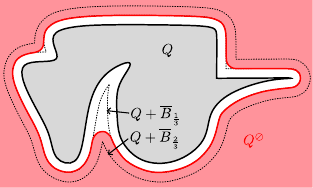}
\caption{An illustration of our notion of restricted coefficient fields.
		The set $Q$ to which the coefficient field  is restricted is depicted in grey. 
		The set $\restcomp$, which sets initialized in $\comp{\pp{\restcomp}}$ will not reach, is depicted in red.}
\label{restr_fig_restcomp}
\end{figure}

Having chosen the constant coefficients such that sets stay trapped in a neighborhood of $Q$ makes it easier to compare the evolution with respect to the original coefficient field to the evolution with respect to the restricted field.

\begin{lemma}[A condition for the equivalence of restricted and unrestricted solutions]\label{restr_lem_QwtQ}
Let $(A,F)\in\Omega$ and $S\subset\Rd$ be a closed set.
Suppose $t_0>0$ and $Q\subset \wt{Q}\subset \Rd$. 
If $S\subset\comp{(\restcomp)}$ as in Lemma \ref{restr_lem_PropRestr}, then
\begin{equation*}
	\reach[(A,F)_{Q}]{t}{S}\subset \reach[(A,F)_{\wt{Q}}]{t}{S}
	\quad\text{ for all }0\leq t.
\end{equation*}
If $\reach[(A,F)_{Q}]{t}{S}+B_\eps\subset Q$ for all $0\leq t\leq t_0$ given some $\eps,t_0>0$, then there holds
\begin{equation*}
	\reach[(A,F)_{Q}]{t}{S}=\reach[(A,F)_{\wt{Q}}]{t}{S}
	\quad\text{ for all }0\leq t\leq t_0.
\end{equation*}
\end{lemma}

\subsection{Measurability of events with respect to $\F(U)$}

To prove the measurability of events with respect to the $\sigma$-algebras $\F(U)$ defined in \eqref{sett_eq_sigmaF}, our main strategy is the application of the following lemma, which states that it is sufficient to check if the event is closed with respect to uniform convergence of the coefficient fields in $U$.

\begin{lemma}\label{meas_lem_closure}
Let $U\subset\Rd$ be a closed set.
Let $E\subset\Omega\subset C(\Rd\times\Sd,\Rdd\times\R)$ be an event which is closed in $\Omega$ with respect to the norm
\begin{equation*}
	\|(A,F)\|_{C_w(U\times\Sd,\Rdd\times\R)}\coloneqq \sup_{x\in U,\,e\in\Sd}\frac{1}{1+\abs{x}^2}\abs{(A(x,e),F(x,e))},
\end{equation*}
which corresponds to weighted uniform convergence and is well defined on $\Omega$ due to the uniform bounds from \eqref{sett_eq_assAbd} and \eqref{sett_eq_assF}.
Then $E$ is measurable with respect to $\F(U)$.
\end{lemma}

\section{Fluctuation bounds for the truncated arrival times in target areas}\label{s_fluct}
In this section, we aim to bound fluctuations of the arrival time at a point $x_0\in\Rd$ around its expectation given a fixed starting set $S\subset\Rd$.
We will follow the approach as laid out in Subsection \ref{subs_str_fluct}. 
However, compared to Subsection \ref{subs_str_fluct} there are multiple main issues: 
\begin{enumerate}[label=\arabic*.), left=12pt]
\item	Since we have no guaranteed minimum speed, $x_0$ might never be reached, that is potentially $\PM{\met{}{x_0}{S}=\infty}>0$.
\item	Since in general sets might shrink, the independence of their evolution from the values of the coefficient fields in their interior breaks down.
\item	Since locally the effective minimum speed might not hold, we will not get uniform bounds for the martingale increments.
		Hence, Azuma's inequality is not applicable.
\end{enumerate}
To deal with these issues, we innovate as follows:
\begin{enumerate}[label=\arabic*.), left=12pt]
\item	We do not work with the actual arrival time at $x_0\in\Rd$, but the arrival time in the surrounding target area $\ov{B}_h(x_0)$, ignoring holes of width up to $2h$ such that the effective minimum speed guaranteed by Assumption \ref{veff_aP_star} applies.
		Since with low probability this might still yield an infinite arrival time, we introduce a truncation, which however is increasingly unlikely to kick in on larger scales, see Definition \ref{intro_def_met} and Lemma \ref{stable_lem_TvsEspeed}. 
\item	We use stable approximations from Assumption \ref{veff_aP_star} to guarantee that sets can not shrink by too much and hence the independence between the coefficients in the interior of the set and its future evolution stays intact. 
		We keep track of the error from the unlikely case that the approximations do not exist.
\item	We introduce an alternative to of Azuma's inequality in Appendix \ref{s_azuma}, so that it suffices to obtain uniform bounds on the martingale increments only with overwhelming and not necessarily full probability.
\end{enumerate}

To be precise, let us briefly recall the precise definition of the arrival times from Definition \ref{intro_def_met} for $S\subset\Rd$, $x_0\in\Rd$ and $h\geq h_0$. 
We write
\begin{align*}
	\met{}{x_0}{S}	&=	\met[(A,F)]{}{x_0}{S}	\coloneqq	\min\Bp{t\geq 0\,:\,x_0\in\reach[(A,F)]{t}{S}},\\
	\metb{h}{x_0}{S}	&=	\metb[(A,F)]{h}{x_0}{S}	\coloneqq	\min\Bp{\min_{y\in\ov{B}_h(x_0)}\met{}{y}{S},\,\timebound{h}{x_0}{S}}\\
	&\qquad\qquad\qquad\qquad\qquad\quad	\text{with }\timebound{h}{x_0}{S}\coloneqq \vmineff^{-1}\dist(x_0,S)+\cstable h,
\end{align*} 
where $t\mapsto \reach[(A,F)]{t}{S}$ is the evolution of the set $S$ with respect to the coefficient field $(A,F)$ from Definition \ref{not_def_reach} and $\cstable>0$ is a parameter from Assumption \ref{veff_aP_star} for the stable approximations.

Note that these are well defined since if $x_0$ or $\ov{B}_h(x_0)$ are reached at some point, there exists a minimum due to the maximum speed of propagation from Lemma \ref{not_lem_vmax}.
Further, these are random variables with respect to $(\Omega,\F(\Rd),\Pm)$, see Lemma \ref{fluct_lem_meas} below for the measurability.
We now state the main result from this section, showing that the fluctuations of these truncated arrival times in the target area $\ov{B}_h(x_0)$ are of order $\sqrt{h\dist(x_0,S)}$.

\begin{proposition}[Fluctuation bounds for truncated arrival times]\label{fluct_prop}
Assume that \ref{sett_aP_fin} and \ref{veff_aP_star} hold.
There are $c,\Cflu,\cflu, \Cfluco=c(\data),\Cflu(\data),\cflu(\data), \Cfluco(\data)>0$ such that 
for all $x_0\in\Rd$, $h\geq h_0$ and $h$-fat sets $S\subset\Rd$ with bounded boundary as well as
\begin{equation}\label{fluct_eq_hrateA}
	\frac{\diam(\partial S)+\dist(x_0,S)}{h}	\leq		\exp\pp{c h^{\ratespeed}}
\end{equation}
there holds
\begin{align}\label{fluct_eq_PMfluct}
	&\PM{\abs{\metb[]{h}{x_0}{S}	- \EV{\metb[]{h}{x_0}{S}}}\geq \lambda}\notag\\
	&\qquad\quad\leq	\left\lbrace	\begin{aligned}
							&\Cflu	\exp\pp{-\frac{\cflu\lambda}{\sqrt{h\dist(x_0,S)}}},	
								&&\text{if }\lambda \leq	\Cfluco\frac{h^{\ratespeed}\sqrt{h\dist(x_0,S)}}{\log\pp{\frac{\dist(x_0,S)}{h}}},\\
							&\Cflu	\exp\pp{-\frac{\cflu h^{\ratespeed}}{\log\pp{\frac{\dist(x_0,S)}{h}}}},
								&&\text{if }\lambda >		\Cfluco\frac{h^{\ratespeed}\sqrt{h\dist(x_0,S)}}{\log\pp{\frac{\dist(x_0,S)}{h}}},\\
							&0,	&&\text{if }\lambda>\timebound{h}{x_0}{S},
						\end{aligned}
					\right.
\end{align}
where $\ratespeed>0$ is the exponential rate from the probability bounds in Assumption \ref{veff_aP_star} for the existence of stable approximation and the applicability of the effective minimum speed of propagation.
\end{proposition}

Analogously to Subsection \ref{subs_str_fluct}, we now define the martingale
\begin{equation}\label{fluct_eq_defMt}
	\M_t\coloneqq \EV[\G_t]{\metb[]{h}{x_0}{S}} - \EV{\metb[]{h}{x_0}{S}},
\end{equation}
where $\pp{\G_t}_{t\geq 0}$ is a filtration with the $\sigma$-algebra $\G_t\subset\F(\Rd)$ containing the information on the approximate state of the set evolution at times $s\leq t$ and on the coefficient field around the area explored up to time $t$.

Let us first make precise the definition of $\G_t$.
As in Subsection \ref{subs_str_fluct}, for any set $M\subset\Rd$ we define its discrete approximation as 
\begin{equation*}
	K_h(M) \coloneqq	\bigcup\pp{y+[0,h]^d \,:\, y\in h\Zd \text{ and } y+[0,h]^d\cap M \neq \emptyset}.
\end{equation*}
Since $\Bp{K_h(M)\,:\,M\subset\Rd}$ is not countable, we need a smaller state space for the discretized set evolutions.
Assuming that there is a stable $(h,\cstable)$-approximation $S_h\subset S$ of the set $S\subset\Rd$ as in Definition \ref{veff_def_fatstab}, we have 
\begin{equation*}
	\Rd\setminus\pp{\comp{S}+\cstable\vmax B_h}
	\subset	S_h
	\subset	\reach{t}{S_h}
	\subset	\reach{t}{S}
	\subset	S+(1+\vmax t)\ov{B}_1
\end{equation*}
with the first inclusion due to the definition of stable $(h,\cstable)$-approximations, the second inclusion due to the stability of $S_h$ and Lemma \ref{stable_lem_preserv}, the third inclusion due to the comparison principle as in Lemma \ref{not_lem_comp} and the final inclusion due to the maximum speed of the evolution as in Lemma \ref{not_lem_vmax}.
Putting these limitations on the state space and since $\partial S$ is bounded, we obtain a countable family $\Bp{K_i}_{i\in\N}\subset\Rd$ such that 
\begin{equation*}
	\Bp{K_i}_{i\in\N} 
	=\Bp{K_h(M)\,:\, \Rd\setminus\pp{\comp{S}+\cstable\vmax B_h}\subset M\subset S+B_R \text{ for some }R>0}.
\end{equation*}
Based on this state space we define the events $\Bp{E_i(t)}_{i\in\N}$ capturing the approximate state of the set evolution at time $t\geq 0$ as
\begin{equation}\label{fluct_eq_Ei}
	E_i(s)	\coloneqq 
	\Bp{(A,F)\in\Omega\,:\, K_h\pp{\reach[(A,F)]{s}{S}}=K_i} \cap \Espeed{h}\pp{\partial S },
\end{equation}
where the restriction to $\Espeed{h}\pp{\partial S}$ from Assumption \ref{veff_aP_star} guarantees the existence of a stable approximation.
Finally, we can now define the appropriate filtration $\Bp{\G_t}_{t\geq 0}$ capturing the information on the current state of the set evolution and the coefficient field in a neighborhood of the explored area as
\begin{align}
	\G_0		&\coloneqq	\Bp{\emptyset, \Omega},\notag\\
	\G_t		&\coloneqq 	\text{$\sigma$-algebra generated by the events }E_i(s)\cap F\notag\\ 
			&\qquad\qquad\qquad\qquad\qquad			\text{for any }~0\leq s\leq t
												\text{ and } F\in\F\pp{K_i+\CG\ov{B}_h}\label{fluct_eq_Gt}
\end{align}
with $\Bp{\F(U)}_U$ from \eqref{sett_eq_sigmaF} denoting the $\sigma$-algebras containing the information on the coefficient field in $U\subset\Rd$ and $\CG=\CG(\data)>0$ chosen large enough.
Essentially, we choose $\CG$ such that expanding $K_i$ by $\CG\ov{B}_h$ covers the additional area which the evolution of a set might have covered before falling back to within $K_i$, at least assuming that the set has a stable $(h,\cstable)$-approximation.
For the precise requirements on $\CG$, see Lemma \ref{fluct_lem_reachRestr}, Lemma \ref{fluct_lem_measEi} and Lemma \ref{fluct_lem_measIn} as well as their proofs.

In order to prove Proposition \ref{fluct_prop}, we largely follow the approach laid out in Subsection \ref{subs_str_fluct}. 
In Subsection \ref{subs_fluct_meas}, we will deal with the necessary measurability issues corresponding to Lemma \ref{str_lem_meas}, essentially using the finite range of dependence to show that the future evolution of a set is independent from the coefficient field explored while reaching its current state. 
In Subsection \ref{subs_fluct_incr} we will use these observations to bound the martingale increments given an event of overwhelming probability and then apply the alternative to Azuma's inequality from Appendix \ref{s_azuma} to obtain Proposition \ref{fluct_prop}.

But first, since in our setting $S$ is not fully guaranteed to be stable, two obvious properties from the minimum speed case treated in Subsection \ref{subs_str_fluct} -- $\bigcup_{i\in\N} E_i(t)=\Omega$ and $S\subset K_i$ for $i\in\N$ -- do not hold and we need replacements.
 
\begin{lemma}\label{fluct_lem_EiCup}
Let $S\subset\Rd$ be an $h$-fat set with bounded boundary.
Let the events $(E_i(t))_i$ for some $t\geq 0$ be defined as in \eqref{fluct_eq_Ei}.
The sets $(E_i(t))_i$ are disjoint and there holds
\begin{equation*}
	\bigcup_{i\in\N} E_i(t) = \Espeed{h}\pp{\partial S}.
\end{equation*}
\end{lemma}

\begin{proof}
By the definition \eqref{fluct_eq_Ei} it is clear that `$\subset$' holds. 
It remains to show '$\supset$'. 
Assume that $(A,F)\in \Espeed{h}\pp{\partial S}$.
Then with \ref{veff_Pspeed} from Assumption \ref{veff_aP_star} there exists a stable $(h,\cstable)$-approximation $S_h$ of $S$.
In particular, $S_h$ is stable with respect to $(A,F)$, satisfies $S_h\subset S$ and $S_h\supset \Rd\setminus\pp{\comp{S}+\cstable\vmax B_h}$.
Therefore, with Lemma \ref{stable_lem_preserv} and \ref{not_lem_comp} we obtain
\begin{equation*}
	S_h	\subset \reach{t}{S_h}	\subset	\reach{t}{S}
	\quad\text{ and }\quad
	\reach{t}{S}\supset\Rd\setminus\pp{\comp{S}+\cstable\vmax B_h}.
\end{equation*}
Due to the maximum speed of propagation from Lemma \ref{not_lem_vmax}, we have 
\begin{equation*}
	\reach{t}{S}\subset S+\pp{1+\vmax t}\ov{B}_1.
\end{equation*}
Thus there exists $i_0\in\N$ with $K_h\pp{\reach{t}{S}}=K_{i_0}$.
Hence, $E_{i_0}(t)$ holds and since the events $\Bp{E_i(t)}_i$ are clearly disjoint, we are done.
\end{proof}

\begin{lemma}\label{fluct_lem_KiS}
Let $S\subset\Rd$ be an $h$-fat set with bounded boundary.
Consider the events $(E_i(t))_i$ for some $t\geq 0$ as defined in \eqref{fluct_eq_Ei}.
Then for any $i\in\N$ with $E_i(t)\neq\emptyset$ there holds
\begin{equation*}
	S\subset K_i +  \pp{1+\pp{1+\vmax\cstable}h}\ov{B}_1.
\end{equation*}
\end{lemma}

\begin{proof}
Let $(A,F)\in E_i(t)$, in particular $(A,F)\in\Espeed{h}\pp{\partial S}$.
Then with \ref{veff_Pspeed} from Assumption \ref{veff_aP_star} there exists a stable $(h,\cstable)$-approximation $S_h$ of $S$ with respect to $(A,F)$.
In particular, we have $S_h\subset S\hsubs \reach[(A,F)]{\cstable h}{S_h}$.
On the one hand, with Lemma \ref{stable_lem_preserv} and the stability of $S_h$ as well as the comparison principle from Lemma \ref{not_lem_comp} we obtain
\begin{equation*}
	S_h	\subset	\reach[(A,F)]{t}{S_h}	\subset	\reach[(A,F)]{t}{S}	\subset	K_i
\end{equation*}
since $K_h\pp{\reach[(A,F)]{t}{S}}=K_i$. 
On the other hand, with $\vmax$ from Lemma \ref{not_lem_vmax} we have
\begin{equation*}
	S	\hsubs	\reach[(A,F)]{\cstable h}{S_h}
		\subset	S_h+(1+\vmax\cstable h)\ov{B}_1.
\end{equation*}
Combining both, we obtain
\begin{equation*}
	S\subset	K_i+(1+\vmax\cstable h)\ov{B}_1 + \ov{B}_h.
\end{equation*}
\end{proof}

\subsection{Measurability issues}\label{subs_fluct_meas}
In this subsection we will provide replacements for the measurability properties in Lemma \ref{str_lem_meas}, which relied on the guaranteed minimum speed.
We first obtain an analogous result to Lemma \ref{str_lem_meas}\ref{str_meas_met}: Arrival times are indeed measurable and further independent from the coefficient field in the interior of the starting set -- as long as a stable approximation for this starting set exists, which is guaranteed by the event $\Espeed{h}\pp{\partial S}$ introduced in Assumption \ref{veff_aP_star}.
If $\Espeed{h}\pp{\partial S}$ does not hold, then the interface might partially collapse into the interior of the starting set, breaking the independence
Hence, we multiply the truncated arrival times with $\charfun[\Espeed{h}\pp{\partial S}]$ to only consider the case that $\Espeed{h}\pp{\partial S}$ holds.

\begin{lemma}[Measurability of arrival times]\label{fluct_lem_meas}
Let $S\subset\Rd$ and $x_0\in\Rd$.
Then
\begin{equation}\label{fluct_eq_measMET}
	\met[]{}{x_0}{S} \,\,\text{ and }\,\,	\metb[]{h}{x_0}{S}	\quad\text{are measurable with respect to }\F(\Rd).
\end{equation}
Assume that in addition $S$ is an $h$-fat set.
Then for the event $\Espeed{h}\pp{\partial S}$ from Assumption \ref{veff_aP_star} there holds
\begin{multline}\label{fluct_eq_measMETstable}
		\met[]{}{x_0}{S}\charfun[\Espeed{h}\pp{\partial S}]	\,\,\text{ and }\,\,	\metb[]{h}{x_0}{S}\charfun[\Espeed{h}\pp{\partial S}]\\
		\quad\text{are measurable with respect to }\F\pp{\ov{\comp{S}} + \max\Bp{\cstable\vmax,\Cmeas}\ov{B}_h}.
\end{multline}
\end{lemma}

\begin{proof}
We only prove the measurability of the truncated arrival times $\metb[]{h}{x_0}{S}$ in the target area $\ov{B}_h(x_0)$, since the proofs for the actual arrival times $\met[]{}{x_0}{S}$ are analogous.
Regarding \eqref{fluct_eq_measMET}, it is sufficient to show for any $T\geq 0$ that 
\begin{equation*}
	\Bp{(A,F)\in\Omega	\,:\,	\met[(A,F)]{h}{x_0}{S}\leq T}\in\F(\Rd).
\end{equation*}
This is clear for $T\geq\timebound{h}{x_0}{S}$.
For  $T<\timebound{h}{x_0}{S}$ we will apply Lemma \ref{meas_lem_closure}, for which we need to check that the set is closed with respect to weighted uniform convergence of the coefficient fields.

Let $(A_k,F_k)\in\Omega$ with $T_k\coloneqq \met[(A_k,F_k)]{h}{x_0}{S}\leq T$ and $(A_k,F_k)\rightarrow (A,F)\in\Omega$ with respect to the norm $\|\cdot\|_{C_w(\Rd\times\Sd,\Rdd\times\R)}$ from Lemma \ref{meas_lem_closure}.
By definition, for $k\in\N$ there exists $y_k\in \ov{B}_h(x_0)\cap\reach[(A_k,F_k)]{T_k}{S}$.
Then there exists a (non-relabeled) subsequence such that $T_k\rightarrow T_0\leq T$ and $y_k\rightarrow y_0\in\ov{B}_h(x_0)$.
We take the lower relaxed limit as defined in \eqref{comp_eq_defLrellim}, 
\begin{equation*}
	v\coloneqq \liminf_{k\rightarrow \infty}\!{}_* \,\ur[(A_k,F_k)]{S},
\end{equation*}
where $\ur[(A_k,F_k)]{S}(x,t)=1-\charfun[{\reach[(A_k,F_k)]{t}{S}}](x)$ are the supersolutions from Lemma \ref{not_lem_ur}.
Due to the stability of viscosity solutions from Theorem \ref{comp_thm_stab} and the weighted uniform convergence of $(A_k,F_k)$, we obtain that $v$ is a supersolution of the level-set equation \eqref{intro_eq_uls} with the coefficient field $(A,F)$.
Since $S\subset\Bp{x\in\Rd\,:\,v(x,0)\leq 0}$, with the comparison principle from Theorem \ref{comp_thm_comp-unb} and Remark \ref{comp_rem_condcomp}, we obtain $\uls[(A,F)]{S}\leq v$ for a continuous solution $\uls[(A,F)]{S}$ from Definition \ref{not_def_reach} and thus in particular
\begin{equation*}
	\reach[(A,F)]{T_0}{S}	=		\Bp{x\in\Rd\,:\,	\uls[(A,F)]{S}(x,T_0)\leq 0}	
							\supset \Bp{x\in\Rd\,:\,	v(x,T_0)\leq 0}	
							\supset 	\Bp{y_0},
\end{equation*}
because $v(y_0,T_0)\leq \liminf_{k\rightarrow 0}\ur[(A_k,F_k)]{S}(y_k,T_k)=0$ due to the definition of the lower relaxed limit. 
Since $y_0\in \ov{B}_h(x_0)$, this implies $\met[(A,F)]{h}{x_0}{S}\leq T_0\leq T$ and hence Lemma \ref{meas_lem_closure} yields the measurability of $\met[]{h}{x_0}{S}$ with respect to $\F(\Rd)$.

Regarding \eqref{fluct_eq_measMETstable}, we write 
\begin{multline*}
	E\coloneqq \big\lbrace	(A,F)\in\Omega\,:\,	\text{there exists }S_h\subset S\text{ stable w.r.t. }(A,F)\\
												\text{and with }S_h\supset\Rd\setminus\pp{\comp{S}+\cstable\vmax B_h} \big\rbrace
\end{multline*}
for the event that there exists a stable set approximating $S$ from inside. 
Clearly, Assumption \ref{veff_aP_star} implies $\Espeed{h}\pp{\partial S}\subset E$, since a stable $(h,\cstable)$-approximation of $S$ satisfies these properties.
Therefore we have 
\begin{equation*}
	\met[]{h}{x_0}{S}\charfun[\Espeed{h}\pp{\partial S}]	
	=	\met[]{h}{x_0}{S}\charfun[E]\charfun[\Espeed{h}\pp{\partial S}],
\end{equation*}
Since $\Espeed{h}\pp{\partial S}\in\F\pp{\partial S + \Cmeas \ov{B}_h}$ by assumption, it is hence sufficient to show for any $T>0$ that 
\begin{align*}
					&	O_0\coloneqq	\Bp{(A,F)\in\Omega\,:\,	\met[(A,F)]{h}{x_0}{S}\charfun[E]=0}		\in\F\pp{\ov{\comp{S}}+\cstable\vmax \ov{B}_h}\\
	\text{and }\quad	&	O_T\coloneqq	\Bp{(A,F)\in\Omega\,:\,	0<\met[(A,F)]{h}{x_0}{S}\charfun[E]\leq T}	\in\F\pp{\ov{\comp{S}}+\cstable\vmax \ov{B}_h}.
\end{align*}
Regarding $O_0$, if $\pp{x_0+\ov{B}_h}\cap S= \emptyset$, then $\met[(A,F)]{h}{x_0}{S}>0$ due to the bounds on the speed of expansion from Lemma \ref{not_lem_vmax} and we have $O_0=\comp{E}$.
If $\pp{x_0+\ov{B}_h}\cap S\neq \emptyset$, then we clearly have  $\met[(A,F)]{h}{x_0}{S}=0$ and hence $O_0=\Omega$.
Hence, to show the measurability of $O_0$, it is enough to show that $E\in\F\pp{\ov{\comp{S}}+\cstable\vmax \ov{B}_h}$.
Regarding $O_T$, note that there has to be $\tau>0$ with $\met[(A,F)]{h}{x_0}{S}\geq \tau$ for all $(A,F)\in O_T$ if $O_T\neq\emptyset$, because otherwise we would obtain $\pp{x_0+\ov{B}_h}\cap S\neq \emptyset$ again due to the bounds on the speed of expansion from Lemma \ref{not_lem_vmax}.
This would imply $O_0=\Omega$ and hence $O_T=\emptyset$, yielding a contradiction.
With these considerations and taking the proof of \eqref{fluct_eq_measMET} from above as a blueprint, both measurability statements are straightforward to show with Lemma \ref{meas_lem_closure} using the stability of viscosity solutions, i.e. Theorem \ref{comp_thm_stab}.
The key insight is that due to the existence of the stable set $S_h\subset S$ from $E$ for any $t\geq 0$ we have 
\begin{equation*}
	\Rd\setminus\pp{\comp{S}+\cstable\vmax B_h}
	\subset	S_h
	\subset	\reach{t}{S_h}
	\subset	\reach{t}{S}
\end{equation*}
where we used that with Lemma \ref{stable_lem_preserv} the stable set $S_h$ does not shrink and then applied the comparison principle from Lemma \ref{not_lem_comp}. 
Thus, the interface never leaves $\ov{\comp{S}}+\cstable\vmax \ov{B}_h$, because of which the evolution is not impacted by the values of the coefficient field outside of this area. 
\end{proof}

In order to obtain replacements for the remaining parts of Lemma \ref{str_lem_meas}, we need the following lemma.
It describes one of the key implications of the existence of a stable $(h,\cstable)$-approximation of a set $S\subset\Rd$:
The area explored before the evolution of $S$ reaches a given state can not be much larger than that state. 
In particular, in order to check if the set evolution ends up in that state, it is sufficient to use the evolution with respect to the coefficient field restricted to a slightly larger area than the state for which we check.
This will imply that the events $E_i$ defined in \eqref{fluct_eq_Ei} are measurable with respect to information from a set only slightly larger than $K_i$.

\begin{lemma}\label{fluct_lem_reachRestr}
Let $h\geq h_0$ and $S\subset\Rd$ be an $h$-fat set.
Let $(A,F)\in\Espeed{h}\pp{\partial S}$.
Let $U\subset\Rd$ be such that
\begin{equation}\label{fluct_eq_SwtU}
	S + (2+\vmax\cstable h)\ov{B}_1
	\subset \hat{U}\coloneqq U+\ov{B}_R
	\quad\text{for some }R\geq 2+\pp{1+\vmax\cstable}h.
\end{equation}
Let $(A,F)_{\hat{U}}$ be the restriction of the coefficient field to $\hat{U}$ as in Definition \ref{restr_def}.
If there holds $\reach[(A,F)_{\hat{U}}]{t_0}{S}\subset U$ for some $t_0\geq 0$, then $\reach[(A,F)]{t}{S}\subset \hat{U}$ for all $0\leq t\leq t_0$ and 
\begin{equation*}
	\reach[(A,F)_{\hat{U}}]{t}{S}=\reach[(A,F)]{t}{S}
	\quad\text{ for all }0\leq t\leq t_0.
\end{equation*}
\end{lemma}

\begin{proof} 
Once we obtain that $\reach[(A,F)_{\hat{U}}]{t}{S}+B_1\subset \hat{U}$ for all $0\leq t\leq t_0$, Lemma \ref{restr_lem_QwtQ} yields $\reach[(A,F)_{\hat{U}}]{t}{S}=\reach[(A,F)]{t}{S}$  for all $0\leq t\leq t_0$ and the result follows immediately.
The idea for proving that $\reach[(A,F)_{\hat{U}}]{t}{S}$ can not be too large is the following:
\begin{itemize}[left = 0pt]
\item	$S$ is sandwiched between a stable set $S_h$ from the inside and from the outside by the same $S_h$ but evolved for a short additional time span $\cstable h$.
\item	$S_h$ has to stay within $U$ at all times, otherwise it and thus $S$ can not drop back into $U$.
\item	This also means that the set bounding $S$ from the outside can not be farther from $U$ than evolving $U$ with the maximum speed from Lemma \ref{not_lem_vmax} for the short time span $\cstable h$.
\end{itemize}
We will now formalize this argument, also considering the restriction of the coefficient field to $\hat{U}$ as in Definition \ref{restr_def}.
If $(A,F)\in\Espeed{h}\pp{\partial S}$ holds, then with \ref{veff_Pspeed} from Assumption \ref{veff_aP_star} there exists a stable $(h,\cstable)$-approximation $S_h$ of $S$.
In particular, $S_h$ is stable with respect to $(A,F)$ and satisfies $S_h\subset S\hsubs \reach[(A,F)]{\cstable h}{S_h}$.
Due to the maximum speed of propagation described in Lemma \ref{not_lem_vmax}, for all $0\leq t\leq \cstable h$ there holds
\begin{equation*}
	\reach[{(A,F)_{\hat{U}}}]{t}{S_h}+B_1	
							\subset S_h+(2+\vmax\cstable h)\ov{B}_1
							\subset	S+(2+\vmax\cstable h)\ov{B}_1
							\subset \hat{U}.
\end{equation*}
Hence, with Lemma \ref{restr_lem_QwtQ} we obtain $\reach[(A,F)]{\cstable h}{S_h}=\reach[(A,F)_{\hat{U}}]{\cstable h}{S_h}$ and thus 
\begin{equation*}
	S_h\subset S\hsubs \reach[(A,F)_{\hat{U}}]{\cstable h}{S_h},
	\quad\text{ with }S_h\text{ stable with respect to }(A,F)_{\hat{U}},
\end{equation*}
where the stability with respect to $(A,F)_{\hat{U}}$ follows since $(A,F)=(A,F)_{\hat{U}}$ on $\hat{U}\supset S_h$.

With the comparison principle from Lemma \ref{not_lem_comp}, which states that for stable sets not only ``$\subset$'' but also ``$\hsubs$'' is preserved under the evolution, and the stability of $S_h$ via Lemma \ref{stable_lem_preserv},
for all $t\geq s\geq 0$ we obtain 
\begin{equation}\label{fluct_eq_Uproof}
	\reach[(A,F)_{\hat{U}}]{s}{S_h}
	\subset	\reach[(A,F)_{\hat{U}}]{t}{S_h}
	\subset	\reach[(A,F)_{\hat{U}}]{t}{S}	
	\hsubs	\reach[(A,F)_{\hat{U}}]{t+\cstable h}{S_h}.
\end{equation}
In particular, the assumption that $\reach[(A,F)_{\hat{U}}]{t_0}{S}\subset U$ implies that $\reach[(A,F)_{\hat{U}}]{t}{S_h}\subset U$ for all $0\leq t\leq t_0$.
Again due to the maximum speed of propagation from Lemma \ref{not_lem_vmax}, we have 
\begin{equation*}
	\reach[{(A,F)_{\hat{U}}}]{t+\cstable h}{S_h}
							\subset \reach[(A,F)_{\hat{U}}]{t}{S_h}+(1+\vmax\cstable h)\ov{B}_1
							\subset	U+(1+\vmax\cstable h)\ov{B}_1
\end{equation*}
for all $0\leq t\leq t_0$.
Together with \eqref{fluct_eq_Uproof}, and since $S_1\hsubs S_2$ implies $S_1\subset S_2+\ov{B}_h$ for any $S_1,S_2\subset\Rd$ by Definition \ref{veff_def_fatstab}, this yields
\begin{align*}
	\reach[(A,F)_{\hat{U}}]{t}{S}
	&\subset	U+(1+(1+\vmax\cstable h))\ov{B}_1
	&&\text{ for all }0\leq t\leq t_0\\
\intertext{and hence}
	\reach[(A,F)_{\hat{U}}]{t}{S} + B_1
	&\subset		U+(2+(1+\vmax\cstable) h)\ov{B}_1
	\subset		\hat{U}
	&&\text{ for all }0\leq t\leq t_0,
\end{align*}
which concludes the proof as described at the start.
\end{proof}

With Lemma \ref{fluct_lem_reachRestr} established, for the event $E_i(t)$, which checks if $\reach{t}{S}\approx K_i$, we are now able to prove the measurability with respect to the $\sigma$-algebra $\F\pp{K_i+\CG\ov{B}_h}$ containing the information of the coefficient field on a set only slightly larger than $K_i$.

\begin{lemma}[Measurability of $E_i$]\label{fluct_lem_measEi}
Let $S\subset\Rd$ be an $h$-fat set with bounded boundary.
For $i\in\N$ and $t\geq 0$, the event $E_i(t)$ as defined in \eqref{fluct_eq_Ei} is measurable with respect to $\F\pp{K_i+\CG\ov{B}_h}$ if $\CG=\CG(\data)>0$ has been chosen large enough.
\end{lemma}

\begin{proof}
Assume that $E_i(t)\neq\emptyset$, otherwise we are done.
We claim that with Lemma \ref{fluct_lem_reachRestr} we obtain
\begin{equation}\label{fluct_eq_Eialt}
	E_i(t) = 
 	\Bp{(A,F)\in\Omega\,:\, K_h\pp{\reach[(A,F)_{K_i+\pp{\CG-1}\ov{B}_h}]{t}{S}}=K_i} \cap 
	\Espeed{h}\pp{\partial S},
\end{equation} 
that is we use $\reach[(A,F)_{K_i+\pp{\CG-1}\ov{B}_h}]{t}{S}$ instead of $\reach[(A,F)]{t}{S}$ as in the definition, replacing the unmodified coefficient field with the restriction to an area only slightly larger than the set, for which we are checking if the final state of the evolution corresponds to it.
Having established \eqref{fluct_eq_Eialt}, we quickly obtain that $E_i(t)$ is measurable with respect to $\F\pp{K_i+\CG\ov{B}_h}$:
It is straightforward to show that $\Bp{(A,F)\in\Omega\,:\, K_h\pp{\reach[(A,F)_{K_i+\pp{\CG-1}\ov{B}_h}]{t}{S}}=K_i}$ is $\F\pp{K_i+\CG\ov{B}_h}$-measurable with Lemma \ref{meas_lem_closure} as for example in the proof of Lemma \ref{fluct_lem_meas}.
With \ref{veff_Pmeas} from Assumption \ref{veff_aP_star} we have $\Espeed{h}\pp{\partial S}\in\F\pp{\partial S + \Cmeas \ov{B}_h}$.
Since $E_i(t)\neq \emptyset$, with Lemma \ref{fluct_lem_KiS} for $\CG$ large enough we have 
\begin{align*}
	\partial S +\Cmeas \ov{B}_h
	\subset K_i + \pp{1+\pp{1+\vmax\cstable}h}\ov{B}_1 + \Cmeas \ov{B}_h 
	\subset K_i+\CG\ov{B}_h
\end{align*}
and since $\F(M_1)\subset\F(M_2)$ for $M_1\subset M_2$, we obtain that $\Espeed{h}\pp{\partial S}$ is measurable with respect to $\F\pp{K_i+\CG\ov{B}_h}$.
Thus, $E_i(t)\in\F\pp{K_i+\CG\ov{B}_h}$ due to \eqref{fluct_eq_Eialt}.

It remains to show that \eqref{fluct_eq_Eialt} holds.
In order to apply Lemma \ref{fluct_lem_reachRestr}, we need to check that $\eqref{fluct_eq_SwtU}$ holds for $U=K_i$, $\hat{U}=K_i+\pp{\CG-1}\ov{B}_h$ and $S$.
Since $E_i(t)\neq\emptyset$, we know from Lemma \ref{fluct_lem_KiS} that  
\begin{equation*}
	S\subset K_i +  \pp{1+\pp{1+\vmax\cstable}h}\ov{B}_1.
\end{equation*}
In particular, for $\CG$ large enough we have
\begin{equation*}
	S+\pp{2+\vmax\cstable h}\ov{B_1}
	\subset	K_i+\pp{3+\pp{1+2\vmax\cstable}h}\ov{B}_1
	\subset	K_i+\pp{\CG-1}\ov{B}_h,
\end{equation*}
satisfying condition \eqref{fluct_eq_SwtU} for Lemma \ref{fluct_lem_reachRestr}.

To show that ``$\supset$'' in \eqref{fluct_eq_Eialt} holds, note that if $K_h\pp{\reach[(A,F)_{K_i+\pp{\CG-1}\ov{B}_h}]{t}{S}}=K_i$ and  $(A,F)\in \Espeed{h}\pp{\partial S}$, Lemma \ref{fluct_lem_reachRestr} yields
	$\reach[(A,F)_{K_i+\pp{\CG-1}\ov{B}_h}]{t}{S}=\reach[(A,F)]{t}{S}$ 
and hence $K_h\pp{\reach[(A,F)]{t}{S}}=K_i$.

Regarding ``$\subset$'' in \eqref{fluct_eq_Eialt}, if $K_h\pp{\reach[(A,F)]{t}{S}}=K_i$ and $(A,F)\in \Espeed{h}\pp{\partial S}$, then the evolution with respect to the restricted coefficient field also ends up in $K_i$ since by Lemma \ref{restr_lem_QwtQ} the evolution with respect to restricted coefficient fields is lower than the one with respect to the original coefficient field.
Now, Lemma \ref{fluct_lem_reachRestr} again yields that the evolutions are the same up to time $t$.
Thus, we have shown that \eqref{fluct_eq_Eialt} holds.
\end{proof}

Using a similar approach, we can leverage the bounds on potential shrinkage from Lemma \ref{fluct_lem_reachRestr} to obtain that arrival times smaller than $t$ are fully determined with the information from $\G_t$, corresponding to Lemma \ref{str_lem_meas}\ref{str_meas_metchar}.
However, again this only works when a stable approximation exists as for the event $\Espeed{h}\pp{\partial S}$ from Assumption \ref{veff_aP_star}.

\begin{lemma}[Measurability of arrival times within the already explored environment]\label{fluct_lem_measIn}
Let $t\geq 0$, $x_0\in\Rd$ and $S\subset\Rd$ be an $h$-fat set with bounded boundary.
Then
\begin{equation}\label{fluct_eq_measIn}
	(A,F)\mapsto \metb[]{h}{x_0}{S}\charfun[\Bp{(A',F')\,:\,\metb[(A',F')]{h}{x_0}{S}\leq t}]((A,F))\charfun[\Espeed{h}\pp{\partial S}]((A,F))
\end{equation}
is measurable with respect to $\G_t$ as defined in \eqref{fluct_eq_Gt} with $\G_0=\Bp{\emptyset,\Omega}$.
In addition, 
\begin{align*}
	\Bp{(A,F)\,:\,\met[(A,F)]{h}{x_0}{S}\leq t}\cap\Espeed{h}\pp{\partial S}	&\quad\text{is $\G_t$-measurable for }t>0,\\
	\Bp{(A,F)\,:\,\met[(A,F)]{h}{x_0}{S}= 0}											&\quad\text{is $\G_0$-measurable}.
\end{align*}
In particular, for the truncated arrival time in the target area $\ov{B}_h(x_0)$ there holds
\begin{equation}\label{fluct_eq_measMTbound}
	\EV[\G_{\timebound{h}{x_0}{S}}]{\metb[]{h}{x_0}{S}\charfun[\Espeed{h}\pp{\partial S}]}
	=	\metb[]{h}{x_0}{S}\charfun[\Espeed{h}\pp{\partial S}].
\end{equation}
\end{lemma}

\begin{proof}
We only show \eqref{fluct_eq_measIn}, the measurability of the sets follows accordingly and \eqref{fluct_eq_measMTbound} is implied by \eqref{fluct_eq_measIn} since 
\begin{equation*}
	\metb[]{h}{x_0}{S}\charfun[\Bp{(A',F')\,:\,\met[(A',F')]{h}{x_0}{S}\leq \timebound{h}{x_0}{S}}]=\metb[]{h}{x_0}{S}.
\end{equation*}
Regarding $t=0$ in \eqref{fluct_eq_measIn}, we clearly have
\begin{equation*}
	\met[]{h}{x_0}{S}\charfun[\Bp{(A',F')\,:\,\met[(A',F')]{h}{x_0}{S}\leq t}]=0.
\end{equation*}
Therefore, for the rest of the proof let $t>0$.
Due to Lemma \ref{fluct_lem_reachRestr} we obtain $\reach[(A,F)]{s}{S}=\reach[(A,F)_{K_i+\pp{\CG-1}\ov{B}_h}]{s}{S}$ for $s\leq t$ and $(A,F)\in E_i(t)$, as long as $\CG=\CG(\data)>0$ from the definition of $\G_t$ in \eqref{fluct_eq_Gt} is chosen large enough.
Hence, with $\Espeed{h}\pp{\partial S}=\bigcup_{i\in\N} E_i(t)$ from Lemma \ref{fluct_lem_EiCup} for the disjoint events $( E_i(t))_i$, for $(A,F)\in\Omega$ there holds
\begin{align*}
	&	\met[(A,F)]{h}{x_0}{S}\charfun[\Bp{(A',F')\,:\,\met[(A',F')]{h}{x_0}{S}\leq t}]((A,F))\charfun[\Espeed{h}\pp{\partial S}]((A,F))\\
	&\qquad	=	\sum_{i\in\N}\met[(A,F)]{h}{x_0}{S}\charfun[\Bp{(A',F')\,:\,\met[(A',F')]{h}{x_0}{S}\leq t}\cap E_i(t)]((A,F))\\
	&\qquad	=	\sum_{i\in\N}\met[(A,F)_{K_i+\pp{\CG-1}\ov{B}_h}]{h}{x_0}{S}\\
	&\qquad\qquad\qquad\qquad				\times	\charfun[\Bp{(A',F')\,:\,\met[(A',F')_{K_i+\pp{\CG-1}\ov{B}_h}]{h}{x_0}{S}\leq t}\cap E_i(t)]((A,F)).
\end{align*}
By the definition of $\G_t$, it only remains to show that the term
\begin{equation*}
	\sum_{i\in\N}\met[(A,F)_{K_i+\pp{\CG-1}\ov{B}_h}]{h}{x_0}{S}\charfun[\Bp{(A',F')\,:\,\met[(A',F')_{K_i+\pp{\CG-1}\ov{B}_h}]{h}{x_0}{S}\leq t}]
\end{equation*}
is $\F\pp{K_i+\CG\ov{B}_h}$-measurable.
This follows from a straightforward application of Lemma \ref{meas_lem_closure} with the proof of Lemma \ref{fluct_lem_meas} as a blueprint.
\end{proof}

Finally, we are able to resolve the conditional expectation for arrival times which have not yet been reached, corresponding to the last part of Lemma \ref{str_lem_meas}.
As in Subsection \ref{subs_str_fluct}, this is the only place where the finite range of dependence from Assumption \ref{sett_aP_fin} enters:
Essentially, it provides the independence of the future development of the set from the coefficients in the already explored environment.

\begin{lemma}[Resolving the conditional expectation]\label{fluct_lem_condE}
There exists $\Ccond=\Ccond(\CG,\data)$ depending on the choice of the constant $\CG$ in the definition of $\G_t$ from \eqref{fluct_eq_Gt}, such that the following holds.
Let $h\geq h_0$.
Let $S\subset\Rd$ be an $h$-fat set with bounded boundary.
Then for any $i\in\N$ and $t\geq 0$ there holds
\begin{multline}\label{fluct_eq_EcondE}
	\EV[\G_t]{\metb[]{h}{x_0}{K_i+\Ccond\ov{B}_h}\charfun[\Espeed{h,i}]\charfun[E_i(t)]}\\
	=\EV{\metb[]{h}{x_0}{K_i+\Ccond\ov{B}_h}\charfun[{\Espeed{h,i}}]}\charfun[{E_i(t)}]
\end{multline}
with
\begin{align*}
	\Espeed{h,i}	&\coloneqq	\Espeed{h}\pp{\partial\pp{K_i+\Ccond\ov{B}_h}}.
\end{align*}
In particular, we obtain
\begin{multline}\label{fluct_eq_condEsum}
	\EV[\G_t]{\metb[]{h}{x_0}{K_h\pp{\reach{t}{S}}+\Ccond\ov{B}_h}\charfun[\Espeed{h}(t)]}\\
	=	\sum_{i\in\N}\EV{\metb[]{h}{x_0}{K_i+\Ccond\ov{B}_h}\charfun[{\Espeed{h,i}}]}\charfun[{E_i(t)}],
\end{multline}
with the union $\Espeed{h}(t)	\coloneqq	\bigcup_{i\in\N}\pp{\Espeed{h,i}\cap E_i(t)}$ of disjoint events.
For all $0\leq t \leq \timebound{h}{x_0}{S}$ there holds
\begin{align}\label{fluct_eq_Et-Emax}
	\Espeed{h}(t)	\supset	\Espeed{h}\pp{\maxM{h}{x_0}{S}},
\end{align}
where the set $\maxM{h}{x_0}{S}$ is given by 
\begin{multline}\label{fluct_eq_defmaxM}
	\maxM{h}{x_0}{S}	\coloneqq	\pp{\ov{\comp{S}}+\cstable\vmax B_h}\\
								\cap \pp{S+\pp{1+\vmax \timebound{h}{x_0}{S}}\ov{B}_1+\pp{\Ccond+\sqrt{d}}\ov{B}_h},
\end{multline}
covering all of the area through which -- assuming that there is a stable $(h,\cstable)$-approximation -- the interface of $S$ could travel until time $\timebound{h}{x_0}{S}$, plus the additional area due to the discretization and expansion by $\Ccond\ov{B}_h$.
\end{lemma}

\begin{proof}
We first prove \eqref{fluct_eq_Et-Emax} by showing that $E_i(t)\neq \emptyset$ implies $\partial (K_i+\Ccond\ov{B}_h)\subset \maxM{h}{x_0}{S}$ and hence $\Espeed{h}\pp{\partial\pp{K_i+\Ccond\ov{B}_h}}\subset \Espeed{h}\pp{\maxM{h}{x_0}{S}}$ due to the monotonicity of the events $\pp{\Espeed{h}\pp{M}}_M$ from Assumption \ref{veff_aP_star}.
Let $E_i(t)\neq \emptyset$.
By the definition of $E_i(t)$ from \eqref{fluct_eq_Ei} there exists $(A,F)\in\Espeed{h}\pp{\partial S}$ with $K_i=	K_h\pp{\reach[(A,F)]{t}{S}}$.
In particular, $S$ has a stable $(h,\cstable)$-approximation $S_h$ with respect to $(A,F)$, which hence satisfies $\Rd\setminus\pp{\comp{S}+\cstable \vmax B_h}\subset S_h \subset S$.
On the one hand, due to the comparison principle and the stability of $S_h$ with Lemma \ref{stable_lem_preserv} we hence obtain
\begin{equation*}
	K_i=	K_h\pp{\reach[(A,F)]{t}{S}}	\supset	\reach[(A,F)]{t}{S}
		\supset	\reach[(A,F)]{t}{S_h}
		\supset	\Rd\setminus\pp{\comp{S}+\cstable \vmax B_h}.
\end{equation*}
On the other hand, due the maximum speed of propagation from Lemma \ref{not_lem_vmax} we have 
\begin{equation*}
	K_i=	K_h\pp{\reach[(A,F)]{t}{S}}	\subset	\reach[(A,F)]{t}{S}	+\bp{-h,h}^d
			\subset	S+\pp{1+t\vmax}\ov{B}_1 +\bp{-h,h}^d.
\end{equation*}
Combining these two observations yields $\partial (K_i+\Ccond\ov{B}_h)\subset \maxM{h}{x_0}{S}$ and hence \eqref{fluct_eq_Et-Emax}.

Regarding \eqref{fluct_eq_EcondE}, we need to show for any $E\in\G_t$ that 
\begin{multline*}
	\EV{\metb[]{h}{x_0}{K_i+\Ccond\ov{B}_h}\charfun[\Espeed{h,i}]\charfun[E_i(t)]\charfun[E]}\\
	=	\EV{\metb[]{h}{x_0}{K_i+\Ccond\ov{B}_h}\charfun[\Espeed{h,i}]}\PM{E_i(t)\cap E}.
\end{multline*}
This follows because
\begin{enumerate}[label=\arabic*.), left=12pt]
\item	\label{fluct_EcondE1}
		$\metb[]{h}{x_0}{K_i+\Ccond\ov{B}_h}\charfun[\Espeed{h,i}]$ is $\F\pp{\ov{\comp{\pp{K_i+\Ccond\ov{B}_h}}} + \max\Bp{\cstable\vmax,\Cmeas}\ov{B}_h}$-measurable,
\item	\label{fluct_EcondE2}
		$E_i(t)\cap E\in\F\pp{K_i+\pp{\CG+3+\sqrt{d}+\vmax\cstable}\ov{B}_h}$,
\item	\label{fluct_EcondE3}
		for $\Ccond=\Ccond(\CG,\data)$ large enough,
				$\F\pp{\ov{\comp{\pp{K_i+\Ccond\ov{B}_h}}} + \max\Bp{\cstable\vmax,\Cmeas}\ov{B}_h}$
		and		$\F\pp{K_i+\pp{\CG+3+\sqrt{d}+\vmax\cstable}\ov{B}_h}$
		are $\Pm$-independent.
\end{enumerate}
Here, \ref{fluct_EcondE1} holds due to \eqref{fluct_eq_measMETstable} from Lemma \ref{fluct_lem_meas} since $K_i+\Ccond\ov{B}_h$ is $h$-fat by definition.
We obtain \ref{fluct_EcondE3} from the finite range of dependence in Assumption \ref{sett_aP_fin} because
\begin{align*}
	\dist&\pp{\ov{\comp{\pp{K_i+\Ccond\ov{B}_h}}} + \max\Bp{\cstable\vmax,\Cmeas}\ov{B}_h,\, K_i+\pp{\CG+3+\sqrt{d}+\vmax\cstable}\ov{B}_h}\\
		&	\geq \dist\pp{\comp{\pp{K_i+\Ccond\ov{B}_h}},\, K_i}	
			-\pp{\CG+3+\sqrt{d}+\vmax\cstable+\max\Bp{\cstable\vmax,\Cmeas}}h\\
		&	\geq		\Ccond h-\pp{\CG+3+\sqrt{d}+\vmax\cstable+\max\Bp{\cstable\vmax,\Cmeas}}h.
\end{align*}
It only remains to show \ref{fluct_EcondE2}.
By definition of $\G_t$ from \eqref{fluct_eq_Gt}, it is sufficient to deal with the case $E=E_j(s)\cap F$, where $j\in\N$, $s\leq t$ and $F\in\F\pp{K_j+\CG\ov{B}_h}$.
From Lemma \ref{fluct_lem_measEi} we know that $E\in\F\pp{K_j+\CG\ov{B}_h}$.

If $E_i(t)\cap E_j(s)=\emptyset$, then $E_i(t)\cap E=\emptyset$ and we are done.
Hence, assume that there is $(A,F)\in E_i(t)\cap E_j(s)$.
Since $\reach[(A,F)]{t}{S}\subset K_i$ and $(A,F)\in\Espeed{h}\pp{\partial S}$ by the definition of $E_i(t)$ in \eqref{fluct_eq_Ei}, the set evolution can not have reached much further than $K_i$ before.
In fact, with Lemma \ref{fluct_lem_reachRestr} we obtain $\reach[(A,F)]{s}{S}\subset K_i+\pp{2+\pp{1+\vmax\cstable}h}\ov{B}_1$ and thus
\begin{equation*}
	K_j	=		K_h\pp{\reach[(A,F)]{s}{S}}
		\subset	K_i+\pp{2+\pp{1+\vmax\cstable}h}\ov{B}_1+[-h,h]^d.
\end{equation*}
In summary, we obtain
\begin{equation*}
	E	\in	\F\pp{K_j+\CG\ov{B}_h}	\subset	\F\pp{K_i+\CG\ov{B}_h+\pp{2+\pp{1+\sqrt{d}+\vmax\cstable}h}\ov{B}_1}
\end{equation*}
and thus have shown \ref{fluct_EcondE2} since $E_i(s)\in\F\pp{K_i+\CG\ov{B}_h}$ with Lemma \ref{fluct_lem_measEi}.

Regarding \eqref{fluct_eq_condEsum}, plugging in the definition of $\Espeed{h}(t)$ and using the definition of $\Bp{E_i(t)}_{i\in\N}$ from \eqref{fluct_eq_Ei} we have 
\begin{align*}
	&\EV[\G_t]{\metb[]{h}{x_0}{K_h\pp{\reach{t}{S}}+\Ccond\ov{B}_h}\charfun[\Espeed{h}(t)]}\\
	&\qquad	=	\sum_{i\in\N}	\EV[\G_t]{\metb[]{h}{x_0}{K_h\pp{\reach{t}{S}}+\Ccond\ov{B}_h}\charfun[\Espeed{h,i}\cap E_i(t)]}\\
	&\qquad	=	\sum_{i\in\N}	\EV[\G_t]{\metb[]{h}{x_0}{K_i+\Ccond\ov{B}_h}\charfun[\Espeed{h,i}\cap E_i(t)]}\\
	&\qquad	=	\sum_{i\in\N}	\EV{\metb[]{h}{x_0}{K_i+\Ccond\ov{B}_h}\charfun[{\Espeed{h,i}}]}\charfun[{E_i(t)}],
\end{align*}
where in the final step we applied \eqref{fluct_eq_EcondE}.
\end{proof}

\subsection{Bounding the martingale increments}\label{subs_fluct_incr}
In this subsection, we will bound the increments of the martingale 
	$\M_t=\EV[\G_t]{\metb[]{h}{x_0}{S}} - \EV{\metb[]{h}{x_0}{S}}$
as introduced at the start of the section with the filtration $\Bp{\G_t}_{t\geq 0}$ defined in \eqref{fluct_eq_Gt}.
Once we obtain these bounds, we will apply our alternative to and our generalization of Azuma's inequality from Appendix \ref{s_azuma} to prove Proposition \ref{fluct_prop}.
In order to obtain the bounds for the martingale increments, we will follow the approach laid out in Subsection \ref{subsub_str_MtMs}.
However, to use the effective minimum speed from Assumption \ref{veff_aP_star}, we will restrict the arrival times to the event $\Espeed{h}\pp{\maxM{h}{x_0}{S}}$ with $\maxM{h}{x_0}{S}$ as defined in \eqref{fluct_eq_defmaxM} essentially covering the maximum range of the interface.
Since this breaks the independence used in Lemma \ref{fluct_lem_condE} to resolve the conditional expectation, we will have to relax this restriction at the appropriate occasion.
The error terms from these operations show up in the final bounds, but they are controlled with high probability, since Assumption \ref{veff_aP_star} guarantees that $\Espeed{h}\pp{\maxM{h}{x_0}{S}}$ holds with overwhelming probability.
Naturally, the probability bounds for very large fluctuations in Proposition \ref{fluct_prop} are limited by this bound from Assumption \ref{veff_aP_star}.

\begin{proposition}[Bounds for the martingale increments]\label{fluct_prop_MtMs}
Let $x_0\in\Rd$ and $S\subset\Rd$ be an $h$-fat set with bounded boundary.
Then there exists a constant $C=C(\data)>0$ such that the martingale $\pp{\M_t}_{t\geq 0}$ as defined in \eqref{fluct_eq_defMt} for all $0\leq s \leq t \leq\timebound{h}{x_0}{S}$ satisfies
\begin{align}\label{fluct_eq_MtMs}
	\abs{\M_t-\M_s}	&\leq		C(t-s) + C h \notag\\
					&\,\,		+	\pp{\delta_{t=0}+\delta_{s=0}}\timebound{h}{x_0}{S}\pp{1-\charfun[{\Espeed{h}\pp{\maxM{h}{x_0}{S}}}]}\notag\\
					&\,\,		+	3\timebound{h}{x_0}{S}\pp{\EV[\G_t]{1-\charfun[{\Espeed{h}\pp{\maxM{h}{x_0}{S}}}]}+\EV[\G_s]{1-\charfun[{\Espeed{h}\pp{\maxM{h}{x_0}{S}}}]}}\notag\\
					&\,\,		+	\timebound{h}{x_0}{S}\pp{1-\PM{\Espeed{h}\pp{\maxM{h}{x_0}{S}}}},
\end{align}
with $\delta_{t=0},\delta_{s=0}$ denoting Dirac-deltas and with $\maxM{h}{x_0}{S}$ as introduced in Lemma \ref{fluct_lem_condE} covering the maximum range of the interface until time $\timebound{h}{x_0}{S}$ and given by
\begin{multline}
	\tag{\ref{fluct_eq_defmaxM}}
	\maxM{h}{x_0}{S}	\coloneqq	\pp{\ov{\comp{S}}+\cstable\vmax B_h}\\
								\cap \pp{S+\pp{1+\vmax \timebound{h}{x_0}{S}}\ov{B}_1+\pp{\Ccond+\sqrt{d}}\ov{B}_h}.
\end{multline}
\end{proposition}

\begin{proof}[Proof of Proposition \ref{fluct_prop_MtMs}]
We will follow the same strategy as in Subsection \ref{subsub_str_MtMs} to obtain good bounds for the increments uniformly on $\Omega$.
However, as described above we will not obtain uniform bounds but might have to restrict ourselves to events such as for example $\Espeed{h}\pp{\maxM{h}{x_0}{S}}$ from Assumption \ref{veff_aP_star}, which guarantees the existence of stable $(h,\cstable)$-approximations and the applicability of the effective minimum speed.
For some shorthand notation, we write $\maxM[_]{h}{x_0}{S}\coloneqq \maxM{h}{x_0}{S}$ denoting the area covering the potential range of the interface and its modifications until time $\timebound{h}{x_0}{S}$ as defined in \eqref{fluct_eq_defmaxM}. 

We first restrict the conditional expectation in $\M_t$ to $\Espeed{h}\pp{\maxM[_]{h}{x_0}{S}}$ in order to be able to use the effective minimum speed from Assumption \ref{veff_aP_star}.
That is, at the cost of an error term $R_1^t$ we obtain
\begin{align*}
	\M_t		&= \EV[\G_t]{\metb[]{h}{x_0}{S}} - \EV{\metb[]{h}{x_0}{S}}\\
			&=	\EV[\G_t]{\met[]{h}{x_0}{S}\charfun[\Espeed{h}\pp{\maxM[_]{h}{x_0}{S}}]} - \EV{\metb[]{h}{x_0}{S}} + R_1^t
\end{align*}
with 
\begin{equation*}
	0 \leq R_1^t	\leq		\timebound{h}{x_0}{S}	\EV[\G_t]{1-\charfun[\Espeed{h}\pp{\maxM[_]{h}{x_0}{S}}]}.
\end{equation*}
Next, we separate the case that $\met[]{h}{x_0}{S}\leq s$, that is that $x_0$ is within the environment already explored at time $t\geq s$, which makes the arrival time measurable with respect to $\G_t$, see Lemma \ref{fluct_lem_measIn}.
In the remaining case that $x_0$ has not been reached at time $s$, that is $\met[]{h}{x_0}{S}> s$, we can write $\metb[]{h}{x_0}{S}=s+\metb[]{h}{x_0}{\reach{s}{S}}$ since Lemma \ref{stable_lem_TvsEspeed} guarantees that the truncation of the arrival times does not kick in because of the restriction to $\Espeed{h}\pp{\maxM[_]{h}{x_0}{S}}$. 
Writing 
	$\Bp{\met[]{h}{x_0}{S}\leq s}=\Bp{(\wt{A},\wt{F})\,:\,\met[(\wt{A},\wt{F})]{h}{x_0}{S}\leq s}$ 
and introducing the error term $R_0^t$, we obtain
\begin{align*}
	\M_t		&=	\EV[\G_t]{\met[]{h}{x_0}{S}\charfun[\Espeed{h}\pp{\partial S}]\charfun[\Espeed{h}\pp{\maxM[_]{h}{x_0}{S}}]} - \EV{\metb[]{h}{x_0}{S}} + R_1^t\\
			&=	\EV[\G_t]{\charfun[\Espeed{h}\pp{\maxM[_]{h}{x_0}{S}}]}\met[]{h}{x_0}{S}\charfun[\Bp{\met[]{h}{x_0}{S}\leq s}]\charfun[\Espeed{h}\pp{\partial S}] \\
			&\quad		+\EV[\G_t]{\pp{s+\met[]{h}{x_0}{\reach{s}{S}}}\charfun[\Espeed{h}\pp{\maxM[_]{h}{x_0}{S}}]}\charfun[\Bp{\met[]{h}{x_0}{S}> s}]\charfun[\Espeed{h}\pp{\partial S}]\\
			&\quad		- \EV{\metb[]{h}{x_0}{S}} + R_1^t + R_0^t,
\end{align*}
where we used that $\Espeed{h}\pp{\maxM[_]{h}{x_0}{S}}\subset\Espeed{h}\pp{\partial S}$ and then applied Lemma \ref{fluct_lem_measIn}.
The error $R_0^t$ captures that for $t=0$ only $\Bp{\met[]{h}{x_0}{S}> 0}\in\Bp{\emptyset,\Omega}$ is $\G_0$-measurable, but not necessarily $\Bp{\met[]{h}{x_0}{S}> 0}\cap\Espeed{h}\pp{\partial S}\in\G_0$:
\begin{align*}
	0\leq
	R_0^t		&=		\delta_{t=0}	\EV[\G_0]{\met[]{h}{x_0}{S}\charfun[\Espeed{h}\pp{\maxM[_]{h}{x_0}{S}}]}\charfun[\Bp{\met[]{h}{x_0}{S}> 0}]\pp{1-\charfun[\Espeed{h}\pp{\partial S}]}\\
				&\leq		\delta_{t=0}\timebound{h}{x_0}{S}\pp{1-\charfun[\Espeed{h}\pp{\partial S}]}.
\end{align*}
Since these considerations analogously hold for $\M_s$, we obtain
\begin{align*}
	\abs{\M_t-\M_s}	&\leq		\abs{\EV[\G_t]{\charfun[\Espeed{h}\pp{\maxM[_]{h}{x_0}{S}}]}-\EV[\G_s]{\charfun[\Espeed{h}\pp{\maxM[_]{h}{x_0}{S}}]}}\\
					&\qquad\qquad\qquad\qquad\quad	\times
									\met[]{h}{x_0}{S}\charfun[\Bp{\met[]{h}{x_0}{S}\leq s}]\charfun[\Espeed{h}\pp{\partial S}]\\
					&\,\,		+	\abs{\EV[\G_t]{\charfun[\Espeed{h}\pp{\maxM[_]{h}{x_0}{S}}]}-\EV[\G_s]{\charfun[\Espeed{h}\pp{\maxM[_]{h}{x_0}{S}}]}}\\
					&\qquad\qquad\qquad\qquad\quad	\times
									s\charfun[\Bp{\met[]{h}{x_0}{S}> s}]\charfun[\Espeed{h}\pp{\partial S}]\\
					&\,\,		+	\abs{\EV[\G_t]{\met[]{h}{x_0}{\reach{s}{S}}\charfun[\Espeed{h}\pp{\maxM[_]{h}{x_0}{S}}]}
									-	\EV[\G_s]{\met[]{h}{x_0}{\reach{s}{S}}\charfun[\Espeed{h}\pp{\maxM[_]{h}{x_0}{S}}]}}\\
					&\qquad\qquad\qquad\qquad\quad	\times
									\charfun[\Bp{\met[]{h}{x_0}{S}> s}]\charfun[\Espeed{h}\pp{\partial S}]\\
					&\,\,		+	\abs{R_1^t - R_1^s + R_0^t - R_0^s}.
\end{align*}
Viewing the first two contributions as another error term $R_2$, we have 
\begin{align}
	\abs{\M_t-\M_s}	&\leq	\abs{\EV[\G_t]{\met[]{h}{x_0}{\reach{s}{S}}\charfun[\Espeed{h}\pp{\maxM[_]{h}{x_0}{S}}]}
									-	\EV[\G_s]{\met[]{h}{x_0}{\reach{s}{S}}\charfun[\Espeed{h}\pp{\maxM[_]{h}{x_0}{S}}]}}\notag	\\
					&\quad	+ R_2	+ 	\abs{R_1^t - R_1^s + R_0^t - R_0^s},		\label{fluct_eq_diffMt}
\end{align}
where with $\met[]{h}{x_0}{S}\charfun[\Bp{\met[]{h}{x_0}{S}\leq s}]\leq s\charfun[\Bp{\met[]{h}{x_0}{S}\leq s}]$ we have
\begin{align*}
	R_2	&\leq		s\abs{\EV[\G_t]{\charfun[\Espeed{h}\pp{\maxM[_]{h}{x_0}{S}}]}-\EV[\G_s]{\charfun[\Espeed{h}\pp{\maxM[_]{h}{x_0}{S}}]}}\\
		&=			s\abs{\EV[\G_t]{1-\charfun[\Espeed{h}\pp{\maxM[_]{h}{x_0}{S}}]}-\EV[\G_s]{1-\charfun[\Espeed{h}\pp{\maxM[_]{h}{x_0}{S}}]}}.
\end{align*}
Ignoring the error terms in \eqref{fluct_eq_diffMt}, we essentially almost eliminated dependence on the information of the already explored area by time $s$, encoded in $\G_s$ -- at least if we knew that $\reach{s}{S}$ was strictly expanding. 

Regarding the conditional expectation with respect to $\G_t$, we want to replace $\reach{s}{S}$ with $\reach{t}{S}$.
For $(A,F)\in\Espeed{h}\pp{\maxM[_]{h}{x_0}{S}}$, let $S_h=S_h^{(A,F)}$ be the stable $(h,\cstable)$-approximation from \ref{veff_Pspeed} in Assumption \ref{veff_aP_star}, which exists since $\partial S\subset \maxM[_]{h}{x_0}{S}$ by definition \ref{fluct_eq_defmaxM}.
Then we have 
	$\reach{\tilde{t}}{S_h}\subset\reach{\tilde{t}}{S}\hsubs\reach{\tilde{t}+\cstable h}{S_h}$ for all $\tilde{t}\geq 0$
on $\Espeed{h}\pp{\maxM[_]{h}{x_0}{S}}$ due to the comparison principle as in Lemma \ref{not_lem_comp}.
On the one hand, note that
\begin{align*}
	\met[]{h}{x_0}{\reach{s}{S}}\charfun[\Espeed{h}\pp{\maxM[_]{h}{x_0}{S}}]
	&\geq	\met[]{h}{x_0}{\reach{s+\cstable h}{S_h}}\charfun[\Espeed{h}\pp{\maxM[_]{h}{x_0}{S}}]\\
	&\geq	\met[]{h}{x_0}{\reach{t+\cstable h}{S_h}}\charfun[\Espeed{h}\pp{\maxM[_]{h}{x_0}{S}}]\\
	&\geq	\met[]{h}{x_0}{\reach{t}{S_h}}\charfun[\Espeed{h}\pp{\maxM[_]{h}{x_0}{S}}]-\cstable h\charfun[\Espeed{h}\pp{\maxM[_]{h}{x_0}{S}}]\\
	&\geq	\met[]{h}{x_0}{\reach{t}{S}}\charfun[\Espeed{h}\pp{\maxM[_]{h}{x_0}{S}}]-\cstable h\charfun[\Espeed{h}\pp{\maxM[_]{h}{x_0}{S}}]
\end{align*}
where we first used that if $S_1\hsubs S_2$ and $x\in S_1+\ov{B}_h$, then $x\in S_2+\ov{B}_h$.
The second inequality is due to $S_h$ being stable, the third and fourth inequality follow from the definition of the arrival times.
On the other hand, we have
\begin{align*}
	\met[]{h}{x_0}{\reach{s}{S}}\charfun[\Espeed{h}\pp{\maxM[_]{h}{x_0}{S}}]
	&\leq		\met[]{h}{x_0}{\reach{t}{S}}\charfun[\Espeed{h}\pp{\maxM[_]{h}{x_0}{S}}] + (t-s)\charfun[\Espeed{h}\pp{\maxM[_]{h}{x_0}{S}}].
\end{align*}
In summary, we obtain
\begin{equation}\label{fluct_eq_Rs-Rt}
	\abs{\met[]{h}{x_0}{\reach{s}{S}}-\met[]{h}{x_0}{\reach{t}{S}}}\charfun[\Espeed{h}\pp{\maxM[_]{h}{x_0}{S}}]
	\leq		\max\Bp{\cstable h,\,t-s}\charfun[\Espeed{h}\pp{\maxM[_]{h}{x_0}{S}}].
\end{equation}

Now that we can replace $\reach{s}{S}$ with $\reach{t}{S}$, we want to resolve the conditional expectation in $\EV[\G_t]{\met[]{h}{x_0}{\reach{t}{S}}\charfun[\Espeed{h}\pp{\maxM[_]{h}{x_0}{S}}]}$ (and in the corresponding term with $s$ instead of $t$) via  Lemma \ref{fluct_lem_condE}.
However, in order to apply Lemma \ref{fluct_lem_condE}, we need to replace $\reach{t}{S}$ with a slightly larger set, for which the evolution is independent of the coefficient field in the already explored area. 
For $\Ccond=\Ccond(\CG,\data)>0$ from Lemma \ref{fluct_lem_condE}, on $\Espeed{h}\pp{\maxM[_]{h}{x_0}{S}}$ we have
\begin{equation*}
	\reach{t}{S}	\subset K_h\pp{\reach{t}{S}}	+ \Ccond\ov{B}_h
				\hsubs	\reach{t+\cstable h + \Ccond\vmineff^{-1} h}{S_h}.
\end{equation*}
The last inclusion `$\hsubs$' above comes from as in the proof of Lemma \ref{stable_lem_TvsEspeed} iteratively applying the effective minimum speed of propagation guaranteed by $\Espeed{h}\pp{\maxM[_]{h}{x_0}{S}}$ since 
\begin{equation*}
	\pp{\reach{t}{S} + \Ccond\pp{\data} \ov{B}_h}\cap		\ov{\comp{(S_h)}}	\subset \maxM[_]{h}{x_0}{S}
\end{equation*}
for $t\leq \timebound{h}{x_0}{S}$ by the definition of $\maxM[_]{h}{x_0}{S}$ in \eqref{fluct_eq_defmaxM}.
We hence have
\begin{align}\label{fluct_eq_Rt-KRt}
	\met[]{h}{x_0}{\reach{t}{S}}\charfun[\Espeed{h}\pp{\maxM[_]{h}{x_0}{S}}]
	&\leq		\met[]{h}{x_0}{K_h\pp{\reach{t}{S}}+\Ccond\ov{B}_h}\charfun[\Espeed{h}\pp{\maxM[_]{h}{x_0}{S}}]\notag\\
	&\leq		\pp{\met[]{h}{x_0}{\reach{t}{S}}+\cstable h + \Ccond\vmineff^{-1} h}\charfun[\Espeed{h}\pp{\maxM[_]{h}{x_0}{S}}].
\end{align}
Replacing $\reach{s}{S}$ with $\reach{t}{S}$, then $\reach{t}{S}$ with $K_h\pp{\reach{t}{S}}+\Ccond\ov{B}_h$ and $\reach{s}{S}$ with $K_h\pp{\reach{s}{S}}+\Ccond\ov{B}_h$ in \eqref{fluct_eq_diffMt}, we obtain
\begin{align*}
	\abs{\M_t-\M_s}	&\leq		\Big|\EV[\G_t]{\met[]{h}{x_0}{K_h\pp{\reach{t}{S}}+\Ccond\ov{B}_h}\charfun[\Espeed{h}\pp{\maxM[_]{h}{x_0}{S}}]}\\
					&\qquad\qquad\qquad\qquad		-	\EV[\G_s]{\met[]{h}{x_0}{K_h\pp{\reach{s}{S}}+\Ccond\ov{B}_h}\charfun[\Espeed{h}\pp{\maxM[_]{h}{x_0}{S}}]}\Big|\\
					&\quad	+ R_3 + R_2	+ 	\abs{R_1^t - R_1^s + R_0^t - R_0^s},
\end{align*}
where due to \eqref{fluct_eq_Rs-Rt} and \eqref{fluct_eq_Rt-KRt} the corresponding error term $R_3$ is bound by 
\begin{equation*}
	\abs{R_3}	\leq	\max\Bp{\cstable h,\,t-s}+\cstable h + \Ccond\vmineff^{-1} h.
\end{equation*}
In order to apply Lemma \ref{fluct_lem_condE}, we now only need to get rid of the restriction to $\Espeed{h}\pp{\maxM[_]{h}{x_0}{S}}$.
Because $S\subset K_h\pp{\reach{t}{S}}+\Ccond\ov{B}_h$ on $\Espeed{h}\pp{\partial S}$ with Lemma \ref{fluct_lem_KiS} and since $\Ccond$ was chosen accordingly, we have $\timebound{h}{x_0}{S}\geq \timebound{h}{x_0}{K_h\pp{\reach{t}{S}}+\Ccond\ov{B}_h}$ and hence obtain
\begin{multline}\label{fluct_eq_mK-mbK}
	0\leq	\metb[]{h}{x_0}{K_h\pp{\reach{t}{S}}+\Ccond\ov{B}_h}\pp{\charfun[\Espeed{h}\pp{t}]-\charfun[\Espeed{h}\pp{\maxM[_]{h}{x_0}{S}}]}\\
	\leq		\timebound{h}{x_0}{S}\pp{1-\charfun[\Espeed{h}\pp{\maxM[_]{h}{x_0}{S}}]},
\end{multline}
where $\Espeed{h}\pp{t}\supset\Espeed{h}\pp{\maxM[_]{h}{x_0}{S}}$ is the less restrictive event defined in Lemma \ref{fluct_lem_condE}, which also is independent from the coefficient field in the already explored area.
Thus, applying Lemma \ref{fluct_lem_condE} after replacing the restriction to $\Espeed{h}\pp{\maxM[_]{h}{x_0}{S}}$ with one to $\Espeed{h}\pp{t}$ at the cost of the error term $R_4$, we obtain
\begin{align*}
	\abs{\M_t-\M_s}	&\leq		\Bigg|\sum_{i\in\N}\EV{\metb[]{h}{x_0}{K_i+\Ccond\ov{B}_h}\charfun[\Espeed{h}\pp{t}]}\charfun[E_i(t)]  \\
					&\qquad\qquad\qquad\qquad	-	\sum_{j\in\N}\EV{\metb[]{h}{x_0}{K_j + \Ccond\ov{B}_h}\charfun[\Espeed{h}\pp{s}]}\charfun[E_j(s)] \Bigg|\\
					&\quad	+ R_4 + R_3 + R_2	+ 	\abs{R_1^t - R_1^s + R_0^t - R_0^s},
\end{align*}
with 
\begin{equation*}
	R_4	\leq		\timebound{h}{x_0}{S}\pp{\EV[\G_t]{1-\charfun[\Espeed{h}\pp{\maxM[_]{h}{x_0}{S}}]}
											+	\EV[\G_s]{1-\charfun[\Espeed{h}\pp{\maxM[_]{h}{x_0}{S}}]}}.
\end{equation*}
With Lemma \ref{fluct_lem_EiCup} we have $\bigcup E_i(t)=\Espeed{h}\pp{\partial S}\supset E_j(s)$ for any $j\in\N$ and therefore
\begin{align}
	\abs{\M_t-\M_s}	&\leq		\sum_{i,j\in\N}\charfun[E_i(t)\cap E_j(s)]	\Big|\EV{\metb[]{h}{x_0}{K_i+\Ccond\ov{B}_h}\charfun[\Espeed{h}\pp{t}]} \notag\\
					&\qquad\qquad\qquad\qquad\qquad\qquad								-	\EV{\metb[]{h}{x_0}{K_j+\Ccond\ov{B}_h}\charfun[\Espeed{h}\pp{s}]}	\Big|\notag\\
					&\quad	+ R_4 + R_3 + R_2	+ 	\abs{R_1^t - R_1^s + R_0^t - R_0^s}.\label{fluct_eq_diffMt4}
\end{align}
Now assume that there exists $(\wt{A},\wt{F})\in E_i(t)\cap E_j(s)\subset\Espeed{h}(\partial S)$ for $i,j\in\N$.
Let $S_h$ be the stable $(h,\cstable)$-approximation of $S$ with respect to $(\wt{A},\wt{F})$.
On the one hand, we have 
\begin{multline*}
	\reach[(\wt{A},\wt{F})]{s}{S}	\hsubs \reach[(\wt{A},\wt{F})]{\cstable h + s}{S_h}	
		\subset \reach[(\wt{A},\wt{F})]{s}{S_h} + \pp{1+\vmax\cstable h}\ov{B}_1\\
		\subset	\reach[(\wt{A},\wt{F})]{t}{S_h} + \pp{1+\vmax\cstable h}\ov{B}_1
		\subset	\reach[(\wt{A},\wt{F})]{t}{S} + \pp{1+\vmax\cstable h}\ov{B}_1
\end{multline*}
due to the comparison principle from Lemma \ref{not_lem_comp}, the maximum speed of propagation from Lemma \ref{not_lem_vmax}, and the stability of $S_h$.
On the other hand, again due to the maximum speed of propagation we have 
\begin{equation*}
	\reach[(\wt{A},\wt{F})]{t}{S}	\subset	\reach[(\wt{A},\wt{F})]{s}{S} + \pp{1+\vmax(t-s)}\ov{B}_1.
\end{equation*}
On the level of the discretization of the states, this implies that 
\begin{align*}
	K_j&=K_h\pp{\reach[(\wt{A},\wt{F})]{s}{S}}\subset	K_i + \pp{1+\pp{\sqrt{d}+\vmax\cstable} h}\ov{B}_1,\\
	K_i&=K_h\pp{\reach[(\wt{A},\wt{F})]{t}{S}}\subset	K_j + \pp{1+\vmax(t-s)+\sqrt{d}h}\ov{B}_1.
\end{align*}
In particular, with the effective minimum speed of propagation, analogously to the proof of Lemma \ref{stable_lem_TvsEspeed}, for $(A,F)\in \Espeed{h}\pp{\maxM[_]{h}{x_0}{S}}$ and $S_k$ denoting the stable $(h,\cstable)$-approximation of $K_k+\Ccond\ov{B}_h$ for $k=i,j$ we obtain 
\begin{align*}
	K_j+\Ccond\ov{B}_h	&\hsubs	\reach[(A,F)]{\cstable h + \vmineff^{-1}\pp{1+\pp{1+\sqrt{d}+\vmax\cstable} h}}{S_i},	\\
	K_i+\Ccond\ov{B}_h	&\hsubs	\reach[(A,F)]{\cstable h + \vmineff^{-1}\pp{1+\vmax(t-s)+\pp{1+\sqrt{d}}h}}{S_j}.
\end{align*}
For the arrival times, with the comparison principle as in Lemma \ref{not_lem_comp} this yields
\begin{align*}
	&\met[]{h}{x_0}{K_i+\Ccond\ov{B}_h}\charfun[\Espeed{h}\pp{\maxM[_]{h}{x_0}{S}}]		\\
	&\,\,\,	\geq		\pp{\met[]{h}{x_0}{K_j+\Ccond\ov{B}_h}
					-\cstable h - \vmineff^{-1}\pp{1+\pp{1+\sqrt{d}+\vmax\cstable} h}}\charfun[\Espeed{h}\pp{\maxM[_]{h}{x_0}{S}}],\\
	&\met[]{h}{x_0}{K_j+\Ccond\ov{B}_h}\charfun[\Espeed{h}\pp{\maxM[_]{h}{x_0}{S}}]		\\
	&\,\,\,\geq		\pp{\met[]{h}{x_0}{K_i+\Ccond\ov{B}_h}
					-\cstable h - \vmineff^{-1}\pp{1+\vmax(t-s)+\pp{1+\sqrt{d}}h}}\charfun[\Espeed{h}\pp{\maxM[_]{h}{x_0}{S}}].
\end{align*}
In order to apply these observations to \eqref{fluct_eq_diffMt4}, we move back to the restriction on $\Espeed{h}\pp{\maxM[_]{h}{x_0}{S}}$,
\begin{align*}
	\abs{\M_t-\M_s}	&\leq		\sum_{i,j\in\N}\charfun[E_i(t)\cap E_j(s)]	\Big|\Ev\Big[\Big(\met[]{h}{x_0}{K_i+\Ccond\ov{B}_h}\\
					&\qquad\qquad\qquad\qquad\qquad\qquad						-\met[]{h}{x_0}{K_j+\Ccond\ov{B}_h}\Big) \charfun[\Espeed{h}\pp{\maxM[_]{h}{x_0}{S}}]\Big]	\Big|\\
					&\quad	+ R_5 + R_4 + R_3 + R_2	+ 	\abs{R_1^t - R_1^s + R_0^t - R_0^s},
\end{align*}
where we again bound the resulting error via \eqref{fluct_eq_mK-mbK} with
\begin{align*}
	R_5 &\leq \timebound{h}{x_0}{S}\abs{\EV{\pp{1-\charfun[\Espeed{h}\pp{\maxM[_]{h}{x_0}{S}}]}\pp{\charfun[\Espeed{h}\pp{t}]-\charfun[\Espeed{h}\pp{s}]}}}\\
		&\leq		\timebound{h}{x_0}{S}\pp{1-\PM{\Espeed{h}\pp{\maxM[_]{h}{x_0}{S}}}}.							
\end{align*}
Now plugging in the inequalities for $E_i(t)\cap E_j(s)\neq\emptyset$, we obtain
\begin{align*}
	\abs{\M_t-\M_s}	&\leq		\cstable h + \vmineff^{-1}\pp{1+\pp{1+\sqrt{d}}h+\vmax\max\Bp{\cstable h, (t-s)}} \\
					&\quad	+ R_5 + R_4 + R_3 + R_2	+ 	\abs{R_1^t - R_1^s + R_0^t - R_0^s}.
\end{align*}
Collecting the bounds for the error terms $R_0^t$, $R_1^t$, $R_2$, $R_3$, $R_4$, and $R_5$ and using that $1\leq h$ thus yields \eqref{fluct_eq_MtMs}.
\end{proof}

Now that with Proposition \ref{fluct_prop_MtMs} we have established the `almost' uniform bounds (see Lemma \ref{azuma_lem_charfun}) for the martingale increments, we are able to prove Proposition \ref{fluct_prop}.
The first two inequalities in \eqref{fluct_eq_PMfluct} will follow from an application of Proposition \ref{azuma_prop_alternative}, our alternative to Azuma's inequality.
The final inequality trivially follows from the truncation imposed on the arrival times.

\begin{proof}[Proof of Proposition \ref{fluct_prop}]
Without loss of generality, let $\dist(x_0,S)>h$, since otherwise $\metb{h}{x_0}{S}=\EV{\metb{h}{x_0}{S}}=0$.
We will now apply Proposition \ref{azuma_prop_alternative}, our replacement for Azuma's inequality to the martingale $(\wt{\M}_n)_{n\in\Nz}$ defined by
\begin{align*}
	\wt{\M}_n	&\coloneqq	
	\left\lbrace	\begin{aligned}
		&\M_{n\Delta t}						&&\text{for }n\Delta t<\timebound{h}{x_0}{S},\\
		&\M_{\timebound{h}{x_0}{S}}			&&\text{for }n\Delta t\geq\timebound{h}{x_0}{S},
	\end{aligned}	\right.
	&\text{with }\,\,	
	\Delta t\coloneqq	h,
\end{align*}
with $\pp{\M_t}_{t}$ defined in \eqref{fluct_eq_defMt} and $\Delta t$ chosen in view of \eqref{fluct_eq_MtMs} from Proposition \ref{fluct_prop_MtMs}.
We set
\begin{equation}\label{fluct_eq_defN}
	N	\coloneqq \left\lceil	\frac{\timebound{h}{x_0}{S_0}}{\Delta t}	\right\rceil
		\leq		C\frac{\dist(x_0,S)}{h},
\end{equation}
with the inequality satisfied for some $C=C(\data)>0$ since $\dist(x_0,S)>h$.

Since $\G_0=\Bp{\emptyset,\Omega}$ there holds $\wt{\M}_0=0$.
With \eqref{fluct_eq_measMTbound} from Lemma \ref{fluct_lem_measIn} we have
\begin{align}
	\wt{M}_N		&=	\EV[\G_{\timebound{h}{x_0}{S}}]{\metb[]{h}{x_0}{S}} -\EV{\metb[]{h}{x_0}{S}}\notag\\
				&=	\metb[]{h}{x_0}{S}\charfun[\Espeed{h}\pp{\partial S}] -\EV{\metb[]{h}{x_0}{S}} + R_1\notag\\
				&=	\metb[]{h}{x_0}{S}	- \EV{\metb[]{h}{x_0}{S}}+	R_1 + R_2\label{fluct_eq_MN}
\end{align}
with 
\begin{align*}
	\abs{R_1}	&\leq	\timebound{h}{x_0}{S}\EV[\G_{\timebound{h}{x_0}{S}}]{1-\charfun[\Espeed{h}\pp{\partial S}]},\\
	\abs{R_2}	&\leq	\timebound{h}{x_0}{S}\pp{1-\charfun[\Espeed{h}\pp{\partial S}]}.
\end{align*}

In order to apply our replacement for Azuma's inequality, we will first show that there is an event $E_{\M}\in\F(\Rd)$ such that 
\begin{equation}\label{fluct_eq_incrMn}
	\abs{\wt{\M}_{n+1}-\wt{\M}_n}\charfun[E_{\M}]	\leq  Ch
	\quad\text{for all }n\in\N
\end{equation}
for some $C=C(\data)>0$, which holds with overwhelming probability, that is
\begin{equation*}
	1-\PM{E_{\M}}\leq C\exp\pp{-\frac{\cspeed}{2} h^{\ratespeed}}.
\end{equation*}

From Assumption \ref{veff_aP_star} we have
\begin{align*}
	1-\PM{\Espeed{h}\pp{\maxM{h}{x_0}{S}}}	
	&\leq	\Cspeed\pp{1+\frac{\diam(\maxM{h}{x_0}{S})}{h}}^d\exp\pp{-\cspeed h^{\ratespeed}}\\
	&\leq	C\pp{1+\frac{\diam(\partial S)+\dist(x_0,S)}{h}}^d\exp\pp{-\cspeed h^{\ratespeed}}
\end{align*}
for some constant $C=C(\data)$.
Regarding the last term in the bounds for the martingale increments from \eqref{fluct_eq_MtMs}, we hence have
\begin{multline}\label{fluct_eq_boundTPE}
	\timebound{h}{x_0}{S}\pp{1-\PM{\Espeed{h}\pp{\maxM{h}{x_0}{S}}}}	\\
	\leq		C\pp{1+\frac{\diam(\partial S)+\dist(x_0,S)}{h}}^{d+1}\exp\pp{-\cspeed h^{\ratespeed}}
	\leq		C\exp\pp{-\frac{\cspeed}{2} h^{\ratespeed}}
\end{multline}
since we assumed in \eqref{fluct_eq_hrateA} that $\partial S$ and $\dist(x_0,S)$ are not exponentially larger than $h$.

Regarding the terms with conditional expectations in \eqref{fluct_eq_MtMs},  Lemma \ref{azuma_lem_charfun} yields that 
\begin{align*}
	E_{\M}	&\coloneqq	\Espeed{h}\pp{\maxM{h}{x_0}{S}}\\
			&\,\,\quad	\cap\bigg\lbrace(A,F)\in\Omega	\,:\,	
							\EV[\G_{n\Delta t\wedge\timebound{h}{x_0}{S}}]{1-\charfun[{\Espeed{h}\pp{\maxM{h}{x_0}{S}}}]}((A,F))\\
			&\qquad\qquad\qquad\qquad\qquad\qquad\qquad\qquad\qquad
						\leq	\frac{h}{\timebound{h}{x_0}{S}}	\quad	\text{for all $n\in\N$}\bigg\rbrace
\end{align*}
as in \eqref{fluct_eq_boundTPE} for some $C=C(\data)>0$ satisfies 
\begin{equation}\label{fluct_eq_ProbEM}
	1-\PM{E_{\M}}	\leq	\pp{1+\frac{\timebound{h}{x_0}{S}}{h}}\pp{1-\PM{\Espeed{h}\pp{\maxM{h}{x_0}{S}}}}
					\leq		C\exp\pp{-\frac{\cspeed}{2} h^{\ratespeed}}.
\end{equation}
Plugging these definitions and \eqref{fluct_eq_boundTPE} into the bounds for the martingale increments from \eqref{fluct_eq_MtMs} in Proposition \ref{fluct_prop_MtMs} yields for all $n\in\N$ that
\begin{align*}
	\abs{\wt{\M}_{n+1}-\wt{\M}_n}\charfun[E_{\M}]	
	&\leq 	Ch + Ch + 0 + 6 h + C\exp\pp{-\frac{\cspeed}{2} h^{\ratespeed}}
\end{align*}
and hence we obtain that this definition of $E_{\M}$ satisfies \eqref{fluct_eq_incrMn}.

Now that we have the bounds from \eqref{fluct_eq_incrMn} for the increments on the set $E_{\M}$ satisfying \eqref{fluct_eq_ProbEM}, we apply our replacement for Azuma's inequality.
Since further $\abs{\wt{\M}_n}\leq \timebound{h}{x_0}{S}$ for all $n\in\Nz$, via Proposition \ref{azuma_prop_alternative} we obtain
\begin{equation*}
 	\EV{\wt{M}_N^{2k}\charfun[E_{\M}]}	\leq		(2k)!(Ch)^{2k}N^k + C(2k)!N^k\timebound{h}{x_0}{S}^{2k}\pp{1-\PM{E_{\M}}}
\end{equation*} 
for all $k\in\N$. 
Based on \eqref{fluct_eq_MN} with the errors $R_1,R_2$, we know that for the endpoint of the martingale we have
\begin{equation*}
	\wt{M}_N\charfun[E_{\M}]=\pp{\metb[]{h}{x_0}{S}	- \EV{\metb[]{h}{x_0}{S}}}\charfun[E_{\M}] + R_1\charfun[E_{\M}]
\end{equation*}
where due to the definition of $E_{\M}$
\begin{equation*}
	\abs{R_1}\charfun[E_{\M}]	\leq h.
\end{equation*}
Therefore, with the value of $C=C(\data)>0$ changing from line to line, we have
\begin{align*}
	&\EV{\pp{\metb[]{h}{x_0}{S}	- \EV{\metb[]{h}{x_0}{S}}}^{2k}}	\\
	&\qquad\qquad	\leq		\EV{\pp{\wt{M}_N-R_1}^{2k}\charfun[E_{\M}]} + C\timebound{h}{x_0}{S}^{2k}\pp{1-\PM{E_{\M}}}\\
	&\qquad\qquad	\leq		\pp{\EV{\wt{M}_N^{2k}\charfun[E_{\M}]}^{\frac{1}{2k}}+h}^{2k} + C\timebound{h}{x_0}{S}^{2k}\pp{1-\PM{E_{\M}}}\\
	&\qquad\qquad	\leq		(2k)!(Ch)^{2k}N^k + C(2k)!N^k\timebound{h}{x_0}{S}^{2k}\pp{1-\PM{E_{\M}}}.
\end{align*}
Normalizing this in view of Proposition \ref{azuma_prop_alternative}, we have
\begin{align*}
	&\EV{\pp{\frac{\metb[]{h}{x_0}{S}	- \EV{\metb[]{h}{x_0}{S}}}{Ch}}^{2k}}	\\
	&\qquad\qquad	\leq		(2k)!N^k + C(2k)!N^k\pp{\frac{\timebound{h}{x_0}{S}}{h}}^{2k}\pp{1-\PM{E_{\M}}}.
\end{align*}
As in Proposition \ref{azuma_prop_alternative}, from these moment bounds we obtain fluctuation bounds corresponding to \eqref{azuma_eq_alt-PM}, which with $\log\pp{1-\PM{E_{\M}}}\approx -h^{\ratespeed}$ as in \eqref{fluct_eq_ProbEM} for $\wt{\lambda}>0$ yields
\begin{multline*}
	\PM{\abs{\frac{\metb[]{h}{x_0}{S}	- \EV{\metb[]{h}{x_0}{S}}}{Ch}}\geq \wt{\lambda}}\\
	\leq		\left\lbrace\begin{aligned}
					&C\exp\pp{-\frac{\wt{\lambda}}{2\sqrt{N}}}		
							&&\text{if }\wt{\lambda}\leq	\frac{ch^{\ratespeed}\sqrt{N}}{\log\pp{\frac{\timebound{h}{x_0}{S}}{h}}},	\\
					&C\exp\pp{-\frac{ch^{\ratespeed}}{\log\pp{\frac{\timebound{h}{x_0}{S}}{h}}}}	
							&&\text{if }\wt{\lambda}\geq	\frac{ch^{\ratespeed}\sqrt{N}}{\log\pp{\frac{\timebound{h}{x_0}{S}}{h}}}.
			\end{aligned}\right.
\end{multline*}
Plugging in $\timebound{h}{x_0}{S}\approx\dist\pp{x_0,S}$ and $N\approx\frac{\dist\pp{x_0,S}}{h}$ from \eqref{fluct_eq_defN} as well as substituting $\lambda=Ch\wt{\lambda}$ yields the two central inequalities in \eqref{fluct_eq_PMfluct}.
\end{proof}

\section{Compensating the propagation of influence by increasing the forcing}\label{s_inf}
In this section we will show that we can compensate far-away influences on the evolution of sets by adding a small constant to the forcing.
More precisely, if a stable set is initially ahead of a second set within a large ball, 
we can guarantee that the set stays ahead near the center of this large ball for a finite time period at the cost of increasing the forcing by a constant -- no matter how much bigger the second set is outside of the large ball, see Figure \ref{inf_fig_ball}.
The larger the radius of the large ball, the less the forcing has to be increased.

However, with our approach we are only able to obtain this result if we assume that the diffusion matrix $A$ for the second-order term is constant in space, the reasons for which we will discuss in Remark \ref{inf_rem} below.

\begin{figure}[h]
\centering
\includegraphics[width=0.7\textwidth]{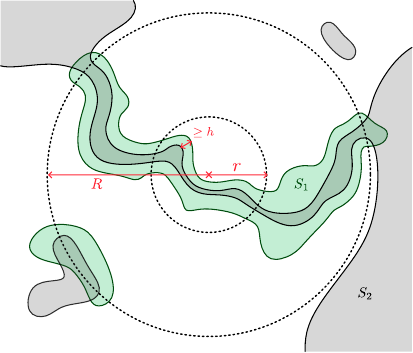}
\caption{	An illustration of the fact that far-away influences on the evolution of sets may be compensated for by increasing the forcing. 
			Within the large ball $B_R$, the set $S_1$ (depicted in green) is distance $h$ ahead of the set $S_2$ (depicted in grey). 
			Within the small ball $B_r$, $S_1$ evolving by $(A,F+\delF)$ will stay ahead of $S_2$ evolving by $(A,F)$ for a time depending on $R-r,h,\delF>0$.}
\label{inf_fig_ball}
\end{figure}

\begin{proposition}[Compensating far-away influence with higher forcing]\label{inf_prop}
Let $(A,F)\in \Omega$ as defined in \eqref{sett_eq_omega} with the additional assumption that $A$ is constant in space in the sense that $A(x,e)=A(0,e)$ for any $x\in\Rd$ and $e\in\Sd$. 
Let $S_1, S_2\subset \Rd$ be closed sets.
Assume that $S_1$ is stable with respect to $(A,F)$ and that for some $h,R>0$
\begin{equation*}
	\pp{S_2+\ov{B}_h	}\cap \ov{B}_R	\subset S_1.
\end{equation*}
Then when increasing the forcing for $S_1$ with a constant $\delF>0$ there holds
\begin{align*}
	\reach[(A,F)]{t}{S_2}\cap \ov{B}_r	&\subset	\reach[(A,F+\delF)]{t}{S_1}
	&&\text{for all }0\leq t\leq T
\end{align*}
for all $T\geq 0$ with
\begin{equation}\label{inf_eq_T}
	T\leq c\min\Bp{\pp{R-1-h-r},\, \pp{R-1-h-r}^{\frac{2}{3}}h^\frac{4}{3},\, \pp{R-1-h-r}^{\frac{2}{3}}\delF^{\frac{4}{3}}},
\end{equation}
where $c=c(\data)>0$.
\end{proposition}

In order to prove this proposition, we first need an auxiliary result, which provides a decreasing (sub-)solution with initial $0$-sublevel set $S_1$ for coefficients $(A,F+\delF)$.

\begin{lemma}\label{inf_lem_stableExt}
Let $(A,F)\in \Omega$ with the additional assumption that $A$ is constant in space in the sense that $A(x,e)=A(0,e)$ for any $x\in\Rd$ and $e\in\Sd$. 
Let $S\subset\Rd$ be stable with respect to $(A,F)$.
Then for any $\delF\in\R_+$, the set $S+\ov{B}_\delta$ is stable with respect to $(A,F+\delF)$ for all $0\leq\delta\leq C_{1F}^{-1}\delF$.
In particular, there exists a solution $\uls{S}$ to \eqref{intro_eq_uls} for coefficients $(A,F+\delF)$ with initial data satisfying $S=\Bp{x\,:\,\uls{S}(x,0)\leq 0}$, which is decreasing.
\end{lemma}

\begin{proof}[Proof of Lemma \ref{inf_lem_stableExt}]
In order to prove that $S+\ov{B}_\delta$ is stable for some $\delta>0$, we need to show that $u_{\delta}=1-\charfun[S+\ov{B}_\delta]$ is a supersolution to $G_{(A,F+\delF)}(x,\nabla u, D^2u)=0$.
Thus, let $(\xi,X)\in \subJ u_{\delta}(x_0)$ for $x_0\in\Rd$.
If $x_0\notin\pp{S+\ov{B}_\delta}$, then $\xi=0$ with $X\leq 0$ and hence $\pp{G_{(A,F+\delF)}}^*(x_0,\xi, X)\geq 0$.
If $x_0\in \pp{S+\ov{B}_\delta}$, then there is $y_0\in S$ with $|x_0-y_0|\leq\delta$. 
Lemma \ref{comp_lem_setsum} yields $(\xi,X)\in \subJ u_0(x_0)$ with $u_0=1-\charfun[S]$.
Since $S$ is stable, $u_0$ is a supersolution and hence we obtain $\pp{G_{(A,F)}}^*(y_0,\xi, X)\geq 0$.
For $\xi=0$, by the definition of the upper semi-continuous envelope and since the second order term is independent of $x_0$ and $y_0$, we have 
\begin{equation*}
	\pp{G_{(A,F+\delF)}}^*(x_0,\xi, X) = \pp{G_{(A,F+\delF)}}^*(x_0,0, X) =\pp{G_{(A,F)}}^*(y_0,0, X)\geq 0.
\end{equation*}
For $\xi\neq 0$, due to the Lipschitz-continuity of $F$ from Assumption \eqref{sett_eq_assF} we obtain
\begin{align*}
	\pp{G_{(A,F+\delF)}}^*(x_0,\xi, X)	
		&=	-\tr\pp{A(\xi)X}+\pp{F(x_0,\xi)+\delF}|\xi|\\
		&\geq	-\tr\pp{A(\xi)X}+\pp{F(y_0,\xi)-C_{1F}\abs{x_0-y_0}+\delF}|\xi|\\
		&=		\pp{G_{(A,F)}}^*(y_0,\xi, X)+\pp{\delF-C_{1F}\delta}|\xi|\\
		&\geq	0.
\end{align*}
Thus, we have shown that $u_{\delta}$ is a supersolution to $G_{(A,F+\delF)}(x,\nabla u, D^2u)=0$ and hence $S+\ov{B}_\delta$ is stable with respect to $(A,F+\delF)$ if $\delta\leq C_{1F}^{-1}\delF$.

To obtain the decreasing solution, we choose
\begin{equation*}
	u_0(x)=	\left\{\begin{aligned}
							&0		&&\text{if }x\in S,	\\
							&1		&&\text{if }x\notin S+B_{C_{1F}^{-1}\delF},\\
							&\frac{C_{1F}}{\delF}\delta		
									&&\text{if }x\in\partial \pp{S+\ov{B}_\delta}\text{ for }\delta\in(0,C_{1F}^{-1}\delF).
					\end{aligned}\right.
\end{equation*}
Clearly, $u_0$ is uniformly continuous with $S=\Bp{x\,:\,u_0(x)\leq 0}$.
Hence, due to Theorem \ref{comp_thm_exis} there exists a unique solution $\uls{S}$ to the level-set equation \eqref{intro_eq_uls} for coefficients $(A,F+\delF)$ with $\uls{S}(\cdot,0)=u_0$.
Since the initial data has been chosen such that all of the initial sublevel sets are stable with respect to $(A,F+\delF)$, all sub-level sets stay stable and do not shrink over time due to Lemma \ref{stable_lem_preserv}.
This implies that $\uls{S}$ is decreasing.
\end{proof}

Equipped with Lemma \ref{inf_lem_stableExt}, we can now prove Proposition \ref{inf_prop}.

\begin{proof}[Proof of Proposition \ref{inf_prop}]
Let $\delF\in\R_+$ and $0<r < R$.
We will obtain the result by comparing 
\begin{itemize}
\item	$0\leq u_1\leq 1$ with $\reach[(A,F+\delF)]{t}{S_1}=\Bp{x\,:\,u_1(x,t)\leq 0}$, the decreasing solution to the level-set equation \eqref{intro_eq_uls} with coefficients $(A,F+\delF)$ from Lemma \ref{inf_lem_stableExt}, and
\item	$\ur{2}$ given by $\ur{2}(x,t)=1-\charfun[{\reach[(A,F)]{t}{S_2}}](x)$, which is a supersolution to \eqref{intro_eq_uls} with coefficients $(A,F)$ by Lemma \ref{not_lem_ur}
\end{itemize}   
after applying suitable modifications:
Instead of $u_1$, we will actually use
\begin{equation*}
	(x,t)\mapsto		u_1(x,t)-\zeta t - \vphi_\eps(x)
\end{equation*}
for small $\zeta, \eps>0$, morally corresponding to slightly speeding up the evolution via $\zeta$ and gently extending $S_1$ when moving away from the origin with $\vphi_\eps\geq 0$ and
\begin{align*}
	\vphi_\eps &= 0	\quad\text{on } \ov{B}_r,
	&	\vphi_\eps(x) 	&\geq	\eps\pp{\abs{x}-r-1}	\quad\text{for all }x\in\Rd,
	&	\abs{\nabla\vphi_\eps},\abs{D^2\vphi_\eps}	&\leq	\eps.
\end{align*}
Since we have no control of the gradients, we use the standard approach of doubling the variables to introduce some regularization, to then directly obtain a comparison. 
Introducing a regularization parameter $\delta>0$ we write
\begin{equation*}
	\Psi(x,y,t)	\coloneqq	u_1(x,t) - u_2(y,t) - \zeta t - \vphi_\eps(x) - \frac{\abs{x-y}^4}{4\delta}.
\end{equation*}
Since $\Psi$ is upper semicontinuous, upper bounded and decaying for $t\rightarrow\infty$, it attains its maximum in $\ov{B}_R\times\ov{B}_R\times[0,\infty)$, say at $(x_0,y_0,t_0)$.
We obtain
\begin{align*}
	-1 \leq	\Psi(0,0,0)	\leq		\Psi(x_0,y_0,t_0)	\leq		1 -	\frac{\abs{x_0-y_0}^4}{4\delta}
\end{align*}
and hence 
\begin{equation}\label{inf_eq_xyleq}
	\abs{x_0-y_0}	\leq		(8\delta)^{\frac{1}{4}}.
\end{equation}
The rest of the proof will consist out of these three steps:
\begin{enumerate}[label=\arabic*.), left=12pt]
\item	We will show that the maximum is attained on the boundary of $\ov{B}_R\times\ov{B}_R\times[0,\infty)$ if the choices of the parameters $\zeta,\eps,\delta$ are suitably restricted.
\item	We will show that $\Psi(x_0,y_0,t_0)\leq	0$ for a suitable choice of the parameters.
\item	Combining the restrictions on the parameters and using that the argument is invariant with respect to reparametrizations of $u_1$ we will obtain the result.
\end{enumerate}

\underline{Step 1: No maximum in the interior.}
Assume that $(x_0,y_0,t_0)$ is in the interior of $\ov{B}_R\times\ov{B}_R\times[0,\infty)$.
We will show that this leads to a contradiction under suitable restrictions on the parameters.
With Theorem \ref{comp_thm_MaxPr}, known as Ishii's Lemma or a parabolic maximum principle for semicontinuous functions, fixing $\gamma>0$ there exist $b_1,b_2\in\R$ and $X,Y\in\Rddsym$ such that
\begin{align}
	&\left\{\begin{aligned}	\pp{b_1,\xi_0+\nabla\vphi_\eps(x_0),X}	&\in\supPc u_1(x_0,t_0)\\
							\pp{b_2,\xi_0,Y}							&\in\subPc \ur{2}(x_0,t_0)
	\end{aligned}\right.,\label{inf_eq_P}\\
	&\quad	b_1-b_2	=	\zeta,\label{inf_eq_b}\\
	&-\pp{\frac{1}{\gamma}+\abs{U}}\Id_{2d\times 2d}	
	\leq		\begin{pmatrix}
				X-D^2\vphi_\eps(x_0)	&	0\\	0	&	-Y
			\end{pmatrix}
	\leq		U+\gamma U^2\label{inf_eq_XYpre}
\end{align}
with the spatial derivatives of the smooth `test function' given by
\begin{equation*}
	\xi_0\coloneqq	\frac{\abs{x_0-y_0}^2}{\delta	}\pp{x_0-y_0}
\end{equation*}
and
\begin{align*}
	U	&\coloneqq	\frac{1}{\delta}\begin{pmatrix}	N	&	-N	\\	-N	&	N	\end{pmatrix}
	&\text{with }	N\coloneqq	\abs{x_0-y_0}^2\Idd + 2\pp{x_0-y_0}\otimes\pp{x_0-y_0}.
\end{align*}
If $x_0\neq y_0$ we choose $\gamma=\delta\abs{x_0-y_0}^{-2}$ and hence obtain from \eqref{inf_eq_XYpre} that
\begin{equation}\label{inf_eq_XY}
	-\frac{C}{\delta}	\abs{x_0-y_0}^2\Id
	\leq		\begin{pmatrix}
				X-D^2\vphi_\eps(x_0)	&	0\\	0	&	-Y
			\end{pmatrix}
	\leq		\frac{C}{\delta}	\abs{x_0-y_0}^2\begin{pmatrix}	\Id	&	-\Id	\\	-\Id	&	\Id\end{pmatrix}.
\end{equation}
In fact, we next show that always $x_0\neq y_0$ and hence this estimate holds.
More precisely, we claim that there is $c=c(\data)>0$ such that 
\begin{align}\label{inf_eq_xygeq}
	\abs{x_0-y_0}	&\geq	c\pp{\delta\zeta}^\frac{1}{2}
	\qquad\qquad \text{if }\,\,	\zeta\delta		\leq c.	
\end{align} 
On the one hand, note that since $\ur{2}$ is a supersolution to \eqref{intro_eq_uls}, due to the inclusion in the subjet from \eqref{inf_eq_P} we have
\begin{equation*}
	\pp{G_{(A,F)}}_*(y_0,\xi_0,Y)	\geq -b_2	=	\zeta - b_1	\geq		\zeta
\end{equation*}
where in the last step we used \eqref{inf_eq_b} that $b_1\leq 0$ since $u_1$ is decreasing due to the choice of the initial data from Lemma \ref{inf_lem_stableExt}.
On the other hand, choosing $\gamma$ as above for $x_0\neq y_0$ and arbitrarily for $x_0=y_0$, from \eqref{inf_eq_XYpre} we obtain 
\begin{equation*}
	-Y\leq C\delta^{-1}\abs{x_0-y_0}^2\Id
\end{equation*}
and hence 
\begin{equation*}
	\pp{G_{(A,F)}}_*(y_0,\xi_0,Y) \leq C C_{1A}\delta^{-1}\abs{x_0-y_0}^2+C_{1F}\abs{\xi_0}.
\end{equation*}
Combining both equations and plugging in the definition of $\xi_0$, we obtain 
\begin{equation*}
	\zeta\delta	\leq		C\abs{x_0-y_0}^2 + C_{1F}\abs{x_0-y_0}^3,
\end{equation*}
which yields \eqref{inf_eq_xygeq} and hence shows that the claim holds.

Next, we set up the contradiction. 
As above, with \eqref{inf_eq_P} and using that $u_1$ is a sub- and $u_2$ a supersolution we obtain
\begin{align*}
	b_1-\tr\pp{A(\xi_0+\nabla\vphi_\eps(x_0))X}+\pp{F\pp{\xi_0+\nabla\vphi_\eps(x_0),x_0}+\delF}|\xi_0+\nabla\vphi_\eps(x_0)|
	&\leq 0\\
	b_2-\tr\pp{A(\xi_0)Y}+F\pp{\xi_0,y_0}|\xi_0|
	&\geq 0
\end{align*}
knowing that $\xi_0\neq 0$ and for now assuming that $\xi_0+\nabla\vphi_\eps(x_0)\neq 0$.
Combining both and using \eqref{inf_eq_b} yields
\begin{multline}\label{inf_eq_sol}
	\zeta\leq	\tr\pp{A(\xi_0+\nabla\vphi_\eps(x_0))X}-\tr\pp{A(\xi_0)Y}\\
				+	F\pp{\xi_0,y_0}|\xi_0|-\pp{F\pp{\xi_0+\nabla\vphi_\eps(x_0),x_0}+\delF}|\xi_0+\nabla\vphi_\eps(x_0)|.
\end{multline}
We will derive suitable restrictions on the parameters for which the right hand side is smaller than $\zeta$, which will yield the contradiction.
First, to ensure that $\xi_0+\nabla\vphi_\eps(x_0)\neq 0$, besides $\delta\zeta\leq c$ as in \eqref{inf_eq_xygeq} we require that 
\begin{align*}
	\eps	&<	c^3\delta^{\frac{1}{2}}\zeta^{\frac{3}{2}}	
	&\text{ such that }
	\abs{\nabla\vphi_\eps(x_0)}	&\leq \epsilon	< \frac{\pp{c(\delta\zeta)^{\frac{1}{2}}}^3}{\delta}	\leq \abs{\xi_0}
\end{align*}
with the constant $c=c(\data)>0$ originating in \eqref{inf_eq_xygeq}.
Further, note that 
\begin{equation}\label{inf_eq_xiErr}
	\abs{\frac{\xi_0}{|\xi_0|}-\frac{\xi_0+\nabla\vphi_\eps(x_0)}{|\xi_0+\nabla\vphi_\eps(x_0)|}}
	\leq		2\frac{|\nabla\vphi_\eps(x_0)|}{|\xi_0|}
	\leq		\frac{2\eps\delta}{\abs{x_0-y_0}^3}.
\end{equation}
Regarding the first order terms, we claim that there is some $C=C(\data)>0$ such that 
\begin{multline}\label{inf_eq_Fcomp}
	F\pp{\xi_0,y_0}|\xi_0|-\pp{F\pp{\xi_0+\nabla\vphi_\eps(x_0),x_0}+\delF}|\xi_0+\nabla\vphi_\eps(x_0)|
	\leq C\eps\\
	\text{if }\,\,\delF\geq C\delta^{\frac{1}{4}}.
\end{multline}
In order to compare the two different locations, using the Lipschitz continuity of $F$ on $\Sd$ from \eqref{sett_eq_assF} and \eqref{inf_eq_xiErr} from above, we first bound the error for neglecting $|\nabla\vphi_\eps(x_0)|\leq \eps$,
\begin{align*}
	&\abs{\pp{F\pp{\frac{\xi_0+\nabla\vphi_\eps(x_0)}{\abs{\xi_0+\nabla\vphi_\eps(x_0)}},x_0}+\delF}|\xi_0+\nabla\vphi_\eps(x_0)|
	-\pp{F\pp{\frac{\xi_0}{\abs{\xi_0}},x_0}+\delF}|\xi_0|}\\
	&\qquad	\leq		\pp{C_{1F}+\delF}\eps + C_{1F}\frac{2\eps}{|\xi_0|}|\xi_0|.
\end{align*}
To bound the error from switching location, we use the Lipschitz continuity of $F$ in $\Rd$ from \eqref{sett_eq_assF} and obtain
\begin{equation*}
	F\pp{\xi_0,y_0}|\xi_0|-\pp{F\pp{\frac{\xi_0}{\abs{\xi_0}},x_0}+\delF}|\xi_0|
	\leq		\pp{C_{1F}|x_0-y_0|-\delF}|\xi_0|
	\leq		0
\end{equation*}
where the last inequality is guaranteed to hold via \eqref{inf_eq_xyleq} if $\delF\geq C_{1F}\pp{8\delta}^{\frac{1}{4}}$.

Regarding the second order terms, we claim that there is $C=C(\data)>0$ such that 
\begin{equation}\label{inf_eq_Acomp}
	\tr\pp{A(\xi_0+\nabla\vphi_\eps(x_0))X}-\tr\pp{A(\xi_0)Y}	\leq  C\eps+C\delta^{-1}\zeta^{-2}\eps^2.
\end{equation}
In order to use \eqref{inf_eq_XY}, we first get an error for inserting $D^2\vphi_\eps(x_0)$,
\begin{equation*}
	\abs{\tr\pp{A(\xi_0+\nabla\vphi_\eps(x_0))X}-\tr\pp{A(\xi_0+\nabla\vphi_\eps(x_0))\pp{X-D^2\vphi_\eps(x_0)}}}\leq C_{1A}\eps.
\end{equation*}
Since with Assumption \eqref{sett_eq_assAbd} we have $A=\sigma\sigma^\intercal$, we can write
\begin{align*}
	&\tr\pp{A(\xi_0+\nabla\vphi_\eps(x_0))\pp{X-D^2\vphi_\eps(x_0)}}-\tr\pp{A(\xi_0)Y}\\
	&\qquad	=	\tr\pp{\sigma^\intercal(\xi_0+\nabla\vphi_\eps(x_0))\pp{X-D^2\vphi_\eps(x_0)}\sigma(\xi_0+\nabla\vphi_\eps(x_0))	
				-\sigma^\intercal(\xi_0)Y\sigma(\xi_0)}\\
	&\qquad	\leq		\frac{C}{\delta}\abs{x_0-y_0}^2\abs{\sigma\pp{\frac{\xi_0+\nabla\vphi_\eps(x_0)}{\abs{\xi_0+\nabla\vphi_\eps(x_0)}}}-\sigma\pp{\frac{\xi_0}{\abs{\xi_0}}}}^2\\
	&\qquad	\leq		\frac{C}{\delta}\abs{x_0-y_0}^2	C_{1A}^2\frac{\eps^2}{\abs{\xi_0}^2}\\
	&\qquad	=		C\delta\abs{x_0-y_0}^{-4}\eps^2\\
	&\qquad	\leq		C\delta^{-1}\zeta^{-2}\eps^2
\end{align*}
where the first inequality is an application of \eqref{inf_eq_XY}, the second follows from the Lipschitz continuity of $\sigma$ and \eqref{inf_eq_xiErr}, and the last inequality follows from \eqref{inf_eq_xygeq}.
Combining the two estimates above yields \eqref{inf_eq_Acomp}.

To obtain the contradiction, we now plug the inequalities \eqref{inf_eq_Fcomp} and \eqref{inf_eq_Acomp} into \eqref{inf_eq_sol} and collect the restrictions on the parameters to obtain
\begin{align}\label{inf_eq_reqInt}
	\zeta	&\leq		C\eps+C\delta^{-1}\zeta^{-2}\eps^2	<	\zeta
	&\text{if }\,\,\left\{\begin{aligned}		\zeta\delta		&\leq	c_1(\data),\\
											\eps 			&<		c_2(\data)\delta^{\frac{1}{2}}\zeta^{\frac{3}{2}},\\
											\delF			&\geq	C(\data)\delta^{\frac{1}{4}},
					\end{aligned}\right.						
\end{align}
where the contradiction holds if we choose $c_2=c_2(\data,c_1)>0$ small enough.

\underline{Step 2: Bounding the maximum.}
Because the maximum is not attained in the interior, $(x_0,y_0,t_0)$ has to be on the boundary of $\ov{B}_R\times\ov{B}_R\times[0,\infty)$.
Thus, there holds $x_0\in\partial B_R$, $y_0\in\partial B_R$ or $t_0=0$. 
If $x_0\in\partial B_R$, then with $|x_0|=R$ and the definition of $\vphi_\eps$ note that
\begin{equation*}
	\Psi(x_0,y_0,t_0)	\leq 1-0-0-{\eps}(R-r-1)-0.
\end{equation*}
Similarly, if $y_0\in\partial B_R$ then due to \eqref{inf_eq_xyleq} also $|x_0|\geq R-\pp{8\delta}^{\frac{1}{4}}$ and hence
\begin{equation*}
	\Psi(x_0,y_0,t_0)	\leq 1-0-0-{\eps}\pp{R-\pp{8\delta}^{\frac{1}{4}}-r-1}-0.
\end{equation*}
If $t_0=0$ and $x_0\in S_1$, then $u_1(x_0,t_0)\leq 0$ and hence $\Psi(x_0,y_0,t_0)\leq 0$.
If however $x_0\notin S_1$, then because $S_2+\ov{B}_h\cap \ov{B}_R\subset S_1$ we have $y_0\notin S_2 $ if $\abs{x_0-y_0}\leq h$ and thus due to $\eqref{inf_eq_xyleq}$ we have
\begin{equation*}
	\Psi(x_0,y_0,t_0)	\leq 1-1-0-0-0 = 0	\qquad\qquad	\text{if }\,\,\pp{8\delta}^{\frac{1}{4}}\leq h.
\end{equation*}
In summary, we obtain 
\begin{align}
	\Psi(x_0,y_0,t_0)&\leq 0 
	&\text{if next to \eqref{inf_eq_reqInt} also }
					\left\{\begin{aligned}	\eps 	&\geq	\frac{1}{R-1-h-r}	,\\
											\delta	&\leq	8h^4.
					\end{aligned}\right.		\label{inf_eq_reqBou}	
\end{align}

\underline{Step 3: Choosing the parameters and comparing $u_1$ and $\ur{2}$.}
For $x\in \ov{B}_r$ and $t\geq 0$, having $\max_{(x,y,t)\in\ov{B}_R\times\ov{B}_R\times[0,\infty)}\Psi(x,y,t)\leq 0 $ implies 
\begin{equation*}
	u_1(x,t)-\ur{2}(x,t)-\zeta t = \Psi(x,x,t)	\leq	 0.
\end{equation*}
As an intermediate step, we want to obtain 
\begin{align}\label{inf_eq_S2S1h}
	\reach[(A,F)]{t}{S_2}\cap \ov{B}_r 
	&\subset \Bp{x\,:\,u_1(x,t)\leq \frac{1}{2}}	
	&\text{for all }0\leq t\leq T.
\end{align}
Thus, we choose $\zeta=\frac{1}{2T}$.
If we can choose the rest of the parameters such that the requirements from \eqref{inf_eq_reqInt} and \eqref{inf_eq_reqBou} are satisfied, then we obtain  \eqref{inf_eq_S2S1h}.
Thus, we choose $\eps$ and $\delta$ as small as possible, that is
\begin{align*}
	\eps &= \frac{1}{R-1-h-r},
	&	\delta	&=	C(\data)\frac{\eps^2}{\zeta^3}	=	C(\data)\frac{T^3}{\pp{R-1-h-r}^2},
\end{align*}
where the constant for $\delta$ is chosen such that the corresponding inequality in \eqref{inf_eq_reqInt} is satisfied.
For \eqref{inf_eq_reqBou} to be satisfied, we hence require 
\begin{equation*}
	\pp{R-1-h-r} \geq	C(\data)\frac{T^{\frac{3}{2}}}{h^2}.
\end{equation*}
For \eqref{inf_eq_reqInt} to be satisfied, when plugging in the choice of $\delta$ we see that we need
\begin{align*}
	\pp{R-1-h-r}	&\geq	C(\data)T,
	&	\delF	&\geq	C(\data)\frac{T^{\frac{3}{4}}}{\pp{R-1-h-r}^\frac{1}{2}}.
\end{align*}
Reformulating these requirements in terms of $T$ yields \eqref{inf_eq_T}.
To conclude the proof, note that all of the bounds above are invariant with respect to reparametrizing  the initial  data of $u_1$.
Repeating the same arguments with $\wt{u}_1\coloneqq \max\Bp{su_1,1}$ for $s\geq 1$, we also arrive at \eqref{inf_eq_S2S1h}, now corresponding to 
\begin{align*}
	\reach[(A,F)]{t}{S_2}\cap \ov{B}_r 
	&\subset		\Bp{x\,:\,\wt{u}_1(x,t)\leq \frac{1}{2}}	
	=			\Bp{x\,:\,u_1(x,t)\leq \frac{1}{2s}}
	&\text{for }0\leq t\leq T.
\end{align*}
Since $s\geq 1$ is arbitrary, we get $\reach[(A,F)]{t}{S_2}\cap \ov{B}_r \subset	\Bp{x\,:\,u_1(x,t)\leq 0}=\reach[(A,F+\delF)]{t}{S_1}$.
\end{proof} 

\begin{remark}\label{inf_rem}
The main reason why we increase the forcing for the comparison is not to obtain Lemma \ref{inf_lem_stableExt} -- this could also be solved using a slightly different concept of stability -- but to control the difference of the first-order terms in \eqref{inf_eq_Fcomp} of Step 1 in the proof of Proposition \ref{inf_prop}. 
The reason why we have to rely on this makeshift approach is that we lack a guaranteed minimum speed of propagation for the interfaces.
This minimum speed would translate to Lipschitz continuity of the level-set solutions, in terms of the notation from the proof allowing us to bound $\xi_0$ and thus obtain better control of $\abs{x_0-y_0}$ and smallness of $\delta^{-1}\abs{x_0-y_0}^4$, which would make increasing the forcing obsolete.

Similarly, we are not able to treat inhomogeneous curvature terms because due to the lack of control on the gradients also the bounds for $X$ and $Y$ obtained via the regularization explode too fast.
\end{remark}

In the following, instead of Proposition \ref{inf_prop} we will use this variation for the half-ball instead of the full ball, see Figure \ref{inf_fig_halfball}.

\begin{figure}[h]
\centering
\includegraphics[width=0.7\textwidth]{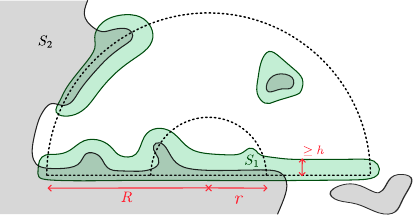}
\caption{An illustration of the fact that far-away influences on the evolution of sets within the half-ball may be compensated for by increasing the forcing. 
	In the upper half-ball $\ov{B}_R\cap\hsp{e}$, the  set $S_1$ (depicted in green) is distance $h$ ahead of the set $S_2$ (depicted in grey) as well as $h$ ahead of the boundary of the half-space. 
	Within the small half-ball $\ov{B}_r\cap\hsp{e}$, $S_1$ evolving by $(A,F+\delF)$ will stay ahead of $S_2$ evolving by $(A,F)$ for a time depending on $R-r,h,\delF>0$.}
\label{inf_fig_halfball}
\end{figure}

\begin{proposition}[Compensating far-away influence with higher forcing in the half-ball]\label{inf_prop_half}
Let $(A,F)\in \Omega$ with the additional assumption that $A(x,e)=A(0,e)$ for $x\in\Rd$, $e\in\Sd$ is constant in space. 
Let $S_1, S_2\subset \Rd$ be closed sets.
Assume that $S_1$ is stable with respect to $(A,F)$ and that for some $h,R>0$ in the upper half-space $\hsp{e}=\Bp{x\in\Rd\,:\,x\cdot e\geq 0}$ in direction $e\in\Sd$ there holds
\begin{align*}
	\pp{S_2+\ov{B}_h	}\cap \ov{B}_R\cap\hsp{e}	&\subset S_1
	&&\text{and}
	&	\pp{\partial\hsp{e}+\ov{B}_h}\cap \ov{B}_R\cap\hsp{e}	&\subset S_1.
\end{align*}
Then when increasing the forcing for $S_1$ with a constant $\delF\in\R_+$ there holds
\begin{align*}
	\reach[(A,F)]{t}{S_2}\cap \ov{B}_r\cap\hsp{e}	&\subset	\reach[(A,F+\delF)]{t}{S_1}
	&&\text{for all }0\leq t\leq T
\end{align*}
for all $T$ with
\begin{equation}\label{inf_eq_Thalf}
	T\leq c\min\Bp{\pp{R-1-h-r},\, \pp{R-1-h-r}^{\frac{2}{3}}h^\frac{4}{3},\, \pp{R-1-h-r}^{\frac{2}{3}}\delF^{\frac{4}{3}}},\tag{\ref{inf_eq_T}}
\end{equation}
where $c=c(\data)>0$.
\end{proposition}

\begin{proof}
The proof is analogous to the proof of Proposition \ref{inf_prop}.
The only difference is that in Step 2 we might also have $x_0\in\partial\hsp{e}\cap \ov{B}_R$ or $y_0\in\partial\hsp{e}$.
Either way, since with the same choice of parameters we have $\abs{x_0-y_0}\leq h$ and hence $x_0\in S_1$ by assumption. 
Because $S_1$ is stable, this implies $x_0\in\reach[(A,F+\delF)]{t}{S_1}$ for all $t\geq 0$, thus $u_1(x_0,t_0)= 0$ and hence also $\Psi(x_0,y_0,t_0)\leq 0$ as in all of the other cases.
\end{proof}

\section{The homogenized speed via arrival times for half-space approximations}\label{s_lin}
For a direction $e\in\Sd$ we denote the corresponding upper and lower half-space with 
\begin{align*}
	\hsp{e}	&=	\Bp{x\in\Rd\,:\,x\cdot e\geq 0},
	&	\hsm{e}	&=	\Bp{x\in\Rd\,:\,x\cdot e\leq 0}.
\end{align*}
Translating the results from \cite{ArCa18} to our setting, we would like to choose the homogenized speed as the average speed with which the half space reaches an `infinitely far away' target area, that is
\begin{equation}\label{lin_eq_intro_limm}
	\vhom(e)	\coloneqq	\pp{\lim_{r\rightarrow\infty}	\frac{1}{r}\EV{\metb{h(r)}{re}{\hsm{e}}}}^{-1}
\end{equation}
for suitable chosen $r\mapsto h(r)<<r$.
However, we can not guarantee that this limit exists.
Since the fluctuation bounds from Proposition \ref{fluct_prop} and some of the following arguments require a set with bounded boundary,
we will approximate the boundary of the half-space with a family of bounded disks
\begin{equation*}
  	\hsma{e}{\beta}{h(r)}{r}\approx \hsm{e}
  	\qquad\text{on the scale $r\geq 0$ for a fixed parameter $\beta$},
\end{equation*}  
see Figure \ref{lin_fig_hs-approx}.
In Corollary \ref{lin_cor_bd}, we will define these sets in a way such that Proposition \ref{inf_prop_half} guarantees good approximative properties with respect to the half-space.

The main task of this section is to show that the expected values of the corresponding arrival times,
\begin{equation*}
	r\mapsto \EV{\metb{h(r)}{re}{\hsma{e}{\beta}{h(r)}{r}}},
\end{equation*}
behave almost linearly, see Proposition \ref{lin_prop_sub} for the approximate sub- and Proposition \ref{lin_prop_sup} for the approximate superlinearity.
Already the approximate sub-linearity then implies the existence of a limit for the averaged arrival times as in \eqref{lin_eq_intro_limm}.
The approximate linearity further provides quantitative bounds for the convergence, see Corollary \ref{lin_cor_exUpB} and Corollary \ref{lin_cor_LowB}. 
We obtain the approximate sub-linearity adapting arguments from \cite[Section 4.2]{ArCa18} to our setting.
However, we only obtain the approximate super-linearity at the cost of slightly changing the forcing due to applications of Proposition \ref{inf_prop_half}.
As a result, the quality of the quantitative bounds depends on the continuity of arrival times with respect to changing the forcing by a constant, which is why we will later need Assumption \ref{sett_aP_gen}.
At the end of this section, we will define the homogenized speed and discuss the implications of Assumption \ref{sett_aP_gen}.

\begin{figure}[h]
\centering
\includegraphics[width=0.95\textwidth]{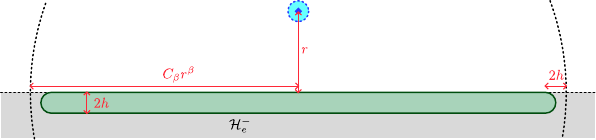}\\
\includegraphics[width=0.95\textwidth]{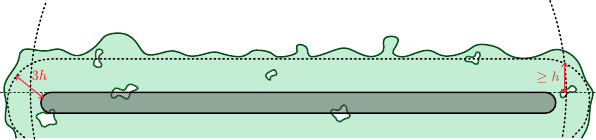}
\caption{	An illustration of the half-space
approximations and the fact that they are quickly $h$ ahead of the half-space in the corresponding half-ball.
			Top: The half-space approximation $\hsma{e}{\beta}{h}{r}$ is depicted in green, the target area $\ov{B}_h(re)$ in blue. 
			Bottom: The evolution $\reach{\pp{\cstable+3\vmineff^{-1}} h}{S_h}$ (depicted in green) of a stable $(h,\cstable)$-approximation $S_h$ of $\hsma{e}{\beta}{h}{r}$ after time $\pp{\cstable+3\vmineff^{-1}} h$ fully $h$-envelopes the set $\pp{\partial\hsp{e}+\ov{B}_h}\cap \ov{B}_R\cap\hsp{e}$.}
\label{lin_fig_hs-approx}
\end{figure}

\subsection{The half-space approximations}
We start by introducing the disks $\hsma{e}{\beta}{h(r)}{r}$ approximating the respective half-space $\hsm{e}$ on the scale $r\geq h_0$, see Figure \ref{lin_fig_hs-approx}.
To obtain the approximative properties, we compensate the far-away influence by slightly increasing the forcing as in Proposition \ref{inf_prop_half}:
We only need to compare the evolution of the approximation to the original half-space (or other large sets) up to the time $\timebound{h}{re}{\hsm{e}}\approx r$ at which we truncate the arrival times.
Plugging this into \eqref{inf_eq_Thalf} from Proposition \ref{inf_prop_half} and choosing a scaling for the width of the disk yields the scaling of the forcing increase.

\begin{corollary}[Justifying half-space approximations with Proposition \ref{inf_prop_half}]\label{lin_cor_bd}
Let $e\in\Sd$ and $\beta\geq \frac{3}{2}$.
For all $r,h\geq h_0$ we approximate $\hsm{e}$ with the fattened disks
\begin{equation*}
	\hsma{e}{\beta}{h}{r}	\coloneqq	\Bp{x\in\ov{B}_{C_{\beta}r^\beta-2h}(-he)	\text{ with } \pp{x+he}\cdot e=0}	+	\ov{B}_h\subset\ov{B}_{C_{\beta}r^\beta}	,
\end{equation*}
where $C_\beta=C_\beta(\data)>0$ is chosen large enough such that \eqref{inf_eq_Thalf} from Proposition \ref{inf_prop} and Proposition \ref{inf_prop_half} is satisfied for all $r\geq h\geq h_0$ with $(T(r,h), 2r, h, R(r), \delF(r))$ given by 
\begin{align*}
			T(r,h)		&=	\timebound{h}{re}{\hsm{e}}\leq C(\data)r	\quad\text{as in Definition \ref{intro_def_met}}
\end{align*}			
and 
\begin{align*}
		R(r)			&=	C_{\beta}r^\beta,
	&	\delF(r)		&=	r^{-\frac{\beta}{2}+\frac{3}{4}}.
\end{align*}
Let $(A,F)\in\Espeed{h}\pp{\ov{B}_{C_{\beta}r^\beta}}$ from Assumption \ref{veff_aP_star} with the additional assumption that $A(x,e)=A(0,e)$ for $x\in\Rd$, $e\in\Sd$ is constant in space. 
Then
\begin{multline}\label{lin_eq_hsma}
	\met[(A,F)]{h}{re}{\hsma{e}{\beta}{h}{r}}	
	\geq		\met[(A,F)]{h}{re}{\hsm{e}}\\
	\geq		\met[\pp{A,F+r^{-\frac{\beta}{2}+\frac{3}{4}}}]{h}{re}{\hsma{e}{\beta}{h}{r}}	-	\pp{\cstable+3\vmineff^{-1}} h.
\end{multline}
The same inequalities hold when replacing $\hsm{e}$ with any other closed set $S\subset\Rd$ which satisfies $\hsma{e}{\beta}{h}{r}\subset S \subset \hsm{e}\cup\pp{\Rd\setminus\ov{B}_{C_{\beta}r^\beta}}$.
\end{corollary}

\begin{proof}
We only show the proof for $S=\hsm{e}$, it is the same for any other closed set $S$ with $\hsma{e}{\beta}{h}{r}\subset S \subset \hsm{e}\cup\pp{\Rd\setminus\ov{B}_{C_{\beta}r^\beta}}$.

Via the comparison principle from Lemma \ref{not_lem_comp} we know that for all $t\geq 0$ there holds $\reach[(A,F)]{t}{\hsma{e}{\beta}{h}{r}}\subset \reach[(A,F)]{t}{\hsm{e}}$.
The first inequality in \eqref{lin_eq_hsma} hence follows via Lemma \ref{veff_lem_henv}. 
For the second inequality, due to Lemma \ref{veff_lem_henv} it is sufficient if we can show that 
\begin{equation}\label{lin_eq_reComp}
	\reach[(A,F)]{t}{\hsm{e}}\cap\ov{B}_{2r}\cap\hsp{e}	\hsubs		\reach[\pp{A,F+r^{-\frac{\beta}{2}+\frac{3}{4}}}]{t+\pp{\cstable+3\vmineff^{-1}} h}{\hsma{e}{\beta}{h}{r}}
	\quad\,\,\text{for all } t\leq \timebound{h}{re}{\hsm{e}}.
\end{equation}
Since $\hsma{e}{\beta}{h}{r}$ is $h$-fat and $(A,F)\in\Espeed{h}\pp{\ov{B}_{C_{\beta}r^\beta}}$, Assumption \ref{veff_aP_star} guarantees that there is a stable $(h,\cstable)$-approximation $S_h$ of $\hsma{e}{\beta}{h}{r}$ with respect to $(A,F)$, for which the minimum effective speed applies.
On the one hand, since $S_h\subset \hsma{e}{\beta}{h}{r}$, by the comparison principle from Lemma \ref{not_lem_comp} there holds 
\begin{equation*}
	\reach[\pp{A,F+r^{-\frac{\beta}{2}+\frac{3}{4}}}]{t}{S_h}	\subset	\reach[\pp{A,F+r^{-\frac{\beta}{2}+\frac{3}{4}}}]{t}{\hsma{e}{\beta}{h}{r}}
	\qquad\text{for all }t\geq 0.
\end{equation*}
On the other hand, due to the definition of $\hsma{e}{\beta}{h}{r}$, with $\hsma{e}{\beta}{h}{r} \hsubs \reach[\pp{A,F}]{\cstable h}{S_h}$ and with the minimum effective speed as in Lemma \ref{stable_lem_vmin-iterative}  we have
\begin{equation*}
	\pp{\partial\hsp{e}+\ov{B}_h}\cap \ov{B}_{C_{\beta}r^\beta}\cap\hsp{e}
	\subset	\pp{\hsma{e}{\beta}{h}{r}+3\ov{B}_h}	\cap \ov{B}_{C_{\beta}r^\beta}
	\hsubs \reach[\pp{A,F+r^{-\frac{\beta}{2}+\frac{3}{4}}}]{\pp{\cstable+3\vmineff^{-1}} h}{S_h}.
\end{equation*}
Filling potential holes of width smaller than $2h$ in $\reach[\pp{A,F+r^{-\frac{\beta}{2}+\frac{3}{4}}}]{\pp{\cstable+3\vmineff^{-1}} h}{S_h}$ as in the proof of Lemma \ref{not_lem_comp}, we obtain a stable set which satisfies the conditions of Proposition \ref{inf_prop_half} as $S_1$ with $S_2=\hsm{e}$. Proposition \ref{inf_prop_half} with the parameters $(T(r,h), \wt{r}(r), h, R(r), \delF(r))$ chosen as above in the assumptions hence yields \eqref{lin_eq_reComp}, which concludes the proof.
\end{proof}

\subsection{Approximate sub-linearity, convergence of the averages and lower bounds}
We will now show that the expected values of the truncated arrival times for the half-space approximations are sub-linear up to a vanishing relative error as $r\rightarrow\infty$.

\begin{proposition}[Approximate sub-linearity of the truncated arrival times]\label{lin_prop_sub}
Assume that \ref{sett_aP_stat}, \ref{sett_aP_fin}, \ref{sett_aP_Aconst} and \ref{veff_aP_star} hold.
Let $\beta\geq\frac{3}{2}$ with $C_\beta>0$ and $\hsma{e}{\beta}{h}{r}$ for $e\in\Sd$ and $r,h\geq h_0$ as defined in Corollary \ref{lin_cor_bd}.

Let $r_2\geq r_1\geq h_0$ and $h_3\geq h_2\geq h_1\geq h_0$ with $\log(h_2)\geq c\log(r_2)$ for some $c>0$.
Then there exists $C=C(c,\beta,\data)>0$ such that
\begin{multline}\label{lin_eq_sub}
	\EV{\metb{h_3}{\pp{r_1+r_2}e}{\hsma{e}{\beta}{h_3}{r_1+r_2}}}\\
		\leq		\EV{\metb{h_1}{r_1 e}{\hsma{e}{\beta}{h_1}{r_1}}}
				+ \EV{\metb{h_2}{r_2 e}{\hsma{e}{\beta}{h_2}{r_2}}}				
				+ C\sqrt{h_2r_2}\log\pp{r_1}.
\end{multline}
\end{proposition} 

\begin{proof}
We set $r_3=r_1+r_2$.
The time until the set $\hsma{e}{\beta}{h_3}{r_3}$ effectively reaches $\ov{B}_{h_3}(r_3 e)$ is clearly less than adding the time it takes for $\hsma{e}{\beta}{h_3}{r_3}$ to fully envelop $r_2 e+\hsma{e}{\beta}{h_1}{r_1}$ to the time it then takes until the `intermediate set' $r_2 e+\hsma{e}{\beta}{h_1}{r_1}$ reaches $\ov{B}_{h_1}(r_3 e)$, see Figure \ref{lin_fig_Nplus}.

\begin{figure}[h]
\centering
\includegraphics[width=0.9\textwidth]{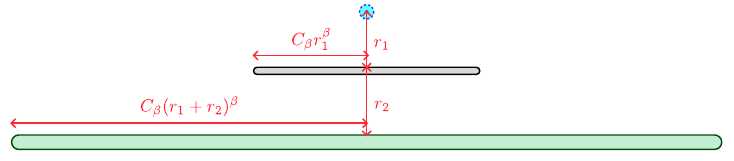}\\\vspace{10pt}
\includegraphics[width=0.9\textwidth]{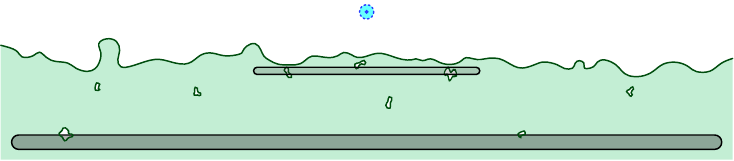}\\\vspace{10pt}
\includegraphics[width=0.9\textwidth]{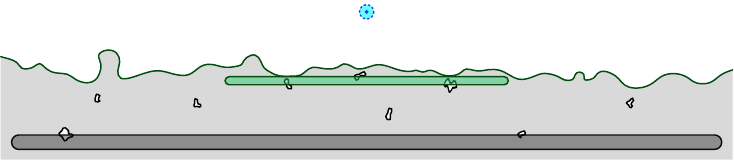}
\caption{	An illustration of the argument for the sub-linearity of the arrival times.
			Top: The large disk is depicted in green with its target area in blue, and the smaller `intermediate' disk is depicted in grey.
			Center: The evolution of the large disk at time $\Tfull$ (depicted in green) $h_2$-envelopes the smaller disk. 
			Bottom: At time $\Tfull$, we compare the evolved set with the smaller intermediate disk (depicted in green).}
\label{lin_fig_Nplus}
\end{figure}

More precisely, we denote the time by which $r_2 e+\hsma{e}{\beta}{h_1}{r_1}$ is fully $h_2$-enveloped with
\begin{equation*}
	\Tfull \coloneqq	\inf\Bp{t\geq 0\,:\,	r_2 e+\hsma{e}{\beta}{h_1}{r_1}\hsubs[h_2]\reach{t}{\hsma{e}{\beta}{h_3}{r_3}}}.
\end{equation*}
Then for $(A,F)\in\Espeed{h_2}\pp{\ov{B}_{C_{\beta}r_3^\beta}}$ there holds 
\begin{equation}\label{lin_eq_Np}
	\met[(A,F)]{h_3}{r_3e}{\hsma{e}{\beta}{h_3}{r_3}}	\leq \Tfull[(A,F)] + \cstable h_2 + \met[(A,F)]{h_1}{r_3e}{r_2 e+\hsma{e}{\beta}{h_1}{r_1}}
\end{equation}
because with $S_{h_2}$ denoting the stable $(h_2,\cstable)$-approximation of $\hsma{e}{\beta}{h_3}{r_3}$ from \ref{veff_aP_star} and with the comparison principle as in Lemma \ref{not_lem_comp}, for all $t\geq 0$ we have
\begin{equation*}
	\reach[(A,F)]{t}{r_2 e+\hsma{e}{\beta}{h_1}{r_1}}
	\hsubs[h_2]	\reach[(A,F)]{t+\Tfull[(A,F)] +\cstable h_2}{S_{h_2}}	
	\subset \reach[(A,F)]{t+\Tfull[(A,F)] +\cstable h_2}{\hsma{e}{\beta}{h_3}{r_3}}.
\end{equation*}
In order to obtain \eqref{lin_eq_sub} from \eqref{lin_eq_Np}, two steps remain:
\begin{enumerate}[label=\arabic*.), left=12pt]
\item	We will show that the expected value of $\Tfull$ on $\Espeed{h_2}\pp{\ov{B}_{C_{\beta}r_3^\beta}}$ can be bounded by the expected value of $\metb{h_2}{r_2 e}{\hsma{e}{\beta}{h_2}{r_2}}$ up to an error essentially corresponding to the fluctuation bounds from Proposition \ref{fluct_prop}.
\item	We take the expected value of \eqref{lin_eq_Np}, use the stationarity from Assumption \ref{sett_aP_stat} and manage the error on the complement of $\Espeed{h_2}\pp{\ov{B}_{C_{\beta}r_3^\beta}}$.
\end{enumerate}

\underline{Step 1: Bounding $\Tfull$.}
For $(A,F)\in\Espeed{h_2}\pp{\ov{B}_{C_{\beta}r_3^\beta}}$, due to the minimum effective speed we can essentially reduce $\Tfull$ to the maximum of a finite amount of truncated arrival times in target areas of radius $h_2$:
Instead of every point in the upper disk $r_2 e+\hsma{e}{\beta}{h_1}{r_1}$, in a first step we only care about reaching every point on a grid.
Whenever the evolution of the large disk $\hsma{e}{\beta}{h_3}{r_3}$ reaches the target area around one of these points, its stable $(h_2,\cstable)$-approximation from Assumption \ref{veff_aP_star} will reach the area just $\cstable h_2$ later.
Since $(A,F)$ admits $\vmineff$ as an effective minimum speed for the stable approximation, $\Bp{x}$ will be $h_2$-enveloped $\vmineff^{-1}h_2$ later.
Once all gridpoints are $h_2$-enveloped, the whole upper disk will be $h_2$-enveloped after $\vmineff^{-1}h_2$ and another time step proportional to the gridsize by the evolution of the stable approximation and thus due to the comparison principle also by the evolution of the large disk.
Hence, for $\lambda>0$ and $\Z^{d-1}_e$ denoting $\Z^{d-1}$ placed in the hyperplane $\Bp{x\in\Rd\,:\,x\cdot e=0}$, we obtain
\begin{align*}
	\Tfull[(A,F)]	
	&\leq	\sup\Bp{\met[(A,F)]{h_2}{x+r_2e}{\hsma{e}{\beta}{h_3}{r_3}} \,:\,x\in \frac{\vmineff}{\sqrt{d}}\lambda\Z^{d-1}_e\cap \hsma{e}{\beta}{h_1}{r_1}}\\
	&\qquad		+ \pp{\cstable+2\vmineff^{-1}} h_2 + \lambda.
\end{align*}
For each target point $x+r_2e$, we can further replace the large disk with the smaller disks $x+\hsma{e}{\beta}{h_2}{r_2}\subset \hsma{e}{\beta}{h_3}{r_3}$, see Figure \ref{lin_fig_Nplus_smallDisk}.
\begin{figure}[h]
\centering
\includegraphics[width=0.9\textwidth]{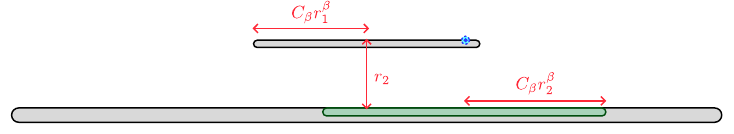}
\caption{	An illustration of the comparison argument from Step 1 of the proof for the sub-linearity of the arrival times.
			A small disk (depicted in green) is placed within the large disk $\hsma{e}{\beta}{h_3}{r_3}$. 
			Its target area (depicted in blue) covers part of the intermediate disk.}
\label{lin_fig_Nplus_smallDisk}
\end{figure}

The comparison principle as in Lemma \ref{not_lem_comp} yields
\begin{align*}	
	\Tfull[(A,F)]		
	&\leq	\sup\Bp{\met[(A,F)]{h_2}{x+r_2e}{x+\hsma{e}{\beta}{h_2}{r_2}} \,:\,x\in \frac{\vmineff}{\sqrt{d}}\lambda\Z^{d-1}_e\cap \hsma{e}{\beta}{h_1}{r_1}}\\
	&\qquad		+ \pp{\cstable+2\vmineff^{-1}} h_2 + \lambda.
\end{align*}
We set
\begin{equation*}
	T_2=\EV{\metb{h_2}{r_2 e}{\hsma{e}{\beta}{h_2}{r_2}}}.
\end{equation*}
Using this bound on $\Tfull[(A,F)]$, for $\lambda>\pp{\cstable+2\vmineff^{-1}} h_2$ and with a union bound and the stationarity of the probability distribution from Assumption \ref{sett_aP_stat}, we obtain
\begin{align*}
	&\PM{\Tfull\charfun[\Espeed{h_2}\pp{\ov{B}_{C_{\beta}r_3^\beta}}]-T_2\geq 3\lambda}\\
	&\,\,	\leq		\PM{\sup_{x\in \frac{\vmineff\lambda}{\sqrt{d}}\Z^{d-1}_e\cap \hsma{e}{\beta}{h_1}{r_1}}\metb{h_1}{x+r_2e}{x+\hsma{e}{\beta}{h_2}{r_2}} -T_2\geq \lambda}\\
	&\,\,	\leq		\sum_{x\in \frac{\vmineff\lambda}{\sqrt{d}}\Z^{d-1}_e\cap \hsma{e}{\beta}{h_1}{r_1}}
					\PM{\metb{h_2}{r_2 e}{\hsma{e}{\beta}{h_2}{r_2}}-T_2\geq \lambda}\\
	&\,\,	\leq		C\frac{r_1^{\beta(d-1)}}{\lambda^{d-1}}
					\PM{\metb{h_2}{r_2 e}{\hsma{e}{\beta}{h_2}{r_2}}-\EV{\metb{h_2}{r_2 e}{\hsma{e}{\beta}{h_2}{r_2}}}\geq  \lambda}
\end{align*}
with $C=C(\beta,\data)>0$.
Inserting the fluctuation bounds from Proposition \ref{fluct_prop}, for $\Lambda\geq \pp{\cstable+2\vmineff^{-1}} h_2$ we now obtain
\begin{align*}
	&\EV{\Tfull\charfun[\Espeed{h_2}\pp{\ov{B}_{C_{\beta}r_3^\beta}}]}-\EV{\metb{h_2}{r_2 e}{\hsma{e}{\beta}{h_2}{r_2}}}\\
	&\,\,	=		\int_0^\infty	3\PM{	\Tfull\charfun[\Espeed{h_2}\pp{\ov{B}_{C_{\beta}r_3^\beta}}]
											-\EV{\metb{h_2}{r_2 e}{\hsma{e}{\beta}{h_2}{r_2}}}\geq 3\lambda}\m\lambda\\		
	&\,\,	\leq		3\Lambda
			+	\int_\Lambda^\infty
					C\frac{r_1^{\beta(d-1)}}{\lambda^{d-1}}\exp\pp{-\frac{\cflu\lambda}{4\sqrt{h_2r_2}}}
				\m\lambda\\
	&\,\,\quad\qquad\qquad\qquad\qquad\qquad\qquad
			+	\int_{h_2^{\ratespeed}\frac{\sqrt{h_2r_2}}{\log\pp{r_2}}}^{Cr_2}
					C\frac{r_1^{\beta(d-1)}}{\lambda^{d-1}}\exp\pp{-C^{-1}\frac{h_2^{\ratespeed}}{\log\pp{r_2}}}
				\m\lambda\\
	&\,\,	\leq		3\Lambda	
			+	C\frac{r_1^{\beta(d-1)}}{\Lambda^{d-1}}\sqrt{h_2r_2}\int_{\frac{\cflu\Lambda}{4\sqrt{h_2r_2}}}^\infty\exp\pp{-s} \m s\\
	&\,\,\quad\qquad\qquad\qquad\qquad\qquad\qquad
			+	C\frac{r_1^{\beta(d-1)}\log(r_2)^{d-1}}{h_2^{\pp{\ratespeed+\frac{1}{2}}(d-1)}r_2^{\frac{d-1}{2}-1}}\exp\pp{-C^{-1}\frac{h_2^{\ratespeed}}{\log\pp{r_2}}}\\
	&\,\,	\leq		3\Lambda	
			+	C\frac{r_1^{\beta(d-1)}}{\Lambda^{d-1}}\sqrt{h_2r_2}\exp\pp{-\frac{\cflu\Lambda}{4\sqrt{h_2r_2}}}
			+	C\exp\pp{-C^{-1}\frac{h_2^{\ratespeed}}{\log\pp{r_2}}},
\end{align*}
where in the last step we used that $\log(h_2)\geq c\log(r_2)$ to absorb the polynomial terms in front.
Now choosing $\Lambda=\frac{4\sqrt{h_2r_2}}{\cflu}\log\pp{r_1^{\beta(d-1)}}$, we have
\begin{multline*}
	\EV{\Tfull\charfun[\Espeed{h_2}\pp{\ov{B}_{C_{\beta}r_3^\beta}}]}-\EV{\metb{h_2}{r_2 e}{\hsma{e}{\beta}{h_2}{r_2}}}\\
	\leq		C\sqrt{h_2r_2}\log\pp{r_1}
			+	C\exp\pp{-C^{-1}\frac{h_2^{\ratespeed}}{\log\pp{r_2}}}
\end{multline*}
for some $C=C(c,\beta,\data)>0$.

\underline{Step 2: Managing the error on the complement of $\Espeed{h}$.}
Restricting to $\Espeed{h_2}\pp{\ov{B}_{C_{\beta}r_3^\beta}}$ at the cost of an error term, applying the estimate \eqref{lin_eq_Np}, and using the stationarity of the probability distribution, after absorbing polynomial prefactors since $\log(h_2)\geq c\log(r_2)\geq c\log\pp{\frac{r_3}{2}}$ we obtain 
\begin{align*}
	&\EV{\metb{h_3}{r_3e}{\hsma{e}{\beta}{h_3}{r_3}}}\\
	&\quad	\leq		\EV{\met{h_3}{r_3e}{\hsma{e}{\beta}{h_3}{r_3}}\charfun[\Espeed{h_1}\pp{\ov{B}_{C_{\beta}r_3^\beta}}]} + Cr_3\pp{1-\PM{\Espeed{h_2}\pp{\ov{B}_{C_{\beta}r_3^\beta}}}}\\
	&\quad	\leq		\EV{\Tfull\charfun[\Espeed{h_2}\pp{\ov{B}_{C_{\beta}r_3^\beta}}]}
					+ Ch_2
					+ \EV{\metb{h_1}{r_1e}{\hsma{e}{\beta}{h_1}{r_1}}}
					+ C\exp\pp{-C^{-1}h_2^{\ratespeed}}.
\end{align*}
Plugging in the final inequality for $\EV{\Tfull\charfun[\Espeed{h_2}\pp{\ov{B}_{C_{\beta}r_3^\beta}}]}$ from Step 1 and absorbing all smaller terms into $C\sqrt{h_2r_2}\log\pp{r_1}$ yields the result. 
\end{proof}

Having obtained the approximate sub-linearity, we are now able to prove the existence of the limit for the averaged arrival times and provide a rate such that the limit can not be approached any slower from below.

\begin{corollary}\label{lin_cor_exUpB}
Assume that \ref{sett_aP_stat}, \ref{sett_aP_fin}, and \ref{sett_aP_Aconst} hold.
Let $\beta\geq\frac{3}{2}$ and $0<\vtheta<1$. 
Let $\hsma{e}{\beta}{h}{r}$ be  the approximation of the half-space from Corollary \ref{lin_cor_bd} with $e\in\Sd$ and $r,h\geq h_0$.
For $r\geq h_0^{-\vtheta}$ we set $h(r)\coloneqq r^{\vtheta}$.

Then for all $\delF\in\R$ with the coefficient field $(A,F+\delF)$ satisfying \ref{veff_aP_star} the limit
\begin{equation*}
	\limm{\delF}{\beta}{e}\coloneqq \lim_{r\rightarrow\infty}	\frac{1}{r}\EV{\metb[(A,F+\delF)]{h(r)}{re}{\hsma{e}{\beta}{h(r)}{r}}}
\end{equation*}
exists, is independent of $\vtheta$ and there exists $C=C(\delF,\beta,\vtheta,\data)>0$ such that for all $r\geq h_0^{-\vtheta}$ there holds
\begin{equation}\label{lin_eq_limmUB}
	\limm{\delF}{\beta}{e}	\leq 	\frac{1}{r}\EV{\metb[(A,F+\delF)]{h(r)}{re}{\hsma{e}{\beta}{h(r)}{r}}} + C\frac{\log(r)}{\sqrt{r^{1-\vtheta}}}.
\end{equation}
The dependence of $C$ on $\delF$ is due to replacing $C_{1F}$ from \eqref{sett_eq_assF} by $C_{1F}+\abs{\delF}$ in $\data$, and hence the same constant can be used for all $\delF'$ with $\abs{\delF'}\leq \abs{\delF}$.
\end{corollary}

\begin{proof}
If Assumptions \ref{sett_aP_stat}, \ref{sett_aP_fin}, and \ref{sett_aP_Aconst} hold for the random variable $(A,F)$, then they clearly also hold for $(A,F+\delF)$.

We first comment on the independence of the limit $\limm{\delF}{\beta}{e}$ with respect to $\vtheta$.
The existence of the limit will be proven below.
Let $0<\vtheta_1<\vtheta_2<1$. 
Let $\limm[i]{\delF}{\beta}{e}$ denote the respective limit with $h(r)$ replaced with $h_i(r)=r^{\vtheta_i}$ for $i=1,2$.
Note that for $E(r)\coloneqq	\Espeed{h_1(r)}\pp{\ov{B}_{C_{\beta}r^\beta}}$ with the probability bounds from Assumption \ref{veff_aP_star} we have $\lim_{r\rightarrow \infty}\Pm{E(r)}= 1$ and hence due to the truncation of the arrival times $i=1,2$ we obtain
\begin{equation}\label{lin_eq_vtheta1}
	\limm[i]{\delF}{\beta}{e} = \lim_{r\rightarrow\infty}	\frac{1}{r}\EV{\metb[(A,F+\delF)]{h_i(r)}{re}{\hsma{e}{\beta}{h_i(r)}{r}}\charfun[E(r)]((A,F))}.
\end{equation}
For coefficients $(A,F)\in E(r)$, Assumption \ref{veff_aP_star} yields stable $(h_1(r),\cstable)$-approximations for the disks $\hsma{e}{\beta}{h_i(r)}{r}$ with $i=1,2$, for which $(A,F)$ and hence $(A,F+\delF)$ admit an effective minimum speed $\vmineff$ in $\ov{B}_{C_{\beta}r^\beta}$ on the scale $h_i(r)$.
Due to this effective minimum speed, the evolution of the stable approximation for a disk $\hsma{e}{\beta}{h_i(r)}{r}$ will $h_1$-envelop the other disk after a time step proportional to $h_1(r)+h_2(r)$.
The comparison principle from Lemma \ref{not_lem_comp} and the compatibility of the arrival times with ``$\hsubs[h_1]$'' from Lemma \ref{veff_lem_henv} yields 
\begin{multline*}
	\abs{\metb[(A,F+\delF)]{h_1(r)}{re}{\hsma{e}{\beta}{h_1(r)}{r}}
	-	\metb[(A,F+\delF)]{h_1(r)}{re}{\hsma{e}{\beta}{h_2(r)}{r}}}\\
	\leq C\pp{h_1(r)+h_2(r)}
\end{multline*}
for $C=C(\data)>0$.
We further replace the target area $re+\ov{B}_{h_1(r)}$ for the second arrival time with $re+\ov{B}_{h_2(r)}$, which yields an additional error proportional to $h_1(r)+h_2(r)$:
Clearly, the larger target area is reached first, but since the stable approximation expands with effective minimum speed $\vmineff$ on the scale $h_1(r)$, the smaller target area will be reached after an additional time step proportional to $h_1(r)+h_2(r)$.
Combining these bounds on the difference with \eqref{lin_eq_vtheta1}, we obtain
\begin{equation*}
	\abs{\limm[1]{\delF}{\beta}{e}-\limm[2]{\delF}{\beta}{e}}
	\leq		\lim_{r\rightarrow\infty}C \frac{h_1(r)+h_2(r)}{r}
	=		0,
\end{equation*}
which yields the independence of the limit $\limm{\delF}{\beta}{e}$ with respect to $\vtheta$.

We now prove the existence of the limit as well as \eqref{lin_eq_limmUB}.
For $r\geq h_0^{-\vtheta}$ we abbreviate
\begin{align*}
	G(r)\coloneqq \frac{1}{r}\EV{\metb[(A,F+\delF)]{h(r)}{re}{\hsma{e}{\beta}{h(r)}{r}}}.
\end{align*}
From Proposition \ref{lin_prop_sub} we have $C=C(\delF,\beta,\vtheta,\data)>0$, such that for $r\geq s$ we obtain
\begin{align}
	G(r)-G(s)
	&=	\frac{1}{r}\pp{rG(r)-sG(s)-(r-s)G(r-s)}
			- \frac{r-s}{r}\pp{G(s)-G(r-s)}\notag\\
	&	\leq		C\frac{\log(r)}{\sqrt{r^{1-\vtheta}}}
					+	\frac{r-s}{r}\pp{G(r-s)-G(s)}.\label{lin_eq_sublinAv}
\end{align}
The proof consists of 3 steps, only using the boundedness of $G(r)$ and \eqref{lin_eq_sublinAv}.
\begin{enumerate}[label=\arabic*.), left=12pt]
\item	For any $s\geq h_0^{-\vtheta}$ the sequence $\pp{G(ns)}_{n\in\N}$ is essentially bounded from above by $G(s)$ in the sense that there is some  $C=C(\delF,\beta,\vtheta,\data)>0$ such that
		\begin{align}\label{lin_eq_bddseq}
				G(ns)&\leq	G(s)+ C\frac{\log(s)}{\sqrt{s^{1-\vtheta}}}
			\qquad	\text{for all }n\in\N.	
		\end{align}
\item	We extend the bounds to all $r\geq s$ with the additional error decreasing for $r\rightarrow\infty$ with some  $C=C(\delF,\beta,\vtheta,\data)>0$ as
		\begin{align}\label{lin_eq_bddcont}
			G(r)		&\leq	G(s) + C\frac{\log(s)}{\sqrt{s^{1-\vtheta}}} + C\frac{s}{r}
			\qquad	\text{for all }r\geq s.		
		\end{align}
\item	The first step implies that the sequence $\pp{G(2^k s)}_{k\in\N}$ is essentially decreasing and hence converges.
		In combination with the second step, this implies the existence of $\lim_{r\rightarrow\infty}G(r)$.
\end{enumerate}
Finally, \eqref{lin_eq_limmUB} then follows from sending $n\rightarrow\infty$ in \eqref{lin_eq_bddseq}.

\underline{Step 1: $G(ns)$ is essentially bounded from above by $G(s)$.}
We will show \eqref{lin_eq_bddseq} by induction, keeping track of the constant to show that it stays bounded.
Assume that for some $k\in\Nz$ there is $C_k>0$ such that
\begin{align}\label{lin_eq_bddInd}
	G(ns)&\leq	G(s)+ C_k\frac{\log(s)}{\sqrt{s^{1-\vtheta}}}
			\qquad	\text{for all }n\in \Bp{1,\ldots,2^k}.	
\end{align}
This clearly holds for $k=0$ with $C_k=0$.
For any even $n\in \Bp{2^k+1,\ldots,2^{k+1}}$, with \eqref{lin_eq_sublinAv} and this assumption  we obtain
\begin{align*}
	G(ns)-G(s)	&=	G(ns)-G\pp{\frac{n}{2}s}+G\pp{\frac{n}{2}s}-G(s)\\
				&\leq	C\frac{\log(ns)}{\sqrt{n^{1-\vtheta}}\sqrt{s^{1-\vtheta}}} 
						+ C_k\frac{\log(s)}{\sqrt{s^{1-\vtheta}}}\\
				&\leq	\wt{C}_{k+1}\frac{\log(s)}{\sqrt{s^{1-\vtheta}}}
\end{align*}
with $\wt{C}_{k+1}=C_k+C\frac{\max\Bp{k,1}}{2^{\frac{1-\vtheta}{2}k}}$.

For any odd $n\in \Bp{2^k+1,\ldots,2^{k+1}}$, with \eqref{lin_eq_sublinAv} and since $n-1\leq 2^{k+1}$ is even, we hence obtain
\begin{align*}
	G(ns)-G(s)	&\leq	C\frac{\log(ns)}{\sqrt{n^{1-\vtheta}}\sqrt{s^{1-\vtheta}}} 
						+	\frac{n-1}{n}\pp{G((n-1)s)-G(s)}\\
				&\leq	C\frac{\log(ns)}{\sqrt{n^{1-\vtheta}}\sqrt{s^{1-\vtheta}}} 
						+ \wt{C}_{k+1}\frac{\log(s)}{\sqrt{s^{1-\vtheta}}},
\end{align*}
which yields \eqref{lin_eq_bddInd} for $k+1$ with $C_{k+1}=C_k+C\frac{\max\Bp{k,1}}{2^{\frac{1-\vtheta}{2}k}}$.
Since the increasing sequence $(C_k)_{k\in\N}$ converges, we hence obtain \eqref{lin_eq_bddseq} with the constant given by $\lim_{k\rightarrow\infty}C_k$.

\underline{Step 2: Bounding $G(r)$ for all $r\geq s$.}
For $r\geq s$, let $n\in\N$ such that $ns\leq r<(n+1)s$.
Then with \eqref{lin_eq_sublinAv} and \eqref{lin_eq_bddseq} from Step 1 we have
\begin{align*}
	G(r)-G(s)	&=		G(r)-G(ns)+G(ns)-G(s)\\
				&\leq	C\frac{\log(r)}{\sqrt{r^{1-\vtheta}}}	+	\frac{r-ns}{r}\pp{G(r-ns)-G(ns)}
						+C\frac{\log(s)}{\sqrt{s^{1-\vtheta}}}\\
				&\leq	C\frac{\log(s)}{\sqrt{s^{1-\vtheta}}} + C\frac{s}{r},
\end{align*}
where in the last step we used $r-ns\leq s$ by the choice of $n$ and that $0\leq G(\cdot)\leq C(\data)$.

\underline{Step 3: Existence of $\lim_{r\rightarrow\infty}G(r)$.} 
Applying Step 1 for each $r_k\coloneqq 2^ks$,  for the sequence $\pp{G(r_k)}_{k\in\N}$ we obtain that
\begin{equation*}
	G(r_\ell)\leq	G(r_k)+ C\frac{\log(r_k)}{\sqrt{r_k^{1-\vtheta}}}
	\qquad	\text{for all }\ell> k.
\end{equation*}
In particular, the sequence $b_n\coloneqq G(r_k)-\sum_{\ell=1}^{k-1}C\frac{\log(r_\ell)}{\sqrt{r_\ell^{1-\vtheta}}}$ is decreasing.
Since it is bounded from below, it converges and thus also $\pp{G(r_k)}_{k\in\N}$.

It remains to prove that also away from this sequence $G(r)$ converges to the same limit for $r\rightarrow \infty$. 
Let $\ov{G}=\lim_{k\rightarrow\infty}G(r_k)$. 
Let $\eps>0$.
There exists $k_0\in\N$ such that
\begin{align*}
	\abs{G(r_k)-\ov{G}}			&\leq \frac{\epsilon}{3},
	&	C\frac{\log(r_k)}{r_k}	&\leq \frac{\epsilon}{3}
	&&\text{for all }k\geq k_0,
\end{align*}
with the constant $C$ as in \eqref{lin_eq_bddcont}.
Let $r\geq \max\Bp{r_{k_0}, 3C\eps^{-1}r_{k_0}}$ with $C$ as in  \eqref{lin_eq_bddcont}.
On the one hand, \eqref{lin_eq_bddcont} yields
\begin{equation*}
	G(r)	\leq		G(r_{k_0})+ C\frac{\log(r_{k_0})}{r_{k_0}} + C\frac{r_{k_0}}{r}
		\leq		\ov{G}+\frac{\eps}{3}+\frac{\eps}{3}+\frac{\eps}{3}
		\leq		\ov{G}+\eps.
\end{equation*}
On the other hand, there is $k\geq k_0$ such that $r_k\geq \max\Bp{r, 3C\eps^{-1}r}$, for which \eqref{lin_eq_bddcont} yields
\begin{equation*}
	G(r)	\geq		G(r_{k})	- C\frac{\log(r)}{r} - C\frac{r}{r_k}
		\geq		\ov{G}-\frac{\eps}{3}-\frac{\eps}{3}-\frac{\eps}{3}
		\geq		\ov{G}-\eps.
\end{equation*}
We have thus shown that $G(r)$ converges to $\ov{G}$ for $r\rightarrow\infty$, concluding the last step of the proof.
\end{proof}

\subsection{Approximate super-linearity and upper bounds for the convergence}
Next, we will obtain the super-linearity of the truncated arrival times, albeit at the cost of changing the forcing by a small constant due to applying Corollary \ref{lin_cor_bd}.

\begin{proposition}[Approximate super-linearity of the truncated arrival times]\label{lin_prop_sup}
Assume that \ref{sett_aP_stat}, \ref{sett_aP_fin}, \ref{sett_aP_Aconst} and \ref{veff_aP_star} hold.
Let $\beta\geq\frac{3}{2}$ with $C_\beta>0$ and $\hsma{e}{\beta}{h}{r}$ for $e\in\Sd$ and $r,h\geq h_0$ as in Corollary \ref{lin_cor_bd}.

Let $r_2\geq r_1\geq h_0$ and $h_3\geq h_2\geq h_1\geq h_0$ with $\frac{\log(h_2)}{\log(r_2)}, \frac{\log(h_1)}{\log(r_1)}\geq c$ and for some $c>0$.
Then there exists $C=C(c,\beta,\data)>0$ such that for $f(r)=r^{-\frac{\beta}{2}+\frac{3}{4}}$ there holds
\begin{align}\label{lin_eq_sup}
	&\EV{\metb[(A,F)]{h_3}{\pp{r_1+r_2}e}{\hsma{e}{\beta}{h_3}{r_1+r_2}}}\notag\\
	&\qquad	\geq		\EV{\metb[\pp{A,F+f(r_1)}]{h_1}{r_1 e}{\hsma{e}{\beta}{h_1}{r_1}}}
					+ \EV{\metb[\pp{A,F+f(r_2)}]{h_2}{r_2 e}{\hsma{e}{\beta}{h_2}{r_2}}}	\notag\\			
	&\qquad\quad		- C\sqrt{h_3r_2}\log\pp{r_2}.
\end{align}
\end{proposition}

\begin{proof}
We set $r_3=r_1+r_2$.
We claim that the time until ${\hsma{e}{\beta}{h_3}{r_3}}$ reaches $\ov{B}_{h_3}(r_3 e)$ is larger than adding the time it takes until ${\hsma{e}{\beta}{h_3}{r_3}}$ first reaches any point of $r_2e+\hsma{e}{\beta}{h_2}{r_3}$ to the time it takes until the smaller `intermediate set' $r_2e+\hsma{e}{\beta}{h_1}{r_1}$ reaches $\ov{B}_{h_1}(r_3 e)$ with a slightly increased forcing, see Figure \ref{lin_fig_Nminus}.

\begin{figure}[h]
\centering
\includegraphics[width=0.9\textwidth]{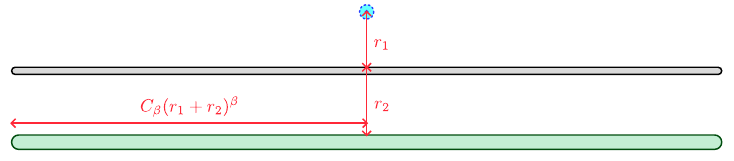}\\\vspace{10pt}
\includegraphics[width=0.9\textwidth]{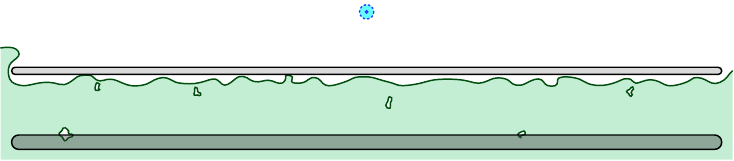}\\\vspace{10pt}
\includegraphics[width=0.9\textwidth]{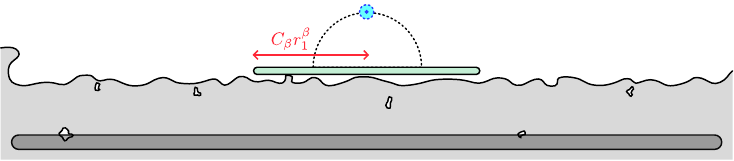}

\caption{	An illustration of the argument for the super-linearity of arrival times.
			Top: The large disk is depicted in green with its target area in blue, and the `intermediate' translation of the large disk is depicted in grey.
			Center: The evolution of the large disk at time $\Tfirst$ (depicted in green) reaches the intermediate disk.
			Bottom: At time $\Tfirst$, we replace the evolved set with the smaller intermediate disk (depicted in green), which stays ahead the evolution of the large disk within the dotted half-ball when increasing the forcing by a constant as in Proposition \ref{inf_prop_half}.}
\label{lin_fig_Nminus}
\end{figure}

More precisely, we denote the time when the disk $r_2e+\hsma{e}{\beta}{h_2}{r_3}$ is first reached with
\begin{equation*}
	\Tfirst\coloneqq	\inf\Bp{t\geq 0\,:\,	\pp{r_2e+\hsma{e}{\beta}{h_2}{r_3}}\cap\reach{t}{\hsma{e}{\beta}{h_3}{r_3}}\neq \emptyset }.
\end{equation*}
We will show that for $(A,F)\in \Espeed{h_2}\pp{r_2e+\ov{B}_{C_\beta r_3^\beta}}\cap\Espeed{h_1}\pp{r_2e+\ov{B}_{C_\beta r_1^\beta}}$ there holds
\begin{multline}\label{lin_eq_Nm}
	\met[(A,F)]{h_3}{r_3e}{\hsma{e}{\beta}{h_3}{r_3}}\\
	\geq		\Tfirst[(A,F)]	+	\met[\pp{A,F+f(r_1)}]{h_1}{r_3 e}{r_2e+\hsma{e}{\beta}{h_1}{r_1}} - 2\pp{\cstable+2\vmineff^{-1}}h_3.
\end{multline}
In order to obtain \eqref{lin_eq_sup}, we will follow these three steps:
\begin{enumerate}[label=\arabic*.), left=12pt]
\item	We will prove \eqref{lin_eq_Nm}.
\item	We will show that the expected value of $\Tfirst$ restricted to $\Espeed{h_2}\pp{\ov{B}_{2C_{\beta}r_3^\beta}}$ can be bounded from below by the expected value of $\met[(A,F+f(r_2))]{h_2}{r_2 e}{\hsma{e}{\beta}{h_2}{r_2}}$ up to an error corresponding to the fluctuation bounds from Proposition \ref{fluct_prop}.
\item	We take the expected value of \eqref{lin_eq_Nm}, use the stationarity from Assumption \ref{sett_aP_stat} and manage the error on the complement of $\Espeed{h_2}\pp{\ov{B}_{C_{\beta}r_3^\beta}}$.
\end{enumerate}

\underline{Step 1: Proving \eqref{lin_eq_Nm}.}
We claim that for $(A,F)\in \Espeed{h_2}\pp{r_2e+\ov{B}_{C_\beta r_3^\beta}}$ and for $r_3\geq C=C(\beta,\data)>0$ large enough there holds
\begin{equation}\label{lin_eq_upNm}
	\pp{r_2e+\pp{\ov{B}_{C_\beta r_1^\beta}\cap\hsp{e}}}	\cap		\reach[(A,F)]{t}{\hsma{e}{\beta}{h_3}{r_3}}	=	\emptyset
	\qquad\text{for all }0\leq t\leq \Tfirst[(A,F)].
\end{equation}
Assuming that this claim holds and if in addition $(A,F)\in\Espeed{h_1}\pp{r_2e+\ov{B}_{C_\beta r_1^\beta}}$, similar to the proof of  Corollary \ref{lin_cor_bd} we apply Proposition \ref{inf_prop_half} in $r_2e+\ov{B}_{C_\beta r_1^\beta}$:
For the set $S_1$ we choose the evolution of the stable $(h_1,\cstable)$-approximation of $r_2e+\hsma{e}{\beta}{h_1}{r_1}$ after time $\pp{\cstable+3\vmineff^{-1}}h_1$ with the holes of width smaller than $2h_1$ filled as in the proof of the comparison principle from Lemma \ref{not_lem_comp}. 
This comparison principle further implies
\begin{align*}
	S_1	&\subset	\reach[(A,F+f(r_1))]{\pp{\cstable+3\vmineff^{-1}}h_1}{r_2e+\hsma{e}{\beta}{h_1}{r_1}}.
\intertext{For the set $S_2$ we choose}
	S_2	&=	\reach[(A,F)]{\Tfirst[(A,F)]}{\hsma{e}{\beta}{h_3}{r_3}}.
\end{align*}
The parameters are given by
\begin{align*}
		r	&=	r_1	,
	&	R	&=	C_\beta r_1^\beta,
	&	\delF	&=	f(r_1) = r_1^{-\frac{\beta}{2}+\frac{3}{4}},
	&	T	&=	\timebound{h_1}{r_3 e}{r_2e + \hsma{e}{\beta}{h_1}{r_1}}.
\end{align*}
The choice of $C_\beta$ from Corollary \ref{lin_cor_bd} guarantees that \eqref{inf_eq_T} holds for these parameters, which together with the claim \eqref{lin_eq_upNm} allows us to apply Proposition \ref{inf_prop_half}. 
We hence obtain
\begin{multline*}
	\pp{r_2e+\pp{\ov{B}_{C_\beta r_1^\beta}\cap\hsp{e}}}	\cap	\reach[(A,F)]{t+\Tfirst[(A,F)]}{\hsma{e}{\beta}{h_3}{r_3}}\\
	\subset\reach[(A,F+f(r_1))]{t+\pp{\cstable+3\vmineff^{-1}}h_1}{r_2e+\hsma{e}{\beta}{h_1}{r_1}}
\end{multline*}
for $0\leq t\leq \timebound{h_1}{r_3e}{r_2e+\hsma{e}{\beta}{h_1}{r_1}}$, the time by which the set is guaranteed to reach $\ov{B}_{h_1}(r_3 e)$.
This implies \eqref{lin_eq_Nm} when considering that if the set $r_2e+\hsma{e}{\beta}{h_1}{r_1}$ reaches $\ov{B}_{h_3}(r_3e)$, it will reach $\ov{B}_{h_1}(r_3e)$ at most $\cstable h_1+\vmineff^{-1}h_3$ later due to the effective minimum speed of the stable approximation from Assumption \ref{veff_aP_star}.

To obtain \eqref{lin_eq_upNm} we compare -- within the domain $r_2e+\pp{\ov{B}_{C_\beta r_3^\beta}\cap\hsp{e}}$ -- the evolution of $\hsma{e}{\beta}{h_3}{r_3}$ to a stable approximation of the upper arc of the boundary domain given by 
	$r_2e+\pp{\partial B_{C_\beta r_3^\beta}\cap\hsp{e}}$, 
which exists since $(A,F)\in \Espeed{h_2}\pp{r_2e+\ov{B}_{C_\beta r_3^\beta}}$, see Figure \ref{lin_fig_Nminus_comp}.
\begin{figure}[h]
\centering
\includegraphics[width=0.75\textwidth]{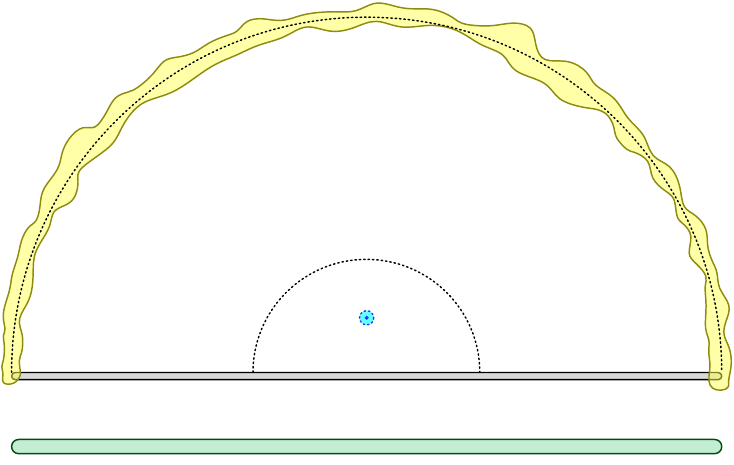}
\caption{An illustration of the comparison argument from Step 1 of the proof for the super-linearity of arrival times.
		The stable approximation of the upper arc of the half-ball is depicted in yellow.
		In this half-ball, we compare its evolution to the evolution of the large disk (depicted in green). 
		Until time $\Tfirst$, by definition this evolution does not reach the intermediate disk (depicted in grey), which makes up the straight part of the half-ball's boundary.}
\label{lin_fig_Nminus_comp}
\end{figure}
To be precise, we apply the comparison principle as in Theorem \ref{comp_thm_comp} on the bounded domain 
	$r_2e+\pp{\ov{B}_{C_\beta r_3^\beta}\cap\hsp{e}}$ 
to the supersolution $\ur{\hsma{e}{\beta}{h_3}{r_3}}$ from Lemma \ref{not_lem_ur} and a solution $\wt{u}\leq 1$ to \eqref{intro_eq_uls} as in Definition \ref{intro_def_met} with the stable approximation of the arc corresponding to the initial $0$-sublevel set. 
Clearly $\ur{\hsma{e}{\beta}{h_3}{r_3}}=1\geq \wt{u}$ for $t=0$ on the domain. 
Regarding the boundary, on the flat part we have $\ur{\hsma{e}{\beta}{h_3}{r_3}}=1\geq \wt{u}$ for $0\leq t\leq \Tfirst$ due to the definition of $\Tfirst$.
On the arc, due to the stability of the approximation we have $\wt{u}=0\leq \ur{\hsma{e}{\beta}{h_3}{r_3}}$.
Taking into account the maximum speed of propagation from Lemma \ref{not_lem_vmax}, the comparison principle yields
\begin{multline*}
	\reach[(A,F)]{t}{\hsma{e}{\beta}{h_3}{r_3}}\cap \pp{r_2e+\pp{\ov{B}_{C_\beta r_3^\beta}\cap\hsp{e}}}\\
	\subset	r_2e + \Rd\setminus\pp{C_\beta r_3^\beta-C h_2-\vmax t}\ov{B}_1
\end{multline*}
for $0\leq t\leq \Tfirst$.
Because $\Tfirst[(A,F)]\leq Cr_2\leq Cr_3<<C_\beta r_3^\beta$ for $(A,F)\in \Espeed{h_3}\pp{r_2e+\ov{B}_{C_\beta r_3^\beta}}$ and $r_3$ large enough, this yields \eqref{lin_eq_upNm} and hence \eqref{lin_eq_Nm}.

\underline{Step 2: Bounding $\Tfirst$.}
For $(A,F)\in\Espeed{h_2}\pp{\ov{B}_{2C_{\beta}r_3^\beta}}$, due to the minimum effective speed we can replace $\Tfirst$ with the minimum of a finite amount of truncated arrival times at the cost of a small error.
Considering only points on a grid within the upper disk, whenever any point of the disk is reached, the closest grid point will be reached just shortly afterwards: 
The first point of contact will be reached $\cstable h_2$ later by the stable approximation of $\hsma{e}{\beta}{h_3}{r_3}$, which will then reach the closest grid point at the rate of the effective minimum speed.
That is, for $\lambda>0$ and $\Z^{d-1}_e$ denoting $\Z^{d-1}$ placed in the hyperplane $\Bp{x\in\Rd\,:\,x\cdot e=0}$, we have
\begin{align*}
	\Tfirst[(A,F)]	
	&\geq	\inf\Bp{\met[(A,F)]{h_2}{x+r_2e}{\hsma{e}{\beta}{h_3}{r_3}}	\,:\,	x\in \frac{\vmineff}{\sqrt{d}}\lambda\Z^{d-1}_e\cap \hsma{e}{\beta}{h_2}{r_3}}\\
	&\quad	-\pp{\cstable+\vmineff^{-1} }h_2-\lambda\\
	&\geq	\inf\Bp{\met[(A,F+f(r_2))]{h_2}{x+r_2e}{x+\hsma{e}{\beta}{h_2}{r_2}}	\,:\,	x\in \frac{\vmineff}{\sqrt{d}}\lambda\Z^{d-1}_e\cap \hsma{e}{\beta}{h_2}{r_3}}\\
	&\quad	-2\pp{\cstable+2\vmineff^{-1} }h_2-\lambda,
\end{align*}
where for the second inequality Corollary \ref{lin_cor_bd} allows us to replace the large disk with the smaller disks $x+\hsma{e}{\beta}{h_2}{r_2}$ placed as in Figure \ref{lin_fig_Nplus_smallDisk}, only with a larger upper disk. 
Abbreviating
\begin{equation*}
	T_2=\EV{\metb[(A,F+f(r_2))]{h_2}{r_2 e}{\hsma{e}{\beta}{h_2}{r_2}}},
\end{equation*}
for $\lambda\geq 2\pp{\cstable+2\vmineff^{-1}} h_2$ we thus obtain
\begin{align*}
	&\PM{\Tfirst\charfun[\Espeed{h_2}\pp{\ov{B}_{2C_{\beta}r_3^\beta}}]-T_2\leq -3\lambda}\\
	&\,\,	\leq		\Pm\Bigg[\inf_{x\in \frac{\vmineff\lambda}{\sqrt{d}}\Z^{d-1}_e\cap \hsma{e}{\beta}{h_2}{r_3}}\met[(A,F+f(r_2))]{h_2}{x+r_2e}{x+\hsma{e}{\beta}{h_2}{r_2}}\charfun[\Espeed{h_2}\pp{\ov{B}_{2C_{\beta}r_3^\beta}}]
					\leq T_2-\lambda\Bigg]\\
	&\,\,	\leq		C\frac{r_3^{\beta(d-1)}}{\lambda^{d-1}}
					\PM{\metb[(A,F+f(r_2))]{h_2}{r_2 e}{\hsma{e}{\beta}{h_2}{r_2}}-T_2\leq -\lambda}\\
	&\,\,\quad				+1-\PM{\Espeed{h_2}\pp{\ov{B}_{2C_{\beta}r_3^\beta}}}
\end{align*}
with $C=C(\beta,\data)>0$, where for the second inequality we used a union bound, the stationarity of the probability distribution from Assumption \ref{sett_aP_stat} and obtained an error term for loosing the restriction to $\Espeed{h_2}\pp{\ov{B}_{2C_{\beta}r_3^\beta}}$.
Applying the fluctuation bounds from Proposition \ref{fluct_prop} and with calculations analogous to Step 1 in the proof of Proposition \ref{lin_prop_sub} while absorbing the additional term $Cr_2\pp{1-\PM{\Espeed{h_2}\pp{\ov{B}_{2C_{\beta}r_3^\beta}}}}$, we obtain
\begin{multline*}
	\EV{\Tfirst\charfun[\Espeed{h_2}\pp{\ov{B}_{2C_{\beta}r_3^\beta}}]}-\EV{\metb[(A,F+f(r_2))]{h_2}{r_2 e}{\hsma{e}{\beta}{h_2}{r_2}}}\\
	\geq		-C\sqrt{h_2r_2}\log\pp{r_2}
			-	C\exp\pp{-C^{-1}\frac{h_2^{\ratespeed}}{\log\pp{r_2}}}
\end{multline*}
for some $C=C(c,\beta,\data)>0$.

\underline{Step 3: Managing the error on the complement of $\Espeed{h}$.}
Taking the expected value of the inequality \eqref{lin_eq_Nm}, using the stationarity of the probability distribution from Assumption \ref{sett_aP_stat}, applying Step 2 and writing $E=\Espeed{h_2}\pp{r_2e+\ov{B}_{C_\beta r_3^\beta}}\cap\Espeed{h_1}\pp{r_2e+\ov{B}_{C_\beta r_1^\beta}}$ we obtain 
\begin{align*}
	&\EV{\metb[(A,F)]{h_3}{r_3e}{\hsma{e}{\beta}{h_3}{r_3}}}\\
	&\quad\geq	\EV{\met[(A,F)]{h_3}{r_3e}{\hsma{e}{\beta}{h_3}{r_3}}\charfun[E]}\\
	&\quad\geq	\EV{\metb[(A,F+f(r_2))]{h_2}{r_2 e}{\hsma{e}{\beta}{h_2}{r_2}}}- C\sqrt{h_3r_2}\log\pp{r_2}	\\
	&\qquad		+	\EV{\met[\pp{A,F+f(r_1)}]{h_1}{r_1 e}{\hsma{e}{\beta}{h_1}{r_1}}\charfun[E]}\\
	&\quad\geq	\EV{\metb[(A,F+f(r_2))]{h_2}{r_2 e}{\hsma{e}{\beta}{h_2}{r_2}}}
				+\EV{\metb[\pp{A,F+f(r_1)}]{h_1}{r_1 e}{\hsma{e}{\beta}{h_1}{r_1}}}\\
	&\qquad				-C\sqrt{h_3r_2}\log\pp{r_2}
\end{align*}
for some $C=C(c,\beta,\data)$. 
For the second inequality, we also used that $Ch_3\leq C\sqrt{h_3r_2}$ since $h_3\leq r_3\leq 2 r_2$.
For the last inequality, we loosened the restriction to $E$ and used that $\frac{\log(h_1)}{\log(r_1)}\geq c$ to absorb the resulting error term, which is controlled by the bounds on $1-\PM{\Espeed{h_2}\pp{r_2e+\ov{B}_{C_\beta r_3^\beta}}}$ and $1-\PM{\Espeed{h_1}\pp{r_2e+\ov{B}_{C_\beta r_1^\beta}}}$ from Assumption \ref{veff_aP_star}. 
\end{proof}

Now that we have the approximate superlinearity, we can use this for the convergence of the averaged arrival times:
We obtain a rate such that the limit can not be approached any slower from above, albeit this only works if we slightly increase the forcing for the arrival times compared to the ones used for the limit process.

\begin{corollary}\label{lin_cor_LowB}
Assume that \ref{sett_aP_stat}, \ref{sett_aP_fin}, and \ref{sett_aP_Aconst} hold.
Let $\hsma{e}{\beta}{h}{r}$ be  the approximation of the half-space from Corollary \ref{lin_cor_bd} with $e\in\Sd$ and $r,h\geq h_0$.
Let $\beta>\frac{3}{2}$, $0<\vtheta<1$, $h(r)\coloneqq r^{\vtheta}$ and $\limm{\delF}{\beta}{e}$ as in Corollary \ref{lin_cor_exUpB} as well as $f(r)=r^{-\frac{\beta}{2}+\frac{3}{4}}$.

Let $\delF\in\R$ and suppose that Assumption \ref{veff_aP_star} is satisfied for all coefficient fields $(A,F+\delF+c)$ with $0\leq c \leq C_1 f(r)$, where $C_1=C_1(\beta)>0$ is determined in the proof below.
Then
\begin{equation*}
	\limm{\delF}{\beta}{e}	\geq 	\frac{1}{r}\EV{\metb[(A,F+\delF+C_1f(r))]{h(r)}{re}{\hsma{e}{\beta}{h(r)}{r}}} - C_2\frac{\log(r)}{\sqrt{r^{1-\vtheta}}}
\end{equation*}
with $C_2=C_2(\delF,\beta,\vtheta,\data)>0$.
The dependence of $C_2$ on $\delF$ is due to replacing $C_{1F}$ from \eqref{sett_eq_assF} by $C_{1F}+\abs{\delF}$ in $\data$, and hence the same constant can also be used for all $\delF'$ with $\abs{\delF'}\leq \abs{\delF}$.
\end{corollary}

\begin{proof}
For $r\geq h_0^{-\vtheta}$ and $\delF\geq 0$ we abbreviate 
\begin{equation*}
	G_{\delF}(r)\coloneqq \frac{1}{r}\EV{\metb[(A,F+\delF)]{h(r)}{re}{\hsma{e}{\beta}{h(r)}{r}}}.
\end{equation*}
With Proposition \ref{lin_prop_sup}, for $k\in\Nz$ we obtain
\begin{align*}
	&G_{\delF}(2^{k+1} r)-G_{\delF+f(2^kr)}(2^k r)\\
	&\quad	=		\frac{1}{2^{k+1}r}\pp{2^{k+1}rG_{\delF}(2^{k+1}r)-2^{k}rG_{\delF+f(2^kr)}(2^kr)-2^krG_{\delF+f(2^k r)}(2^kr)}\\
	&\quad	\geq		-C\frac{\log(2^{k+1}r)}{\sqrt{(2^{k+1}r)^{1-\vtheta}}}\\
	&\quad	\geq		-C\frac{\max\Bp{k,1}}{2^{\frac{1-\vtheta}{2} k}}\frac{\log(r)}{\sqrt{r^{1-\vtheta}}}.
\end{align*}
Chaining this inequality, in particular increasing the forcing at every step, we obtain
\begin{align*}
	G_{\delF}(2^{k+1} r)-G_{\delF+\sum_{i=0}^k f(2^kr)}(r)
	\geq		-C\pp{\sum_{i=0}^k\frac{\max\Bp{k,1}}{2^{\frac{1-\vtheta}{2} k}}}\frac{\log(r)}{\sqrt{r^{1-\vtheta}}}.
\end{align*}
Note that the sum of the forcing increments converges and scales like $f(r)$:
Because we chose $\beta>\frac{3}{2}$, there is $C_1=C_1(\beta)>0$ such that
\begin{equation*}
	\sum_{i=0}^k f(2^kr)	=		r^{-\frac{\beta}{2}+\frac{3}{4}}\sum_{i=0}^k 2^{-\pp{\frac{\beta}{2}-\frac{3}{4}}k}
						\leq	 	C_1r^{-\frac{\beta}{2}+\frac{3}{4}}
						=		C_1f(r).
\end{equation*}
Since the arrival times decrease when increasing the forcing, we have $G_{\delF}(r)\geq G_{\delF'}(r)$ for any $0\leq\delF\leq \delF'$ and hence obtain
\begin{equation*}
	G_{\delF}(2^{k+1} r)	
	\geq		G_{\delF+C_1 f(r)}(r)	-	C\frac{\log(r)}{\sqrt{r^{1-\vtheta}}}.
\end{equation*}
From Corollary \ref{lin_cor_LowB} we know that $\limm{\delF}{\beta}{e}=\lim_{s\rightarrow\infty}G_{\delF}(s)$ exists.
Taking the limit for $k\rightarrow\infty$ in the inequality above thus concludes the proof.
\end{proof}

\subsection{Definition of the homogenized speed and comments on Assumption \ref{sett_aP_gen}}
Corollary \ref{lin_cor_exUpB} guarantees that the limit of the averaged arrival times for the half-space approximations exist.
This allows us to define our candidates for the homogenized speed for different disk width scalings.

\begin{definition}[The candidates for the homogenized speed]\label{lin_def_vhom}
Assume that \ref{sett_aP_stat}, \ref{sett_aP_fin}, and \ref{sett_aP_Aconst} hold.
Then for any $\delF\in\R$ with $(A,F+\delF)$ satisfying \ref{veff_aP_star}, $\beta>\frac{3}{2}$, and $e\in\Sd$ we set
\begin{align*}
	\vhom[\delF,\beta](e)\coloneqq \pp{\lim_{r\rightarrow\infty}	\frac{1}{r}\EV{\metb[(A,F+\delF)]{h(r)}{re}{\hsma{e}{\beta}{h(r)}{r}}}}^{-1},
\end{align*}
where $\hsma{e}{\beta}{h(r)}{r}$ are the half-space approximations with $h(r)\coloneqq r^{\vtheta}$ for some $0<\vtheta<1$.
The existence of the limits and their independence with respect to $\vtheta$ follows from Corollary \ref{lin_cor_exUpB}.
It is straightforward to see that 
\begin{equation*}
	\vmineff\leq \vhom[\delF,\beta](e)	\leq \vmax(\delF)
	\qquad\text{for all  $\beta>\frac{3}{2}$, and $e\in\Sd$}
\end{equation*}
with $\vmax(\delF)$ for $(A,F+\delF)$ corresponding to the maximum speed $\vmax$ for $(A,F)$ from Lemma \ref{not_lem_vmax}.
The lower bound follows from Assumption \ref{veff_aP_star} via Lemma \ref{stable_lem_TvsEspeed}.
\end{definition}

In terms of the homogenized speed, the quantitative bounds from Corollary \ref{lin_cor_exUpB} and Corollary \ref{lin_cor_LowB} yield
\begin{align}
	\EV{\metb[(A,F+\delF)]{h(r)}{re}{\hsma{e}{\beta}{h(r)}{r}}}
	&\geq		\vhom[\delF,\beta]^{-1}(e)r - C\log(r)\sqrt{r^{1+\vtheta}}\label{lin_eq_vhomUpB}
\intertext{ and, if all coefficient fields $(A,F+\delF+c)$ with $0\leq c \leq C_1 f(r)$ satisfy \ref{veff_aP_star}, }
	\EV{\metb[(A,F+\delF)]{h(r)}{re}{\hsma{e}{\beta}{h(r)}{r}}}
	&\leq		\vhom[\delF-C_1f(r),\beta]^{-1}(e)r + C\log(r)\sqrt{r^{1+\vtheta}}.\label{lin_eq_vhomLowB}
\end{align}
Note that based on these inequalities we can fully quantify the convergence only if the function $\delF\mapsto\vhom[\delF,\beta]$ is H\"older continuous.
Assumption \ref{sett_aP_gen}, which we argued to be plausible in Remark \ref{sett_rem_hoel}, provides this H\"older continuity for some width scaling $\genhoel>\frac{3}{2}$.

However, morally the H\"older continuity should not be dependent on a specific width scaling $\genhoel$.
In fact, in the following lemma we show that the H\"older continuity for $\genhoel$ from Assumption \ref{sett_aP_gen} implies H\"older continuity with the same constant and coefficient for any $\beta>\frac{3}{2}$ and that all candidates from Definition \ref{lin_def_vhom} yield the same homogenized speed  independent of $\beta$.

\begin{lemma}\label{lin_lem_hoelInv}[The homogenized speed is independent of the width scaling]
Assume that the random coefficient field $(A,F)\in\Omega$ with probability distribution $\Pm$ satisfies Assumptions \ref{sett_aP_stat}, \ref{sett_aP_fin}, \ref{sett_aP_Aconst} and \ref{sett_aP_gen}.
Assume that there is $\delFdel\in(0,(\Choel)^{-1}]$ such that all coefficient fields $(A,F+\delF)$ with $\delF\in\bp{-\delFdel,\delFdel}$ satisfy Assumption \ref{veff_aP_star}.

Then for any $\beta>\frac{3}{2}$, $\delF>-\delFdel$ and $e\in\Sd$ the candidates for the homogenized speed from Definition \ref{lin_def_vhom} satisfy
\begin{equation*}
	\abs{\vhom[\delF,\beta](e)-\vhom[0,\beta](e)}	\leq		\Choel\abs{\delF}^{\hoel}	
	\qquad\text{ for all $\delF\in[-\delFdel,\delFdel]$}
\end{equation*}
with $\Choel>0$ and $0<\hoel\leq 1$ as in Assumption \ref{sett_aP_gen}.
In addition, there holds 
\begin{equation}\label{lin_eq_def_vhom0}
	\vhom(e)	\coloneqq	\vhom[0,\genhoel](e)	=	\vhom[0,\beta](e)
	\qquad\text{for all }\beta>\frac{3}{2}.
\end{equation}
\end{lemma}

\begin{proof}
Let $\beta_1\geq \beta_2>\frac{3}{2}$ be two different scalings for the width of the half-space approximations. 
We claim that 
\begin{equation}\label{lin_eq_vhomBetas}
	\vhom[\delF,\beta_1](e)	\geq \vhom[\delF,\beta_2](e)	\geq		\vhom[\wt{\delF},\beta_1](e)
	\quad\text{for any $\delF>\wt{\delF}\geq-\delFdel$, $e\in\Sd$.}
\end{equation}
If this claim holds, we immediately obtain that $\vhom[0,\beta]=\vhom[0,\genhoel]$ for all $\beta>\frac{3}{2}$ since due to Assumption \ref{sett_aP_gen} we have $\vhom[0,\genhoel](e)=\lim_{\delF\rightarrow 0}\vhom[\delF,\genhoel](e)$.
We further obtain the H\"older continuity for arbitrary $\beta>\frac{3}{2}$ as follows.
For $\beta>\genhoel$ and $\delF\geq0$, let $\wt{\delF}>\delF$. Then \eqref{lin_eq_vhomBetas} yields
\begin{align*}
	\abs{\vhom[\delF,\beta](e)-\vhom[0,\beta](e)}	
	&=	\vhom[\delF,\beta](e)-\vhom[0,\beta](e)
	\leq		\vhom[\wt{\delF},\genhoel](e)-\vhom[0,\genhoel](e)
	\leq		\Choel\abs{\wt{\delF}}^{\hoel},
\end{align*}
where we used the H\"older continuity of $\wt{\delF}\mapsto \vhom[\wt{\delF},\genhoel](e)$ from Assumption \ref{sett_aP_gen}.
Taking the limit $\wt{\delF}\searrow\delF$ yields the desired inequality.
The cases of $\delF\leq 0$ as well as $\genhoel>\beta>\frac{3}{2}$ with $\delF\geq 0$ and $\delF\leq 0$ respectively are dealt with analogously.

It remains to prove \eqref{lin_eq_vhomBetas}.
Let $\delF>-\delFdel$.
For width scalings $\beta_1\geq \beta_2$, due to the definition of the half-space approximations from Corollary \ref{lin_cor_bd} we have the inclusion $\hsma{e}{\beta_1}{h(r)}{r}\supset \hsma{e}{\beta_2}{h(r)}{r}$ for all $r>h_0$ large enough.
Hence, the comparison principle from Lemma \ref{not_lem_comp} and the definition of the truncated arrival times yield
\begin{align*}
	\frac{1}{r}\metb[(A,F+\delF)]{h(r)}{re}{\hsma{e}{\beta_1}{h(r)}{r}}	
	\leq		\frac{1}{r}\metb[(A,F+\delF)]{h(r)}{re}{\hsma{e}{\beta_2}{h(r)}{r}}
\end{align*}
for all $(A,F)\in\Omega$ if $r>h_0$ is large enough. 
Taking the expected value and the limit $r\rightarrow\infty$ we thus obtain $\vhom[\delF,\beta_1](e)\geq \vhom[\delF,\beta_2](e)$.

Regarding the second inequality in \eqref{lin_eq_vhomBetas}, we apply Corollary \ref{lin_cor_bd}, that is we compensate for the propagation of influence from far away by increasing the forcing by an appropriate constant to obtain
\begin{equation*}
	\metb[(A,F+\delF)]{h(r)}{re}{\hsma{e}{\beta_2}{h(r)}{r}}
	\leq		\metb[(A,F+\delF-f(r))]{h(r)}{re}{\hsma{e}{\beta_1}{h(r)}{r}} + Ch(r)
\end{equation*}
with $f(r)=r^{-\frac{\beta_2}{2}+\frac{3}{4}}$.
Since $\beta_2>\frac{3}{2}$, for any $\wt{\delF}<\delF$ we have $\delF-f(r)\geq \wt{\delF}$ for all $r\geq h_0$ large enough and we hence obtain
\begin{equation*}
	\metb[(A,F+\delF)]{h(r)}{re}{\hsma{e}{\beta_2}{h(r)}{r}}
	\leq		\metb[(A,F+\wt{\delF})]{h(r)}{re}{\hsma{e}{\beta_1}{h(r)}{r}} + Ch(r).
\end{equation*}
Multiplying with $\frac{1}{r}$, taking the expected value and the limit $r\rightarrow \infty$ we thus obtain $\vhom[\delF,\beta_2](e)	\geq		\vhom[\wt{\delF},\beta_1](e)$, using that $\frac{h(r)}{r}=r^{\vtheta-1}\rightarrow 0$ as $r\rightarrow \infty$.
We have thus shown \eqref{lin_eq_vhomBetas}, concluding the proof.
\end{proof}

If Assumption \ref{sett_aP_gen} holds, then the bounds \eqref{lin_eq_vhomUpB} and \eqref{lin_eq_vhomLowB} for the arrival times of the half-space approximations based on Corollary \ref{lin_cor_exUpB} and Corollary \ref{lin_cor_LowB}  improve as follows, where we also use Lemma \ref{lin_lem_hoelInv}.

\begin{corollary}[Bounds for the truncated arrival times of half-space approximations]\label{lin_cor_ErrBou}
Assume that the random coefficient field $(A,F)\in\Omega$ with probability distribution $\Pm$ satisfies Assumptions \ref{sett_aP_stat}, \ref{sett_aP_fin}, \ref{sett_aP_Aconst}, and \ref{sett_aP_gen}.
Assume that there is $\delFdel\in(0,1]$ such that all coefficient fields $(A,F+\delF)$ with $\delF\in\bp{-\delFdel,\delFdel}$ satisfy Assumption \ref{veff_aP_star}.
Let $h(r)=r^\vtheta$ for some $0<\vtheta<1$ and $e\mapsto\vhom(e)$ be the homogenized speed as defined in Lemma \ref{lin_lem_hoelInv}.
Then there exist $C=C(\data)>0$ and $r=r_0\pp{\data,\delFdel}$ such that for the truncated arrival times of the half-space approximations $\hsma{e}{\beta}{h(r)}{r}$ with $e\in\Sd$, width scaling $\beta>\frac{3}{2}$ and $r\geq r_0$ there holds
\begin{align*}
	\EV{\metb[(A,F+\delF)]{h(r)}{re}{\hsma{e}{\beta}{h(r)}{r}}}
	&\geq		\vhom^{-1}(e)r - C\log(r)\sqrt{r^{1+\vtheta}} - Cr\abs{\delF}^\hoel
\intertext{for any $\delF\in\bp{-\delFdel,\delFdel}$ with the H\"older coefficient $\hoel$ from Assumption \ref{sett_aP_gen} and further}
	\EV{\metb[(A,F)]{h(r)}{re}{\hsma{e}{\beta}{h(r)}{r}}}
	&\leq		\vhom^{-1}(e)r + C\log(r)\sqrt{r^{1+\vtheta}} + Cr^{1-\pp{\frac{\beta}{2}-\frac{3}{4}}\hoel}.
\end{align*}
\end{corollary}

\section{Final homogenization argument and the proof of Theorem \ref{main_thm}}\label{s_hom}
In this section we will prove the overall homogenization result for arrival times, as stated in Theorem \ref{main_thm}.
In the first subsection, we will list the used assumptions and recall the definition of the homogenized speed.
In the second subsection, we will show that in order to prove the result we only need to look at two types of sets: Balls and round holes.
In the third subsection, we will prove bounds for the homogenization error of the arrival times given these two sets and comment on the necessity of the isotropy assumption.

\subsection{Assumptions and the homogenized speed}
We assume that the random coefficient field $(A,F)$ with the probability distribution $\Pm$ satisfies Assumptions \ref{sett_aP_stat}, \ref{sett_aP_fin}, \ref{sett_aP_Aconst}, \ref{sett_aP_ani} and \ref{sett_aP_gen}.
We require that there is $\delFdel\in(0,1]$ such that all coefficient fields $(A,F+\delF)$ with $\delF\in\bp{-\delFdel,\delFdel}$ satisfy Assumption \ref{veff_aP_star}.

Due to Assumptions \ref{sett_aP_stat}, \ref{sett_aP_fin}, \ref{sett_aP_Aconst} and \ref{veff_aP_star} we know from Corollary \ref{lin_cor_exUpB} that the limits
\begin{align*}
	\vhom[\delF,\beta](e)	&\coloneqq \pp{\lim_{r\rightarrow\infty}\frac{1}{r}\EV{\metb[(A,F+\delF)]{h(r)}{re}{\hsma{e}{\beta}{h(r)}{r}}}}^{-1}
\end{align*}
exist for all $\delF\in \bp{-\delFdel,\delFdel}$, $e\in\Sd$, and width scalings $\beta>\frac{3}{2}$, where $\hsma{e}{\beta}{h(r)}{r}$ are the half-space approximations from Corollary \ref{lin_cor_bd}.
Since Assumption \ref{sett_aP_gen} holds, we know from Lemma \ref{lin_lem_hoelInv} that for $\delF=0$ these limits are independent of $\beta$.
Due to the isotropy of the probability distribution from Assumption \ref{sett_aP_ani} -- which we use here for the first time -- all these limits are the same for each direction $e\in\Sd$. 
For the homogenized speed we hence set
\begin{align}\label{hom_eq_vhom}
	\vhom\coloneqq\vhom[0,\beta](e)	\qquad\text{for arbitrary $e\in\Sd$, $\beta>\frac{3}{2}$}.
\end{align}
In order to relate truncated arrival times of arbitrary sets to the homogenized speed, we will align half-space approximations with the boundary of these sets and use the quantitative error bounds from Corollary \ref{lin_cor_ErrBou}. 

\subsection{Reduction from arbitrary sets to a self-similar grower and shrinker}\label{subs_hom_reduc}
In order to prove Theorem \ref{main_thm}, we want to compare the arrival time for a fat set $S\subset\Rd$ at some target $x_0\in\Rd$ with respect to the homogenized first-order evolution from Definition \ref{intro_def_homR} to the heterogeneous second-order evolution on large scales. 
To be more precise, for $0<\eps<<1$ we normalize the arrival time of $\eps^{-1}S$ in the target area $\eps^{-1}x_0+\ov{B}_{h(R_\eps)}$, where 
\begin{align*}
	R_\eps		&\coloneqq	\eps^{-1}\dist\pp{x_0,S},\\
	h(r)			&\coloneqq 	r^{\vtheta}				\qquad\qquad\qquad\text{for some fixed $0<\vtheta<1$},
\end{align*}
to ultimately compare
\begin{align*}
	&\met[{hom,\vhom}]{}{x_0}{S}=\vhom^{-1}\dist\pp{x_0,S}
	&&\text{to }
	&\eps\metb[(A,F)]{h(R_\eps)}{\eps^{-1}x_0}{\eps^{-1}S}.
\end{align*}
However, instead of dealing with arbitrary sets $S$, we want to reduce this comparison to two special sets.

Since $S$ is fat, there exists $\eps_0=\eps_0(S)>0$ such that for all $0<\eps\leq\eps_0$ the set $\eps^{-1}S$ is $h(R_\eps)$-fat.
In particular, there exists $y_\eps\in\Rd$ such that 
\begin{align*}
	\ov{B}_{h(R_\eps)}(y_\eps)	&\subset	\eps^{-1} S
	&	&\text{ and }
	&	\dist\pp{\eps^{-1}x_0,\ov{B}_{h(R_\eps)}(y_\eps)}	&=	R_\eps.
\end{align*}
Thus, as shown in Figure \ref{hom_fig_sandwich} we can sandwich $\eps^{-1}S$ between
\begin{equation}\label{hom_eq_sandw}
	\ov{B}_{h(R_\eps)}(y_\eps)	\subset \eps^{-1}S	\subset	\Rd\setminus B_{R_\eps}(\eps^{-1}x_0).
\end{equation}
and therefore obtain
\begin{multline}\label{hom_eq_sandwAT}
	\eps\metb{h(R_\eps)}{\eps^{-1}x_0}{\ov{B}_{h(R_\eps)}(y_\eps)}
	\geq		\eps\metb{h(R_\eps)}{\eps^{-1}x_0}{\eps^{-1}S}\\
	\geq		\eps\metb{h(R_\eps)}{\eps^{-1}x_0}{\Rd\setminus B_{R_\eps}(\eps^{-1}x_0)}.
\end{multline}
Recall from Definition \ref{intro_def_homR} and Definition \ref{intro_def_met} that for $\wt{S}\subset\Rd$ and $x\in\Rd$ the homogenized arrival times are given by 
\begin{equation*}
	\met[{hom,\vhom[\delF]}]{}{x}{\wt{S}}	\coloneqq	\min\Bp{t\geq 0\,:\,x\in\reach[{hom,\vhom[\delF]}]{t}{\wt{S}}},
\end{equation*}
where $t\mapsto\reach[{hom,\vhom[\delF]}]{t}{\wt{S}}$ is the evolution of the set $\wt{S}$ according to \eqref{intro_eq_HomM_ls} with homogeneous coefficient field $\vhom[\delF]$.
The key consequence of an isotropic homogenized speed resulting from Assumption \ref{sett_aP_ani} is that
\begin{align*}
	\reach[{hom,\vhom[\delF]}]{t}{\wt{S}}	=	\wt{S}+t\vhom[\delF]\ov{B}_1	
	\qquad\text{for all }t>0
\end{align*}
and thus in particular from \eqref{hom_eq_sandw} we obtain
\begin{multline}\label{hom_eq_sandwHOM}
	\eps\met[{hom,\vhom}]{}{\eps^{-1}x_0}{\ov{B}_{h(R_\eps)}(y_\eps)}
	=	\met[{hom,\vhom}]{}{x_0}{S}
	=	\vhom^{-1}\dist\pp{x_0,S}\\
	=	\eps\met[{hom,\vhom}]{}{\eps^{-1}x_0}{\Rd\setminus B_{R_\eps}(\eps^{-1}x_0)}.
\end{multline}
To obtain a homogenization result, it is therefore sufficient to only treat two shapes: A ball and a round hole.
In the view of \eqref{hom_eq_sandwAT}, we want to show that the ball expands faster with respect to the heterogeneous evolution than with respect to the homogeneous evolution, that is
\begin{equation*}
	\eps\EV{\metb{h(R_\eps)}{\eps^{-1}x_0}{\ov{B}_{h(R_\eps)}(y_\eps)}}
	\leq	 \eps\met[{hom,\vhom}]{}{\eps^{-1}x_0}{\ov{B}_{h(R_\eps)}(y_\eps)}	+	\mathrm{Err}_{hom}
\end{equation*}
for a suitable error term.
Analogously, we want to show that the round hole collapses slower with respect to the heterogeneous evolution than with respect to the homogeneous one, that is
\begin{multline*}
	\eps\EV{\metb{h(R_\eps)}{\eps^{-1}x_0}{\Rd\setminus B_{R_\eps}(\eps^{-1}x_0)}}\\
	\geq		\eps\met[{hom,\vhom}]{}{\eps^{-1}x_0}{\Rd\setminus B_{R_\eps}(\eps^{-1}x_0)}	-	\mathrm{Err}_{hom}.
\end{multline*}
If we are able to show both, then for the original set \eqref{hom_eq_sandwAT} and \eqref{hom_eq_sandwHOM} yield
\begin{equation*}
	\abs{\eps\EV{\metb{h\pp{\eps^{-1}\dist(x_0,S)}}{\eps^{-1}x_0}{\eps^{-1}S}}	-	\met[{hom,\vhom}]{}{x_0}{S}}
	\leq		\mathrm{Err}_{hom}
\end{equation*}
and thus the kind of homogenization result we aim for.
In the following we will provide such a comparison for the arrival times of the ball (Proposition \ref{hom_prop_ball}) and for the arrival times of the round hole (Proposition \ref{hom_prop_hole}).

\begin{figure}[h]
\centering
\includegraphics[width=0.3\textwidth]{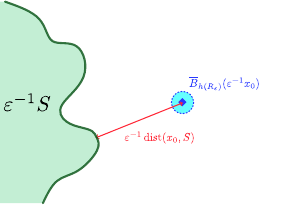}
$\,\,$
\includegraphics[width=0.3\textwidth]{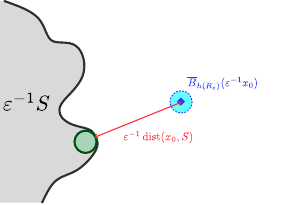}
$\,\,$
\includegraphics[width=0.3\textwidth]{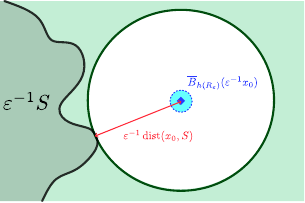}
\caption{	An illustration of the comparison argument used to reduce the homogenization problem to expanding balls and shrinking holes.
			Left: The original set $\eps^{-1}S\subset\Rd$ is depicted in green. 
			The target area  $\ov{B}_{h(R_\eps)}(\eps^{-1}x_0)$ is depicted in blue.
			Center: The ball $\ov{B}_{h(R_\eps)}(y_\eps)$ (depicted in green) is placed at the point in $\eps^{-1}S$ closest to the target area.
			Right: The set $\Rd\setminus B_{R_\eps}(\eps^{-1}x_0)$ (depicted in green) separates $\eps^{-1}S$ and the target area.}
\label{hom_fig_sandwich}
\end{figure}

\subsection{Bounding the homogenization error for the self-similar sets}\label{subs_hom_err}

\begin{proposition}[Upper bound for arrival times]\label{hom_prop_ball}
Assume that the random coefficient field $(A,F)\in\Omega$ with probability distribution $\Pm$ satisfies Assumptions \ref{sett_aP_stat}, \ref{sett_aP_fin}, \ref{sett_aP_Aconst}, \ref{sett_aP_ani}, and \ref{sett_aP_gen}.
Assume that there is $\delFdel\in(0,1]$ such that all coefficient fields $(A,F+\delF)$ with $\delF\in\bp{-\delFdel,0}$ satisfy Assumption \ref{veff_aP_star}.
Let $\vhom>0$ be the homogenized speed defined in \eqref{hom_eq_vhom} and $h(r)=r^\vtheta $ for $0<\vtheta<1$.

Then there exists a constant $C=C(\delFdel,\vtheta,\data)>0$ such that for any $e\in\Sd$, $R>0$ and $\eps>0$ with $R_\eps\coloneqq \eps^{-1}R$ there holds
\begin{multline}\label{hom_eq_ball}
	\eps\EV{\metb[(A,F)]{h(R_\eps)}{\pp{R_\eps+ h(R_\eps)}e}{\ov{B}_{h(R_\eps)}}}	- \vhom^{-1}R\\
	\leq	 	C R^{1-\alpha}\log(\max\Bp{R,2})^{\frac{3-\vtheta}{2}}\eps^{\alpha}\log(\eps^{-1})^{\frac{3-\vtheta}{2}}
\end{multline}
with $\alpha\coloneqq	\pp{1+4\pp{\frac{1}{1-\vtheta}+\frac{1}{\hoel}}}^{-1}$, where $\hoel$ is the H\"older coefficient from Assumption \ref{sett_aP_gen}.
\end{proposition}

\begin{proof}
The idea is to show that we can expect the evolution of the ball according to the heterogeneous coefficient field to expand not much slower than the first order evolution with respect to the homogenized speed, that is
\begin{align*}
	\reach{t+\mathrm{Err}_t}{\ov{B}_{h(R_\eps)}}
	&\hsups[h(R_\eps)]	\ov{B}_{h(R_\eps)}+t\vhom \ov{B}_1	
	=		\reach[{hom,\vhom}]{t}{\ov{B}_{h(R_\eps)}}
\end{align*}
for all $0\leq t\leq \vhom^{-1}R$ with a suitably controlled random error term $t\mapsto \mathrm{Err}_t$.

More precisely, we will iteratively choose a suitable sequence of ball enlargements $0<r_1,\ldots, r_N$  and show that the times until one ball fully envelops the next largest one correspond to expansion not much slower than at speed $\vhom$.
That is, with the enlargements adding up to the radii 
	$R_n\coloneqq h(R_\eps)+\sum_{k=1}^n r_k$
with $R_N=R_\eps+h(R_\eps)$, for $n\in\Bp{0,\ldots,N}$ we define the times
\begin{align}\label{hom_eq_ballTk}
	t_n	=			t_n^{(A,F)}
		\coloneqq	\sup\Bp{\metb[(A,F)]{h(r_{n})}{y}{ \ov{B}_{R_{n-1}}	}	\,:\, y\in\partial\ov{B}_{R_n}}
		\qquad\text{for $n\in\Bp{1,\ldots,N}$},
\end{align}
see Figure \ref{hom_fig_ball_tn}.
\begin{figure}[h]
\centering
\includegraphics[width=0.35\textwidth]{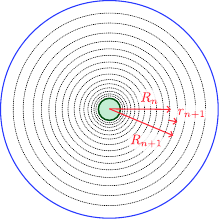}
$\quad$
\includegraphics[width=0.35\textwidth]{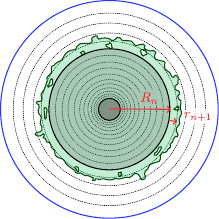}
\caption{	An illustration of the argument for bounding the expansion speed of balls from below.
			Left: The initial ball with radius $R_0=h(R_\eps)$ is depicted in green.
			The boundary of the ball with radius $R_N$ corresponding to the target distance is depicted in blue.
			In-between, dotted lines denote the boundaries of intermediate balls with radii $(R_n)_n$ corresponding to the consecutive enlargements $(r_n)_n$.
			Right: For $n\in\Bp{0,\ldots,N-1}$, after time $t_{n+1}$ the evolution $\reach{t_{n+1}}{\ov{B}_{R_n}}$ (depicted in green) of the ball $\ov{B}(R_{n})$ (depicted in light grey) fully envelops the boundary of the next larger ball $\ov{B}_{R_{n+1}}$.}
\label{hom_fig_ball_tn}
\end{figure}
We will show that for a suitable choice of the enlargements $(r_n)_n$ there are relatively small errors $(\tau_n)_n$ such that
\begin{align}\label{hom_eq_ballTkbound}
	\EV{t_n}\leq		\vhom^{-1}r_n	+	\tau_n
	\qquad\text{for all }n\in\Bp{1,\ldots,N}.
\end{align}
We claim that the overall arrival time is smaller than the sum of the times for all full ball expansions, that is that
\begin{align}\label{hom_eq_ballTkstacked}
	\metb[(A,F)]{h(R_\eps)}{\pp{R_\eps+ h(R_\eps)}e}{\ov{B}_{h(R_\eps)}}
	\leq		\sum_{n=1}^N \pp{t_n^{(A,F)} + C h(r_n)}
\end{align}
for $(A,F)\in\bigcap_{n=1}^N\Espeed{h(r_n)}(\ov{B}_{R_n})$.
Taking the expected value of \eqref{hom_eq_ballTkstacked} and applying \eqref{hom_eq_ballTkbound} will yield \eqref{hom_eq_ball} with the right hand side given by
\begin{equation}\label{hom_eq_ballTkstacked_Error}
	\eps\sum_{n=1}^N \pp{\tau_n + C h(r_n)} + CR\sum_{n=1}^N\pp{1-\PM{\Espeed{h(r_n)}(\ov{B}_{R_n})}}.
\end{equation}
Our goal is to choose the enlargements $(r_n)_n$ such that this overall error is as small as possible.
In the first part of the proof, we will establish the claim in \eqref{hom_eq_ball_reachTk}.
In the second part of the proof, we will estimate $\EV{t_n}$ for $n\in\N$ and the different contributions to the error $\tau_n$ for a given enlargement choice $r_n$ in order to deduce a recursive formula for $(r_n)_n$ which optimizes the scaling of $(\tau_n)_n$ and the overall error.
In the third part of the proof, we will analyze the resulting sequence and calculate the overall error.

\underline{Step 1: Establishing \eqref{hom_eq_ballTkstacked}.}
The claim \eqref{hom_eq_ballTkstacked} follows directly from 
\begin{equation}\label{hom_eq_ball_reachTk}
	\partial B_{R_n}	\hsubs[h(r_n)]		\reach[(A,F)]{\sum_{k=1}^n \pp{t_k^{(A,F)} + C h(r_k)}}{\ov{B}_{h(R_\eps)}},
	\qquad\text{for all $n=1,\ldots,N$}
\end{equation}
for $(A,F)\in\bigcap_{n=1}^N\Espeed{h(r_n)}(\ov{B}_{R_n})$, which we obtain via induction.

For $n=0$ with $R_0=h(R_\eps)$ we have $\partial B_{R_0}\subset \ov{B}_{h(R_\eps)}$.
Let $n\geq 1$ with 
\begin{equation*}
	\partial B_{R_{n-1}}	\hsubs[h(r_{n-1})]		\reach[(A,F)]{T_{n-1}}{\ov{B}_{h(R_\eps)}}
	\qquad\text{for } 
	T_{n-1}\coloneqq \sum_{k=1}^{n-1} \pp{t_k^{(A,F)} + C h(r_k)}.
\end{equation*}
Since $(A,F)\in\Espeed{h(r_n)}(\ov{B}_{R_n})$, due to Assumption \ref{veff_aP_star} there is a stable $(h(r_n),\cstable)$-approximation $S_h$ of $\ov{B}_{h(R_\eps)}$, for which the comparison principle from Lemma \ref{not_lem_comp} yields
\begin{equation*}
	\partial B_{R_{n-1}}	\hsubs[h(r_{n-1})]		\reach[(A,F)]{T_{n-1}}{\ov{B}_{h(R_\eps)}}
						\hsubs[h(r_n)]			\reach[(A,F)]{T_{n-1}+Ch(r_n)}{S_h}.
\end{equation*}
Since by Assumption \ref{veff_aP_star} $(A,F)$ admits an effective minimum speed on the scale $h(r_n)$ for $S_h$, an additional time step of size $C h(r_n)$ yields $\partial B_{R_{n-1}}\hsubs[h(r_n)]\reach[(A,F)]{T_{n-1}+Ch(r_n)}{S_h}$.
Using the stability of $\reach[(A,F)]{T_{n-1}+Ch(r_n)}{S_h}$ inherited from $S_h$ via Lemma \ref{stable_lem_preserv}, we can fill any of its holes which intersect $\ov{B}_{R_{n-1}}$ as in the proof of the comparison principle from Lemma \ref{not_lem_comp}\ref{not_list_lemcomp2} without changing any arrival times outside of $\ov{B}_{R_{n-1}}$.
Due to the comparison principle, the definition of $t_n$ in \eqref{hom_eq_ballTk} and the stability we thus obtain 
\begin{equation*}
	x+\ov{B}_{h(r_n)}\cap \reach[(A,F)]{T_{n-1}+Ch(r_n)+t_n^{(A,F)}}{S_h} \neq \emptyset
	\qquad\text{for all }x\in \partial B_{R_{n}}.
\end{equation*}
With the effective minimum speed, after another time step of size $C h(r_n)$ we hence obtain 
\begin{equation*}
	\partial B_{R_{n}}\hsubs[h(r_n)]\reach[(A,F)]{T_{n-1}+t_n^{(A,F)}+Ch(r_n)}{S_h}\subset \reach[(A,F)]{T_{n-1}+t_n^{(A,F)}+Ch(r_n)}{\ov{B}_{h(R_\eps)}},
\end{equation*}
where we used that $S_h\subset \ov{B}_{h(R_\eps)}$ and the comparison principle.
This yields \eqref{hom_eq_ball_reachTk} by induction and thus establishes \eqref{hom_eq_ballTkstacked}.

\underline{Step 2: Estimate on $\EV{t_n}$.}
Assume that $r_1,\ldots,r_n$ and thus $R_n=h(R_\eps)+\sum_{k=1}^n r_k$ are already chosen.
In order to relate the arrival times in the definition of $t_{n+1}$ from \eqref{hom_eq_ballTk} to the homogenized speed $\vhom$, we need to find proper placements for half-space approximations $\hsma{e}{\beta}{h(r_{n+1})}{r_{n+1}}$, which were used to define the homogenized speed.
For each $y=\pp{R_n+r_{n+1}}e$ with $e\in\Sd$, we place a half-space approximation $\hsma{e}{\beta}{h(r_{n+1})}{r_{n+1}}$ inside $ \ov{B}_{R_n}$, but as close to $y$ as possible, see Figure \ref{hom_fig_ball-gap}.
Due to the curvature of the ball, as noted in Figure \ref{hom_fig_ball-gap}, there is a gap of size $\rho_{n+1}>0$ between the boundary of the ball and the center of the half-space approximation given by
\begin{equation}\label{hom_eq_rho}
 	\pp{R_n-\rho_{n+1}}^2+\pp{C_\beta r_{n+1}^\beta}^2
 	=	R_n^2,
 	\quad\text{implying}\quad
 	\rho_{n+1}\leq C\frac{r_{n+1}^{2\beta}}{R_n},
 \end{equation} 
where $C_\beta r_{n+1}^\beta$ is the width of $\hsma{e}{\beta}{h(r_{n+1})}{r_{n+1}}$ as defined in Corollary \ref{lin_cor_bd}.
\begin{figure}[h]
\centering
\includegraphics[width=0.7\textwidth]{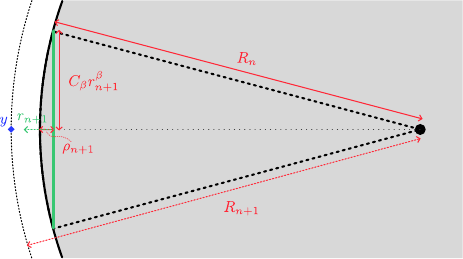}
\caption{	A sketch of the argument used to estimate the gap between the boundary of balls and the aligned half-space approximations in the proof of the lower bounds for the expansion speed of balls.	
			The ball $\ov{B}_{R_n}$ is depicted in grey with a target point $y\in \partial \ov{B}_{R_{n+1}}$ depicted in blue.			
			A half-space approximation (depicted in green) is aligned inside $\ov{B}_{R_n}$, which results in a gap of size $\rho_{n+1}>0$ between the center of the half-space approximation and the boundary of the ball.
}
\label{hom_fig_ball-gap}
\end{figure}
We hence bound the arrival time of the ball $\ov{B}_{R_n}$ at the point $y=\pp{R_n+r_{n+1}}e$ -- as in the definition of $t_{n+1}$ from \eqref{hom_eq_ballTk} -- by the arrival time from the half-space approximation $\hsma{e}{\beta}{h(r_{n+1})}{r_{n+1}}+\pp{R_n-\rho_{n+1}}e$ at the point $y-\rho_{n+1}e$ plus an error term corresponding to bridging the remaining gap at minimum effective speed, that is
\begin{align*}
	&\metb[(A,F)]{h(r_{n+1})}{y}{ \ov{B}_{R_n}	}\\
	&\quad\leq		\metb[(A,F)]{h(r_{n+1})}{r_{n+1}e+\pp{R_n-\rho_{n+1}}e}{\hsma{e}{\beta}{h(r_{n+1})}{r_{n+1}}+\pp{R_n-\rho_{n+1}}e} \\
	&\qquad			+	C\pp{\rho_{n+1}+h(r_{n+1})}
\end{align*}
for $(A,F)\in\Espeed{h(r_{n+1})}\pp{\ov{B}_{R_{n+1}}}$.
We have thus reduced the expansion of the ball to the expansion of infinitely many half-space approximations at the cost of the error term  $C\pp{\rho_{n+1}+h(r_{n+1})}$, obtaining
\begin{multline*}
	t_{n+1}\charfun[\Espeed{h(r_{n+1})}\pp{\ov{B}_{R_{n+1}}}]	-	 C\pp{\rho_{n+1}+h(r_{n+1})}	\\
		\leq		\sup_{e\in\Sd}\metb{h(r_{n+1})}{r_{n+1}e+\pp{R_n-\rho_{n+1}}e}{\hsma{e}{\beta}{h(r_{n+1})}{r_{n+1}}+\pp{R_n-\rho_{n+1}}e}.
\end{multline*}
Analogously to the proof of Proposition \ref{lin_prop_sub}, for $(A,F)\in\Espeed{h(r_{n+1})}\pp{\ov{B}\pp{R_{n+1}}}$ we can reduce the supremum over infinitely many arrival times to one over finitely many, control the fluctuations with a union bound and the fluctuation bounds from Proposition \ref{fluct_prop}, and finally use the stationarity and isotropy from Assumptions \ref{sett_aP_stat} and \ref{sett_aP_ani} to obtain 
\begin{multline*}
	\EV{t_{n+1}}
	\leq		\EV{\metb{h(r_{n+1})}{r_{n+1}e}{\hsma{e}{\beta}{h(r_{n+1})}{r_{n+1}}}} \\
			+ C\sqrt{h(r_{n+1})r_{n+1}}\log\pp{R_{n+1}}
			+  C\pp{\rho_{n+1}+h(r_{n+1})}
\end{multline*}
for some $e\in\Sd$ and $C=C(\vtheta,\data)$.
Now, assuming without loss of generality that $r_{n+1}\geq C=C(\data,\delFdel)$ is large enough to apply Corollary \ref{lin_cor_ErrBou} and plugging in the bound for $\rho_{n+1}$ from \eqref{hom_eq_rho} yields
\begin{align}
	\EV{t_{n+1}}
	&\leq	\vhom^{-1}r_{n+1}
			+	Cr_{n+1}^{1-\pp{\frac{\beta}{2}-\frac{3}{4}}\hoel}
			+	C\sqrt{h(r_{n+1})r_{n+1}}\log\pp{R_{n+1}}
			+ 	C\frac{r_{n+1}^{2\beta}}{R_n}\notag\\
	&\eqqcolon		\vhom^{-1}r_{n+1}	+	\tau_{n+1,force}	+	\tau_{n+1,fluct}	+	\tau_{n+1,curve}.\label{hom_eq_ballstep}
\end{align}
In order to balance the three error contributions $\tau_{n+1,force}, \tau_{n+1,fluct}, \tau_{n+1,curve}$ and to optimize the overall error, we make two choices: 
First, for the width scaling of the half-space approximations we choose $\beta=\frac{3}{2}+\frac{1-\vtheta}{\hoel}$, balancing the error $\tau_{n+1,force}$ due to the forcing change and the H\"older continuity from Assumption \ref{sett_aP_gen} with the error $\tau_{n+1,fluct}$ from the fluctuation bounds.
Second, we choose the enlargement $r_{n+1}$ to balance these errors with the error $\tau_{n+1,curve}$ due to the gap.

In summary, we have 
\begin{align}\label{hom_eq_ballTaubound}
	\EV{t_{n+1}}	\leq		\vhom^{-1}r_{n+1} + C\log(R_{n+1})\sqrt{r_{n+1}^{1+\vtheta}}
	\qquad\text{for}\quad	
	r_{n+1}^{2\pp{\frac{3}{2}+\frac{1-\vtheta}{\hoel}}-\frac{1+\vtheta}{2}}\approx	R_n,
\end{align}
establishing the scaling for our recursive choice of $r_{n+1}$ based on $r_1,\ldots,r_n$.

\underline{Step 3: Analyzing the recursive formula and calculating the overall error.}
We choose $r_{n+1}=R_{n+1}-R_n$ as in \eqref{hom_eq_ballTaubound}, that is 
\begin{align}
	R_{n+1}		&\coloneqq	R_n+c_{n+1}R_n^{\gamma}	=	\pp{1+c_{n+1}R_n^{\gamma-1}}R_n,\notag\\
	R_0			&\coloneqq	h(R_\eps)=R_{\eps}^\vtheta,\notag
\intertext{with}
	\gamma	&\coloneqq \pp{2\pp{\frac{3}{2}+\frac{1-\vtheta}{\hoel}}-\frac{1+\vtheta}{2}}^{-1}
			=			\frac{2}{1-\vtheta+4\pp{\frac{1-\vtheta}{\hoel}+1}}		\in(0,1),\label{hom_eq_gamma}
\end{align}
and with the  coefficients $\pp{c_{n}}_n\subset\pp{\frac{1}{2},1}$ chosen such that there is $N\in\N$ with $R_N=R_\eps$.
Since the radii are increasing, for $1\leq n \leq N$ with $R_0\leq R_n\leq R_N=R_\eps+h(R_\eps)\leq 2R_\eps$ we have
\begin{align*}
	R_\eps	\geq		R_0\prod_{n=1}^N \pp{1+\frac{1}{2}R_n^{\gamma-1}}
			\geq		\pp{1+\frac{1}{4}R_\eps^{\gamma-1}}^N R_0 		
\end{align*}
and thus
\begin{align*}
	N	\leq	\frac{\log\pp{\frac{R_\eps}{R_0}}}{\log\pp{1+\frac{1}{4}R_\eps^{\gamma-1}}}
		\leq		CR_\eps^{1-\gamma}\log\pp{R_\eps}.
\end{align*}
Based on \eqref{hom_eq_ballTkstacked}, \eqref{hom_eq_ballTkstacked_Error}, \eqref{hom_eq_ballstep} and \eqref{hom_eq_ballTaubound} due to the choice of the incremental enlargements from Step 1, we obtain
\begin{align*}
	\EV{\metb[(A,F)]{h(R_\eps)}{\pp{R_\eps+ h(R_\eps)}e}{\ov{B}_{h(R_\eps)}}}
	&\leq		\sum_{k=1}^N \EV{t_n} + C h(R_\eps)\\
	&\leq		\vhom^{-1}R_\eps	 + C h(R_\eps)	+  C\log\pp{R_\eps}\sum_{n=1}^N	\sqrt{r_{n}^{1+\vtheta}},
\end{align*}
where the term $CR\sum_{n=1}^N\pp{1-\PM{\Espeed{h(r_n)}(\ov{B}_{R_n})}}$ from \eqref{hom_eq_ballTkstacked_Error} can be bound by a constant due to the probability bound from Assumption \ref{veff_aP_star}, which allows us to absorb all polynomial prefactors since $h(r_n)\geq h(r_1)\geq \frac{1}{2}h(R_\eps)^\gamma$.
Regarding the main error term, applying Jensen's inequality we obtain
\begin{multline*}
	\sum_{n=1}^N		r_{n}^{\frac{1+\vtheta}{2}}
	=N	\frac{1}{N} \sum_{n=1}^N		r_{n}^{\frac{1+\vtheta}{2}}
	\leq		N\pp{\frac{1}{N}\sum_{n=1}^N r_n}^{\frac{1+\vtheta}{2}}
	=	N^{\frac{1-\vtheta}{2}}R_\eps^{\frac{1+\vtheta}{2}}\\
	\leq		CR_\eps^{\pp{1-\gamma}\frac{1-\vtheta}{2}+\frac{1+\vtheta}{2}}\pp{\log\pp{R_\eps}}^{\frac{1-\vtheta}{2}}.
\end{multline*}
Multiplying with $\eps$ and plugging in $R_\eps=\eps^{-1}R$ as well as the definition of $\gamma\in(0,1)$ from \eqref{hom_eq_gamma} we hence obtain \eqref{hom_eq_ball}.
\end{proof}

\begin{proposition}[Lower bound for arrival times]\label{hom_prop_hole}
Assume that the random coefficient field $(A,F)\in\Omega$ with probability distribution $\Pm$ satisfies Assumptions \ref{sett_aP_stat}, \ref{sett_aP_fin}, \ref{sett_aP_Aconst}, \ref{sett_aP_ani}, and \ref{sett_aP_gen}.
Assume that there is $\delFdel\in(0,1]$ such that all coefficient fields $(A,F+\delF)$ with $\delF\in\bp{0,\delFdel}$ satisfy Assumption \ref{veff_aP_star}.
Let $\vhom>0$ be the homogenized speed defined in \eqref{hom_eq_vhom} and $h(r)=r^\vtheta $ for $0<\vtheta<1$.

Then there exists a constant $C=C(\vtheta,\data)>0$ such that for any $R>0$ and $\eps>0$ with $R_\eps\coloneqq \eps^{-1}R$ there holds
\begin{multline}\label{hom_eq_hole}
	\eps\EV{\metb{h(R_\eps)}{0}{\Rd\setminus B_{R_\eps}}}	- \vhom^{-1}R\\
	\geq	 	-	C R^{1-\alpha}\log(\max\Bp{R,2})^{\frac{3-\vtheta}{2}}\eps^{\alpha}\log(\eps^{-1})^{\frac{3-\vtheta}{2}}
\end{multline}
with $\alpha\coloneqq	\pp{1+4\pp{\frac{1}{1-\vtheta}+\frac{1}{\hoel}}}^{-1}$, where $\hoel$ is the H\"older coefficient from Assumption \ref{sett_aP_gen}.
\end{proposition}

\begin{proof}
Up to some smaller details, this proof is analogous to the proof of Proposition \ref{hom_prop_ball}, only that the geometric setup is mirrored.
The idea here is to show that we can expect the round hole to shrink not much faster when evolving according to the heterogeneous coefficient field compared to the homogenized speed, that is
\begin{align*}
	\reach{t-\mathrm{Err}_t}{\Rd\setminus B_{R_\eps}}
	&\hsubs[h(R_\eps)]	\Rd\setminus\pp{R_\eps-t\vhom}B_{1}
	=		\reach[{hom,\vhom}]{t}{\Rd\setminus B_{R_\eps}}
\end{align*}
for all $0\leq t\leq \vhom^{-1}R$ with a suitably controlled random error term $t\mapsto \mathrm{Err}_t$.

To be more precise, we again iteratively choose a suitable sequence of enlargements $0<r_1,\ldots,r_N$ resulting in shrinking radii $R_n\coloneqq R_\eps-\sum_{k=1}r_k$ for $n\in\Bp{1,\ldots,N}$.
We will show that the time until the evolution $\Rd\setminus B_{R_{n-1}}$ first reaches the boundary of the next smaller hole $\Rd\setminus B_{R_{n}}$ is not much faster than $\vhom^{-1}r_{n}$.
That is, for $n\in\Bp{1,\ldots,N}$ we define the times
\begin{align}\label{hom_eq_holeTk}
	t_n	=			t_n^{(A,F)}
		\coloneqq	\inf\Bp{\metb[(A,F)]{h(r_n)}{y}{\Rd\setminus B_{R_{n-1}}}	\,:\,	y\in\partial  B_{R_n}},
\end{align}
see Figure \ref{hom_fig_hole_tn}.
\begin{figure}[h]
\centering
\includegraphics[width=0.35\textwidth]{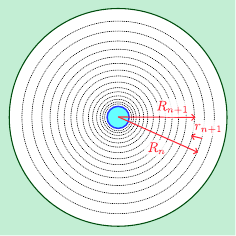}
$\quad$
\includegraphics[width=0.35\textwidth]{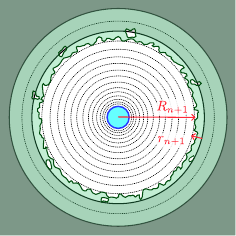}
\caption{	An illustration of the argument for bounding the shrinking speed of holes from above.
			Left: The exterior of the initial hole with radius $R_0=R_\eps$ is depicted in green.
			The target area $\ov{B}_{R_N}$ with radius $R_N=h(R_\eps)$ is depicted in blue.
			In-between, dotted lines denote the boundaries of intermediate holes with radii $(R_n)_n$ corresponding to the consecutive enlargements $(r_n)_n$.
			Right: For $n\in\Bp{0,\ldots,N-1}$, after time $t_{n+1}$ the evolution $\reach{t_{n+1}}{\Rd\setminus B_{R_n}}$ (depicted in green) of $\Rd\setminus B_{R_n}$ (depicted in light grey) first reaches $\partial B_{R_{n+1}}$.
		}
\label{hom_fig_hole_tn}
\end{figure}
Clearly, following the actual evolution to the center of the hole is slower than fully filling up the complement every time the boundary of the next smaller hole is reached and then continuing the evolution.
That is, with $R_N=h(R_\eps)$ we have 
\begin{align}\label{hom_eq_holeTimeTk}
	\metb[(A,F)]{h(R_\eps)}{0}{\Rd\setminus B_{R_\eps}}
	\geq		\sum_{n=1}^N t_n^{(A,F)}
\end{align}
for $(A,F)\in\Espeed{h(R_\eps)}\pp{\ov{B}_{R_\eps}}$, because using the definition of the $(t_n)_n$ and iteratively applying the comparison principle yields  $\reach{\sum_{k=1}^n t_n}{\Rd\setminus B_{R_\eps}}\subset	\Rd\setminus B_{R_{n}}$ for all $n$.
Mirroring the proof of Proposition \ref{hom_prop_ball}, we will choose the enlargements $(r_n)_n$ such that there are relatively small errors $(\tau_n)_n$ satisfying
\begin{align}\label{hom_eq_holeTkbound}
	\EV{t_n}\geq		\vhom^{-1}r_n	-	\tau_n
	\qquad\text{for all }n\in\Bp{1,\ldots,N},
\end{align}
which will yield \eqref{hom_eq_hole} with the error term given by $\eps\sum_{k=1}^N\tau_n$.

For the first part of the proof, we will estimate $\EV{t_n}$ for $n\in\N$ and the different contributions to the error $\tau_n$ for a given enlargement choice $r_n$ in order to deduce a recursive formula for $(r_n)_n$ which optimizes the scaling of $(\tau_n)_n$ and the overall error.
Since the resulting scaling will be essentially the same as in the proof of Proposition \ref{hom_prop_ball}, in the second part of the proof we will then briefly summarize the analysis of the resulting sequence and the calculation of the overall error. 

\underline{Step 1: Estimate on $\EV{t_n}$.}
Assume that $r_1,\ldots,r_n$ and thus $R_n=R_\eps-\sum_{k=1}^n r_k$ are already chosen.
Again, in order to relate the arrival times in the definition of $t_{n+1}$ from \eqref{hom_eq_holeTk} to the homogenized speed $\vhom$, we need to properly place half-space approximations $\hsma{e}{\beta}{h(r)}{r}$, ideally tangentially to the hole. 
However, since we now want to bound the arrival times from below, we need to place the disk approximating the half-space `ahead' of the hole's curve, such that the boundary of the disk is aligned with the boundary of the hole, see Figure \ref{hom_fig_hole-gap}.
\begin{figure}[h]
	\centering
	\includegraphics[width=0.7\textwidth]{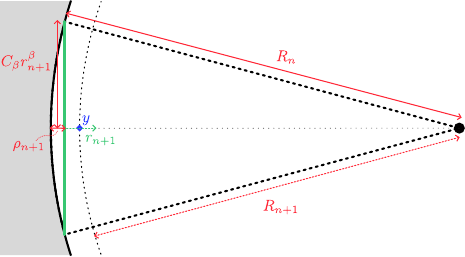}
	\caption{	A sketch of the argument used to estimate the gap between the boundary of holes and the aligned half-space approximations in the proof of the upper bounds for the shrinking speed of holes.	
				The set  $\Rd\setminus B_{R_n}$  is depicted in grey with a target point $y\in \partial B_{R_{n+1}}$ depicted in blue.			
				A half-space approximation (depicted in green) is aligned ahead of $\Rd\setminus B_{R_n}$, which results in a gap of size $\rho_{n+1}>0$ between the center of the half-space approximation and the boundary of the hole.
}
	\label{hom_fig_hole-gap}
\end{figure}
The gap $\rho_{n+1}>0$ between the center of the disk and the boundary of the hole is the same as in the mirrored situation from the proof of Proposition \ref{hom_prop_ball}, given by 
\begin{equation*}
 	\pp{R_n-\rho_{n+1}}^2+\pp{C_\beta r_{n+1}^\beta}^2
 	=	R_n^2,
 	\quad\text{implying}\quad
 	\rho_{n+1}\leq C(\beta,\data)\frac{r_{n+1}^{2\beta}}{R_n},
 \end{equation*} 
where $C_\beta r_{n+1}^\beta$ is the width of $\hsma{e}{\beta}{h(r_{n+1})}{r_{n+1}}$ as defined in Corollary \ref{lin_cor_bd}.

Consider the definition of $t_{n+1}$ from \eqref{hom_eq_ballTk}.
For a target point $y=-\pp{R_n-r_{n+1}}e$ on the boundary of the smaller hole, we now apply Corollary \ref{lin_cor_bd} to compare the half-space approximation $\hsma{e}{\beta}{h(r_{n+1})}{r_{n+1}}-\pp{R_n-\rho_{n+1}}e$ to the hole $\Rd\setminus B_{R_n}$.
For $f(r)=r^{-\frac{\beta}{2}+\frac{3}{4}}$ from Corollary \ref{lin_cor_bd}, and $(A,F)\in\Espeed{h(r_{n+1})}\pp{\ov{B}_{R_n}}$, using the thus guaranteed minimum effective speed we obtain
\begin{align*}
	&\metb[(A,F)]{h(r_{n+1})}{y}{ \Rd\setminus B_{R_n}}	\\
	&\,\,	\geq		\metb[(A,F)]{h(r_{n+1})}{r_{n+1}e-\pp{R_n-\rho_{n+1}}e}{ \Rd\setminus B_{R_n}}	-	C\pp{\rho_{n+1}+h(r_{n+1})}\\
	&\,\,	\geq		\metb[(A,F+f(r_{n+1}))]{h(r_{n+1})}{r_{n+1}e-\pp{R_n-\rho_{n+1}}e}{\hsma{e}{\beta}{h(r_{n+1})}{r_{n+1}}-\pp{R_n-\rho_{n+1}}e} \\
	&\,\,	\quad	-	C\pp{\rho_{n+1}+h(r_{n+1})}.
\end{align*}
We have thus reduced the shrinking of the hole to the expansion of infinitely many half-space approximations at the cost of the error term $C\pp{\rho_{n+1}+h(r_{n+1})}$ and changing the forcing by a constant, obtaining
\begin{multline*}
	t_{n+1}^{(A,F)}	+	 C\pp{\rho_{n+1}+h(r_{n+1})}	\\
	\geq		\inf_{e\in\Sd}\metb[(A,F+f(r_{n+1}))]{h(r_{n+1})}{r_{n+1}e-\pp{R_n-\rho_{n+1}}e}{\hsma{e}{\beta}{h(r_{n+1})}{r_{n+1}}-\pp{R_n-\rho_{n+1}}e}.
\end{multline*}
for $(A,F)\in\Espeed{h(r_{n+1})}\pp{\ov{B}_{R_{n+1}}}$.
Analogously to the proof of Proposition \ref{lin_prop_sup}, we can reduce the infimum to one over finitely many arrival times, control the fluctuations with the bounds from Proposition \ref{fluct_prop} and use the stationarity and isotropy from Assumptions \ref{sett_aP_stat} and \ref{sett_aP_ani} to obtain
\begin{multline*}
	\EV{t_{n+1}}
	\geq		\EV{\metb[(A,F+f(r_{n+1}))]{h(r_{n+1})}{r_{n+1}e}{\hsma{e}{\beta}{h(r_{n+1})}{r_{n+1}}}}\\
			-C\sqrt{h(r_{n+1})r_{n+1}}\log\pp{R_n}	
			-C\pp{\rho_{n+1}+h(r_{n+1})}
\end{multline*}
for some $e\in\Sd$ and $C=C(\beta,\vtheta,\data)>0$.
The expected value of the arrival times from the half-space approximations now provides the link to the homogenized speed, with the change in the forcing resulting in an error depending on the H\"older coefficient $\hoel$ from Assumption \ref{sett_aP_gen}.
That is, we apply Corollary \ref{lin_cor_ErrBou} to obtain
\begin{align*}
	\EV{t_{n+1}}
	&\geq		\vhom^{-1}r_{n+1}	-	Cr_{n+1}f(r_{n+1})^{\hoel}	-	C\sqrt{r_{n+1}^{1+\vtheta}}\log\pp{R_n} - C\rho_{n+1}\\
	&\eqqcolon	\vhom^{-1}r_{n+1}	-	\tau_{n+1,force}	-	\tau_{n+1,fluct}	-	\tau_{n+1,curve}.
\end{align*}
Balancing these errors just as in the proof of Proposition \ref{hom_prop_ball}, we choose $\beta=\frac{3}{2}+\frac{1-\vtheta}{\hoel}$ and obtain 
\begin{equation}\label{hom_eq_holeTaubound}
	\EV{t_{n+1}}	\geq		\vhom^{-1}r_{n+1} - C\log(R_{n})\sqrt{r_{n+1}^{1+\vtheta}}
	\qquad\text{for}\quad	
	r_{n+1}^{2\pp{\frac{3}{2}+\frac{1-\vtheta}{\hoel}}-\frac{1+\vtheta}{2}}\approx	R_n,
\end{equation}
resulting in the same error bounds and sequence of radii as above but in reverse order.

\underline{Step 2: Analyzing the recursive formula and calculating the overall error.}
Again, we look at the sequence of radii for the hole instead of the single enlargements. 
Instead of increasing as in the proof of Proposition \ref{hom_prop_ball}, these radii are decreasing with $R_{n+1}=R_n-r_{n+1}$.
We set 
\begin{align*}
	R_{n+1}	&\coloneqq	R_n-c_{n+1}R_n^{\gamma}	=	\pp{1-c_{n+1}R_n^{\gamma-1}}R_n,\\
	R_0		&\coloneqq	R_\eps
\intertext{with}
	\gamma	&\coloneqq \pp{2\pp{\frac{3}{2}+\frac{1-\vtheta}{\hoel}}-\frac{1+\vtheta}{2}}^{-1}
			=			\frac{2}{1-\vtheta+4\pp{\frac{1-\vtheta}{\hoel}+1}}		\in(0,1).
\end{align*}
and with the coefficients $(c_{n})_n\subset\pp{\frac{1}{2},1}$ chosen such that there is $N\in \N$ with $R_N=h(R_\eps)$.
To bound this number of steps, note that since we have $h(R_\eps)=R_N\leq R_n\leq R_\eps$ for all $1\leq n\leq N$, there holds
\begin{align*}
	h(R_\eps)	\leq		\pp{1-\frac{1}{2} R_\eps^{\gamma-1}}^N R_\eps
\end{align*}
and thus 
\begin{equation*}
	N	\leq		\frac{\log(\frac{R_\eps}{h(R_\eps)})}{\log\pp{\frac{1}{1-\frac{1}{2} R_\eps^{\gamma-1}}}}
		\leq		C R_\eps^{1-\gamma}\log(R_\eps).
\end{equation*}
Plugging \eqref{hom_eq_holeTaubound} into \eqref{hom_eq_holeTkbound} and the resulting bound for $\EV{t_n}$ into \eqref{hom_eq_holeTimeTk}, we obtain
\begin{align*}
	\EV{\metb{h(R_\eps)}{0}{\Rd\setminus B_{R_\eps}}} 
	&\geq		\sum_{n=1}^N\EV{t_n}\\
	&\geq		\vhom^{-1}R_\eps - C\log\pp{R_\eps} \sum_{n=1}^N r_{n}^{\frac{1+\vtheta}{2}}\\
	&\geq		\vhom^{-1}R_\eps - C\log\pp{R_\eps}	N^{\frac{1-\vtheta}{2}}R_\eps^{\frac{1+\vtheta}{2}},
\end{align*}
where for the last inequality we used Jensen's inequality as at the end of the proof of Proposition \ref{hom_prop_ball}.
In the same way, plugging in the bound for $N$ and multiplying with $\eps$ we obtain \eqref{hom_eq_hole}.
\end{proof} 

In summary, combining Proposition \ref{hom_prop_ball} and Proposition \ref{hom_prop_hole} with the arguments from Subsection \ref{subs_hom_reduc}, which allowed us to replace arbitrary sets with balls and holes, proves Theorem \ref{main_thm}.

\begin{remark}[The anisotropic case]\label{hom_rem_aniso}
In principle, Assumption \ref{sett_aP_ani} is not necessary for our argument to work.
Its role is to guarantees that as in \eqref{hom_eq_sandwHOM} we can reduce the analysis of arrival times for arbitrary sets to obtaining bounds for the arrival times of just the ball and its inverted shape, a round hole. 
This kind of reduction is still valid for an anisotropic homogeneous speed $e\mapsto\vhom(e)$ if the convex shape $K\subset\Rd$ given by
\begin{equation*}
	tK= \reach[hom,\vhom]{t}{\Bp{0}}	\qquad\text{for all }t>0
\end{equation*}
has bounded curvature and also satisfies
\begin{equation*}
	\reach[hom,\vhom]{t}{\overline{\Rd\setminus \pp{-K}}}=(1-t)\overline{\Rd\setminus \pp{-K}}	\qquad\text{for all }0<t<1.
\end{equation*}
In the isotropic case, we have $K=\vhom \ov{B}_1$ and hence clearly both conditions are satisfied.
Both conditions do not hold in general and depend on the properties of $e\mapsto \vhom(e)$.
For example, the second condition is satisfied if the homogenized evolution satisfies a dynamic programming principle, that is that the evolution of a set $S\subset\Rd$ is given by the union over the separate evolution of all included points:
\begin{equation*}
	\reach[hom,v]{t}{S}	=	\bigcup_{y\in S}\reach[hom,v]{t}{\Bp{y}}
	\qquad\text{for all }t\geq 0.
\end{equation*}
A sufficient criterion for the dynamic programming principle to hold is level-set convexity of the Hamiltonian given by the $1$-homogeneous extension of the speed, that is the convexity of the shape given by $\Sd\ni e\mapsto\vhom(e)^{-1}e$, see for example \cite{ACS14}. 
\begin{figure}[h]
\centering
\includegraphics[width=0.35\textwidth]{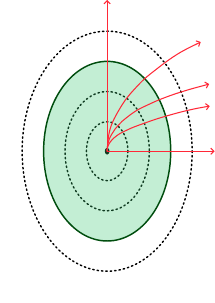}
$\quad$
\includegraphics[width=0.35\textwidth]{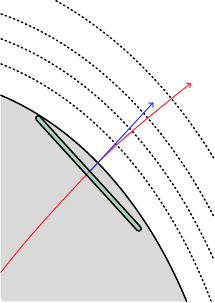}
\caption{	A depiction of the evolution of a self-similar grower with respect to a homogeneous first-order motion in the non-isotropic case.
			Left: Some normal trajectories (depicted in red) for an invariant shape $K$ (depicted in green). 
			Right: A half-space approximation (depicted in green) is tangentially aligned inside the invariant shape $K$ (depicted in grey).
			The normal of the half-space approximation (depicted in blue) is compared to the normal trajectory of the set evolution.}
\label{hom_fig_trajectories}
\end{figure}

The only major adjustment for the part of the proof presented here in Subsection \ref{subs_hom_err} is that we would have to more carefully keep track of the homogenized  motions' `trajectories':
These are the paths obtained from following the outwards normal of the invariant shape $K$ scaled up to the current point of the trajectory, see Figure \ref{hom_fig_trajectories}.
In the isotropic case these trajectories of course are straight and the velocity of the interface stays constant.
In the anisotropic case, these trajectories might be curved and hence the velocity might vary along the trajectory, but the curvature is controlled by the curvature of the invariant shape $K$.
In the proof, for a suitably chosen time step we need to show that the evolution of a tangentially fitted half-space approximation reaches the endpoint of the corresponding trajectory in a comparable time, see Figure \ref{hom_fig_trajectories}.
For the half-space approximation, Corollary \ref{lin_cor_ErrBou} up to an error guarantees the evolution along a straight trajectory with the average speed corresponding to the homogenized velocity for that direction.
Choosing the time steps small enough, the trajectory for the homogenized  motion will be close to that straight line.
The difference in velocities can be controlled by showing that $e\mapsto \vhom^{-1}(e)$ is H\"older continuous, adapting \cite[Lemma 4.7]{ArCa18} to our setting.
This leads to an additional error term in the first step of the proof of Proposition \ref{hom_prop_ball} and Proposition \ref{hom_prop_hole} respectively, which can be handled analogously.
\end{remark}

\begin{appendix}

\section{Results from the theory of viscosity solutions}\label{s_visc}
We consider the level-set evolution equation \eqref{intro_eq_SurfM_ls} in the framework of viscosity solutions.
The notion of viscosity solutions was developed in the 1980s and provides a framework which yields well-posedness and comparison principles for a wide range of scalar fully nonlinear partial differential equations of second order. 
For a general introduction see \cite{CIL92} with a detailed history of the development in the notes at the end of sections.
Here, we mainly follow the presentation in \cite{Giga06}, which summarizes and provides results for surface evolution equations.

The results presented in this section apply to evolution equations of the form
\begin{equation}\label{comp_eq_evo}
	\pat u + G(x, \nabla u, D^2 u)=0
\end{equation}
where $G\colon \Rd\times\pp{\Rd\!\setminus\!\Bp{0}}\times \Rddsym\rightarrow \R$ satisfies
\begin{enumerate}[label=(G\arabic*),left=0pt]
\item 	(Continuity)\label{comp_aG_cont}
		$G\colon \Rd\times\pp{\Rd\!\setminus\!\Bp{0}}\times \Rddsym\rightarrow \R$ is continuous.
\item	(Degenerate ellipticity)\label{comp_aG_ell}
		For any $x\in\Rd$, $\xi\in\Rd\setminus\Bp{0}$ there holds
		\begin{equation*}
			G(x,\xi,X)\leq G(x,\xi,Y)
			\quad\text{ for }X\geq Y, \, X,Y\in\Rddsym.
		\end{equation*}
\item	(Mild singularity)\label{comp_aG_fin} 
		Given the semicontinuous envelopes $G^*$ and $G_*$ as defined below in \eqref{comp_eq_envelopes}, for any $x\in\Rd$ there holds
		$-\infty<G_*(x,0,0)=G^*(x,0,0)<\infty$.
\item	(Geometricity) \label{comp_aG_geo} $G$ satisfies for all $\lambda>0$ and $\sigma\in\R$ that
		\begin{align*}
			G(x,\lambda \xi, \lambda X)&= \lambda G(x,\xi,X)
			\quad\text{for all }x\in\Rd,\, \xi\in\Rd\setminus\Bp{0},\, X\in\Rddsym,\\
			G(x,\xi, X+\sigma \xi\otimes\xi)&=G(x,\xi,X)
			\quad\text{for all }x\in\Rd,\, \xi\in\Rd\setminus\Bp{0},\, X\in\Rddsym.
		\end{align*}
\end{enumerate}

In our case, recall that for $(A,F)\in\Omega$ in \eqref{intro_eq_GAF} we set
\begin{multline*}
	G_{(A,F)}\colon \Rd\times\pp{\Rd\!\setminus\!\Bp{0}}\times \Rddsym\rightarrow \R,\\
	(x,\xi,X)\mapsto - \tr\pp{A\pp{x, \frac{\xi}{\abs{\xi}}}X} + F\pp{x, \frac{\xi}{\abs{\xi}}}\abs{\xi},
\end{multline*}
which due to the definition of $\Omega$ satisfies \ref{comp_aG_cont}, \ref{comp_aG_ell}, and \ref{comp_aG_geo}.

Regarding \ref{comp_aG_fin}, given a metric space $M$ with $U\subset M$, for $g\colon U\rightarrow \R\cup\Bp{\pm\infty}$ the upper and lower semicontinuous envelopes are given by
\begin{equation}\label{comp_eq_envelopes}
\begin{aligned}
	g^*\colon \ov{U}\rightarrow\R\cup\Bp{\pm\infty},	\quad&	g^*(x)\coloneqq \lim_{r\searrow 0} \sup\Bp{g(y)\,:\,y\in U\cap B_{r}(x)},\\
	g_*\colon \ov{U}\rightarrow\R\cup\Bp{\pm\infty},	\quad&	g_*(x)\coloneqq \lim_{r\searrow 0} \inf\Bp{g(y)\,:\,y\in U\cap B_{r}(x)}.
\end{aligned}
\end{equation}
For example, $G_{(A,F)}=G_{(A,F)}^*=\pp{G_{(A,F)}}_*$ on $\Rd\times\pp{\Rd\!\setminus\!\Bp{0}}\times \Rddsym$ and
\begin{multline}\label{comp_eq_Gstar}
	C_{1A}\min\Bp{0,\lambda_1(X)} 	\leq \pp{G_{(A,F)}}_*(x,0,X) \\
	\leq  G_{(A,F)}^*(x,0,X) 		\leq C_{1A}\max\Bp{0,\lambda_d(X)}
\end{multline}
for $x\in\Rd$ and $X\in\Rddsym$, where $\lambda_1(X)\leq \lambda_d(X)$ are the smallest and largest eigenvalue of $X$ respectively.
In particular, $G_{(A,F)}$ satisfies \ref{comp_aG_fin}.

With respect to the stability of sub- and supersolutions, the natural limit procedure is taking the \textit{upper} and \textit{lower relaxed limits}.
For a family of functions $(g_\eps)_{\eps>0}$ on $U\subset M$ with values in $\R\cup\Bp{\pm\infty}$, these are respectively defined on $\ov{U}$ by 
\begin{align}
	\label{comp_eq_defUrellim}
	\pp{\limsup_{\eps\rightarrow 0}\!{}^*g_\eps}(x)	&\coloneqq	\lim_{\eps\rightarrow 0}\sup\Bp{g_\delta(y) \,:\, y\in U\cap B_\eps(x) \text{ and } 0<\delta<\eps},\\
	\label{comp_eq_defLrellim}
	\pp{\liminf_{\eps\rightarrow 0}\!{}_*g_\eps}(x)	&\coloneqq	\lim_{\eps\rightarrow 0}\inf\Bp{g_\delta(y) \,:\, y\in U\cap B_\eps(x) \text{ and } 0<\delta<\eps}.
\end{align}

In order to obtain values for the derivatives necessary to evaluate \eqref{comp_eq_evo} given just semicontinuous functions, semijets are introduced. 
These semijets effectively approximate the behavior of the function around a point with smooth functions from above and from below.

\begin{definition}[Semijets]\label{comp_def_jets}
For an upper semicontinuous function $u\colon U\rightarrow\R\cup\Bp{-\infty}$ with $U\subset\Rd$ locally compact,
an element $(\xi,X)\in\Rd\times\Rddsym$ is called a \textit{superjet} of $u$ at some $x_0\in U$ if $u(x_0)$ is finite and 
\begin{equation*}
	u(x)-u(x_0)
	\leq		\xi\cdot(x-x_0) + \frac{1}{2}(x-x_0)^\intercal X(x-x_0) + \sorder{\abs{x-x_0}^2}
\end{equation*}
for $x\in U$ and $x\rightarrow x_0$.
As usual, $\sorder{s}$ denotes a function such that $\sorder{s}/s\rightarrow0$ as $h\rightarrow 0$.
For the set of all superjets of $u$ at $x_0\in U$ we write $\supJ[U]u(x_0)$.

For a lower semicontinuous function $u\colon U\rightarrow\R\cup\Bp{+\infty}$ an element $(\xi,X)\in\Rd\times\Rddsym$ is a \textit{subjet} of $u$ at $x_0\in U$ if $(-\xi,-X)\in -\supJ[U](-u)(x_0)$.
For the set of all subjets of $u$ at $x_0\in U$ we write $\subJ[U]u(x_0)$.
\end{definition}

For the proof of Lemma \ref{inf_lem_stableExt} we use the following property of semijets with respect to characteristic functions.

\begin{lemma}\label{comp_lem_setsum}
Let $S, B\subset \Rd$ be closed sets.
Let $v_{S}= 1-\charfun[S]$ and $v_{S+B}= 1-\charfun[S+B]$.
Then for any $x_0\in S$, $y_0\in \Bp{x_0}+B$ there holds
\begin{equation*}
	\subJ[\Rd]v_{S+B}(y_0)\subset\subJ[\Rd]v_{S}(x_0).
\end{equation*}
\end{lemma}

\begin{proof}
We will show that $\Rd\times\Rddsym\setminus\subJ[\Rd]v_{S+B}(y_0)\supset\Rd\times\Rddsym\setminus\subJ[\Rd]v_{S}(x_0)$.
Let $(\xi,X)\in\Rd\times\Rddsym\setminus\subJ[\Rd]v_{S}(x_0)$.
Then there exists a sequence $\pp{x_n}_n\subset S$ and a constant $c>0$ such that $x_0	=\lim_{n\rightarrow\infty}x_n$ but
\begin{align*}
	\frac{0-0 + \xi\cdot(x_n-x_0) + \frac{1}{2}(x_n-x_0)^\intercal X(x_n-x_0)}{\abs{x_n-x_0}^2}\geq c
 		\quad\text{for all }n\in\N.
\end{align*} 
Note that $y_0-x_0\in B$.
Hence, we have $y_n\coloneqq x_n+\pp{y_0-x_0}\in S+B$.
The sequence  $\pp{y_n}_n$ satisfies $y_0	=\lim_{n\rightarrow\infty}y_n$ and
\begin{align*}
	\frac{v_{S+B}(y_n)-v_{S+B}(y_0) + \xi\cdot(y_n-y_0) + \frac{1}{2}(y_n-y_0)^\intercal X(y_n-y_0)}{\abs{y_n-y_0}^2}\geq c
 		\quad\text{for all }n\in\N.
\end{align*} 
This implies $(\xi,X)\notin\subJ[\Rd]v_{S+B}(y_0)$ and concludes the proof.
\end{proof}

With the semijets providing a surrogate for the derivatives, we can define semicontinuous sub- and supersolutions.

\begin{definition}
Let $U\subset\Rd$ be an open set.
\begin{enumerate}[label=(\roman*), left=0pt]
\item	$u\colon U\rightarrow \R\cup\Bp{-\infty}$ is a subsolution of $G(x,\nabla u,D^2u)=0$ in $U$ if $u^*(x)<\infty$ for $x\in U$ and
		\begin{equation*}
			G_*(x,\xi,X) \leq 0
			\quad\text{for all }x\in U \text{ and } (\xi,X)\in \supJ[U]u^*(x).
		\end{equation*}
\item	We call $u\colon U\rightarrow \R\cup\Bp{-\infty}$ a supersolution of $G(x,\nabla u,D^2u)=0$ in $U$ if $u_*(x)>-\infty$ for $x\in U$ and
		\begin{equation*}
			G^*(x,\xi,X) \geq 0
			\quad\text{for all }x\in U \text{ and } (\xi,X)\in \subJ[\Ut]u_*(x).
		\end{equation*}
\end{enumerate}
\end{definition}

For equations such as \eqref{comp_eq_evo} it is convenient to separate the spatial and time derivatives, using parabolic semijets instead of the standard semijets.

\begin{definition}[Parabolic semijets]\label{comp_def_pjets}
Fixing a locally compact set $\Ut\subset\Rd\times\R$, let $u\colon \Ut\rightarrow\R\cup\Bp{-\infty}$ be an upper semicontinuous function.
\begin{enumerate}[label=(\roman*), left=0pt]
\item	An element $(a,\xi,X)\in\R\times\Rd\times\Rddsym$ is called a \textit{parabolic superjet} of $u$ at some $(x_0,t_0)\in\Ut$ if
		\begin{multline*}
			u(x,t)-u(x_0,t_0)
			\leq 	a(t-t_0) + \xi\cdot(x-x_0) + \frac{1}{2}(x-x_0)^\intercal X(x-x_0)\\
					+\sorder{\abs{t-t_0}+\abs{x-x_0}^2}
					\quad\text{for $(x,t)\in \Ut$ as $(x,t)\rightarrow(x_0,t_0)$.}
		\end{multline*}
		The set of all parabolic superjets of $u$ at $(x_0,t_0)$ is denoted with $\supP[\Ut]u(x_0,t_0)$.
\item	The closure of the superjets at $(x_0,t_0)\in\Rd\times\R$ is restricted to sequences with a continuous approach of $u(x_0,t_0)$:
		\begin{align*}
			\supPc[\Ut]u(t_0,x_0)\coloneqq
			\Big\lbrace (a,&\xi,X)\in \R\times\Rd\times\Rddsym	\,:\,
				\text{there is } (t_j,x_j,a_j,\xi_j,X_j)_{j\in\N} \text{ with}\\
				&(x_j,t_j,u(x_j,t_j),a_j,\xi_j,X_j)\overset{j\rightarrow\infty}{\rightarrow}(x_0,t_0,u(x_0,t_0),a_0,\xi_0,X_0)\\
				&\text{and }(a_j,\xi_j,X_j)\in \supP[\Ut]u(x_j,t_j) 
			\Big\rbrace.
		\end{align*}
\item 	We call $(a,\xi,X)\in\R\times\Rd\times\Rddsym$ a \textit{relaxed parabolic superjet} of $u$ at $(x_0,t_0)\in\Ut$ if the inequality only holds in the past:
		\begin{multline*}
			u(x,t)-u(x_0,t_0)
			\leq 	a(t-t_0) + \xi\cdot(x-x_0) + \frac{1}{2}(x-x_0)^\intercal X(x-x_0)\\
					+\sorder{\abs{t-t_0}+\abs{x-x_0}^2}
					\quad\text{for $(x,t)\in \Ut,\,t\leq t_0$ as $(x,t)\rightarrow(x_0,t_0)$.}
		\end{multline*}
		The set of all relaxed parabolic superjets of $u$ at $(x_0,t_0)$ is denoted with $\supPt[\Ut]u(x_0,t_0)$.
\end{enumerate}
The sets of parabolic subjets $\subP[\Ut]u(x_0,t_0),\,\subPc[\Ut]u(x_0,t_0),\,\subPt[\Ut]u(x_0,t_0)$ for lower semicontinuous functions are defined analogously.
Additionally, if $\Ut$ is an open neighborhood of $(x_0,t_0)$, then the semijets are the same as for any other open neighborhood and we might omit $\Ut$ in the notation.
\end{definition}

Since we are only interested in either problems on $\Rd$ or with Dirichlet boundary conditions, we only define sub- and supersolutions in open sets. 
For compatibility with more complicated boundary conditions see e.g. \cite[Definition 2.3.1]{Giga06}.

\begin{definition}[Parabolic (Viscosity) Sub- and Supersolutions]
Let $T>0$ and $\Ut\subset\Rd\times[0,T]$ be an open set.
\begin{enumerate}[label=(\roman*), left=0pt]
\item	$u\colon\Ut\rightarrow \R\cup\Bp{-\infty}$ is a subsolution of \eqref{comp_eq_evo} in $\Ut$ if $u^*(x,t)<\infty$ for $(x,t)\in\Ut$ and
		\begin{equation*}
			a + G_*(x,\xi,X) \leq 0
			\quad\text{for all }(x,t)\in \Ut \text{ and } (a,\xi,X)\in \supP[\Ut]u^*(x,t).
		\end{equation*}
\item	We call $u\colon\Ut\rightarrow \R\cup\Bp{-\infty}$ a supersolution of \eqref{comp_eq_evo} in $\Ut$ if $u_*(x,t)>-\infty$ for $(x,t)\in\Ut$ and
		\begin{equation*}
			a + G^*(x,\xi,X) \geq 0
			\quad\text{for all }(x,t)\in \Ut \text{ and } (a,\xi,X)\in \subP[\Ut]u_*(x,t).
		\end{equation*}
\end{enumerate}
\end{definition}

\begin{lemma}[Localization in space]
Let $T>0$ and $U\subset\Rd$ be open.
If $u$ is a subsolution (resp. supersolution) of \eqref{comp_eq_evo} in $U\times(0,T)$, then $u$ is a subsolution (resp. supersolution) of \eqref{comp_eq_evo} in $\Ut$ for any open set $\Ut\subset U\times(0,T)$.
\end{lemma}

\begin{proof}
The lemma follows immediately via the independence of $\supP[\Ut]u^*(t,x)$ from $\Ut$ as long as $\Ut$ is a neighborhood of $(t,x)$.
\end{proof}

\begin{theorem}[Stability, see Theorem 2.2.1 in \cite{Giga06}]\label{comp_thm_stab}
Let $T>0$ and $\Ut\subset \Rd\times(0,T)$ be open.
Let $(G_\eps)_{\eps>0}$ and $G$ satisfy \ref{comp_aG_cont} and \ref{comp_aG_fin} with
\begin{align*}
	G_*&\leq \liminf_{\eps\rightarrow 0}\!{}_*G_{\eps} 
	&&
	\pp{\text{resp. }G^*\geq \limsup_{\eps\rightarrow 0}\!{}^*G_{\eps}}
	&&
	\text{in }\Rd\times\Rd\times \Rddsym
\end{align*}
If $(u_\eps)_{\eps>0}$ are subsolutions (resp. supersolutions) in $\Ut$ of
\begin{equation*}
	\pat u_\eps + G_\eps\pp{x,\nabla u, D^2 u}=0,
\end{equation*}
Then 
	$\ov{u}\coloneqq \limsup_{\eps\rightarrow 0}\!{}^*u_{\eps}$ 
	$\pp{\text{resp. }\underline{u}\coloneqq \liminf_{\eps\rightarrow 0}\!{}_*u_{\eps}}$ 
is a subsolution (resp. supersolution) in $\Ut$ of \eqref{comp_eq_evo}, as long as $\ov{u}<\infty$ (resp. $\underline{u}>-\infty$) in $\Ut$.
\end{theorem}

In general, a comparison principle does not hold for parabolic equations with spatial inhomogeneities in the second order term. 
We do not state the precise general assumptions, since this would take up too much space.
However, in our case a comparison principle holds due to the structure of \eqref{intro_eq_uls} and the assumptions for admissible coefficient fields.

\begin{theorem}[Comparison principle for bounded domains, see Corollary 3.6.6 in \cite{Giga06}]\label{comp_thm_comp}
Let $(A,F)\in\Omega$ or $(A,F)=(\wt{A},\wt{F})_Q$ for some $(\wt{A},\wt{F})\in\Omega$ and a closed set $Q\subset\Rd$.
Let $T>0$ and $U\subset\Rd$ be an open and bounded set.
Let $u$ be a sub- and $v$ be a supersolution of \eqref{intro_eq_uls} in $U\times (0,T)$.
If $-\infty<u^*\leq v_*<\infty$ on the parabolic boundary $\pp{\ov{U}\times\Bp{0}}\cup\pp{\partial U\times[0,T]}$, then $u^*\leq v_*$ in $U\times(0,T)$.
\end{theorem}

For unbounded domains, the requirements are more complicated to check. 

\begin{theorem}[Comparison principle for unbounded domains]\label{comp_thm_comp-unb}
Let $(A,F)\in\Omega$ or $(A,F)=(\wt{A},\wt{F})_Q$ for some $(\wt{A},\wt{F})\in\Omega$ and a closed set $Q\subset\Rd$.
Let $T>0$ and $U\subset\Rd$ be open.
Let $u$ be a sub- and $v$ be a supersolution of \eqref{intro_eq_uls} in $U\times (0,T)$.
Assume that 
\begin{enumerate}[label=(\roman*), left=0pt]
\item	\label{comp_comp1_growth}
		there is some $C>0$ independent of $(x,t)\in U\times(0,T)$ such that
		\begin{equation*}
			u(x,t)\leq C(|x|+1)
			\quad\text{ and }\quad
			v(x,t)\geq -C(|x|+1),
		\end{equation*}
\item	\label{comp_comp2_moddiff}
		there is a modulus $m$ such that 
		\begin{equation*}
			u^*(x,t)-v_*(y,t)	\leq		m(|x-y|)
			\quad\text{ for all }(x,y,t)\in \partial_p \pp{(U\times U)\times(0,T)},
		\end{equation*}
		where $\partial_p \pp{(U\times U)\times(0,T)}$ denotes the parabolic boundary of $(U\times U)\times(0,T)$.
\item	\label{comp_comp3_growthdiff}
		there is some $C>0$ independent of $(x,y,t)\in \partial_p \pp{(U\times U)\times(0,T)}$, such that
		\begin{equation*}
			u^*(x,t)-v_*(y,t)\leq C(|x-y|+1)
			\quad\text{ on }\partial_p \pp{(U\times U)\times(0,T)}.
		\end{equation*}
\end{enumerate}
Then $u^*\leq v_*$ in $U\times(0,T)$.
\end{theorem}

\begin{proof}
See \cite[Theorem 4.2]{GGIS91}.
The comment in \cite{GGIS91} right below Theorem 4.2 implies that $G_{(A,F)}$ satisfies the conditions for applying Theorem 4.2 when ignoring the first order term. 
However, it is easy to see that adding the first order term does not break any of the conditions.
\end{proof}

\begin{remark}\label{comp_rem_condcomp}
Sufficient conditions for Assumptions \ref{comp_comp1_growth}, \ref{comp_comp2_moddiff} and \ref{comp_comp3_growthdiff} in Theorem \ref{comp_thm_comp} to be satisfied are as follows:
\begin{enumerate}[label=\alph*), left=0pt]
\item	\label{comp_cond1_bound}
		If $u$ and $v$ are uniformly bounded, then \ref{comp_comp1_growth} and \ref{comp_comp3_growthdiff} hold.
\item	\label{comp_cond2_uni0}
		If $U=\Rd$ and $u(\cdot,0)$ or $v(\cdot,0)$ is uniformly continuous with $u^*\leq v_*$ on $U\times\Bp{0}$, then \ref{comp_comp2_moddiff} holds choosing $m$ as the modulus of continuity.
\item	\label{comp_cond3_u}
		If $u$ is uniformly continuous with $u\leq v_*$ on $U\times\Bp{0}$ and $u\leq \inf_{U\times(0,T)}v$ on $\partial U\times(0,T)$, then \ref{comp_comp2_moddiff} holds choosing $m$ as the modulus of continuity for $u$.
\item	\label{comp_cond4_v}
		If $v$ is uniformly continuous with $u^*\leq v$ on $U\times\Bp{0}$ and $v\geq \sup_{U\times(0,T)}u$ on $\partial U\times(0,T)$, then \ref{comp_comp2_moddiff} holds choosing $m$ as the modulus of continuity for $v$.
\end{enumerate}
\end{remark} 

A classical ingredient to proving comparison principles is a maximum principle for semicontinuous functions, also known as Ishii's lemma.
In particular, this principle is used whenever the comparison principle itself is not applicable, but the machinery behind it still leads to results.
Here we use a parabolic version.

\begin{theorem}[Theorem 8.3 in \cite{CIL92}]\label{comp_thm_MaxPr}
Let $n\in\N$. Let $T>0$, let $d_j\in\N$ and $U_j\subset \R^{d_j}$ be open sets for $j=1,\ldots,n$, and let $u_j\colon U_j\times(0,T)$ be upper semicontinuous functions.
Let $\vphi\colon U_1\times\ldots\times U_n\times (0,T)$ be once continuously differentiable in $t$ and twice continuously differentiable in $(x_1,\ldots,x_n)\in U_1\times\ldots\times U_n$. 

Assume that $(t,x_1,\ldots,x_n)\mapsto u_1(x_1,t)+\ldots+u_n(x_n,t)-\vphi(x_1,\ldots,x_n,t)$ has a maximum in $(\wt{x},_1,\ldots,\wt{x}_n,\wt{t})$ 
and that 
\begin{equation}\label{comp_eq_condMax}
\begin{aligned}
	&\text{there exists } r>0 \text{ such that for all }L>0 \text{ there is a } C>0 \text{ such that }\\
	&\text{whenever }(a_j,\xi_j,X_j)\in\supJ u_j(x_j,t)
	\text{ for } \abs{x_j-\wt{x}_j}+\abs{t-\wt{t}}\leq r\\
	&\text{with } \abs{u_j(x_j,t)}+\abs{\xi_j}+\abs{X_j}\leq L,
	\text{ then } a_j\leq C.
\end{aligned}
\end{equation}
Then for each $\gamma>0$ there are $b_1,\ldots,b_n\in\R$ and $X_j\in\R^{d_j\times d_j}_{\mathrm{sym}}$, $j=1,\ldots,n$, such that with $Z=D_{x}^2\vphi(\wt{x},_1,\ldots,\wt{x}_n,\wt{t})$ there hold
\begin{equation*}
\begin{aligned}
	(i)\quad&	(b_j,\nabla_{x_j}\vphi(\wt{x},_1,\ldots,\wt{x}_n,\wt{t}), X_j)\in \supPc u(\wt{x}_j,\wt{t}),\\
	(ii)\quad&	-\pp{\frac{1}{\gamma}+\abs{Z}}\Id	
				\leq \begin{pmatrix}X_1 & \cdots & 0 \\ \vdots & \ddots & \vdots \\ 0 & \cdots & X_n\end{pmatrix}
				\leq Z+\gamma Z^2,\\
	(iii)\quad&	b_1+\ldots+b_n = \pat\vphi(\wt{x},_1,\ldots,\wt{x}_n,\wt{t}).
\end{aligned}
\end{equation*}
\end{theorem}

\begin{remark}
\eqref{comp_eq_condMax} holds for example whenever $u_j$ is a subsolution or $-u_j$ is a supersolution of a parabolic equation.
\end{remark}

If $G$ satisfies a parabolic comparison principle, then replacing the parabolic semijets with the relaxed parabolic semijets from Definition \ref{comp_def_pjets} yields an equivalent definition.

\begin{proposition}[Localization in time]\label{comp_prop_local}
Let $T>0$ and $U\subset\Rd$ be an open set.
Let $(A,F)\in\Omega$ or $(A,F)=(\wt{A},\wt{F})_Q$ for some $(\wt{A},\wt{F})\in\Omega$ and a closed set $Q\subset\Rd$.
Then $u$ is a subsolution (resp. supersolution) of \eqref{intro_eq_uls} in $U\times(0,T)$ if and only if
		\begin{align*}
			a + \pp{G_{(A,F)}}_*(x,\xi,X) \leq 0
			\quad&\text{for all }(x,t)\in \Ut \text{ and } (a,\xi,X)\in \supPt[\Ut]u^*(x,t)\\
			\Big( a + \pp{G_{(A,F)}}^*(x,\xi,X) \geq 0
			\quad&\text{for all }(x,t)\in \Ut \text{ and } (a,\xi,X)\in \subPt[\Ut]u_*(x,t)\Big).
		\end{align*}
\end{proposition}

\begin{proof}
See \cite[Theorem 1]{Juu01}, which requires only degenerate ellipticity and that a parabolic comparison principle as in Theorem \ref{comp_thm_comp} holds.
While the result in \cite{Juu01} is stated for a continuous functional without singularities, the proof works in the same way given a singularity using the semi-continuous envelopes from \eqref{comp_eq_envelopes}.
While the relaxed parabolic semijets in \cite{Juu01} are defined via test functions, it is easy to see that -- as for the usual parabolic semijets -- this is equivalent to our formulation, see e.g. \cite[Lemma 3.2.7]{Giga06}.
\end{proof}

For sufficiently nice initial data and some further regularity assumptions on $G$, there exists a unique solution to \eqref{comp_eq_evo}. 
We state the result only for $G_{(A,F)}$ to avoid listing all of the extra assumptions.
The appropriate space given a metric space $M$ is defined by 
\begin{equation*}
	UC_*(M) \coloneqq
	\big\lbrace
		g\in C(M) \,:\, (g\wedge L)\vee (-L) \text{ is uniformly continuous for all }L>0
	\big\rbrace.
\end{equation*}

\begin{theorem}[Existence and uniqueness of solutions]\label{comp_thm_exis}
Let $T>0$ and $(A,F)\in\Omega$ or $(A,F)=(\wt{A},\wt{F})_Q$ for some $(\wt{A},\wt{F})\in\Omega$ and a closed set $Q\subset\Rd$.
Then for every $u_0\in UC_*(\Rd)$ there exists a unique solution $u\in UC_*(\Rd\times[0,T))$ of \eqref{intro_eq_uls} in $\Rd\times(0,T)$ with $u(\cdot, 0)=u_0$.
\end{theorem}

\begin{proof}
With Theorem \ref{comp_thm_comp-unb}, a comparison principle for unbounded domains holds for \eqref{comp_eq_evo}.
With this, the assumptions for \cite[Theorem 4.3.5]{Giga06} hold, which yields the unique existence result.
\end{proof}

\begin{remark}[Suitable initial data given a set as initial data]\label{comp_rem_iniData}
Given a closed set $S$, a suitable function $u_0\in UC_*(\Rd)$ with $S=\Bp{x\in\Rd\,:\,u_0(x)\leq 0}$ is given for example by
\begin{equation*}
	u_0(x)\coloneqq \dist(x,S).
\end{equation*}
\end{remark}

While the existence result requires uniformly continuous initial data, the evolution of a fixed level set is independent of the specific initial data, as long as the level set is preserved.

\begin{theorem}[Invariance, see Theorem 4.2.1 in \cite{Giga06}]\label{comp_thm_inv}
Let $T>0$.
Assume that $G$ satisfies \ref{comp_aG_cont} and \ref{comp_aG_geo}.
Let $\theta\colon\R\rightarrow\R$ be nondecreasing and upper semicontinuous (resp. lower semicontinuous).
Let $u$ be a subsolution (resp. supersolution) of \eqref{comp_eq_evo} in $\Rd\times(0,T)$. 
Then $\theta\circ u$ also is a subsolution (resp. supersolution) of \eqref{comp_eq_evo} in $\Rd\times(0,T)$. 
\end{theorem}

\begin{theorem}[Uniqueness of set evolutions, see Theorem 4.2.11 in \cite{Giga06}]\label{comp_thm_unique}
Let $T>0$ and $(A,F)\in\Omega$ or $(A,F)=(\wt{A},\wt{F})_Q$ for some $(\wt{A},\wt{F})\in\Omega$ and a closed set $Q\subset\Rd$.
Let $u$ and $\wt{u}$ be solutions of \eqref{intro_eq_uls} in $\Rd\times(0,T)$ where
\begin{equation*}
	\Bp{x\in\Rd\,:\, u(x,0)\leq 0} = \Bp{x\in\Rd\,:\, \wt{u}(x,0)\leq 0}.
\end{equation*}
Then for every $t\in(0,T)$ there holds
\begin{equation*}
	\Bp{x\in\Rd\,:\, u(x,t)\leq 0} = \Bp{x\in\Rd\,:\, \wt{u}(x,t)\leq 0}.
\end{equation*}
\end{theorem}

\section{Proofs for the preliminaries}\label{s_pre-proofs}
In this section we collect our proofs for the preliminary properties of the set evolutions, stable sets, restricted coefficient fields and measurability from Section \ref{s_assump} and Section \ref{s_pre}.

\subsection{Proofs for Section \ref{s_assump}}

\begin{proof}[Proof of Lemma \ref{sett_lem_isotropy}]
Due to Assumption \ref{sett_aP_stat}, Assumption \ref{sett_aP_fin}, and Assumption \ref{sett_aP_Aconst} there exists $\wt{A}\in  C^{0,1}(\Sd, \Rddsym)$ satisfying \eqref{sett_eq_assAbd} and \eqref{sett_eq_assAgeom} such that 
\begin{equation*}
	\PM{(A,F)\in\Omega\,:\,A(x,e)=\wt{A}(e) \text{ for all }x\in\Rd, e\in\Sd}=1.
\end{equation*}
Assumption \ref{sett_aP_ani_actual} from Remark \ref{sett_rem_ani_actual} further yields that
\begin{equation}\label{prepro_eq_iso-rot}
	\wt{A}(e) = R^\intercal \wt{A}(Re)R
	\qquad\text{for all $R\in\operatorname{SO}(d)$ and $e\in\Sd$}.
\end{equation}
We claim that for any $e\in\Sd$ there exists $a(e)\geq 0$ such that 
\begin{equation}\label{prepro_eq_iso-claim}
	\wt{A}(e) = a(e)\pp{Id-e\otimes e}.
\end{equation}
Then \eqref{prepro_eq_iso-rot} would immediately imply that $a$ is constant, concluding the proof.

It remains to show \eqref{prepro_eq_iso-claim}.
Let $e_1\in\Sd$.
First note that $\wt{A}(e_1)$ is positive semi-definite and symmetric due to \eqref{sett_eq_assAbd} and hence has an eigendecomposition with only non-negative eigenvalues.
With \eqref{sett_eq_assAbd} we obtain that $\wt{A}(e_1)e_1=0$ and hence that $e_1$ is an eigenvector with eigenvalue $0$.
For $d=2$, this already implies \ref{prepro_eq_iso-claim}.
For $d\geq 3$, we choose eigenvectors $e_2,\ldots, e_d\in\Sd$ with eigenvalues $\lambda_2,\ldots,\lambda_d\geq 0$ such that $e_1,\ldots,e_d$ is an orthonormal basis of $\Rd$.
For $i\in\Bp{3,\ldots,d}$, let $R\in\operatorname{SO}(d)$ be the rotation given by $R e_j=e_j$ for $j\notin \Bp{2,i}$, $Re_2=e_i$ and $Re_i= - e_2$.
Then with \eqref{prepro_eq_iso-rot} we obtain
\begin{equation*}
	\lambda_2 e_2 = A(e_1)e_2 = R^\intercal \wt{A}(Re_1) Re_2 = R^\intercal \wt{A}(e_1) e_i = \lambda_i R^\intercal e_i = \lambda_i e_2.
\end{equation*}
Since $i\in\Bp{3,\ldots,d}$ was arbitrary, this implies $\wt{A}(e_1) = \lambda_2\pp{Id-e_1\otimes e_1}$ and thus establishes the claim \eqref{prepro_eq_iso-claim}.
\end{proof}

\begin{proof}[Proof of Lemma \ref{veff_lem_henv}]
Let $S_3, S_4\subset\Rd$. We claim that
\begin{equation}\label{veff_eq_hsubs-cons}
	\text{if $S_3\hsubs S_4$, then $S_3+\ov{B}_r\subset S_4+\ov{B}_{\max\Bp{r,h}}$ for all $r>0$.}
\end{equation}
Note that for closed sets $S\subset\Rd$ we can rewrite the arrival time from Definition \ref{intro_def_met} as
\begin{equation*}
	\met[(A,F)]{h}{x_0}{S}=\min\Bp{\min\Bp{t\geq 0 \,:\, x_0\in\reach[(A,F)]{t}{S}+\ov{B}_h},\,\timebound{h}{x_0}{S}}.
\end{equation*}
For any $0\leq t\leq \timebound{h}{x_0}{S_2}$ with $x_0\in\reach[(A,F)]{t}{S_1}+\ov{B}_h$, we in particular have $x_0\in\pp{\reach[(A,F)]{t}{S_1}\cap\ov{B}_h(x_0)}+\ov{B}_h $ and thus  $x_0\in\reach[(A,F)]{t}{S_2}+\ov{B}_h$ due to \eqref{veff_eq_reach-hsubs} and \eqref{veff_eq_hsubs-cons}.
This implies 
\begin{multline}\label{veff_eq_hsubs-met-T}
	\met[(A,F)]{h}{x_0}{S_1}	\\
	\geq		\min\Bp{\min\Bp{t\geq 0 \,:\, x_0\in\reach[(A,F)]{t}{S_2}+\ov{B}_h},\,\timebound{h}{x_0}{S_2},\,\timebound{h}{x_0}{S_1}}.
\end{multline}
Note that for closed sets we can rewrite the distance as 
\begin{equation*}
	\dist(x_0,S)=\min\Bp{r\geq 0	\,:\,	x_0\in S+\ov{B}_r}.
\end{equation*}
Thus, if $\dist(x_0,S_1)\leq h$, then \eqref{veff_eq_hsubs-cons} implies that also $\dist(x_0,S_2)\leq h$ and hence $\met[(A,F)]{h}{x_0}{S_2}=0$, which yields $\met[(A,F)]{h}{x_0}{S_1}	\geq \met[(A,F)]{h}{x_0}{S_2}$.
Alternatively, if $\dist(x_0,S_1)> h$, then the above characterization of distance and \eqref{veff_eq_reach-hsubs} in combination with \eqref{veff_eq_hsubs-cons} imply that $\dist(x_0,S_1)\geq \dist(x_0,S_2)$ and hence $\timebound{h}{x_0}{S_1}\geq \timebound{h}{x_0}{S_2}$.
Plugging this into \eqref{veff_eq_hsubs-met-T} yields $\met[(A,F)]{h}{x_0}{S_1}	\geq \met[(A,F)]{h}{x_0}{S_2}$.

It remains to show that \eqref{veff_eq_hsubs-cons} holds.
Let $x\in S_3+\ov{B}_r$ with $y\in S_3$ such that $x-y\in\ov{B}_r$.
Without loss of generality assume that $x\in \comp{(S_4)}$. 
If $x$ is in the same connected component of $\comp{(S_4)}$ as $y$, then $x\in S_4+\ov{B}_h$ by Definition \ref{veff_def_fatstab} since $S_3\hsubs S_4$. 
If $x$ is not in the same connected component, then there exists $z\in S_4$ on the line segment connecting $x$ and $y$. 
In particular, there holds $x-z\in\ov{B}_r$ and hence $x\in S_4+\ov{B}_r$.
This establishes \eqref{veff_eq_hsubs-cons} and thus concludes the proof.
\end{proof}

\subsection{General properties of the set evolutions}

\begin{proof}[Proof of Lemma \ref{not_lem_ur}]
Let $\uls{S}$ be as in Definition \ref{not_def_reach}. 
Then writing $\ur[(A,F)]{S}$ as 
\begin{equation*}
	\ur[(A,F)]{S}(x,t)
	= 1-\charfun[\reach{t}{S}](x) 
	= \pp{\charfun[\Bp{r\in\R\,:\,r>0}]\circ\uls{S}}(x,t)
\end{equation*}
yields the result by Theorem \ref{comp_thm_inv}.
\end{proof}


\begin{proof}[Proof of Lemma \ref{not_lem_vmax}]
Let $x_0\in\Rd$ with  $\dist(x_0,S)>\delta+\ov{T}\vmax[\delta]$, for some $\ov{T}>T$.
We will prove the lemma by showing that $x_0\notin\reach[(A,F)]{T}{S}$.
Let $\ur{S}=1-\charfun[\reach{\cdot}{S}]$ be the supersolution to \eqref{intro_eq_uls} from Lemma \ref{not_lem_ur}.
A comparable subsolution $v$ of \eqref{intro_eq_uls} in $\Rd\times(0,\ov{T})$ is given by a closing hole moving according to the maximum forcing and with a curvature contribution corresponding to the hole of radius $\delta$ reached after time $\ov{T}$, see Figure \ref{not_fig_vmax}. 
\begin{figure}[h]
\centering
\includegraphics[width=0.5\textwidth]{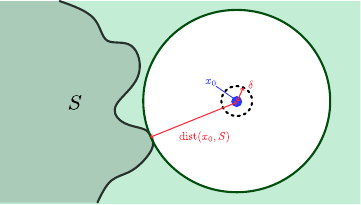}
\caption{	An illustration of the argument bounding the maximum speed of front propagation.
			The evolution of a set $S\subset\Rd$ (depicted in grey) is compared to the evolution of a hole (its exterior is depicted in green) around the target point.
			For the evolution of the hole, the forcing is set to its maximum value.}
\label{not_fig_vmax}
\end{figure}

That is, $v$ is given by
\begin{equation*}
	v(x,t)\coloneqq \charfun[{\pp{\delta+(\ov{T}-t)\vmax[\delta]} B_1(x_0)}](x),
\end{equation*}
where we use that on the one hand $0\leq A \leq C_{1A}$ and $\abs{F}\leq \max\Bp{C_{1F},\,6C_{1A}}$, while on the other hand the interface has `curvature' larger or equal to $\delta^{-1}$ for $t\leq \ov{T}$.
Relaxing the estimate $\abs{F}\leq C_{1F}$ to $\abs{F}\leq \max\Bp{C_{1F},\,6C_{1A}}$ covers the case that $(A,F)$ is the restriction of an admissible coefficient field, see Definition \ref{restr_def}.
Note that $\ur{S}(\cdot,0)\geq v(\cdot,0)$ since $\dist(x_0,S)>\delta+\ov{T}\vmax[\delta]$.
Now applying the comparison principle (see Theorem \ref{comp_thm_comp}) to the supersolution $\ur{S}$ and the subsolution $v$ on the ball $(\delta+\ov{T}\vmax[\delta]) B_1(x_0)$ yields 
	$\reach{t}{S}\subset\comp{\pp{(\delta+(\ov{T}-t)\vmax[\delta]) B_1(x_0)}}$ 
for all $t<\ov{T}$.
\end{proof}

\subsection{Properties of stable sets}

\begin{proof} [Proof of Lemma \ref{stable_lem_preserv}]
Let $T>0$.
We first show the non-shrinking property and then the preservation of stability for any $t< T$. 
Since $T$ is arbitrary, this yields the properties for all $t>0$.
The idea of the proof is to take the union of all sets up to a certain time. 
We show that this is a supersolution, which due to the comparison principle means that original evolution of the set has to be at least as large and thus can not shrink.
\\

\underline{Step 1: Non-shrinking.} 
Let $\uls{S}$ be as in Definition \ref{not_def_reach}, without loss of generality assume $\uls{S} = \min\Bp{1,\uls{S}}$. 
Let $\ur{S}=1-\charfun[\reach{\cdot}{S}]$ be the supersolution to \eqref{intro_eq_uls} from Lemma \ref{not_lem_ur}.
We claim that $\wt{u}$ given by $\wt{u}(x,t)\coloneqq \inf_{s\in[0,t]} \ur{S}(x,s)$ is also a supersolution of \eqref{intro_eq_uls}.
If this was the case, then we can apply the comparison principle (see Theorem \ref{comp_thm_comp-unb}) to $\wt{u}$ and $\uls{S}$ on $\Rd\times(0,T)$ with Remark \ref{comp_rem_condcomp}\ref{comp_cond2_uni0}.
Since by definition $\wt{u}(\cdot,0)\geq \uls{S}(\cdot,0)$, this yields $\wt{u}\geq \uls{S}$.
Hence, 
\begin{equation*}
	\reach{t}{S}=\Bp{x\,:\, \uls{S}(x,t)\leq 0}
	\supset \Bp{x\,:\, \wt{u}(x,t)\leq 0}
	\supset \Bp{x\,:\, \ur{S}(x,t)\leq 0}
	=\reach{t}{S}
\end{equation*}
and therefore $\reach{t}{S} = \Bp{x\,:\, \wt{u}(x,t)\leq 0}$, in particular $\reach[(A,F)]{s}{S}\subset \reach[(A,F)]{t}{S}$ for any $0\leq s \leq t <T$ by the definition of $\wt{u}$.

It remains to show that indeed $\wt{u}$ is a supersolution.
Let $(a,\xi,X)\in\subP \wt{u}(x_0,t_0)$ for some $(x_0,t_0)\in\Rd\times(0,T)$. 
There exists $0\leq t_1\leq t_0$ such that $\wt{u}(x_0,t_0)=\ur{S}(x_0,t_1)$.\\
\underline{Case 1: $t_1>0$.}
Since $\wt{u}$ is non-increasing and by the definition of $\subP \wt{u}(x_0,t_0)$,
\begin{align*}
	\ur{S}(x_0+x,t_1+t)
	&\geq \wt{u}(x_0+x,t_1+t) 
	\geq \wt{u}(x_0+x,t_0+t) \\
	&\geq \ur{S}(x_0,t_1) + at + \xi\cdot x + \frac{1}{2}x^\intercal Xx + \sorder{\abs{t}+\abs{x}^2}
\end{align*}
and therefore $(a,\xi,X)\in\subP \ur{S}(x_0,t_1)$.
Because $\ur{S}$ is a supersolution, we obtain
\begin{equation*}
	a+G_{(A,F)}^*(x_0,\xi,X)\geq 0.
\end{equation*}
\underline{Case 2: $t_1=0$.}
By the definition of $\wt{u}$, we have $\ur{S}(x_0,0)=\wt{u}(x_0,t)$ for all $t\leq t_0$.
On the one hand, by the definition of the parabolic subjet this implies $a\geq 0$.
On the other hand, since $\wt{u}$ is non-increasing,
\begin{align*}
	\ur{S}(x_0+x,0)
	&\geq \wt{u}(x_0+x,t_0)
	\geq \ur{S}(x_0,0) + \xi\cdot x + \frac{1}{2}x^\intercal Xx + \sorder{\abs{x}^2},
\end{align*}
and thus $(\xi,X)\in \subJ \pp{\ur{S}(\cdot,0)}(x_0)$. 
Since $S$ is stable, by definition $\ur{S}(\cdot,0)$ is a supersolution of \eqref{box_eq_stable}, which yields
\begin{equation*}
	0\leq G_{(A,F)}^*(x_0,\xi,X) \leq a+G_{(A,F)}^*(x_0,\xi,X).
\end{equation*}

\underline{Step 2: Preservation of stability.}
We need to show that $\ur{S}(\cdot,t_0)$ is a supersolution of \eqref{box_eq_stable}.
Let $(\xi,X)\in\subJ\pp{u(\cdot,t_0)}(x_0)$ for some $x_0\in\Rd$.
Then $(0,\xi,X)\in\subPt u(x_0,t_0)$, where $\subPt u(x_0,t_0)$ is the relaxed parabolic superjet introduced in Definition \ref{comp_def_pjets}: 
Since $t\mapsto \ur{S}(x,t)$ is non-increasing by Step 1, for $t\leq 0$ we have
\begin{equation*}
	\ur{S}(x_0+x,t_0+t)
	\geq \ur{S}(x_0+x,t_0)
	\geq \ur{S}(x_0,t_0) + \xi\cdot x + \frac{1}{2}x^\intercal Xx + \sorder{\abs{x}^2}.
\end{equation*}
Since $\ur{S}$ is a supersolution (by Proposition \ref{comp_prop_local} also with relaxed superjets), we obtain 
\begin{equation*}
	0+G_{(A,F)}^*(x_0,\xi,X)\geq 0.
\end{equation*}
Hence, $\ur{S}(\cdot,t_0)=1-\charfun[\reach{t_0}{S}]$ is a supersolution to \eqref{box_eq_stable} and thus $\reach{t_0}{S}$ is stable.
\end{proof}


\begin{proof}[Proof of Lemma \ref{not_lem_comp}]
We first prove \ref{not_list_lemcomp1}. 
Let $T>0$.
Let $\ur{S_1}(x,t)= 1-\charfun[\reach{t}{S_1}](x)$ be the supersolution of \eqref{intro_eq_uls} from Lemma \ref{not_lem_ur}.
Let $\uls{S_2}\colon \Rd\times[0,\infty)\rightarrow [0,1]$ be a uniformly continuous solution of \eqref{intro_eq_uls} on $\Rd\times(0,T)$ with $S_2=\Bp{x\in\Rd \,:\, \uls{S_2}(x,0)\leq 0}$ as in Definition \ref{not_def_reach} and hence $\reach{t}{S_2}=\Bp{x\in\Rd \,:\, \uls{S_2}(x,t)\leq 0}$. 
We now apply Theorem \ref{comp_thm_comp-unb}, the comparison principle for unbounded domains, with Remark \ref{comp_rem_condcomp}\ref{comp_cond1_bound} and \ref{comp_cond2_uni0} to $v=\ur{S_1}$ and $u=\uls{S_2}$: 
Since $S_1\subset S_2$, we have $\uls{S_2}\leq \ur{S_1}$ on $\Rd\times\Bp{0}$ and hence the comparison yields $\ur{S_1}\geq \uls{S_2}$ on $\Rd\times(0,T)$.
Since $T>0$ was arbitrary, comparing the 0-sublevel sets implies $\reach{t}{S_1}\subset \reach{t}{S_2}$ for all $t\geq 0$.

Regarding \ref{not_list_lemcomp2}, let $\wt{S}_2$ be $S_2$ with all holes of width smaller than $2h$ removed: 
\begin{equation*}
	\wt{S}_2\coloneqq S_2 \cup \bigcup\Bp{H\,:\,H\text{ is a connected component of }\comp{(S_2)}\text{ with } H\subset S_2+\ov{B}_h}.
\end{equation*}
By the definition of `$\hsubs$', we have $\wt{S}_2\hsubs S_2 \subset \reach{t}{S_2}$ for all $t\geq 0$ using Lemma \ref{stable_lem_preserv} since $S_2$ is stable. 
If we can show that $\reach{t}{S_1}\cap \comp{(\wt{S}_2)}\subset \reach{t}{S_2}$ for all $t\geq 0$, then we are done. 
This follows by applying Theorem \ref{comp_thm_comp-unb} to $\ur{S_1}$ and $\uls{S_2}$ as in the first part of the proof, only now not on $\Rd\times(0,T)$ but on $\comp{(\wt{S}_2)}\times(0,T)$.
Since $S_2$ is stable, Lemma \ref{stable_lem_preserv} yields
\begin{equation*}
	\uls{S_2}=0\leq \inf_{\Rd\times(0,\infty)} \ur{S_1} 
	\quad\text{ on }\quad
	\partial S_2\times[0,T]\supset \partial\comp{(\wt{S}_2)}\times[0,T].
\end{equation*}
Hence with Remark \ref{comp_rem_condcomp}\ref{comp_cond1_bound} and \ref{comp_cond3_u}, the conditions for  Theorem \ref{comp_thm_comp-unb} hold and  we obtain $\ur{S_1}\geq \uls{S_2}$ on $\comp{(\wt{S}_2)}\times(0,T)$ and thus $\reach{t}{S_1}\cap \comp{(\wt{S}_2)}\subset \reach{t}{S_2}$ for all $t\geq 0$, which concludes the proof.
\end{proof}

\begin{proof}[Proof of Lemma \ref{stable_lem_vmin-iterative}]
We prove the Lemma by induction, iteratively applying the effective minimum speed of propagation from Definition \ref{veff_def_fatstab} in combination with the comparison principle from Lemma \ref{not_lem_comp}.
For $n=0$ we clearly have $S+n\ov{B}_h	\hsubs	S_{n}\coloneqq\reach[(A,F)]{n\vmineff^{-1}h}{S}$.
Assume that $S+(n-1)\ov{B}_h	\hsubs	S_{n-1}$  for some $n\in\N$. 
Due to the definition of ``$\hsubs$'', see Definition \ref{veff_def_fatstab}, for each $x\in S+n\ov{B}_h$ we either already have $\Bp{x}\hsubs S_{n-1}$ or there exists $y\in S_{n-1}\cap\pp{S+n\ov{B}_h}$ on the shortest path connecting $x$ to $S+(n-1)\ov{B}_h$, see Figure \ref{stable_fig_growth}.
Either due to the stability of $S$ or due to the effective minimum speed and $x\in \ov{B}_h(y)\subset \pp{S_{n-1}\cap U}+\ov{B}_h$ we hence obtain $\Bp{x}\hsubs	\reach[(A,F)]{\cstable h+n\vmineff^{-1}h}{S}$ and thus $S+n\ov{B}_h	\hsubs	S_n$, concluding the proof.
\begin{figure}[h]
\centering
\includegraphics[width=0.6\textwidth]{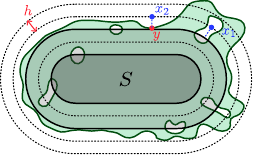}
\caption{	An illustration of a setup for the iterative application of the effective minimum speed.
			The set $S_{n-1}$ (depicted in green) $h$-envelopes $S+(n-1)\ov{B}_h$ (depicted in light grey).
			For each point $x\in S+n\ov{B}_h$, either $\Bp{x}\hsubs S_{n-1}$ or there exists $y\in S_{n-1}$ on the shortest path connecting $x$ to $S+(n-1)\ov{B}_h$, exemplified for $x_1$ and $x_2$ (both depicted in blue).}
\label{stable_fig_growth}
\end{figure}
\end{proof}


\begin{proof}[Proof of Lemma \ref{stable_lem_TvsEspeed}]
The proof is a rather straightforward application of Lemma \ref{stable_lem_vmin-iterative} using the properties which Assumption \ref{veff_aP_star} guarantees for $(A,F)\in\Espeed{h}(U)$.
Since by assumption $\partial S\subset U$, for $(A,F)\in\Espeed{h}(U)$ there exists a stable $(h,\cstable)$-approximation $S_h$ with $S_h\subset S\hsubs \reach[(A,F)]{\cstable h}{S_h}$, for which $(A,F)$ admits $\vmineff$ as an effective minimum speed of propagation on the scale $h$ in $U$.
Since $S_1\hsubs S_2$ implies $S_1+\ov{B}_h\subset S_2+\ov{B}_h$ as in \eqref{veff_eq_hsubs-cons}, with Lemma \ref{stable_lem_vmin-iterative} we obtain
\begin{equation*}
	S+n\ov{B}_h	\subset \reach[(A,F)]{\cstable h}{S_h}+n\ov{B}_h
				\hsubs \reach[(A,F)]{\cstable h+n\vmineff^{-1}h}{S_h}
\end{equation*}
for $n\in\N$ as long as $\pp{S+n\ov{B}_h}\cap\pp{\partial S\cup \comp{S}}\subset U\cup S$.
In particular, due to the assumption on $U$ and using that the evolution of the stable set $S_h$ is increasing as in Lemma \ref{stable_lem_preserv} to not round down we have
\begin{equation*}
	S+\left\lfloor \frac{\dist(x_0,S)}{h} \right\rfloor \ov{B}_h	
	\hsubs	\reach[(A,F)]{\cstable h+\dist(x_0,S)\vmineff^{-1}}{S_h}
	\subset	\reach[(A,F)]{\cstable h+\dist(x_0,S)\vmineff^{-1}}{S},
\end{equation*}
where the last inclusion is due to $S_h\subset S$ and the comparison principle from Lemma \ref{not_lem_comp}.
This implies 
\begin{equation*}
	x_0\in S+\frac{\dist(x_0,S)}{h}\ov{B}_h	\hsubs \reach[(A,F)]{\cstable h+\dist(x_0,S)\vmineff^{-1}}{S}+\ov{B}_h
\end{equation*}
and hence $\met[(A,F)]{h}{x_0}{S}\leq \cstable h+\dist(x_0,S)\vmineff^{-1} = \timebound{h}{x_0}{S}$.
\end{proof}

\subsection{Properties of restricted coefficient fields}

\begin{proof}[Proof of Lemma \ref{restr_lem_PropRestr}]
Without loss of generality, assume that $C=1$ and $u\geq 0$. 
The result for general $u$ then follows since $e^{u-C}\colon \Rd\times[0,\infty)\rightarrow[0,1]$ has the same level sets and is a supersolution due to Theorem \ref{comp_thm_inv}.
For $0<\eps<\frac{1}{2}$ we denote
\begin{equation*}
	\restcomp_\eps\coloneqq \bigcup\Bp{\ov{B}_{\frac{1-\eps}{3}}(x)\,:\,\dist(x,Q)\geq \frac{2}{3}}.
\end{equation*}
Note that $\restcomp = \bigcup_{0<\eps<1}\restcomp_\eps$.
We claim that for all $0<\eps<\frac{1}{2}$
\begin{equation}\label{restr_eq_claim}
	v_\eps(\cdot,t)\coloneqq \charfun[\restcomp_\eps] \text{ is a stationary subsolution of \eqref{intro_eq_uls} with coefficients }(A,F)_Q.
\end{equation} 
Assuming that \eqref{restr_eq_claim} holds, we apply Theorem \ref{comp_thm_comp-unb}, the comparison principle for unbounded domains, to $u$ and $v_\eps$ on $\Rd\times(0,T)$ for any $T>0$. 
The conditions \ref{comp_comp1_growth} and \ref{comp_comp3_growthdiff} from Theorem \ref{comp_thm_comp-unb} are satisfied since $u$ and $v_\eps$ are bounded. 
Regarding condition \ref{comp_comp2_moddiff}, note that whenever $v_\eps(x,0)=1$, then $v_\eps(x,0)=1\leq u(y,0)$ for all $y\in B_\frac{\eps}{6}(x)$, because
\begin{equation*}
	\dist\pp{\Bp{x\,:\,v_\eps(x,0)=1}, \Bp{x\,:\,u(x,0)\neq 1}}
	\geq		\dist\pp{\restcomp_\eps, \comp{\pp{\restcomp}}}
	\geq		\frac{\eps}{3}.
\end{equation*}
Hence, Theorem \ref{comp_thm_comp-unb} yields $v_\eps\leq u$ on $\Rd\times[0,\infty)$ for any $0<\eps<\frac{1}{2}$ and therefore 
\begin{equation*}
	\emptyset	=	\Bp{x\,:\, u(\cdot,t)<1} \cap \bigcup_{0<\eps<1} \restcomp_\eps
				=	\Bp{x\,:\, u(\cdot,t)<C} \cap \restcomp.
\end{equation*}
Now, with respect to $u$ being a supersolution to \eqref{intro_eq_uls} with coefficients $(A,F)_{\wt{Q}}$, by the definition of the restricted coefficient fields we have $G_{(A,F)_Q}\leq G_{(A,F)_{\wt{Q}}}$ on $Q+\ov{B}_{\frac{2}{3}}$.
Hence, for $t_0>0$ and $x_0\in Q+\ov{B}_{\frac{2}{3}}$ and $(a,\xi,X)\in\subJ \ur[(A,F)_Q]{S}(x_0,t_0)$, we obtain 
\begin{equation*}
	a+G_{(A,F)_{\wt{Q}}}^*(x_0,\xi,X)
	\geq 	a+G_{(A,F)_Q}^*(x_0,\xi,X)
	\geq		0
\end{equation*}
because $u$ is a supersolution to \eqref{intro_eq_uls} with coefficients $(A,F)_{Q}$.
For $t_0>0$ and $x_0\notin Q+\ov{B}_{\frac{2}{3}}$, we know that $u=C$ in a neighborhood of $(x_0,t_0)$ since $\Bp{x\,:\, u(\cdot,t)<C}\subset\comp{\pp{\restcomp}}\subset Q+\ov{B}_{\frac{2}{3}}$ for all $t>0$.
Hence, if $(a,\xi,X)\in\subJ \ur[(A,F)_Q]{S}(x_0,t_0)$, then $(a,\xi,X)=(0,0,X)$ with $X\leq 0$ and thus 
	$a+G_{(A,F)_{\wt{Q}}}^*(x_0,\xi,X)\geq 0$.

Thus, we have shown that $u$ is a supersolution to \eqref{intro_eq_uls} with coefficients $(A,F)_{\wt{Q}}$.
We obtain \ref{restr_prop1_restr} and \ref{restr_prop3_super} by choosing $u=\ur[(A,F)_Q]{S}$ as in Lemma \ref{not_lem_ur}.
We obtain \ref{restr_prop2_stable} by plugging in the stationary supersolution $u(x,t)=1-\charfun[S](x)$.

\underline{It remains to prove Claim \eqref{restr_eq_claim}:}
For $(x_0,t_0)\notin \partial \restcomp_\eps\times(0,\infty)$, there is a neighborhood of $(x_0,t_0)$ on which $v_\eps$ is constant.
Hence, if $(a,\xi,X)\in\supP v_\eps(x_0,t_0)$, then we have $(a,\xi,X)=(0,0,X)$ with $X\geq 0$ and thus 
	$a+\big(G_{(A,F)_{Q}}\big)_*(x_0,\xi,X)\leq 0$.
	
For $(x_0,t_0)\in \partial \restcomp_\eps\times(0,\infty)$, we know that $v_\eps(x_0,t)=1$, 
	$\frac{1}{3}\leq \dist(x_0,Q)\leq \frac{2}{3}$ 
and that there exists $x_1$ such that $x_0\in\partial B_\frac{1-\eps}{3}(x_1)$ and $v_\eps=1$ on $\ov{B}_\frac{1-\eps}{3}(x_1)$.
Hence, a superjet $(a,\xi,X)\in\supP v_\eps(x_0,t_0)$ has to satisfy $a=0$.
If $\xi=0$, 
then by approaching $x_0$ from within $B_\frac{1}{3}(x_1)$ we obtain $X\geq 0$ and thus 
	$a+\big(G_{(A,F)_{Q}}\big)_*(x_0,\xi,X)\leq 0$,
exactly as for $(x_0,t_0)\notin \partial \restcomp\times(0,\infty)$.

If $\xi\neq 0$, we know that $\xi\cdot\nu\geq 0$, where $\nu\coloneqq \frac{3}{1-\eps}\pp{x_1-x_0}$.
Let $\nu^\perp$ be a normal orthogonal component of $\nu$.
We approach $x_0$ along $\partial B_\frac{1-\eps}{3}(x_1)$ with 
\begin{equation*}
	s\mapsto 		x_0+s\nu^\perp + f(s)\nu,
	\quad\text{ where }\quad
	f(s)\coloneqq	\frac{1-\eps}{3}\pp{1-\sqrt{1-\pp{\frac{3}{1-\eps}s}^2}}.
\end{equation*}
Since $v_\eps$ is constant on the ball, by the definition of $\supP v_\eps(x_0,t_0)$, we have
\begin{equation*}
	0\leq \lim_{s\rightarrow 0}\frac{1}{s^2}\bigg(
	\xi\cdot\pp{s\nu^\perp + f(s)\nu}
	+\frac{1}{2} \pp{s\nu^\perp + f(s)\nu}^\intercal X \pp{s\nu^\perp + f(s)\nu}
	\bigg).
\end{equation*}
Note that $f(s)=\frac{3}{2(1-\eps)}s^2+\mathcal{O}(|s|^3)$ for $s\rightarrow 0$.
Comparing the coefficients for the different powers of $s$, this yields $\xi\cdot\nu^\perp = 0$ (thus $\xi=\abs{\xi}\nu$) and further 
\begin{align*}
	0	&\leq		\frac{3}{1-\eps}\xi\cdot\nu + \nu^\perp \cdot X\nu^\perp
		= 		\frac{3}{1-\eps}\abs{\xi} + \nu^\perp \cdot X\nu^\perp.
\end{align*}
Since $\nu^\perp$ was arbitrary and because $\frac{\xi}{\abs{\xi}}=\nu$, we obtain
\begin{equation}\label{rest_eq_Xbound}
	\pp{\Idd-\frac{\xi}{\abs{\xi}}\otimes\frac{\xi}{\abs{\xi}}}X\pp{\Idd-\frac{\xi}{\abs{\xi}}\otimes\frac{\xi}{\abs{\xi}}}
	\geq		-\frac{3}{1-\eps}\abs{\xi}.
\end{equation}
Using that $\xi\neq 0$, that $G_{(A,F)_{\wt{Q}}}$ is strongly geometric due to \eqref{sett_eq_assAgeom} and \cite[Theorem 1.6.12 and Lemma 1.6.2]{Giga06}, we obtain
\begin{align*}
	G_{(A,F)_{\wt{Q}}}^*(x_0,\xi,X)
	&=	G_{(A,F)_{\wt{Q}}}\pp{x_0,\xi,\pp{\Id-\frac{\xi}{\abs{\xi}}\otimes\frac{\xi}{\abs{\xi}}}X\pp{\Id-\frac{\xi}{\abs{\xi}}\otimes\frac{\xi}{\abs{\xi}}}}\\
	&=	-\tr\pp{A(x_0,\nu)\pp{\Id-\frac{\xi}{\abs{\xi}}\otimes\frac{\xi}{\abs{\xi}}}X\pp{\Id-\frac{\xi}{\abs{\xi}}\otimes\frac{\xi}{\abs{\xi}}}}
		+F_Q(x_0,\nu)\abs{\xi}\\
	&\leq \frac{3}{1-\eps}C_{1A}\abs{\xi} - \max\Bp{C_{1F},\,6C_{1A}}\abs{\xi}
	\leq 0
\end{align*}
where in the last step we used \eqref{rest_eq_Xbound} and that $A$ is nonnegative definite due to \eqref{sett_eq_assAbd}, the definition of $F_Q$ from \eqref{box_eq_defFQ} and that $\dist(x_0,Q)\in\bp{\frac{1}{3},\frac{2}{3}}$ as well as $0<\eps<\frac{1}{2}$.
\end{proof}


\begin{proof}[Proof of Lemma \ref{restr_lem_QwtQ}]
\underline{$\reach[(A,F)_{Q}]{t}{S}\subset\reach[(A,F)_{\wt{Q}}]{t}{S}$}:
Let $\wt{u}_S\colon \Rd\times[0,\infty)\rightarrow [0,1]$ be a uniformly continuous solution of \eqref{intro_eq_uls} for $(A,F)_{\wt{Q}}$, with $S=\Bp{x\in\Rd \,:\, \wt{u}_S(x,0)\leq 0}$ as in Definition \ref{not_def_reach}.
From Lemma \ref{restr_lem_PropRestr} we know that 
	$\ur{S}(x,t)= 1-\charfun[{\reach[{(A,F)_{Q}}]{t}{S}}](x)$ 
is a supersolution to \eqref{intro_eq_uls} for coefficients $(A,F)_{\wt{Q}}$.
Since $\ur{S}\geq \wt{u}_S$ on $\Rd\times\Bp{0}$, with Remark \ref{comp_rem_condcomp}\ref{comp_cond1_bound} and \ref{comp_cond2_uni0} we can apply the comparison principle Theorem \ref{comp_thm_comp-unb} and obtain $\ur{S}\geq \wt{u}_S$ on $\Rd\times(0,\infty)$.
Comparing the $0$-sublevel sets, we obtain $\reach[(A,F)_{Q}]{t}{S}\subset\reach[(A,F)_{\wt{Q}}]{t}{S}$.

\underline{$\reach[(A,F)_{Q}]{t}{S}\supset\reach[(A,F)_{\wt{Q}}]{t}{S}$}:
Let $\uls{S}$ be be a uniformly continuous solution of \eqref{intro_eq_uls} for coefficients $(A,F)_{Q}$ satisfying $S=\Bp{x\in\Rd \,:\, \uls{S}(x,0)\leq 0}$ as in Definition \ref{not_def_reach}.
Since $\uls{S}$ is uniformly continuous and for all $0\leq t\leq t_0$ there holds
\begin{equation*}
	Q	\supset	\reach[(A,F)_{Q}]{t}{S} + B_\eps
		=		\Bp{x\in\Rd \,:\, \uls{S}(x,t)\leq 0} + B_\eps,
\end{equation*}
there exists $C>0$, such that $\Bp{x\in\Rd \,:\, \uls{S}(x,t)\leq C}\subset Q$ for all $0\leq t\leq t_0$.
Without loss of generality, assume that $C=1$ and $0\leq \uls{S}\leq 1$ (otherwise use Theorem \ref{comp_thm_inv} to continue with $\pp{0\vee C^{-1}\uls{S}}\wedge 1$ instead).

We claim that $\uls{S}$ is a subsolution to \eqref{intro_eq_uls} not just for the coefficients $(A,F)_{Q}$, but also $(A,F)_{\wt{Q}}$:
For $(x_0,t_0)\in Q\times(0,\infty)$, if $(a,\xi,X)\in\supP \uls{S}(x_0,t_0)$, we have 
\begin{equation*}
	a+\big(G_{(A,F)_{\wt{Q}}}\big)_*(x_0,\xi,X) = a+\big(G_{(A,F)_{Q}}\big)_*(x_0,\xi,X) \leq 0
\end{equation*}
since $\uls{S}$ is a solution of \eqref{intro_eq_uls} for coefficients $(A,F)_Q$ and $(A,F)_{\wt{Q}}=(A,F)_{Q}$ on $\Sd\times Q$.
For $(x_0,t_0)\in \comp{Q}\times(0,\infty)$, there is a neighborhood of $(x_0,t_0)$ on which $\uls{S}=1$. 
Hence, if $(a,\xi,X)\in\supP \uls{S}(x_0,t_0)$, then we have $(a,\xi,X)=(0,0,X)$ with $X\geq 0$ and thus in this case we also have
	$a+\big(G_{(A,F)_{\wt{Q}}}\big)_*(x_0,\xi,X)\leq 0$.

Let $\tilde{u}^\mathcal{R}_S= 1-\charfun[{\reach[{(A,F)_{\wt{Q}}}]{t}{S}}](x)$ be the supersolution of \eqref{intro_eq_uls} for coefficients $(A,F)_{\wt{Q}}$ from Lemma \ref{not_lem_ur}.
Since $\tilde{u}^\mathcal{R}_S\geq \uls{S}$ on $\Rd\times\Bp{0}$, via the comparison principle -- Theorem \ref{comp_thm_comp-unb} and Remark \ref{comp_rem_condcomp}\ref{comp_cond1_bound} and \ref{comp_cond2_uni0} -- we obtain $\tilde{u}^\mathcal{R}_S\geq \uls{S}$ on $\Rd\times(0,\infty)$ and hence finish the proof by comparing the $0$-sublevel sets.
\end{proof}

\subsection{Measurability criterion}
\begin{proof}[Proof of Lemma \ref{meas_lem_closure}]
Note that $\Omega$ with the topology induced by $\|\cdot\|_{C_w(U\times\Sd,\Rdd\times\R)}$ is separable due to the uniform Lipschitz bounds from \eqref{sett_eq_assAbd} and \eqref{sett_eq_assF}.
It is straightforward to show that any subset of a separable metric space is again separable: For each element of the original dense set take a sequence inside the subset with the distance to the point converging to the distance between the point and the subset, which combined yields a dense countable set in the subset.
In particular, $E$ is separable with respect to the topology induced by $\|\cdot\|_{C_w(U\times\Sd,\Rdd\times\R)}$.
For brevity, in this proof we write $C_w(U\times\Sd)$ for $C_w(U\times\Sd,\Rdd\times\R)$.
Since $E\subset\Omega$ is closed with respect to $\|\cdot\|_{C_w(U\times\Sone)}$, there exist $\pp{(A_k,F_k)}_k\in\N\subset E$ such that 
\begin{align*}
	E&=\overline{\bigcup_{k\in\N} \Bp{(A_k,F_k)}}{}^{C_w(U\times\Sone)}\cap\Omega
	= \bigcap_{\ell\in\N}\pp{	
			\bigcup_{k\in\N} \ov{B}_{\frac{1}{\ell}}^{C_w(U\times\Sone)}\pp{(A_k,F_k)}\cap\Omega
								}.
\end{align*}
This shows that $E$ is measurable with respect to $\F(U)$, since 
	$\ov{B}_{\delta}^{C_w(U\times\Sone)}\pp{(\wt{A},\wt{F})}\cap\Omega$ 
is $\F(U)$-measurable for any $\delta>0$ and $(\wt{A},\wt{F})\in\Omega$, which can be seen as follows.
Let $\pp{(x_k,e_k)}_{k\in\N}$ be dense in $U\times\Sone$, again using that subsets of separable space are separable.
Using continuity, we have
\begin{align*}
	&\ov{B}_{\delta}^{C_w(U\times\Sone)}\pp{(\wt{A},\wt{F})}\cap\Omega\\
	&\quad	=	\Big\lbrace(A,F)\in\Omega\,:\,\abs{(A,F)(x,e)-(\wt{A},\wt{F})(x,e)}\leq \delta\pp{1+\abs{x}^2}\\
	&\qquad\qquad\qquad\qquad\qquad\qquad\qquad\qquad\qquad\qquad\quad
				\text{ for all $(x,e)\in U\times\Sone$}\Big\rbrace\\
	&\quad	=	\bigcap_{k\in\N} \Bp{(A,F)\in\Omega\,:\,	\abs{(A,F)(x_k,e_k)-(\wt{A},\wt{F})(x_k,e_k)}\leq \delta\pp{1+\abs{x_k}^2}},
\end{align*}
which clearly is measurable with the definition of $\F(U)$ in \eqref{sett_eq_sigmaF}.
\end{proof}

\section{Azuma's inequality -- a generalization and an alternative}\label{s_azuma}
In Subsection \ref{subs_azuma-alt}, we introduce and prove our alternative for Azuma's inequality, which we cited in Proposition \ref{str_prop_azuma} and which provides fluctuation bounds for martingales if the increments based on uniform bounds for the increments.
In contrast, our results only require uniform bounds on a set of very large probability.
Compared to Azuma's inequality, our alternative only yields exponential instead of Gaussian tail bounds.

While it is possible to generalize Azuma's inequality to martingales with only likely bounded increments, this does not provide meaningful Gaussian fluctuation bounds in our setting -- at least if it was possible to obtain such bounds, a different strategy would be required.
In order to illustrate this, we provide a concentration inequality with Gaussian tail bounds in Subsection \ref{subs_azuma-gen}, generalizing Azuma's inequality.
We further provide the probability bounds which we would obtain in the setting of Proposition \ref{fluct_prop}.
However, in this setting the probability of the uniform bounds holding is not strong enough to recover good fluctuation bounds in the same regime as with the alternative from Subsection \ref{subs_azuma-alt}. 

\subsection{An alternative to Azuma's inequality}\label{subs_azuma-alt}
\begin{lemma}\label{azuma_lem_charfun}
Let $(\Omega,\G,\Pm)$ be a probability space. 
Let $(\G_n)_{n\in\Nz}\subset\G$ be a filtration.
Then for any $E\in\G$ and $\delta>0$ there holds
\begin{equation*}
	\PM{\Bp{\omega\in\Omega\,:\, \exists n\in\Nz:\,\EV[\G_n]{\charfun[E]}(\omega)\geq \delta}}	\leq		\delta^{-1}\PM{E}.
\end{equation*}
\end{lemma}

\begin{proof}
For $n\in\Nz$ we set
\begin{align*}
		E_n			&\coloneqq	\Bp{\omega\in\Omega	\,:\,	 \EV[\G_n]{\charfun[E]}(\omega)\geq \delta}\in\G_n,
	&	\wt{E}_n		&\coloneqq	E_n\setminus \bigcup_{i=0}^{n-1}	E_i\in\G_n.
\end{align*}
Due to the measurability of these events, we obtain
\begin{multline*}
	\PM{E}	=	\EV{\charfun[E]}	\geq		\sum_n	\EV{\charfun[E]\charfun[\wt{E}_n]}
								=		\sum_n	\EV{\EV[\G_n]{\charfun[E]}\charfun[\wt{E}_n]}\\
								\geq		\sum_{n}	\delta	\PM{\wt{E}_n}
								=		\delta	\PM{\bigcup_{n}\wt{E}_n}.
\end{multline*}
Since 
	$\PM{\Bp{\omega\in\Omega\,:\, \exists n\in\Nz:\,\EV[\G_n]{\charfun[E]}(\omega)\geq \delta}}	=	\PM{\bigcup_{n}E_n}$,
this concludes the proof.
\end{proof}

We now introduce our alternative to Azuma's inequality. 
Assuming that we would have uniform bounds on the martingale increments this result would be strictly weaker than the original cited in Proposition \ref{str_prop_azuma}: 
While we obtain the same typical size for the fluctuations, the stochastic integrability is worse since we only have exponential instead of Gaussian tail bounds.
However, we can afford the probability of the set on which we do not have uniform bounds to be much larger compared to the straightforward generalization of Azuma's inequality, which we present below in Proposition \ref{azuma_prop_azuma}.

\begin{proposition}[An alternative to Azuma's inequality]\label{azuma_prop_alternative}
Let $\pp{\MA_n}_{n\in\N}$ be a martingale with respect to a filtration $\pp{\G_n}_{n\in\Nz}$ on a probability space $\pp{\Omega,\G,\Pm}$.
Suppose that $\MA_0=0$, $\abs{\MA_n}\leq T$  for some $T\geq 1$ and all $n\in\Nz$.
Set $\delM_n\coloneqq \MA_n-\MA_{n-1}$ for $n\in\N$ and let $E\in\G$ be a set such that $\abs{\delM_n}(\omega)\leq 1$ for all $\omega\in E$.
Then, for all $k\in\Nz$ and $N\in\Nz$ there holds
\begin{align}
	\label{azuma_eq_pre-alternative}
	\EV{\charfun[E]\MA_N^{2k}}	&\leq \pp{2k}!N^k \PM{E}+ C(T)\pp{2k}!N^k T^{2k}\PM{\comp{E}}.\\
\intertext{with $C(T)\coloneqq\frac{2T^2}{T^2+2-e}$ and, as a consequence,}
	\label{azuma_eq_alternative}
	\EV{\MA_N^{2k}}				&\leq \pp{2k}!N^k \PM{E}+ \pp{1+C(T)\pp{2k}!N^k}T^{2k}\PM{\comp{E}}.
\end{align}
In particular, there exists $C>0$ such that for all $\lambda>0$ this implies
\begin{equation}\label{azuma_eq_alt-PM}
	\PM{\abs{\MA_N}\geq \lambda}	
	\leq		\left\lbrace\begin{aligned}
					&C\exp\pp{-\frac{\lambda}{2\sqrt{N}}}		&&\text{if }\lambda\leq	\frac{-\log\pp{\PM{\comp{E}}}}{2\log(T)}\sqrt{N},	\\
					&C\PM{\comp{E}}^{\frac{1}{2\log(T)}}	&&\text{if }\lambda\geq	\frac{-\log\pp{\PM{\comp{E}}}}{2\log(T)}\sqrt{N}.
			\end{aligned}\right.
\end{equation}
\end{proposition}

\begin{proof}
We prove the proposition by induction. 
For both $k=0$ with arbitrary $N\in\Nz$ and $N\in\Bp{0,1}$ with arbitrary $k\in\Nz$ clearly \eqref{azuma_eq_pre-alternative} holds.
Now fix $(k,N)\in\N\times\N$, $N\geq 2$ and assume that \eqref{azuma_eq_pre-alternative} holds for all $(\ell,n)\in\Nz\times\Nz$ with either $\ell<k$, $n\leq N$ or $n<N$, $\ell\leq k$.
Showing that then \eqref{azuma_eq_pre-alternative} holds for $(k,N)$ yields the result.

Using that
	$\EV{\delM_N \MA_{N-1}^{2k-1}}=\EV{\EV[\G_{N-1}]{\delM_N \MA_{N-1}^{2k-1}}}=\EV{\EV[\G_{N-1}]{\delM_N} \MA_{N-1}^{2k-1}}$ 
and $\EV[\G_{N-1}]{\delM_N}=0$ hold since $\MA_n$ is a martingale , we obtain
\begin{align*}
	\EV{\charfun[E]\MA_N^{2k}}
	&=		\EV{\charfun[E]\pp{\MA_{N-1}+\delM_N}^{2k}}\\
	&=		\EV{\charfun[E]\MA_{N-1}^{2k}}
			-2k \EV{\charfun[\comp{E}]\delM_N \MA_{N-1}^{2k-1}}
			+\sum_{\ell=0}^{2k-2} \binom{2k}{\ell}\EV{\charfun[E]\delM_{N}^{2k-\ell}\MA_{N-1}^\ell}\\
	&\leq	\EV{\charfun[E]\MA_{N-1}^{2k}}
			+4k T^{2k}\PM{\comp{E}}
			+\sum_{\ell=0}^{k-1} \binom{2k}{2\ell}\EV{\charfun[E]\MA_{N-1}^{2\ell}}\\
	&\quad	+\sum_{\ell=0}^{k-2} \binom{2k}{2\ell+1}\EV{\charfun[E](N-1)\MA_{N-1}^{2\ell}},
\end{align*}
where in the last step we used $\abs{\delM_N}\leq 2T$, $\charfun[E]\abs{\delM_N}\leq 1$ and $\charfun[E]\abs{\MA_{N-1}}\leq N-1$.
Applying \eqref{azuma_eq_pre-alternative}, which holds due to the induction assumption, now yields
\begin{align*}
	\EV{\charfun[E]\MA_N^{2k}}
	&\leq	\pp{2k}!N^k\PM{E} \bigg(\pp{\frac{N-1}{N}}^k 
								+ \frac{1}{N}\bigg(	\sum_{\ell=0}^{k-1}\frac{1}{\pp{2k-2l}!}\frac{(N-1)^\ell}{N^{k-1}}\\
	&\qquad\qquad\qquad\qquad\qquad\quad				+\sum_{\ell=0}^{k-2}\frac{1}{\pp{2k-2l-1}!\pp{2l+1}}\frac{(N-1)^{\ell+1}}{N^{k-1}}\bigg)
								\bigg)\\
	&\quad	+C(T)\pp{2k}!N^k T^{2k} \PM{\comp{E}}\bigg(\pp{\frac{N-1}{N}}^k 
								+ \frac{1}{N}\bigg(\frac{4k}{C(T)\pp{2k}!N^{k-1}}\\
	&\qquad\qquad\qquad\qquad\qquad\qquad\quad		+\sum_{\ell=0}^{k-1}\frac{1}{\pp{2k-2l}!}\frac{(N-1)^\ell}{N^{k-1}}T^{2(\ell-k)}\\
	&\qquad\qquad\qquad\qquad						+\sum_{\ell=0}^{k-2}\frac{1}{\pp{2k-2l-1}!\pp{2l+1}}\frac{(N-1)^{\ell+1}}{N^{k-1}}T^{2(\ell-k)}\bigg)
								\bigg)
\end{align*}
Factoring out and using that $\frac{(N-1)^\ell}{N^{k-1}}\leq \pp{\frac{N-1}{N}}^{k-1}$ since $N\geq 2$, we have 
\begin{align*}
	\EV{\charfun[E]\MA_N^{2k}}
	&\leq	\pp{2k}!N^k\PM{E}\pp{\frac{N-1}{N}}^{k-1}\pp{1-\frac{1}{N}+\frac{1}{N}\pp{\mathrm{I}_k+\mathrm{II}_k}}\\
	&\quad	+C(T)\pp{2k}!N^k T^{2k}\PM{\comp{E}}\pp{\frac{N-1}{N}}^{k-1} \\
	&\qquad\qquad\qquad\qquad\qquad		\times\pp{1-\frac{1}{N}+\frac{1}{N}\pp{\mathrm{III}_{k}+\mathrm{I}_{k,T}+\mathrm{II}_{k,T}}},
\end{align*}
where, adjusting the indices of the sums, 
\begin{align*}
	\mathrm{I}_k			&\coloneqq	\sum_{\ell=1}^k	\frac{1}{\pp{2\ell}!},&
	\mathrm{II}_k		&\coloneqq	\sum_{\ell=2}^k	\frac{1}{\pp{2\ell-1}!\pp{2k-2\ell+1}},\\
	\mathrm{I}_{k,T}		&\coloneqq	\sum_{\ell=1}^k	\frac{1}{\pp{2\ell}!T^{2\ell}}
						\leq 		\frac{1}{T^2}\mathrm{I}_k,&
	\mathrm{II}_{k,T}	&\coloneqq	\sum_{\ell=2}^k	\frac{1}{\pp{2\ell-1}!\pp{2k-2\ell+1}T^{2\ell}}
						\leq 		\frac{1}{T^4}\mathrm{II}_k,\\
	\mathrm{III}_k		&\coloneqq	\frac{4k}{C(T)\pp{2k}!}
						\leq 		\frac{2}{C(T)}.
\end{align*}
Note that $\mathrm{I}_k+\mathrm{II}_k\leq e-2<1$ and hence $\mathrm{III}_k+\mathrm{I}_{k,T}+\mathrm{II}_{k,T}\leq 1$ by definition of $C(T)$.
Plugging this back into the estimate for $\EV{\charfun[E]\MA_N^{2k}}$ yields \eqref{azuma_eq_pre-alternative} and thus finishes the proof by induction.

Regarding \eqref{azuma_eq_alt-PM}, with \eqref{azuma_eq_alternative} we obtain for any $k\in\N$ that
\begin{align*}
	\PM{\abs{\MA_N}\geq \lambda}	
	&\leq		\lambda^{-2k}\EV{\MA_N^{2k}}\\
	&\leq		\lambda^{-2k}(2k)!N^k\pp{1+CT^{2k}\PM{\comp{E}}}.
\end{align*}
Using that $(2k)!\leq C\sqrt{k}\pp{\frac{2k}{e}}^{2k}\leq	C\pp{\frac{2k}{\sqrt{e}}}^{2k}$ for some $C>0$, we obtain
\begin{align*}
	\PM{\abs{\MA_N}\geq \lambda}	
	&\leq		C\exp\pp{-2k\pp{\frac{1}{2}+\log\pp{\frac{\lambda}{\sqrt{N}}}-\log\pp{2k}}}		\pp{1+CT^{2k}\PM{\comp{E}}}.
\end{align*}
Just focusing on the argument of the exponential, we want to choose $k$ as large as possible while the overall sign stays negative, that is 
	$k_\lambda\coloneqq	\left\lfloor \frac{\lambda}{2\sqrt{N}}\right\rfloor$.
However, this only makes sense while this is not counteracted by the last factor stays under control, that is for $T^{2k}\PM{\comp{E}}\leq 1$, which is guaranteed if 
	$k\leq k_0\coloneqq	\left\lfloor \frac{-\log\pp{\PM{\comp{E}}}}{2\log\pp{T}}\right\rfloor$.
We hence obtain \eqref{azuma_eq_alt-PM} by choosing $k=k_\lambda$ if $k_\lambda\leq k_0$ and $k=k_0$ if $k_\lambda\geq k_0$.
\end{proof}

\subsection{A generalization of Azuma's inequality}\label{subs_azuma-gen}
For our generalization of Azuma's inequality, we will use the following lemma.

\begin{lemma}[Hoeffding's lemma for conditional expectations]\label{azuma_lem_hoeff}
Let $X\in\R$ be a random variable with respect to some probability space $(\Omega,\G,\Pm)$ and $\wt{\G}\subset \G$ be a coarser $\sigma$-algebra.
Suppose there are $\wt{\G}$-measurable random-variables $a,b\in\R$ such that $a\leq X\leq b$.
Then for all $s\in\R$ there holds
\begin{equation*}
	\EV[\wt{\G}]{\exp\pp{sX}}	\leq		\exp\pp{s\EV[\wt{\G}]{X}+\frac{s^2\pp{b-a}^2}{8}}.
\end{equation*}
\end{lemma}

\begin{proof}
Analogous to the proof of Hoeffding's lemma for the expected value.
\end{proof}

For the following generalization of Azuma's inequality, note that the limit up to which we obtain the tail bounds is much stricter than in Proposition \ref{azuma_prop_alternative}.
In particular, it is not even necessarily better than the typical size of the fluctuations.

\begin{proposition}[A generalization of Azuma's inequality]\label{azuma_prop_azuma}
Let $\pp{\MA_n}_{n\in\N}$ be a martingale with respect to a filtration $\pp{\G_n}_{n\in\Nz}$ on a probability space $\pp{\Omega,\G,\Pm}$.
Suppose that $\abs{\MA_n}\leq T$  for some $T>0$ and all $n\in\Nz$.
Set $\delM_n\coloneqq \MA_n-\MA_{n-1}$ for $n\in\N$ and let $E\in\G$ be a set such that $\abs{\delM_n}(\omega)\leq c_n$ for all $\omega\in E$, $n\in\N$ and some $\pp{c_n}_n\subset\R_+$.
Then, for all $N\in\N$ and $s>0$ there holds
\begin{equation*}
	\EV{\exp\pp{s\pp{\MA_N-\MA_0}}}\leq	 \exp\pp{s^2\sum_{n=1}^N c_n^2}\pp{1 + N\pp{1+\frac{4TN}{s\sum_{n=1}^N c_n^2}}\PM{\comp{E}}e^{2sT}}
\end{equation*}
and the same for $\EV{\exp\pp{s\pp{\MA_0-\MA_N}}}$.
In particular, there holds
\begin{align}\label{azuma_eq_PM}
	\PM{\abs{\MA_N-\MA_0}\geq \lambda}
	&\leq		4\exp\pp{-\frac{\lambda^2}{4\sum_{n=1}^N c_n^2}}		&&\text{for all }0<\lambda\leq\lambda_0,
\end{align}
where $\lambda_0=-\log\pp{\pp{N+\frac{8TN^2}{\sqrt{\sum_{n=1}^N c_n^2}}}\PM{\comp{E}}}\frac{\sum_{n=1}^N c_n^2}{T}$.
\end{proposition}

\begin{proof}
We provide bounds for the exponential moments of $\MA_N-\MA_0$, the bounds for $\MA_0-\MA_N$ follow analogously.
For any $s\geq 0$, using that $\MA_n$ is $\G_{N-1}$-measurable for all $n\leq N-1$, for any $E_N\in\G_{N-1}$ we obtain
\begin{multline*}
	\EV{\exp\pp{s\pp{\MA_N-\MA_0}}}	=	\EV{(\charfun[\comp{E}]+\charfun[E])(\charfun[\comp{(E_N)}]+\charfun[E_N])\prod_{n=1}^{N}\exp\pp{s\delM_n}}\\
			\leq		\PM{\comp{E}}e^{2sT} + \PM{\comp{(E_N)}}e^{2sT} + \EV{\charfun[E_N]\prod_{n=1}^{N-1}\exp\pp{s\delM_n}\EV[\G_{N-1}]{\charfun[E]\exp\pp{s\delM_N}}}.
\end{multline*}
We now want to apply Hoeffding's Lemma in the form of Lemma \ref{azuma_lem_hoeff}.
Note that by assumption $\abs{\charfun[E]\delM_N}\leq c_N$ and due to the martingale property
\begin{equation}\label{azuma_eq_condE}
	\EV[\G_{N-1}]{\charfun[E]\delM_N}	=	\EV[\G_{N-1}]{\delM_N} - \EV[\G_{N-1}]{\charfun[\comp{E}]\delM_N}
	\leq		2T\EV[\G_{N-1}]{\charfun[\comp{E}]}.
\end{equation}
This kind of leftover term is small with probability proportional to $\PM{\comp{E}}$. 
To control this leftover term, we define suitable sets $E_n\in\G_{n-1}$, namely
\begin{align*}
	E_n\coloneqq \Bp{\omega\in\Omega \text{ such that }  2T\EV[\G_{n-1}]{\charfun[\comp{E}]}(\omega)  \leq \frac{s\sum_{k=1}^N c_k^2}{2N}},
\end{align*}
for whose probability as in the proof of Lemma \ref{azuma_lem_charfun} we have the bound
\begin{align*}
	\PM{\comp{(E_n)}}\leq	\frac{4TN}{s\sum_{n=1}^N c_n^2}\PM{\comp{E}}.
\end{align*}
Now, first applying Hoeffding's Lemma in the form of Lemma \ref{azuma_lem_hoeff} with \eqref{azuma_eq_condE} and then using the definition of $E_N$ yields
\begin{align*}
	\EV{\exp\pp{s\pp{\MA_N-\MA_0}}}
	&	\leq		\pp{\PM{\comp{E}}+\PM{\comp{(E_N)}}}e^{2sT} \\
	&	\quad		+ \EV{\charfun[E_N]\prod_{n=1}^{N-1}\exp\pp{s\delM_n}\exp\pp{{s2T\EV[\G_{N-1}]{\charfun[\comp{E}]}}}}e^{\frac{s^2 c_N^2}{2}}\\
	&	\leq		\pp{1+\frac{4TN}{s\sum_{n=1}^N c_n^2}}\PM{\comp{E}}e^{2sT} \\
	&	\quad		+ \EV{\prod_{n=1}^{N-1}\exp\pp{s\delM_n}}e^{\frac{s^2 c_N^2}{2}+\frac{s^2\sum_{n=1}^N c_n^2}{2N}}.
\end{align*}
Repeating this procedure with $\EV{\prod_{n=1}^{m}\exp\pp{s\delM_n}}$ for all $m\in\Bp{1,\ldots,N-1}$, we obtain
\begin{multline*}
	\EV{\exp\pp{s\pp{\MA_N-\MA_0}}}	\leq		N\pp{1+\frac{4TN}{s\sum_{n=1}^N c_n^2}}\exp\pp{s^2\sum_{n=2}^N c_n^2}\PM{\comp{E}}e^{2sT} \\
											+ \exp\pp{s^2\sum_{n=1}^N c_n^2},
\end{multline*}
which yields the exponential moment bounds.
Regarding \eqref{azuma_eq_PM}, we have
\begin{align*}
	&\PM{\MA_N-\MA_0\geq \lambda}
					\leq		e^{-s\lambda}\EV{\exp\pp{s\pp{\MA_N-\MA_0}}}\\
	&\qquad\qquad	\leq		\exp\pp{-s\lambda+s^2\sum_{n=1}^N c_n^2}\pp{1 + N\pp{1+\frac{4TN}{s\sum_{n=1}^N c_n^2}}\PM{\comp{E}}e^{2sT}}.
\end{align*}
Just focusing on the argument of the first exponential, to minimize it we would like to choose 
	$s=s_\lambda\coloneqq	\frac{\lambda}{2\sum_{n=1}^N c_n^2}$.
However, we still need to make sure that the second term stays under control. 
Since the probability bounds for \eqref{azuma_eq_PM} are only meaningful if they are smaller than $1$ and we are only considering half of the total contribution, we only need to consider $\lambda\geq \sqrt{\sum_{n=1}^N c_n^2}$ because otherwise just the first term breaks these bounds. 
In particular, we would like to choose at least
	$s\geq s_0\coloneqq	\frac{1}{2\sqrt{\sum_{n=1}^N c_n^2}}$.
With this restriction, we obtain for the second term that
\begin{equation*}
	N\pp{1+\frac{4TN}{s\sum_{n=1}^N c_n^2}}\PM{\comp{E}}e^{2sT}
	\leq 	N\pp{1+\frac{8TN}{\sqrt{\sum_{n=1}^N c_n^2}}}\PM{\comp{E}}e^{2sT}
	\leq		1
\end{equation*}
for all 
\begin{equation*}
	s_0 \leq s \leq s_1\coloneqq		\frac{-\log\pp{\pp{N+\frac{8TN^2}{\sqrt{\sum_{n=1}^N c_n^2}}}\PM{\comp{E}}}}{2T}.
\end{equation*}
For $s_\lambda\leq s_1$, we choose $s=s_\lambda$.
This is the case for all
\begin{equation*}
	\lambda\leq \lambda_0\coloneqq-\log\pp{\pp{N+\frac{8TN^2}{\sqrt{\sum_{n=1}^N c_n^2}}}\PM{\comp{E}}}\frac{\sum_{n=1}^N c_n^2}{T}.
\end{equation*}
Note that we can ignore the case that $s_{\lambda_0}=s_1<s_0$ because then the resulting probability bounds are anyways larger than 1.
For $\lambda>\lambda_0$, we would want to choose $s=s_1$, but this is not necessary since in the overall context of this paper the resulting bounds are worse than the ones provided by Proposition \ref{azuma_prop_alternative}.
Repeating the same argument for $\PM{\MA_N-\MA_0\leq -\lambda}$ yields \eqref{azuma_eq_PM}.
\end{proof}

If in the proof of Proposition \ref{fluct_prop} we apply this generalization of Azuma's inequality instead of Proposition \ref{azuma_prop_alternative}, we obtain Gaussian instead of exponential tail bounds for the same fluctuation size up to some lower bound, see the lemma below. 
However, for this lower bound to actually be small we would have to require that $h^{1+2\ratespeed}\gg\dist(x_0,S)$.
This is different to the lower bound resulting from Proposition \ref{azuma_prop_alternative}, where we only need that $h$ is large.

\begin{lemma}[Fluctuation bounds with improved stochastic integrability]
Assume that \ref{sett_aP_fin} and \ref{veff_aP_star} hold.
There are $c,\Cflu,\cflu, \Cfluco=c(\data),\Cflu(\data),\cflu(\data), \Cfluco(\data)>0$ such that 
for all $x_0\in\Rd$, $h\geq h_0$ and $h$-fat sets $S\subset\Rd$ with bounded boundary as well as
\begin{equation*}
	\frac{\diam(\partial S)+\dist(x_0,S)}{h}	\leq		\exp\pp{c h^{\ratespeed}}
\end{equation*}
there holds
\begin{multline*}
	\PM{\abs{\metb[]{h}{x_0}{S}	- \EV{\metb[]{h}{x_0}{S}}}\geq \lambda}\notag\\
	\leq			\Cflu	\exp\pp{-\frac{\cflu^2\lambda^2}{h\dist(x_0,S)}}
				\qquad\text{for all }0<\lambda \leq	\Cfluco h^{1+\ratespeed},
\end{multline*}
where $\ratespeed>0$ is the exponential rate from the probability bounds in Assumption \ref{veff_aP_star} for the existence of stable approximation and the applicability of the effective minimum speed of propagation.
\end{lemma}

\begin{proof}
To obtain these fluctuation bounds, we repeat the argument from the proof of Proposition \ref{fluct_prop} with Proposition \ref{azuma_prop_azuma} instead of Proposition \ref{azuma_prop_alternative}.
That is, we apply our generalization of Azuma's inequality to the martingale $(\wt{\M}_n)_n$ from the proof of Proposition \ref{fluct_prop} with
\begin{equation}
	\tag{\ref{fluct_eq_MN}}
 	\wt{\M}_N=	\metb[]{h}{x_0}{S}	- \EV{\metb[]{h}{x_0}{S}}+	R_1 + R_2
 \end{equation} 
for $N\leq C h^{-1}\dist(x_0,S)$ and 
\begin{align*}
	\abs{R_1}	&\leq	\timebound{h}{x_0}{S}\EV[\G_{\timebound{h}{x_0}{S}}]{1-\charfun[\Espeed{h}\pp{\partial S}]},\\
	\abs{R_2}	&\leq	\timebound{h}{x_0}{S}\pp{1-\charfun[\Espeed{h}\pp{\partial S}]}.
\end{align*}
From the proof of Proposition \ref{fluct_prop} we further have the event $E_\M\in\F(\Rd)$ such that 
\begin{equation*}
	\abs{\wt{\M}_{n+1}-\wt{\M}_n}\charfun[E_{\M}]	\leq  Ch
	\quad\text{for all }n\in\N.
\end{equation*}
For the failure of this event we have the stretched exponential bounds 
\begin{equation*}
	1-\PM{E_{\M}}\leq C\exp\pp{-\frac{\cspeed}{2} h^{\ratespeed}}.
\end{equation*}
Using \eqref{fluct_eq_MN} and the associated error bounds right below to then apply Proposition \ref{azuma_prop_azuma} yields
\begin{align*}
	&\EV{\exp\pp{s\pp{\metb[]{h}{x_0}{S}	- \EV{\metb[]{h}{x_0}{S}}}}}	\\
	&\qquad\qquad	\leq		\EV{\exp\pp{s\pp{\wt{M}_N-R_1}}\charfun[E_{\M}]} + \exp\pp{s\timebound{h}{x_0}{S}}\pp{1-\PM{E_{\M}}}\\
	&\qquad\qquad	\leq		\exp\pp{s  h}\EV{\exp\pp{s\pp{\wt{M}_N}}\charfun[E_{\M}]} + \exp\pp{s\timebound{h}{x_0}{S}}\pp{1-\PM{E_{\M}}}\\
	&\qquad\qquad	\leq		\exp\pp{Cs^2Nh^2 +  s h }\\
	&\qquad\qquad			\qquad\qquad\times\pp{1+C\pp{N+\frac{\timebound{h}{x_0}{S}N}{sh^2}}\pp{1-\PM{E_{\M}}}e^{2s\timebound{h}{x_0}{S}}}.
\end{align*}
This is essentially the same moment bound as in Proposition \ref{azuma_prop_azuma}.
Following the proof of \eqref{azuma_eq_PM} from Proposition \ref{azuma_prop_azuma} -- except for using that $e^{-s\lambda+sh}\leq e^{-\frac{1}{2}s\lambda}$ for $\lambda \geq 2h$ to deal with the additional term $e^{sh}$ -- we obtain 
\begin{align*}
	\PM{\abs{\metb[]{h}{x_0}{S}	- \EV{\metb[]{h}{x_0}{S}}}\geq \lambda}
	\leq		4\exp\pp{-\frac{\lambda^2}{CN h^2}}		&&\text{for all }0<\lambda\leq\lambda_0,
\end{align*}
with $\lambda_0=-\log\pp{\frac{CTN\sqrt{N}}{h}\pp{1-\PM{E_{\M}}}}\frac{Nh^2}{T}$.
Plugging in $\timebound{h}{x_0}{S}\approx\dist\pp{x_0,S}$ and $N\approx\frac{\dist\pp{x_0,S}}{h}$ from \eqref{fluct_eq_defN} as well as using Assumption \eqref{fluct_eq_hrateA} with \eqref{fluct_eq_ProbEM} to obtain $\log\pp{\frac{CTN\sqrt{N}}{h}\pp{1-\PM{E_{\M}}}}\approx -h^{\ratespeed}$ yields the first inequality in \eqref{fluct_eq_PMfluct}.
\end{proof}

\end{appendix}

%


\printbibliography[title={References}]
\end{document}